\documentclass[11pt, a4paper,reqno]{amsart}
\pdfoutput=1
\usepackage{tikz,enumitem,rotating,latexsym,bm,stmaryrd,caption,}
\usetikzlibrary{positioning,intersections,decorations,shadings}
\usetikzlibrary{shapes}
\usetikzlibrary{arrows}
\usetikzlibrary{patterns}
\usepackage{tkz-euclide}
\usetikzlibrary{calc}
\usetikzlibrary{shapes.geometric}
\tikzset{snake it/.style={decorate, decoration=snake}}
\tikzset{zigzag/.style={decorate, decoration=zigzag}}
\usepackage{bbm}
\definecolor{ao(english)}{rgb}{0.0, 0.5, 0.0}
\definecolor{darkgreen}{rgb}{0.0, 0.5, 0.0}
\def\dover#1{\underline{\underline{#1}}}
\newcommand{\uvee}{{\underline{v} }}
\newcommand{\vx}{{\color{violet} \underline{x} }}
\usepackage{comment}
 
\newcommand{\sgnlamu}{{\sf sgn}(\la,\mu)}
\newcommand{\sgnnula}{{\sf sgn}(\nu,\la)}
\newcommand{\sgnalla}{{\sf sgn}(\alpha,\la)} 
\newcommand{\sgnnual}{{\sf sgn}(\nu,\alpha)} 
\newcommand{\sgnalnu}{{\sf sgn}(\alpha,\nu)}
\newcommand{\sgnnumu}{{\sf sgn}(\nu,\mu)}  
\newcommand{\first}{{\sf first}}  
\newcommand{\last}{{\sf last}}

\pgfdeclareverticalshading{rainbow}{100bp}
{color(0bp)=(orange);color(25bp)=(orange);
color(50bp)=(magenta); color(75bp)=(cyan);
 color(100bp)=(cyan)}

\usetikzlibrary{decorations.pathmorphing}

	\definecolor{eng}{rgb}{0.0, 0.5, 0.0}
\definecolor{apple}{rgb}{0.55, 0.71, 0.0}
\definecolor{cadmium}{rgb}{0.0, 0.42, 0.24}
\definecolor{darkspringgreen}{rgb}{0.09, 0.45, 0.27}
\definecolor{amethyst}{rgb}{0.6, 0.4, 0.8}
\definecolor{ao}{rgb}{0.0, 0.0, 1.0}
\definecolor{atomictangerine}{rgb}{1.0, 0.6, 0.4}
\definecolor{carmine}{rgb}{0.59, 0.0, 0.09}

 \newcommand{\blor}{
 \begin{tikzpicture}
\draw[orange,->](0,0) to(0.2,0);
\draw[blue,->](0,0) to(-0.2,0);
  \end{tikzpicture}}
 \newcommand{\orbl}{
 \begin{tikzpicture}
\draw[blue,->](0,0) to(0.2,0);
\draw[orange,->](0,0) to(-0.2,0);
  \end{tikzpicture}}
\newcommand{\blbl}{
 \begin{tikzpicture}
\draw[blue,->](0,0) to(0.2,0);
\draw[blue,->](0,0) to(-0.2,0);
  \end{tikzpicture}}
\newcommand{\oror}{
 \begin{tikzpicture}
\draw[orange,->](0,0) to(0.2,0);
\draw[orange,->](0,0) to(-0.2,0);
  \end{tikzpicture}}

\definecolor{toggle}{rgb}{1.0, 0.94, 0.96}

\usepackage[normalem]{ulem}
 \newcommand{\plus}{ }
 \newcommand{\pDelta}{\Delta}
 \newcommand{\dom}{{\sf dom}}
  \newcommand{\cod}{{\sf cod}}
  \newcommand{\on}{{\sf 1}}

 \newcommand{\gsigma}{{{{\color{gray}\bm\sigma}}}}

 \newcommand{\brown}{{{{\color{brown}\bm\pi}}}}
 \newcommand{\violet}{{{{\color{violet}\bm\rho}}}}
  \newcommand{\purple}{{{{\color{violet}\bm\rho}}}}
 \newcommand{\green}{{{{\color{green!80!black}\bm\beta}}}}
 \newcommand{\orange}{{{{\color{orange}\bm\gamma}}}}
 \newcommand{\pink}{{{{\color{magenta}\bm\alpha}}}}
 \newcommand{\blue}{{{{\color{cyan}\bm\tau}}}}
 
 \newcommand{\colMAP}{\imath}
 \newcommand{\MAP}{\jmath }
 \newcommand{\TRNC}{\varphi}
 \newcommand{\defect}{{\sf def}}
 \newcommand{\grey}{{{{\color{gray}\bm\alpha}}}}

 \newcommand{\two}{{{{\color{orange}2}}}}
 \newcommand{\zero}{{{{\color{magenta}0}}}}
 \newcommand{\one}{{{{\color{cyan}1}}}}
 \newcommand{\trid}{{\sf trid}}

 \newcommand{\cpi}{{{{\color{violet}\bm\pi}}}}
 \newcommand{\csigma}{{{{\color{magenta}\bm\sigma}}}}
 \newcommand{\ctau}{{{{\color{cyan}\bm\tau}}}}
\newcommand{\csigmaw}{{{{\color{magenta}\bm\sigma}}}}
\newcommand{\ctauw}{{{{\color{cyan}\bm\tau}}}}

\newcommand{\ct}{{\sf ct}}

 \newcommand{\al}{\alpha}
 \newcommand{\crho}{{{{ \color{green!80!black}\bm\rho}}}}
   
\newcommand{\exx}{{b_\al }}
\newcommand{\deltax}{{{\color{magenta}\delta_\alpha}}}
 
\newcommand{\eps}{ \varepsilon}

\newcommand{\gap}{{\sf gap}}
\newcommand{\fork}{{\sf fork}}
\newcommand{\braid}{{\sf braid}}
\newcommand{\spot}{{\sf spot}}

\newcommand{\isit}{{i}}
\newcommand{\isitone}{\al(i+1)}
\newcommand{\Shl}{\widehat{S}_{h\ell}}
\newcommand{\w}{{\underline{w}}}
\newcommand{\vvv}{{\underline{v} }}
\newcommand{\uuu}{{\underline{u} }}
\newcommand{\y}{{\underline{y}}}
\newcommand{\x}{{\underline{x}}}

\newcommand{\dil}{{\sf dil}_\ctau }

\newcommand{\unvvv}{{{z}}}
\newcommand{\grade}{q}
\newcommand{\dgrm}{\mathscr{D}}
\newcommand{\dgrmdeg}{\mathscr{D}^{\rm deg}}
\newcommand{\dgrmf}{\mathscr{D}_{\rm f}}
\newcommand{\dgrmBS}{\mathscr{D}_{\rm BS}}
\newcommand{\dgrmBSdeg}{\mathscr{D}_{\rm BS}^{\rm deg}}
\newcommand{\dgrmBSdegsum}{\mathscr{D}_{\rm BS}^{{\rm deg},\oplus}}
\newcommand{\dgrmBSF}{\mathscr{D}_{\rm BS}^F}
\newcommand{\dgrmF}{\mathscr{D}^F}
\newcommand{\dgrmBSsumshift}{\mathscr{D}_{\rm BS}^{\oplus,(-)}}
\newcommand{\dgrmstd}{\mathscr{D}_{\rm std}}
\newcommand{\dgrmBSstd}{\mathscr{D}_{{\rm BS},{\rm std}}}
\newcommand{\dgrmBSpast}{\mathscr{D}_{{\rm BS},p|\ast}}
\newcommand{\dgrmBSpastsumshift}{\mathscr{D}_{{\rm BS},p|\ast}^{\oplus,\langle - \rangle}}
\newcommand{\dgrmpast}{\mathscr{D}_{p|\ast}}
\newcommand{\eee}{{\sf e}}

\newcommand{\Wf}{W_{\rm f}}
\newcommand{\W}{W}
\newcommand{\Wp}{\W_p}
\newcommand{\Ssf}{S_{{\rm f}}}
\newcommand{\Ss}{S}
\newcommand{\Ssp}{\Ss_p}
\newcommand{\Sspexpr}{\expr{\Ss}_p}
\newcommand{\Sspone}{\Ss_{p \cup 1}}
\newcommand{\sh}{s_{\rm h}}
\newcommand{\linkexpr}{\expr{\Ss}_{p|1}}
\newcommand{\Wpcosets}{\prescript{p}{}{\W}}
\newcommand{\Alc}{\text{\bf Alc}}
 \newcommand{\Pdiptwo}{{\sf M}_{i,i+2}}
 \newcommand{\Pdipj}{{{\sf M}_{i,j}}}

  \newcommand{\level}{{\sf level}}

 \newcommand{\vstwo}{{{{\color{violet}\bm \sigma_2}}}}
 \newcommand{\vsi}{{{{\color{violet}\bm \sigma_i}}}}
  \newcommand{\vsione}{{{{\color{violet}\bm \sigma_{i+1}}}}}
  \newcommand{\vsone}{{{{\color{violet}\bm \sigma_1}}}}
    \newcommand{\vsk}{{{{\color{violet!70!white}\bm \sigma_k}}}}
      \newcommand{\vsj}{{{{\color{violet}\bm \sigma_j}}}}

 \newcommand{\nodelabel}{0}

\def\down{\vee}
\def\up{\wedge}

\usepackage{standalone}

\usepackage{etex}
 
\tikzset{
  variable line width/.style={
    every variable line width/.append style={#1},
    to path={%
      \pgfextra{%
        \draw[every variable line width/.try,line width=\pgfkeysvalueof{/tikz/thickness}] (\tikztostart) -- (\tikztotarget);
      }%
      (\tikztotarget)
    },
  },
  thickness/.initial=0.6pt,
  every variable line width/.style={line cap=round, line join=round},
}

\usepackage{todonotes}
\newcommand{\warning}[1]{\todo[color=red!75]{#1}}

\usepackage{tikz}
\usetikzlibrary{matrix,  intersections, calc, decorations.pathreplacing} 

\usepackage{tikz-cd}

\newlength{\superthick}
\newlength{\cornerradius}
\tikzstyle{corner}=[rounded corners=\cornerradius]
\tikzstyle{dot}=[circle, inner sep=0pt, minimum size=4.8pt]
\tikzstyle{string}=[line width=\superthick]
\tikzstyle{std}=[string,dash pattern=on 0.9pt off 0.9pt]
\definecolor{realcyan}{rgb}{0,1,1}

 \usepackage{amsmath,amsthm,amsfonts,amssymb,mathrsfs,pb-diagram}
\usepackage[
bookmarks=true,colorlinks=true,linktoc=page,citecolor=green!80!black,linkcolor=green!80!black,urlcolor=green!80!black]{hyperref}
\usepackage{caption}
\usepackage{lipsum,wasysym}
\usepackage{mathtools}
\usepackage[a4paper,margin=2.3cm]{geometry}
\usepackage{cleveref}
 \usepackage{amsmath}
\mathchardef\mhyphen="2D
\usepackage{color}
\usepackage{xcolor}
\usepackage{ifthen}
\usepackage{sidecap}   
\definecolor{mediumblue}{rgb}{0.0, 0.0, 0.8}

\newcommand{\Res}{{\rm Res}}
\newcommand{\Rem}{{\rm Rem}}
\newcommand{\Add}{{\rm Add}}
\newcommand{\Ind}{{\rm Ind}}
\newcommand{\hstar}{\mathfrak{h}^*}
\newcommand{\mptn}{{\mathscr{P}_{m,n}}}
\newcommand{\restr}{{\mathscr{R}_{m,n}}}
\newcommand{\mtau}{{\mathscr{P}^\ctau_{(W,P)}}}
\newcommand{\mptnmax}{{\mathscr{P}_n^\circ}}
\newcommand{\mptnl}{{\mathscr{P}^\ell_{(W,P)}}}
\newcommand{\mptntau}{{\mathscr{P}_{{m,n}}^\ctau}}

\renewcommand{\geq}{\geqslant}
\renewcommand{\leq}{\leqslant}
 \newcommand{\Q}{{\mathbb Q}}

\tikzset{wei/.style= 
{red,double=red,double
distance=0.5pt}}

\newcommand{\fA}{\mathfrak{A}}
\newcommand{\fB}{\mathfrak{B}}
\newcommand{\fC}{\mathfrak{C}}
\newcommand{\fD}{\mathfrak{D}}
\newcommand{\fE}{\mathfrak{E}}
\newcommand{\fF}{\mathfrak{F}}
\newcommand{\fG}{\mathfrak{G}}
\newcommand{\fH}{\mathfrak{H}}
\newcommand{\fI}{\mathfrak{I}}
\newcommand{\fJ}{\mathfrak{J}}
\newcommand{\fK}{\mathfrak{K}}
\newcommand{\fL}{\mathfrak{L}}
\newcommand{\fM}{\mathfrak{M}}
\newcommand{\fN}{\mathfrak{N}}
\newcommand{\fO}{\mathfrak{O}}
\newcommand{\fP}{\mathfrak{P}}
\newcommand{\fQ}{\mathfrak{Q}}
\newcommand{\fR}{\mathfrak{R}}
\newcommand{\fS}{S}
\newcommand{\fT}{\mathfrak{T}}
\newcommand{\fU}{\mathfrak{U}}
\newcommand{\fV}{\mathfrak{V}}
\newcommand{\fW}{\mathfrak{W}}
\newcommand{\fX}{\mathfrak{X}}
\newcommand{\fY}{\mathfrak{Y}}
\newcommand{\fZ}{\mathfrak{Z}}
\newcommand{\fa}{\mathfrak{a}}
\newcommand{\fb}{\mathfrak{b}}
\newcommand{\fc}{\mathfrak{c}}
\newcommand{\fd}{\mathfrak{d}}
\newcommand{\fe}{\mathfrak{e}}
\newcommand{\ff}{\mathfrak{f}}
\newcommand{\ffg}{\mathfrak{g}}
\newcommand{\fh}{\mathfrak{h}}
\newcommand{\ffi}{\mathfrak{i}}
\newcommand{\fj}{\mathfrak{j}}
\newcommand{\fk}{\mathfrak{k}}
\newcommand{\fl}{\mathfrak{l}}
\newcommand{\fm}{\mathfrak{m}}
\newcommand{\fn}{\mathfrak{n}}
\newcommand{\fo}{\mathfrak{o}}
\newcommand{\fp}{\mathfrak{p}}
\newcommand{\fq}{\mathfrak{q}}
\newcommand{\fr}{\mathfrak{r}}
\newcommand{\fs}{\mathfrak{s}}
\newcommand{\ft}{\mathfrak{t}}
\newcommand{\fu}{\mathfrak{u}}
\newcommand{\fv}{\mathfrak{v}}
\newcommand{\fw}{\mathfrak{w}}
\newcommand{\fx}{\mathfrak{x}}
\newcommand{\fy}{\mathfrak{y}}
\newcommand{\fz}{\mathfrak{z}}

\newcommand{\sA}{\mathscr{A}}
\newcommand{\sB}{\mathscr{B}}
\newcommand{\sC}{\mathscr{C}}
\newcommand{\sD}{\mathscr{D}}
\newcommand{\sE}{\mathscr{E}}
\newcommand{\sF}{\mathscr{F}}
\newcommand{\sG}{\mathscr{G}}
\newcommand{\sH}{\mathscr{H}}
\newcommand{\sI}{\mathscr{I}}
\newcommand{\sJ}{\mathscr{J}}
\newcommand{\sK}{K}
\newcommand{\sL}{\mathscr{L}}
\newcommand{\sM}{\mathscr{M}}
\newcommand{\sN}{\mathscr{N}}
\newcommand{\sO}{\mathscr{O}}
\newcommand{\sP}{\mathscr{P}}
\newcommand{\sQ}{\mathscr{Q}}
\newcommand{\sR}{\mathscr{R}}
\newcommand{\sS}{\mathscr{S}}
\newcommand{\sT}{\mathscr{T}}
\newcommand{\sU}{\mathscr{U}}
\newcommand{\sV}{\mathscr{V}}
\newcommand{\sW}{\mathscr{W}}
\newcommand{\sX}{\mathscr{X}}
\newcommand{\sY}{\mathscr{Y}}
\newcommand{\sZ}{\mathscr{Z}}

\newcommand{\hd}{\operatorname{hd}}
\newcommand{\cA}{ \mathscr{A}}
\newcommand{\cB}{\mathcal{B}}
\newcommand{\cC}{\mathcal{C}}
\newcommand{\cD}{\mathscr{D}}
\newcommand{\cE}{{\rm M}}
\newcommand{\cF}{f}
\newcommand{\cG}{\mathcal{G}}
\newcommand{\cH}{K_{\infty }  }
\newcommand{\cI}{\mathcal{I}}
\newcommand{\cJ}{\mathcal{J}}
\newcommand{\cK}{K}
\newcommand{\cL}{\mathcal{L}}
\newcommand{\cLR}{\mathcal{LR}}
\newcommand{\cM}{\mathcal{M}}
\newcommand{\cN}{\mathcal{N}}
\newcommand{\cO}{\mathcal{O}}
\newcommand{\cP}{\mathcal{P}}
\newcommand{\cQ}{\mathcal{Q}}
\newcommand{\cR}{\mathcal{R}}
\newcommand{\cS}{\mathcal{S}}
\newcommand{\cT}{\mathcal{T}}
\newcommand{\cU}{\mathcal{U}}
\newcommand{\cV}{\mathcal{V}}
\newcommand{\cW}{\mathcal{W}}
\newcommand{\cX}{\mathcal{X}}
\newcommand{\cY}{\mathcal{Y}}
\newcommand{\cZ}{\mathcal{Z}}

\newcommand{\Z}{\mathbb{Z}}
\newcommand{\R}{\mathbb{R}}
\newcommand{\N}{\mathbb{N}}
\newcommand{\C}{\mathbb{C}}

\DeclareMathOperator{\reg}{reg}
\DeclareMathOperator{\image}{im}

 \newcommand{\alphar}{{{\color{magenta}\boldsymbol \alpha}}}
 \newcommand{\bet}{{{\color{cyan}\boldsymbol \tau}}}
  \newcommand{\gam}{{{{\color{orange!95!black}\boldsymbol\gamma}}}}

 \newcommand{\betar}{{{\color{green!70!black}\boldsymbol \beta}}}

\tikzset{wei2/.style={red,double=red,double
distance=0.5pt}}

\allowdisplaybreaks
\numberwithin{equation}{section}
\usepackage{scalefnt}

\newtheorem{thm}{Theorem}[section]
\newtheorem{cor}[thm]{Corollary}
 \newtheorem{algorithm}{Algorithm}[section]

\newtheorem{conj}[thm]{Conjecture}
\newtheorem{lem}[thm]{Lemma}
\newtheorem{exl}[thm]{Example}
\newtheorem{prop}[thm]{Proposition}

\newtheorem*{prop*}{Proposition}
\newtheorem*{thmA}{Theorem A}

\newtheorem*{thmB}{Theorem B}

\newtheorem*{cor*}{Corollary}

\newtheorem*{conj*}{Conjecture D}

\newtheorem*{conj1*}{Conjecture B}
\newtheorem*{Acknowledgements*}{Acknowledgements}

\theoremstyle{rmk}

\theoremstyle{defn}
\newtheorem{rmk}[thm]{Remark}
\newtheorem{defn}[thm]{Definition}
\newtheorem{eg}[thm]{Example}

\newcommand{\great}{>}
\newcommand{\less}{<}
\newcommand{\greatoreq}{\geq}
\newcommand{\lessoreq}{\leq}
\newcommand{\codeg}{\mathrm{codeg}}
\newcommand{\triv}{\mathrm{triv}}
\newcommand{\id}{\mathrm{id}}

\newcommand{\TL}{\mathrm{TL}}
\newcommand{\rad}{\mathrm{rad}}
\newcommand{\res}{\mathrm{res}}
\newcommand{\ik}{{k}}
\newcommand{\Std}{{\rm Std}}
\newcommand{\SStd}{{\rm DStd}}
\renewcommand{\det}{{\rm det}}
\newcommand{\epsilonLIRONdontchange}{\epsilon}
\newcommand{\TSStd}{\operatorname{\mathcal{T}}}
\newcommand{\Shape}{\operatorname{Shape}} 
\newcommand{\Path}{{\rm Path}}
\newcommand{\CPath}{{\rm Path}_{\underline{w}}}
\newcommand{\la}{\lambda}
\newcommand{\I}{i}
\newcommand{\J}{j}
\newcommand{\K}{k}
\newcommand{\M}{m}
\renewcommand{\L}{l}
\def\Ca{\mathcal C}
\newcommand{\Lead}{\operatorname{Lead}}

\newcommand{\mcompose}{ {\; \color{magenta}\circledast \; } }
\newcommand{\ocompose}{ {\; \color{orange}\circledast \; } }
\newcommand{\gcompose}{ {\; \color{green!80!black}\circledast \; } }
\newcommand{\greycompose}{ {\; \color{gray}\circledast \; } }
\newcommand{\compose}{ {\; \color{cyan}\circledast \; } }
\newcommand{\pcompose}{ {\; \color{violet}\circledast \; } }

\newcommand{\SSTS}{\mathsf{S}}
\newcommand{\tSSTT}{\overline{\mathsf{T}}} 
\newcommand{\tla}{\overline{\x}}
\newcommand{\tmu}{\overline{\y }}
\newcommand{\SSTT}{\mathsf{T}}  
\newcommand{\SSTP}{\mathsf{P}}  
\newcommand{\SSTU}{\mathsf{U}}  
\newcommand{\SSTV}{\mathsf{V}}  
\newcommand{\SSTQ}{\mathsf{Q}}  
\newcommand{\SSTR}{\mathsf{R}}  
\newcommand{\SSTX}{\mathsf{X}} 
\newcommand{\SSTY}{\mathsf{Y}} 
\newcommand{\sts}{\mathsf{s}}  
\newcommand{\stt}{\mathsf{t}}  
\newcommand{\stu}{\mathsf{u}}  
\newcommand{\stv}{\mathsf{v}}  
\newcommand{\stw}{\mathsf{w}}  
\newcommand{\stx}{\mathsf{x}}  
\newcommand{\sty}{\mathsf{y}}  
\newcommand{\stq}{\mathsf{q}}  
\newcommand{\str}{\mathsf{r}}  

\newcommand{\ZZ}{{\mathbb Z}}
\newcommand{\NN}{{\mathbb N}}
\newcommand{\g}{\ell}
 
\newcommand{\CC}{{\mathbb{C}}}
\newcommand{\RR}{{\mathbb L}}
\newcommand{\Hyp}{{\mathbb E}_{\alpha,me}}
\newcommand{\length}{{t}}
\DeclareMathOperator\noedge{\:\rlap{\hspace*{0.25em}/}\text{---}\:}
\DeclareMathOperator{\Hom}{Hom}
\def\Mod{\textbf{-Mod}}
\let\<=\langle
\let\>=\rangle

\newcommand\Dec[1][A]{\mathbf{D}_{#1}(t)}
\newcommand\Cart[1][A]{\mathbf{C}_{#1}(t)}

\newcommand{\north}{top }
\newcommand{\northT}{{\sf T}}
\newcommand{\south}{bottom } 
\newcommand{\southT}{{\sf B}}

\newcommand\mydots{\makebox[1em][c]{\color{cyan}.\hfil\color{magenta}.\hfil\color{cyan}.\hfil\color{magenta}.}}
\tikzset{
ultra thin/.style= {line width=0.05pt},
very thin/.style=  {line width=0.2pt},
thin/.style=       {line width=0.1pt},
semithick/.style=  {line width=0.6pt},
thick/.style=      {line width=0.8pt},
very thick/.style= {line width=1.2pt},
ultra thick/.style={line width=1.6pt}
}

\crefname{ques}{Question}{Questions}
\crefname{defn}{Definition}{Definitions}
\crefname{thm}{Theorem}{Theorems}
\crefname{prop}{Proposition}{Propositions}
\crefname{lem}{Lemma}{Lemmas}
\crefname{cor}{Corollary}{Corollaries}
\crefname{conj}{Conjecture}{Conjectures}
\crefname{section}{Section}{Sections}
\crefname{subsection}{Subsection}{Subsections}
\crefname{eg}{Example}{Examples}
\crefname{figure}{Figure}{Figures}
\crefname{rem}{Remark}{Remarks}
\crefname{rmk}{Remark}{Remarks}
\crefname{equation}{equation}{equation}

\Crefname{ques}{Question}{Questions}
\Crefname{defn}{Definition}{Definitions}
\Crefname{thm}{Theorem}{Theorems}
\Crefname{prop}{Proposition}{Propositions}
\Crefname{lem}{Lemma}{Lemmas}
\Crefname{cor}{Corollary}{Corollaries}
\Crefname{conj}{Conjecture}{Conjectures}
\Crefname{section}{Section}{Sections}
\Crefname{subsection}{Subsection}{Subsections}
\Crefname{eg}{Example}{Examples}
\Crefname{figure}{Figure}{Figures}
\Crefname{rem}{Remark}{Remarks}
\Crefname{rmk}{Remark}{Remarks}

  \newcommand{\Mull}{{\rm M}}
    
\usepackage[hang,flushmargin]{footmisc}
 
\newcommand\Dim[2][t]{\text{\rm Dim}_{#1}#2}

\newcommand\REMOVETHESE[2]{{{{\mathsf{M}}_{#1}^{#2}	}}}
\newcommand\ADDTHIS[2]{{{{\mathsf{P}}_{#1}^{#2}}}}

 \newcommand{\Spotzero}{\SSTS_{0,\al}}
  \newcommand{\Spotone}{\SSTS_{1,\al}}
   \newcommand{\Spottwo}{\SSTS_{2,\al}}
    \newcommand{\Spotthree}{\SSTS_{3,\al}}
 \newcommand{\Spotq}{\SSTS_{q,\al}}
  \newcommand{\Spotb}{\SSTS_{\exx,\al}}
  \newcommand{\Spotqplus}{\SSTS_{q+1,\al}}

 \newcommand{\Forkq}{{\sf F}_{q, \al}}
  \newcommand{\Forkqplus}{{\sf F}_{q+1, \al}}
 \newcommand{\Forkone}{{\sf F}_{1, \al}}
 \newcommand{\Forkzero}{{\sf F}_{0, \al}} 
  \newcommand{\Forktwo}{{\sf F}_{2, \al}}
    \newcommand{\Forkthree}{{\sf F}_{3, \al}}
  \newcommand{\Forkb}{{\sf F}_{\exx, \al}}

 \newcommand{\TForkq}{{\sf F}_{q,\al }}
  \newcommand{\TForkqplus}{{\sf F}_{q+1,\al}}
 \newcommand{\TForkone}{{\sf F}_{1,\al }}
 \newcommand{\TForkzero}{{\sf F}_{0,\al }} 
  \newcommand{\TForktwo}{{\sf F}_{2,\al }}

\renewcommand{\labelitemi}{$\circ $}

\def\Item{\item\abovedisplayskip=0pt\abovedisplayshortskip=5pt~\vspace*{-\baselineskip}} 

\begin{document}

 \title [Quiver presentations and Schur--Weyl duality 
 for  Khovanov arc algebras ]{
Quiver presentations and Schur--Weyl duality
\\
 for  Khovanov arc algebras 
   }

 \author{C. Bowman}
       \address{Department of Mathematics, 
University of York, Heslington, York,  UK}
\email{chris.bowman-scargill@york.ac.uk}
  
 \author{M. De Visscher}
	\address{Department of Mathematics, City, University of London,   London, UK}
\email{maud.devisscher.1@city.ac.uk}

 \author{A. Dell'Arciprete}
       \address{Department of Mathematics, 
University of York, Heslington, York,  UK}
\email{alice.dellarciprete@york.ac.uk}

		\author{A.  Hazi}
	 
     \address{School of Mathematics, University of Leeds, Leeds, LS2 9JT}
\email{a.hazi@leeds.ac.uk}

		\author{R. Muth}
	 
     \address{Department of Mathematics and Computer Science,
Duquesne University,
Pittsburgh, PA 15282}
\email{muthr@duq.edu}

		\author{ C. Stroppel}
 \address{ Mathematical Institute, Endenicher Allee 60, 53115 Bonn}
 \email{stroppel@math.uni-bonn.de}
 
 \maketitle

 \vspace{-0.2cm}

 \begin{abstract}
We provide an ${\rm Ext}$-quiver and relations presentation of the Khovanov arc algebras and prove a precise analogue of the Kleshchev--Martin conjecture in this setting. 
 \end{abstract}

 \section{Introduction}

 The  Khovanov arc algebras,  $H^{m}_{ n}$, were first  introduced  by Khovanov (in the case $m=n$) in his  pioneering   construction of homological knot invariants for  tangles \cite{MR1740682,MR1928174}.   These homological  knot invariants have  subsequently been developed by Rasmussen and put to use in   Piccirillo's proof that the Conway knot is not slice \cite{MR2729272,MR4076631}. 
  The Khovanov arc algebras and their quasi-hereditary covers,  $K ^{m}_{ n}$,  
  have   been studied  from the point of view of 
their cohomological and representation theoretic structure  \cite{MR2600694,MR2781018,MR2955190,MR2881300,BarWang},
 symplectic geometry \cite{MR4422212}, 
  and  they  provide the  
exciting possibility of constructing  algebraic invariants suitable for Crane--Frenkel's  approach to the smooth 4-dimensional Poincar\'e  conjecture \cite{Manolescu}.

     In previous work, we gave a quadratic presentation of the 
     algebras $K ^{m}_{ n}$ as   $\ZZ$-algebras  
     in terms of Dyck path combinatorics, which upon base change to any field $\Bbbk$ 
     specialised to be the 
      ${\rm Ext}$-quiver and relations for the  $\Bbbk$-algebra $K ^{m}_{ n}$ \cite{compan4}; 
    we hence completely determined the representation  theoretic structure 
       of $K ^{m}_{ n}$
      in an entirely  characteristic-free manner.  
  In \cite{BDDHMS}  we further proved that  over any field the algebra 
$K ^{m}_{ n}$     is an $(|n-m|-1)$-faithful cover of $H^{m}_{ n}$ (in the sense of Rouquier \cite{rouquier}) thus establishing a strong cohomological 
connection between these algebras. 
In this paper, we 
 push  the 
quadratic  presentation of  $K ^{m}_{ n}$ 
   through this Schur--Weyl duality to provide 
  complete  ${\rm Ext}$-quiver and relations presentations of  $H^{m}_{ n}$, hence   completely determining the   representation  theoretic structure  of the Khovanov arc algebras. 
We will see that the
 structure of the algebras $H^{m}_{ n}$ is not as  rigid as that of their quasi-hereditary covers: the Khovanov arc algebras are cubic and  
 their structure does depend, to some extent, on the characteristic, $p$, of the underlying  field $\Bbbk$.  We assume, without loss of generality, that $m\leq n$.

 \begin{thmA}\label{thma}
The ${\rm Ext}$-quiver of $H^{m}_{ n}$ has vertices labelled by the 
 partitions $\la$ in an $(m\times n)$-rectangle such that $(m,m-1, \dots, 2,1)\subseteq \la$.  
 The edges are labelled by addable and removable Dyck paths:
  for $m=n$ and $p\neq 2$ there are $m$ additional loops at  $\la=(m^a, (m-a)^{m-a}) $ for $1\leq a \leq m$; 
 for $m=n$ and $p=2$ every vertex has an additional loop. 
The $\Bbbk$-algebra     $H^{m}_{ n}$ is the quotient of  the path algebra of its  ${\rm Ext}$-quiver modulo relations \eqref{rel1} to  \eqref{loop-relation}.
 \end{thmA}

  Examples of these ${\rm Ext}$-quivers are given in \cref{pic1,pic2,pic3}.

Schur--Weyl duality is a more general 
phenomenon  that  relates  Lie theoretic objects (such as categories $\mathcal{O}$ of Lie algebras, or (quantum) reductive algebraic  groups)   
 to interesting  finite dimensional  algebras (such as the group algebras of symmetric groups, Iwahori--Hecke algebras, and the Brauer and walled Brauer algebras).   
Representation theorists  have long endeavoured to pass cohomological   information back-and-forth by way of these dualities.
A benchmark for our understanding (or lack thereof) for this
back-and-forth process 
has been provided by the famous Kleshchev--Martin conjecture: this states that 
 simple modules for the symmetric group  do not admit self-extensions (providing that the Schur algebra is at least a 0-faithful cover). 
 We prove that the exact analogue of this statement holds in our context.

 \begin{thmB}
Let $\Bbbk$ be a field. 
The 
 ${\rm Ext}$-quiver  of $H^{m}_{ n}$  is loop-free
  if and only if 
 $K ^{m}_{ n}$ is an $i$-faithful   quasi-hereditary cover  for some $i\geq 0$ if and only if $m\neq n$.
  \end{thmB}

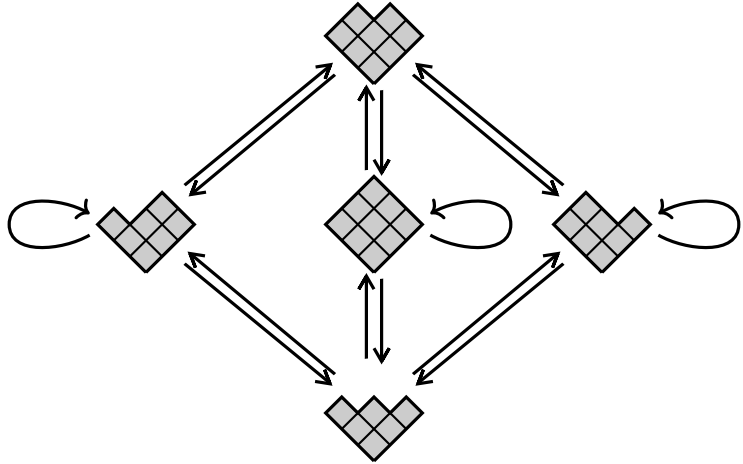
\begin{figure}[ht!]

 $$\begin{tikzpicture}
 \clip(-5,0.5) rectangle (5,-5.65);

 \path(0,0) coordinate (top)--++(180:0.1)  coordinate (topL)--++(0:0.2)  coordinate (topR);
  \path(0,-5) coordinate (bottom)--++(180:0.1)  coordinate (bottomL)--++(0:0.2)  coordinate (bottomR);;

  \path(0,-2.5) coordinate (middle) --++(180:0.1)  coordinate (middleL)--++(0:0.2)  coordinate (middleR);;  
    \path(3,-2.5) coordinate (right) --++(90:0.1)  coordinate (rightT)--++(-90:0.2)  coordinate (rightB);;  
    \path(-3,-2.5) coordinate (left)--++(90:0.1)  coordinate (leftT)--++(-90:0.2)  coordinate (leftB);;  ;

 \draw[very thick]  (topL)--(leftT)  (topR)--(leftB);

 \draw[very thick]  (topL)--(rightB)  (topR)--(rightT);

 \draw[very thick]  (bottomL)--(leftB)  (bottomR)--(leftT);

 \draw[very thick]  (bottomL)--(rightT)  (bottomR)--(rightB);

          \draw[very thick]  (topL)--(bottomL)  (topR)--(bottomR);

 \draw(0,0) node 
 {
$ \begin{tikzpicture}[scale=0.3]
 
 \path  (0,0) coordinate (Y);
 \path  (Y)--++(135:1)--++(45:2)--++(135:0.5)--++(-135:0.5)coordinate (X);
\fill [white] (X) circle(69pt);
\begin{scope}
 \draw[very thick,fill=gray!40] 
 (Y)--++(135:3)--++(45:2)
--++(-45:1) --++(45:1) 
 --++(-45:2) --++(-135:3); 
  \clip (Y)--++(135:3)--++(45:2)
--++(-45:1) --++(45:1) 
 --++(-45:2) --++(-135:3); 

  	\foreach \i in {0,1,2,3,4,5,6,7,8,9,10,11,12}
		{
			\path (Y)--++(45:1*\i) coordinate (c\i); 
			\path (Y)--++(135:1*\i)  coordinate (d\i); 
			\draw[thick, ] (c\i)--++(135:14);
			\draw[thick, ] (d\i)--++(45:14);
		}
	 	\end{scope}

 \end{tikzpicture}$
 };

 \draw(0,-2.5) node 
 {
$ \begin{tikzpicture}[scale=0.3]
 
 \path  (0,0) coordinate (Y);
 \path  (Y)--++(135:1)--++(45:2)--++(135:0.5)--++(-135:0.5)coordinate (X);
\fill [white] (X) circle(69pt);
\begin{scope}
 \draw[very thick,fill=gray!40] (Y)--++(135:3)--++(45:3)--++(-45:3) --++(-135:3)--(0,0); 
  \clip  (Y)--++(135:3)--++(45:3)--++(-45:3) --++(-135:3)--(0,0);  	
  	\foreach \i in {0,1,2,3,4,5,6,7,8,9,10,11,12}
		{
			\path (Y)--++(45:1*\i) coordinate (c\i); 
			\path (Y)--++(135:1*\i)  coordinate (d\i); 
			\draw[thick, ] (c\i)--++(135:14);
			\draw[thick, ] (d\i)--++(45:14);
		}
	 	\end{scope}

 \end{tikzpicture}$
 };

 \draw(3,-2.5) node 
 {
$ \begin{tikzpicture}[scale=0.3]
 
 \path  (0,0) coordinate (Y);
 \path  (Y)--++(135:1)--++(45:2)--++(135:0.5)--++(-135:0.5)coordinate (X);
\fill [white] (X) circle(69pt);
\begin{scope}
 \draw[very thick,fill=gray!40] 
 (Y)--++(135:3)--++(45:2)
--++(-45:2) --++(45:1) 
 --++(-45:1) --++(-135:3); 
  \clip(Y)--++(135:3)--++(45:2)
--++(-45:2) --++(45:1) 
 --++(-45:1) --++(-135:3); 

  	\foreach \i in {0,1,2,3,4,5,6,7,8,9,10,11,12}
		{
			\path (Y)--++(45:1*\i) coordinate (c\i); 
			\path (Y)--++(135:1*\i)  coordinate (d\i); 
			\draw[thick, ] (c\i)--++(135:14);
			\draw[thick, ] (d\i)--++(45:14);
		}
	 	\end{scope}

 \end{tikzpicture}$
 };

 \draw(-3,-2.5) node 
 {
$ \begin{tikzpicture}[scale=0.3,xscale=-1]
 
 \path  (0,0) coordinate (Y);
 \path  (Y)--++(135:1)--++(45:2)--++(135:0.5)--++(-135:0.5)coordinate (X);
\fill [white] (X) circle(69pt);
\begin{scope}
 \draw[very thick,fill=gray!40] 
 (Y)--++(135:3)--++(45:2)
--++(-45:2) --++(45:1) 
 --++(-45:1) --++(-135:3); 
  \clip(Y)--++(135:3)--++(45:2)
--++(-45:2) --++(45:1) 
 --++(-45:1) --++(-135:3); 

  	\foreach \i in {0,1,2,3,4,5,6,7,8,9,10,11,12}
		{
			\path (Y)--++(45:1*\i) coordinate (c\i); 
			\path (Y)--++(135:1*\i)  coordinate (d\i); 
			\draw[thick, ] (c\i)--++(135:14);
			\draw[thick, ] (d\i)--++(45:14);
		}
	 	\end{scope}

 \end{tikzpicture}$
 };

 \draw(0,-5) node 
 {
$ \begin{tikzpicture}[scale=0.3]
 
 \path  (0,-5) coordinate (Y);
 \path  (Y)--++(135:1)--++(45:2)--++(135:0.5)--++(-135:0.5)coordinate (X);
\fill [white] (X) circle(69pt);
\begin{scope}
 \draw[very thick,fill=gray!40] 
 (Y)--++(135:3) --++(45:1) 
 --++(-45:1)
 --++(45:1)  --++(-45:1)
 --++(45:1) 
 --++(-45:1) --++(-135:3); 
  \clip (Y)--++(135:3) --++(45:1) 
 --++(-45:1)
 --++(45:1)  --++(-45:1)
 --++(45:1) 
 --++(-45:1) --++(-135:3);

  	\foreach \i in {0,1,2,3,4,5,6,7,8,9,10,11,12}
		{
			\path (Y)--++(45:1*\i) coordinate (c\i); 
			\path (Y)--++(135:1*\i)  coordinate (d\i); 
			\draw[thick, ] (c\i)--++(135:14);
			\draw[thick, ] (d\i)--++(45:14);
		}
	 	\end{scope}

 \end{tikzpicture}$
 };

\path         (-3.6,-2.5)  --++(-135:0.2) coordinate (X) ;
\path         (-3.6,-2.5)  --++(135:0.2) coordinate (Y) ;
\path         (-3.6,-2.5)  --++(180:1.2) coordinate (Z) ;

\draw[very thick,->]        (X) to [out=-150,in=-90] (Z) to[out=90,in=150] (Y);

\path         (3.6,-2.5)  --++(-45:0.2) coordinate (X) ;
\path         (3.6,-2.5)  --++(45:0.2) coordinate (Y) ;
\path         (3.6,-2.5)  --++(0:1.2) coordinate (Z) ;

\draw[very thick,->]        (X) to [out=-30,in=-90] (Z) to[out=90,in=30] (Y);

\path         (0.6,-2.5)  --++(-45:0.2) coordinate (X) ;
\path         (0.6,-2.5)  --++(45:0.2) coordinate (Y) ;
\path         (0.6,-2.5)  --++(0:1.2) coordinate (Z) ;

\draw[very thick,->]        (X) to [out=-30,in=-90] (Z) to[out=90,in=30] (Y);

          \path (topL)--++(-141:17.25pt) coordinate (here);
         \draw[very thick] (here)  --++(-141+30:0.2) --++(-141+30+180:0.2) --++(-141-30:0.2); 
     \path (leftB)--++(180-140:15pt+7pt) coordinate (here);
         \draw[very thick] (here)  --++(180-141+30:0.2) --++(180-141+30+180:0.2) --++(180-141-30:0.2);

          \path (topR)--++(180+141:17.25pt) coordinate (here);
         \draw[very thick] (here)  --++(180+141-30:0.2) --++(180+141-30+180:0.2) --++(180+141+30:0.2); 
     \path (rightB)--++(+140:15pt+7pt) coordinate (here);
         \draw[very thick] (here)  --++(+141-30:0.2) --++(+141-30+180:0.2) --++(+141+30:0.2);


          \path (rightT)--++(-140:21.5pt) coordinate (here);
         \draw[very thick] (here)  --++(-141+30:0.2) --++(-141+30+180:0.2) --++(-141-30:0.2); 
     \path (bottomR)--++(180-141:17.5pt) coordinate (here);
         \draw[very thick] (here)  --++(180-141+30:0.2) --++(180-141+30+180:0.2) --++(180-141-30:0.2);

          \path (leftT)--++(180+140:21.5pt) coordinate (here);
         \draw[very thick] (here)  --++(180+141-30:0.2) --++(180+141-30+180:0.2) --++(180+141+30:0.2); 
     \path (bottomL)--++(+141:17.5pt) coordinate (here);
         \draw[very thick] (here)  --++(+141-30:0.2) --++(+141-30+180:0.2) --++(+141+30:0.2);

            \path (topL)--++(-90:19.5pt) coordinate (here);
         \draw[very thick] (here)  --++(180+90-30:0.2) --++(180+90-30+180:0.2) --++(180+90+30:0.2);  
            \path (middleR)--++(90:19.5pt) coordinate (here);
          \draw[very thick] (here)  --++(90-30:0.2) --++(90-30+180:0.2) --++(90+30:0.2);

          \path (middleL)--++(-90:19.5pt) coordinate (here);
         \draw[very thick] (here)  --++(180+90-30:0.2) --++(180+90-30+180:0.2) --++(180+90+30:0.2);  
            \path (bottomR)--++(90:19.5pt) coordinate (here);
          \draw[very thick] (here)  --++(90-30:0.2) --++(90-30+180:0.2) --++(90+30:0.2);  
 
\end{tikzpicture}
$$

\caption{The ${\rm Ext}$-quiver of $  {H}^3_3$ for $p\neq 2$.}
\label{pic1}
\end{figure}

  \begin{figure}[ht!]
   $$\begin{tikzpicture}

 \path(0,0) coordinate (top)--++(180:0.1)  coordinate (topL)--++(0:0.2)  coordinate (topR);
  \path(0,-5) coordinate (bottom)--++(180:0.1)  coordinate (bottomL)--++(0:0.2)  coordinate (bottomR);;

  \path(0,-2.5) coordinate (middle) --++(180:0.1)  coordinate (middleL)--++(0:0.2)  coordinate (middleR);;  
    \path(3,-2.5) coordinate (right) --++(90:0.1)  coordinate (rightT)--++(-90:0.2)  coordinate (rightB);;  
    \path(-3,-2.5) coordinate (left)--++(90:0.1)  coordinate (leftT)--++(-90:0.2)  coordinate (leftB);;  ;    
   
 \draw[very thick]  (topL)--(leftT)  (topR)--(leftB);

 \draw[very thick]  (topL)--(rightB)  (topR)--(rightT);

 \draw[very thick]  (bottomL)--(leftB)  (bottomR)--(leftT);

 \draw[very thick]  (bottomL)--(rightT)  (bottomR)--(rightB);

          \draw[very thick]  (topL)--(bottomL)  (topR)--(bottomR);

 \draw(0,0) node 
 {
$ \begin{tikzpicture}[scale=0.3]
 
 \path  (0,0) coordinate (Y);
 \path  (Y)--++(135:1)--++(45:2)--++(135:0.5)--++(-135:0.5)coordinate (X);
\fill [white] (X) circle(69pt);
\begin{scope}
 \draw[very thick,fill=gray!40] 
 (Y)--++(135:3)--++(45:2)
--++(-45:1) --++(45:1) 
 --++(-45:2) --++(-135:3); 
  \clip (Y)--++(135:3)--++(45:2)
--++(-45:1) --++(45:1) 
 --++(-45:2) --++(-135:3); 

  	\foreach \i in {0,1,2,3,4,5,6,7,8,9,10,11,12}
		{
			\path (Y)--++(45:1*\i) coordinate (c\i); 
			\path (Y)--++(135:1*\i)  coordinate (d\i); 
			\draw[thick, ] (c\i)--++(135:14);
			\draw[thick, ] (d\i)--++(45:14);
		}
	 	\end{scope}

 \end{tikzpicture}$
 };

 \draw(0,-2.5) node 
 {
$ \begin{tikzpicture}[scale=0.3]
 
 \path  (0,0) coordinate (Y);
 \path  (Y)--++(135:1)--++(45:2)--++(135:0.5)--++(-135:0.5)coordinate (X);
\fill [white] (X) circle(69pt);
\begin{scope}
 \draw[very thick,fill=gray!40] (Y)--++(135:3)--++(45:3)--++(-45:3) --++(-135:3)--(0,0); 
  \clip  (Y)--++(135:3)--++(45:3)--++(-45:3) --++(-135:3)--(0,0);  	
  	\foreach \i in {0,1,2,3,4,5,6,7,8,9,10,11,12}
		{
			\path (Y)--++(45:1*\i) coordinate (c\i); 
			\path (Y)--++(135:1*\i)  coordinate (d\i); 
			\draw[thick, ] (c\i)--++(135:14);
			\draw[thick, ] (d\i)--++(45:14);
		}
	 	\end{scope}

 \end{tikzpicture}$
 };

 \draw(3,-2.5) node 
 {
$ \begin{tikzpicture}[scale=0.3]
 
 \path  (0,0) coordinate (Y);
 \path  (Y)--++(135:1)--++(45:2)--++(135:0.5)--++(-135:0.5)coordinate (X);
\fill [white] (X) circle(69pt);
\begin{scope}
 \draw[very thick,fill=gray!40] 
 (Y)--++(135:3)--++(45:2)
--++(-45:2) --++(45:1) 
 --++(-45:1) --++(-135:3); 
  \clip(Y)--++(135:3)--++(45:2)
--++(-45:2) --++(45:1) 
 --++(-45:1) --++(-135:3); 

  	\foreach \i in {0,1,2,3,4,5,6,7,8,9,10,11,12}
		{
			\path (Y)--++(45:1*\i) coordinate (c\i); 
			\path (Y)--++(135:1*\i)  coordinate (d\i); 
			\draw[thick, ] (c\i)--++(135:14);
			\draw[thick, ] (d\i)--++(45:14);
		}
	 	\end{scope}

 \end{tikzpicture}$
 };

 \draw(-3,-2.5) node 
 {
$ \begin{tikzpicture}[scale=0.3,xscale=-1]
 
 \path  (0,0) coordinate (Y);
 \path  (Y)--++(135:1)--++(45:2)--++(135:0.5)--++(-135:0.5)coordinate (X);
\fill [white] (X) circle(69pt);
\begin{scope}
 \draw[very thick,fill=gray!40] 
 (Y)--++(135:3)--++(45:2)
--++(-45:2) --++(45:1) 
 --++(-45:1) --++(-135:3); 
  \clip(Y)--++(135:3)--++(45:2)
--++(-45:2) --++(45:1) 
 --++(-45:1) --++(-135:3); 

  	\foreach \i in {0,1,2,3,4,5,6,7,8,9,10,11,12}
		{
			\path (Y)--++(45:1*\i) coordinate (c\i); 
			\path (Y)--++(135:1*\i)  coordinate (d\i); 
			\draw[thick, ] (c\i)--++(135:14);
			\draw[thick, ] (d\i)--++(45:14);
		}
	 	\end{scope}

 \end{tikzpicture}$
 };

 \draw(0,-5) node 
 {
$ \begin{tikzpicture}[scale=0.3]
 
 \path  (0,-5) coordinate (Y);
 \path  (Y)--++(135:1)--++(45:2)--++(135:0.5)--++(-135:0.5)coordinate (X);
\fill [white] (X) circle(69pt);
\begin{scope}
 \draw[very thick,fill=gray!40] 
 (Y)--++(135:3) --++(45:1) 
 --++(-45:1)
 --++(45:1)  --++(-45:1)
 --++(45:1) 
 --++(-45:1) --++(-135:3); 
  \clip (Y)--++(135:3) --++(45:1) 
 --++(-45:1)
 --++(45:1)  --++(-45:1)
 --++(45:1) 
 --++(-45:1) --++(-135:3);

  	\foreach \i in {0,1,2,3,4,5,6,7,8,9,10,11,12}
		{
			\path (Y)--++(45:1*\i) coordinate (c\i); 
			\path (Y)--++(135:1*\i)  coordinate (d\i); 
			\draw[thick, ] (c\i)--++(135:14);
			\draw[thick, ] (d\i)--++(45:14);
		}
	 	\end{scope}

 \end{tikzpicture}$
 };

\path         (-3.6,-2.5)  --++(-135:0.2) coordinate (X) ;
\path         (-3.6,-2.5)  --++(135:0.2) coordinate (Y) ;
\path         (-3.6,-2.5)  --++(180:1.2) coordinate (Z) ;

\draw[very thick,->]        (X) to [out=-150,in=-90] (Z) to[out=90,in=150] (Y);

\path         (3.6,-2.5)  --++(-45:0.2) coordinate (X) ;
\path         (3.6,-2.5)  --++(45:0.2) coordinate (Y) ;
\path         (3.6,-2.5)  --++(0:1.2) coordinate (Z) ;

\draw[very thick,->]        (X) to [out=-30,in=-90] (Z) to[out=90,in=30] (Y);

\path         (0.6,-2.5)  --++(-45:0.2) coordinate (X) ;
\path         (0.6,-2.5)  --++(45:0.2) coordinate (Y) ;
\path         (0.6,-2.5)  --++(0:1.2) coordinate (Z) ;

\draw[very thick,->]        (X) to [out=-30,in=-90] (Z) to[out=90,in=30] (Y);

          \path (topL)--++(-141:17.25pt) coordinate (here);
         \draw[very thick] (here)  --++(-141+30:0.2) --++(-141+30+180:0.2) --++(-141-30:0.2); 
     \path (leftB)--++(180-140:15pt+7pt) coordinate (here);
         \draw[very thick] (here)  --++(180-141+30:0.2) --++(180-141+30+180:0.2) --++(180-141-30:0.2);

          \path (topR)--++(180+141:17.25pt) coordinate (here);
         \draw[very thick] (here)  --++(180+141-30:0.2) --++(180+141-30+180:0.2) --++(180+141+30:0.2); 
     \path (rightB)--++(+140:15pt+7pt) coordinate (here);
         \draw[very thick] (here)  --++(+141-30:0.2) --++(+141-30+180:0.2) --++(+141+30:0.2);


          \path (rightT)--++(-140:21.5pt) coordinate (here);
         \draw[very thick] (here)  --++(-141+30:0.2) --++(-141+30+180:0.2) --++(-141-30:0.2); 
     \path (bottomR)--++(180-141:17.5pt) coordinate (here);
         \draw[very thick] (here)  --++(180-141+30:0.2) --++(180-141+30+180:0.2) --++(180-141-30:0.2);

          \path (leftT)--++(180+140:21.5pt) coordinate (here);
         \draw[very thick] (here)  --++(180+141-30:0.2) --++(180+141-30+180:0.2) --++(180+141+30:0.2); 
     \path (bottomL)--++(+141:17.5pt) coordinate (here);
         \draw[very thick] (here)  --++(+141-30:0.2) --++(+141-30+180:0.2) --++(+141+30:0.2);

            \path (topL)--++(-90:19.5pt) coordinate (here);
         \draw[very thick] (here)  --++(180+90-30:0.2) --++(180+90-30+180:0.2) --++(180+90+30:0.2);  
            \path (middleR)--++(90:19.5pt) coordinate (here);
          \draw[very thick] (here)  --++(90-30:0.2) --++(90-30+180:0.2) --++(90+30:0.2);

          \path (middleL)--++(-90:19.5pt) coordinate (here);
         \draw[very thick] (here)  --++(180+90-30:0.2) --++(180+90-30+180:0.2) --++(180+90+30:0.2);  
            \path (bottomR)--++(90:19.5pt) coordinate (here);
          \draw[very thick] (here)  --++(90-30:0.2) --++(90-30+180:0.2) --++(90+30:0.2);

\path         (0,0.6)  --++(-135-90:0.2) coordinate (X) ;
\path         (0,0.6)  --++(135-90:0.2) coordinate (Y) ;
\path         (0,0.6)  --++(180-90:1.2) coordinate (Z) ;

\draw[very thick,->]        (X) to [out=-150-90,in=-90-90] (Z) to[out=90-90,in=150-90] (Y);

\path         (0,-5-0.6)  --++(-135-90-180:0.2) coordinate (X) ;
\path         (0,-5-0.6)  --++(135-90-180:0.2) coordinate (Y) ;
\path         (0,-5-0.6)  --++(180-90-180:1.2) coordinate (Z) ;

\draw[very thick,->]        (X) to [out=-150-90-180,in=-90-90-180] (Z) to[out=90-90-180,in=150-90-180] (Y);

\end{tikzpicture}
$$
\caption{The ${\rm Ext}$-quiver of $H^3_3$ for $p=2$.}
\label{pic2}
\end{figure}
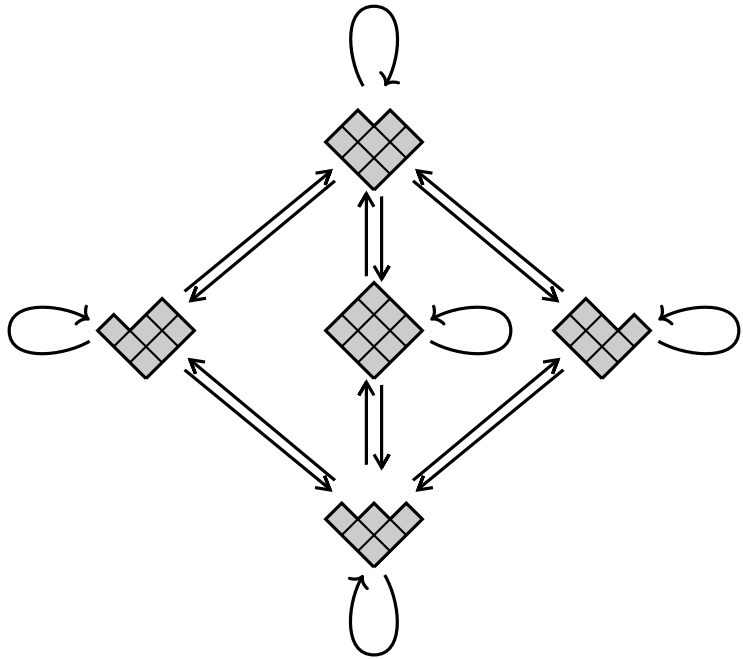

\begin{figure}[ht!]

 $$\begin{tikzpicture}

 \clip(-4,0.35) rectangle (4,-5.55);

%
%
 
 \path(0,0) coordinate (top)--++(180:0.1)  coordinate (topL)--++(0:0.2)  coordinate (topR);
  \path(0,-5) coordinate (bottom)--++(180:0.1)  coordinate (bottomL)--++(0:0.2)  coordinate (bottomR);;

  \path(0,-2.5) coordinate (middle) --++(180:0.1)  coordinate (middleL)--++(0:0.2)  coordinate (middleR);;  
    \path(3,-2.5) coordinate (right) --++(90:0.1)  coordinate (rightT)--++(-90:0.2)  coordinate (rightB);;  
    \path(-3,-2.5) coordinate (left)--++(90:0.1)  coordinate (leftT)--++(-90:0.2)  coordinate (leftB);;  ;

%
%
%
%
%
%
  
 \draw[very thick]  (topL)--(leftT)  (topR)--(leftB);

 \draw[very thick]  (topL)--(rightB)  (topR)--(rightT);

 \draw[very thick]  (bottomL)--(leftB)  (bottomR)--(leftT);

 \draw[very thick]  (bottomL)--(rightT)  (bottomR)--(rightB);

          \draw[very thick]  (topL)--(bottomL)  (topR)--(bottomR);

 \draw(0,0) node 
 {
$ \begin{tikzpicture}[scale=0.3]
 
 \path  (0,0) coordinate (Y);
  \path  (Y)--++(135:0.5)--++(45:2)--++(135:0.5)--++(-135:0.5)coordinate (X);
\fill [white] (X) circle(69pt);
\begin{scope}
 \draw[very thick,fill=gray!40] 
 (Y)--++(135:2)--++(45:2)
--++(-45:1) --++(45:1) 
 --++(-45:1) --++(-135:3);   \clip (Y)--++(135:2)--++(45:2)
--++(-45:1) --++(45:1) 
 --++(-45:1) --++(-135:3); 
  	\foreach \i in {0,1,2,3,4,5,6,7,8,9,10,11,12}
		{
			\path (Y)--++(45:1*\i) coordinate (c\i); 
			\path (Y)--++(135:1*\i)  coordinate (d\i); 
			\draw[thick, ] (c\i)--++(135:14);
			\draw[thick, ] (d\i)--++(45:14);
		}
	 	\end{scope}

 \end{tikzpicture}$
 };

 \draw(0,-2.5) node 
 {
$ \begin{tikzpicture}[scale=0.3]
 
 \path  (0,0) coordinate (Y);
 \path  (Y)--++(135:0.5)--++(45:2)--++(135:0.5)--++(-135:0.5)coordinate (X);
\fill [white] (X) circle(69pt);
\begin{scope}
 \draw[very thick,fill=gray!40] (Y)--++(135:2)--++(45:3)--++(-45:2) --++(-135:3)--(0,0); 
  \clip (Y)--++(135:2)--++(45:3)--++(-45:2) --++(-135:3)--(0,0); 
  	\foreach \i in {0,1,2,3,4,5,6,7,8,9,10,11,12}
		{
			\path (Y)--++(45:1*\i) coordinate (c\i); 
			\path (Y)--++(135:1*\i)  coordinate (d\i); 
			\draw[thick, ] (c\i)--++(135:14);
			\draw[thick, ] (d\i)--++(45:14);
		}
	 	\end{scope}

 \end{tikzpicture}$
 };

 \draw(3,-2.5) node 
 {
$ \begin{tikzpicture}[scale=0.3]
 
 \path  (0,0) coordinate (Y);
  \path  (Y)--++(135:0.5)--++(45:2)--++(135:0.5)--++(-135:0.5)coordinate (X);
\fill [white] (X) circle(69pt);
\begin{scope}
 \draw[very thick,fill=gray!40] 
 (Y)--++(135:2)--++(45:2)
--++(-45:1) --++(45:1) 
 --++(-45:1) --++(-135:3); 
  \clip(Y)--++(135:2)--++(45:2)
--++(-45:1) --++(45:1) 
 --++(-45:1) --++(-135:3);

  	\foreach \i in {0,1,2,3,4,5,6,7,8,9,10,11,12}
		{
			\path (Y)--++(45:1*\i) coordinate (c\i); 
			\path (Y)--++(135:1*\i)  coordinate (d\i); 
			\draw[thick, ] (c\i)--++(135:14);
			\draw[thick, ] (d\i)--++(45:14);
		}
	 	\end{scope}

 \end{tikzpicture}$
 };

 \draw(-3,-2.5) node 
 {
$ \begin{tikzpicture}[scale=0.3,xscale=-1]
 
 \path  (0,0) coordinate (Y);
  \path  (Y)--++(135:0.5)--++(45:2)--++(135:0.5)--++(-135:0.5)coordinate (X);
\fill [white] (X) circle(69pt);
\begin{scope}
 \draw[very thick,fill=gray!40] 
 (Y)--++(135:3)--++(45:1)
--++(-45:2) --++(45:1) 
 --++(-45:1) --++(-135:2); 
  \clip (Y)--++(135:3)--++(45:1)
--++(-45:2) --++(45:1) 
 --++(-45:1) --++(-135:2);

  	\foreach \i in {0,1,2,3,4,5,6,7,8,9,10,11,12}
		{
			\path (Y)--++(45:1*\i) coordinate (c\i); 
			\path (Y)--++(135:1*\i)  coordinate (d\i); 
			\draw[thick, ] (c\i)--++(135:14);
			\draw[thick, ] (d\i)--++(45:14);
		}
	 	\end{scope}

 \end{tikzpicture}$
 };

 \draw(0,-5) node 
 {
$ \begin{tikzpicture}[scale=0.3]
 
 \path  (0,-5) coordinate (Y);
  \path  (Y)--++(135:0.5)--++(45:2)--++(135:0.5)--++(-135:0.5)coordinate (X);
\fill [white] (X) circle(69pt);
\begin{scope}
 \draw[very thick,fill=gray!40] 
 (Y)--++(135:2) --++(45:1) 
 --++(-45:1)
 --++(45:1) 
 --++(-45:1) --++(-135:2); 
  \clip(Y)--++(135:2) --++(45:1) 
 --++(-45:1)
 --++(45:1) 
 --++(-45:1) --++(-135:2);

  	\foreach \i in {0,1,2,3,4,5,6,7,8,9,10,11,12}
		{
			\path (Y)--++(45:1*\i) coordinate (c\i); 
			\path (Y)--++(135:1*\i)  coordinate (d\i); 
			\draw[thick, ] (c\i)--++(135:14);
			\draw[thick, ] (d\i)--++(45:14);
		}
	 	\end{scope}

 \end{tikzpicture}$
 };

          \path (topL)--++(-141:17.25pt) coordinate (here);
         \draw[very thick] (here)  --++(-141+30:0.2) --++(-141+30+180:0.2) --++(-141-30:0.2); 
     \path (leftB)--++(180-140:15pt+7pt) coordinate (here);
         \draw[very thick] (here)  --++(180-141+30:0.2) --++(180-141+30+180:0.2) --++(180-141-30:0.2);

          \path (topR)--++(180+141:17.25pt) coordinate (here);
         \draw[very thick] (here)  --++(180+141-30:0.2) --++(180+141-30+180:0.2) --++(180+141+30:0.2); 
     \path (rightB)--++(+140:15pt+7pt) coordinate (here);
         \draw[very thick] (here)  --++(+141-30:0.2) --++(+141-30+180:0.2) --++(+141+30:0.2);


          \path (rightT)--++(-140:21.5pt) coordinate (here);
         \draw[very thick] (here)  --++(-141+30:0.2) --++(-141+30+180:0.2) --++(-141-30:0.2); 
     \path (bottomR)--++(180-141:17.5pt) coordinate (here);
         \draw[very thick] (here)  --++(180-141+30:0.2) --++(180-141+30+180:0.2) --++(180-141-30:0.2);

          \path (leftT)--++(180+140:21.5pt) coordinate (here);
         \draw[very thick] (here)  --++(180+141-30:0.2) --++(180+141-30+180:0.2) --++(180+141+30:0.2); 
     \path (bottomL)--++(+141:17.5pt) coordinate (here);
         \draw[very thick] (here)  --++(+141-30:0.2) --++(+141-30+180:0.2) --++(+141+30:0.2);

            \path (topL)--++(-90:19.5pt) coordinate (here);
         \draw[very thick] (here)  --++(180+90-30:0.2) --++(180+90-30+180:0.2) --++(180+90+30:0.2);  
            \path (middleR)--++(90:19.5pt) coordinate (here);
          \draw[very thick] (here)  --++(90-30:0.2) --++(90-30+180:0.2) --++(90+30:0.2);

          \path (middleL)--++(-90:19.5pt) coordinate (here);
         \draw[very thick] (here)  --++(180+90-30:0.2) --++(180+90-30+180:0.2) --++(180+90+30:0.2);  
            \path (bottomR)--++(90:19.5pt) coordinate (here);
          \draw[very thick] (here)  --++(90-30:0.2) --++(90-30+180:0.2) --++(90+30:0.2);

\end{tikzpicture}
$$

\caption{The ${\rm Ext}$-quiver of $H^2_3$ for all $p\geq2$.}
\label{pic3}
\end{figure}
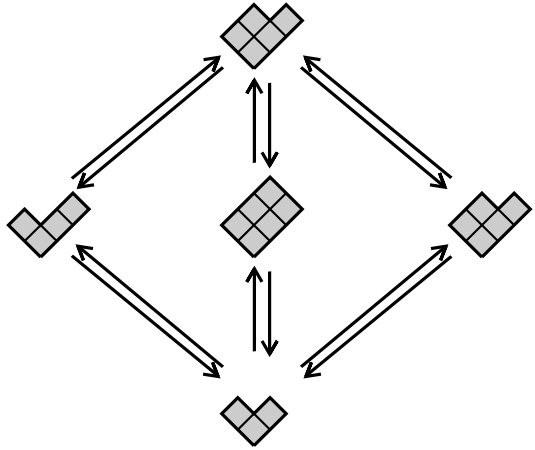

The paper is structured as follows. Section 2 recalls the combinatorics of partitions and Dyck paths necessary for the statements of our results. 
Sections  3 and 4 develops this Dyck  combinatorics further, providing the language needed to discuss the simple heads of cell modules of 
Khovanov arc algebras.  
Section 5 recalls the definition of the Khovanov arc algebras, their quasi-hereditary covers  {\em  the extended Khovanov arc algebras}  and 
discusses their cellular structures; here we prove that every cell  module has a simple head and deduce the  full submodule structures of these cell modules. 
In Section 6 we define an abstract $\ZZ$-algebra via Dyck combinatorial generators and relations and prove that this algebra projects surjectively onto the Khovanov arc algebra.
  In Section 7 we show that this map is an isomorphism of algebras,
therefore  the  arc algebras inherit these 
  Dyck combinatorial presentations.    
  Finally, in Section 8 we make the few tweaks necessary to refine these presentations into ${\rm Ext}$-quiver and relations presentations over a field $\Bbbk$; as our algebras are basic, this simply requires that we find a minimal set of generators. 
In Section 8 we also  recouch the Kleshchev--Martin conjecture in our language and propose  a vast generalisation of this conjecture to all anti-spherical Hecke categories.

\begin{Acknowledgements*}
 The first and third authors were funded by  EPSRC grant 
 EP/V00090X/1.  
 
 \end{Acknowledgements*}

  \section{Partitions, cup and cap diagrams, and $p$-Kazhdan-Lusztig polynomials}

     We begin by reviewing and unifying the combinatorics of Khovanov arc algebras 
 \cite{MR2600694,MR2781018,MR2955190,MR2881300} and the Hecke categories of interest in this paper \cite{compan,compan2}.
     
 
Let $S_n$ denote the symmetric group of degree $n$. Throughout this paper, we will work with the parabolic Coxeter system $(W,P) = (S_{m+n}, S_m \times S_n)$. 
For the entire   paper we assume, without loss of generality,   that $m\leq n$.  
We label the simple reflections with the slightly unusual subscripts $s_i, \, -m+1 \leq i \leq n-1$ so that $P = \langle s_i \, | \, i\neq 0\rangle \leq W$. We view $W$ as the group of permutations of the $n+m$ points on a horizontal strip numbered by the half integers $i\pm \tfrac{1}{2}$ where the simple reflection $s_i$ swaps the points $i-\tfrac{1}{2}$ and $i+\tfrac{1}{2}$ and fixes every other point. 
The right cosets of $P$ in $W$ can then be identified by labelled horizontal strips called {\sf weights}, where each point $i\pm \tfrac{1}{2}$ is labelled by either $\up$ or $\down$ in such a way that the total number of $\up$ is equal to $m$ (and so the total number of $\down$ is equal to $n$). Specifically, the trivial coset $P$ is represented by the weight with negative points labelled by  $\up$  and positive points labelled by $\down$. The other cosets are obtained by permuting the labels of the identity weight. An example is given in  \ref{Figweightpartition}. 
For more details on this combinatorics, see     \cite[Section 2]{compan4}.

  \begin{figure}[ht!]
  $$ \scalefont{0.7}
 \begin{tikzpicture} [scale=0.82]

\path (0,0) coordinate (origin2);

\begin{scope}

     \foreach \i in {0,1,2,3,4,5,6,7,8,9,10,11,12}
{
\path (origin2)--++(45:0.5*\i) coordinate (c\i);
\path (origin2)--++(135:0.5*\i)  coordinate (d\i);
  }

\path(origin2)  ++(135:2.5)   ++(-135:2.5) coordinate(corner1);
\path(origin2)  ++(45:2)   ++(135:7) coordinate(corner2);
\path(origin2)  ++(45:2)   ++(-45:2) coordinate(corner4);
\path(origin2)  ++(135:2.5)   ++(45:6.5) coordinate(corner3);

\draw[thick] (origin2)--(corner1)--(corner2)--(corner3)--(corner4)--(origin2);

\clip(corner1)--(corner2)--++(90:0.3)--++(0:6.5)--(corner3)--(corner4)
--++(90:-0.3)--++(180:6.5) --(corner1);

\path[name path=pathd1] (d1)--++(90:7);
 \path[name path=top] (corner2)--(corner3);
 \path [name intersections={of = pathd1 and top}];
   \coordinate (A)  at (intersection-1);
     \path(A)--++(-90:0.1) node { $\up$ };

\path[name path=pathd3] (d3)--++(90:7);
 \path[name path=top] (corner2)--(corner3);
 \path [name intersections={of = pathd3 and top}];
   \coordinate (A)  at (intersection-1);
  \path(A)--++(90:-0.1) node { $\up$ };

\path[name path=pathd5] (d5)--++(90:7);
 \path[name path=top] (corner2)--(corner3);
 \path [name intersections={of = pathd5 and top}];
   \coordinate (A)  at (intersection-1);
  \path(A)--++(90:0.1) node { $\down$ };

\path[name path=pathd7] (d7)--++(90:7);
 \path[name path=top] (corner2)--(corner3);
 \path [name intersections={of = pathd7 and top}];
   \coordinate (A)  at (intersection-1);
     \path(A)--++(90:-0.1) node { $\up$ };

\path[name path=pathd9] (d9)--++(90:7);
 \path[name path=top] (corner2)--(corner3);
 \path [name intersections={of = pathd9 and top}];
   \coordinate (A)  at (intersection-1);
    \path(A)--++(-90:-0.1) node { $\down$ };

\path[name path=pathc1] (c1)--++(90:7);
 \path[name path=top] (corner2)--(corner3);
 \path [name intersections={of = pathc1 and top}];
   \coordinate (A)  at (intersection-1);
  \path(A)--++(-90:-0.1) node { $\down$ };

\path[name path=pathc3] (c3)--++(90:7);
 \path[name path=top] (corner2)--(corner3);
 \path [name intersections={of = pathc3 and top}];
   \coordinate (A)  at (intersection-1);
    \path(A)--++(90:-0.1) node { $\up$ };

\path[name path=pathc5] (c5)--++(90:7);
 \path[name path=top] (corner2)--(corner3);
 \path [name intersections={of = pathc5 and top}];
   \coordinate (A)  at (intersection-1);
    \path(A)--++(-90:-0.1) node { $\down$ };

\path[name path=pathc7] (c7)--++(90:7);
 \path[name path=top] (corner2)--(corner3);
 \path [name intersections={of = pathc7 and top}];
   \coordinate (A)  at (intersection-1);
   \path(A)--++(-90:0.1) node { $\up$ };

 \path[name path=pathd1] (d1)--++(-90:7);
 \path[name path=bottom] (corner1)--(corner4);
 \path [name intersections={of = pathd1 and bottom}];
   \coordinate (A)  at (intersection-1);
   \path (A)--++(90:-0.1) node { $\up$  };

 \path[name path=pathd3] (d3)--++(-90:7);
 \path[name path=bottom] (corner1)--(corner4);
 \path [name intersections={of = pathd3 and bottom}];
   \coordinate (A)  at (intersection-1);
   \path (A)--++(90:-0.1) node { $\up$  };

  \path[name path=pathd5] (d5)--++(-90:7);
 \path[name path=bottom] (corner1)--(corner4);
 \path [name intersections={of = pathd5 and bottom}];
   \coordinate (A)  at (intersection-1);
   \path (A)--++(90:-0.1) node { $\up$  };

 \path[name path=pathd7] (d7)--++(-90:7);
 \path[name path=bottom] (corner1)--(corner4);
 \path [name intersections={of = pathd7 and bottom}];
   \coordinate (A)  at (intersection-1);
   \path (A)--++(90:-0.1) node { $\up$  };

 \path[name path=pathd9] (d9)--++(-90:7);
 \path[name path=bottom] (corner1)--(corner4);
 \path [name intersections={of = pathd9 and bottom}];
   \coordinate (A)  at (intersection-1);
   \path (A)--++(90:-0.1) node { $\up$  };

 \path[name path=pathc1] (c1)--++(-90:7);
 \path[name path=bottom] (corner1)--(corner4);
 \path [name intersections={of = pathc1 and bottom}];
   \coordinate (A)  at (intersection-1);
   \path (A)--++(90:0.1) node { $\down$  };

    \path[name path=pathc3] (c3)--++(-90:7);
 \path[name path=bottom] (corner1)--(corner4);
 \path [name intersections={of = pathc3 and bottom}];
   \coordinate (A)  at (intersection-1);
   \path (A)--++(90:0.1) node { $\down$  };

 \path[name path=pathc5] (c5)--++(-90:7);
 \path[name path=bottom] (corner1)--(corner4);
 \path [name intersections={of = pathc5 and bottom}];
   \coordinate (A)  at (intersection-1);
   \path (A)--++(90:0.1) node { $\down$  };

 \path[name path=pathc7] (c7)--++(-90:7);
 \path[name path=bottom] (corner1)--(corner4);
 \path [name intersections={of = pathc7 and bottom}];
   \coordinate (A)  at (intersection-1);
   \path (A)--++(90:0.1) node { $\down$  };

\clip(corner1)--(corner2)--(corner3)--(corner4)--(corner1);

  \foreach \i in {1,3,5,7,9,11}
{
 \draw[thick, gray](c\i)--++(90:7);
 \draw[thick, gray](c\i)--++(-90:7);
\draw[thick, gray](d\i)--++(90:7);
\draw[thick, gray](d\i)--++(-90:7);
   }

\end{scope}

\begin{scope}

\draw[ very thick, fill=gray!20](0,0) --++(45:4)--++(135:1)
--++(-135:1)--++(135:1)--++(-135:1)--++(135:2)--++(-135:1)--++(135:1)
--++(-135:1)--++(-45:5);

\clip(0,0) --++(45:4)--++(135:1)
--++(-135:1)--++(135:1)--++(-135:1)--++(135:2)--++(-135:1)--++(135:1)
--++(-135:1)--++(-45:5);
 
\path (0,0) coordinate (origin2);

 \foreach \i\j in {0,1,2,3,4,5,6,7,8,9,10,11,12}
{
\path (origin2)--++(45:0.5*\i) coordinate (a\i);
\path (origin2)--++(135:0.5*\i)  coordinate (b\j);}

 \foreach \i\j in {0,2,4,6,8,10,12}
{
\draw[line width=2  ](a\i)--++(135:14);
\draw[line width=2](b\j)--++(45:14);

\path (origin2)--++(45:0.5*\i)--++(135:14) coordinate (x\i);
\path (origin2)--++(135:0.5*\i)--++(45:14)  coordinate (y\j);

   }

\draw[  thick,magenta,gray](c1) --++(135:1) coordinate (cC1);
\draw[  thick,darkgreen,gray](cC1) --++(135:1) coordinate (cC1);
\draw[  thick,orange,gray](cC1) --++(135:1) coordinate (cC1);
\draw[  thick,lime,gray](cC1) --++(135:1) coordinate (cC1);
\draw[  thick,violet,gray](cC1) --++(135:1) coordinate (cC1);

\draw[  thick,gray,gray](c3) --++(135:1) coordinate (cC1);
\draw[  thick,magenta,gray](cC1) --++(135:1) coordinate (cC1);
\draw[  thick,darkgreen,gray](cC1) --++(135:1) coordinate (cC1);
\draw[  thick,orange,gray](cC1) --++(135:1) coordinate (cC1);
\draw[  thick,lime,gray](cC1) --++(135:1) coordinate (cC1);

  \draw[  thick,cyan,gray](c5) --++(135:1) coordinate (cC1);
 \draw[  thick,gray,gray](cC1) --++(135:1) coordinate (cC1);
\draw[  thick,magenta,gray](cC1) --++(135:1) coordinate (cC1);
\draw[  thick,darkgreen,gray](cC1) --++(135:1) coordinate (cC1);
\draw[  thick,orange,gray](cC1) --++(135:1) coordinate (cC1);
\draw[  thick,lime,gray](cC1) --++(135:1) coordinate (cC1);

  \draw[  thick,pink,gray](c7) --++(135:1) coordinate (cC1);
  \draw[  thick,cyan,gray](cC1) --++(135:1) coordinate (cC1);
 \draw[  thick,gray,gray](cC1) --++(135:1) coordinate (cC1);
\draw[  thick,magenta,gray](cC1) --++(135:1) coordinate (cC1);
\draw[  thick,darkgreen,gray](cC1) --++(135:1) coordinate (cC1);
\draw[  thick,orange,gray](cC1) --++(135:1) coordinate (cC1);
\draw[  thick,lime,gray,gray](cC1) --++(135:1) coordinate (cC1);

\draw[  thick,magenta,gray](d1) --++(45:1) coordinate (x1);
\draw[  thick,gray,gray](x1) --++(45:1) coordinate (x1);
 \draw[  thick,cyan,gray](x1) --++(45:1) coordinate (x1);
 \draw[  thick,pink,gray](x1) --++(45:1) coordinate (x1);

\draw[  thick,darkgreen,gray](d3) --++(45:1) coordinate (x1);
\draw[  thick,magenta,gray](x1) --++(45:1) coordinate (x1);
\draw[  thick,gray,gray](x1) --++(45:1) coordinate (x1);
 \draw[  thick,cyan,gray](x1) --++(45:1) coordinate (x1);
 \draw[  thick,pink,gray](x1) --++(45:1) coordinate (x1);

\draw[  thick, gray](d5) --++(45:1) coordinate (x1);
\draw[  thick, gray](x1) --++(45:1) coordinate (x1);
\draw[  thick, gray](x1) --++(45:1) coordinate (x1);
\draw[  thick, gray](x1) --++(45:1) coordinate (x1);
 \draw[  thick, ,gray](x1) --++(45:1) coordinate (x1);
 \draw[  thick ,gray](x1) --++(45:1) coordinate (x1);

\draw[  thick, gray](d7) --++(45:1) coordinate (x1);
\draw[  thick, gray](x1) --++(45:1) coordinate (x1);

\path(x1) --++(45:1) coordinate (x1);
\path(x1) --++(45:1) coordinate (x1);
\path(x1) --++(45:1) coordinate (x1);
\path(x1) --++(45:1) coordinate (x1);
\path(x1) --++(45:1) coordinate (x1);

\draw[very thick, gray](d9) --++(45:1) coordinate (x1);

\end{scope}

\path(origin2)  ++(135:2.5)   ++(-135:2.5) coordinate(corner1);
\path(origin2)  ++(45:2)   ++(135:7) coordinate(corner2);
\path(origin2)  ++(45:2)   ++(-45:2) coordinate(corner4);
\path(origin2)  ++(135:2.5)   ++(45:6.5) coordinate(corner3);

\draw[thick] (origin2)--(corner1)--(corner2)--(corner3)--(corner4)--(origin2);
\clip(origin2)--(corner1)--(corner2)--(corner3)--(corner4)--(origin2);

  \draw[ line width =2 ]
(0,0) --++(45:4)--++(135:1)--++(-135:1)--++(135:1)--++(-135:1) 
 --++(135:2)--++(-135:1) --++(135:1)--++(-135:1)--(0,0) ;
\end{tikzpicture}
$$\caption
{
Along the top we picture the weight $\down \up \down \up \up \down  \up  \down\up \down \in \Lambda_{5,5}$,
which  corresponds to the partition $(5,4,2,1)\in \mathscr{P}_{5,5}$. 
Filling each box in the partition with an $s_i$ generator constructs the corresponding coset (which when applied to the minimal element of $\Lambda_{5,5}$ permutes the $\up$s and $\down$s so as to arrive at the given weight).
}\label{Figweightpartition}
\end{figure}
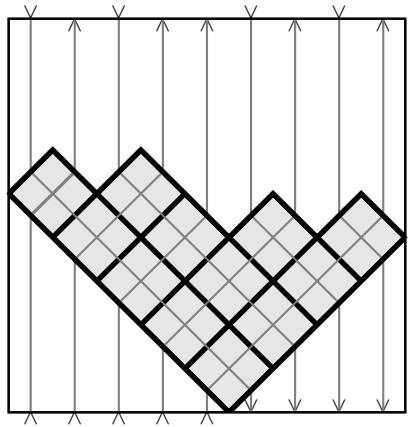

%
%
%
%

Formally, a {\sf partition}    $\lambda $ of $\ell$  is defined to be a weakly decreasing  sequence   of non-negative integers $\lambda = (\lambda_1, \lambda_2, \ldots )$ which  sum to $\ell$.  We call $\ell (\lambda) := \ell = \sum_{i}\lambda_i$ the length of the partition $\lambda$.
We define the Young diagram of a partition to be the collection of tiles 
$$[\la]=\{[r,c] \mid 1\leq c \leq \la_r\}$$
depicted in Russian style with rows at $135^\circ$ and columns at $45^\circ$.  We identify a partition with its Young diagram.
We let $\la^t$ denote the transpose partition given by reflection 
of the Russian Young diagram through the vertical axis.  
 Given  $m,n\in \NN$ we let  ${\mathscr P}_{m,n}$ denote the set of all partitions which fit into an $m\times n$ rectangle, that is 
$${\mathscr P}_{m,n}= \{ \la \mid \la_1\leq m, \la_1^t \leq n\}.$$
For $\lambda \in {\mathscr P}_{m,n}$, the $x$-coordinate of a tile $[r,c] \in \lambda$ is   equal to   $r-c \in \{-m+1, -m+2 , \ldots , n-2, n-1\}$ and we define this $x$-coordinate to be  the {\sf content}  of the tile and we  write ${\sf cont}[r,c]=r-c$.  
Given $\lambda \in {\mathscr P}_{m,n}$, we define the set ${\rm Add}(\lambda)$ to be the set of all tiles $[r,c]\notin \lambda$ such that $\lambda \cup [r,c]\in {\mathscr P}_{m,n}$. Similarly, we define the set ${\rm Rem}(\lambda)$ to be the set of all tiles $[r,c]\in \lambda$ such that $\lambda \setminus [r,c] \in {\mathscr P}_{m,n}$. 

 The following definitions come from \cite{MR2918294}.
    
 \begin{defn} 
\begin{itemize}[leftmargin=*]
 \item To each weight $\lambda$ we associate a {\sf cup diagram} $\underline{\lambda}$ and a {\sf cap diagram} $\overline{\lambda}$. To construct $\underline{\lambda}$,  
repeatedly find a pair of vertices labeled $\down$ $\up$ in order from left to right that are neighbours in the sense that there are only vertices already joined by cups in between. Join these new vertices together with a cup. Then repeat the process until there are no more such $\down$ $\up$ pairs.
Finally draw south-westerly rays   at all the remaining $\up$  vertices
  and south-easterly rays  at all the remaining  $\down$ vertices.
   The cap diagram $\overline{\lambda}$ is obtained by flipping $\underline{\lambda}$ horizontally. We stress that the vertices of the cup and cap diagrams  are not labeled.

\item Let $\lambda$ and $\mu$ be weights. We can glue $\underline{\mu}$ under $\lambda$ to obtain a new diagram $\underline{\mu}\lambda$. We say that $\underline{\mu}\lambda$ is  {\sf oriented} if (i) the vertices at the ends of each cup in $\underline{\mu}$ are labelled by exactly one $\down$  and one $\up$ in the weight $\lambda$ and (ii) it is impossible to find two rays in $\underline{\mu}$ whose top vertices are labeled $\down$ $\up$  in that order from left to right in the weight $\lambda$. 
Similarly, we obtain a new diagram $\lambda \overline{\mu}$ by gluing $\overline{\mu}$ on top of $\lambda$. We say that $\lambda \overline{\mu}$ is oriented if $\underline{\mu} \lambda$ is oriented. 
\item Let $\lambda$, $\mu$ be weights such that $\underline{\mu}\lambda$ is oriented. We set the {\sf degree} of the diagram $\underline{\mu}\lambda$ (respectively $\lambda \overline{\mu}$) to the the number of clockwise oriented cups (respectively caps) in the diagram. 
\item Let $\la, \mu ,\nu$ be weights such that $\underline{\mu}\la$ and $\la \overline{\nu}$ are oriented. Then  we form a new diagram $\underline{\mu}\la\overline{\nu}$ by gluing $\underline{\mu}$ under and $\overline{\nu}$ on top of $\la$. We set ${\rm deg}(\underline{\mu}\la \overline{\nu}) = {\rm deg}(\underline{\mu}\la)+{\rm deg}(\la \overline{\nu})$.
\end{itemize}
\end{defn}

An example is provided in Figure \ref{figure4}.

For the purposes of this paper, for $p\geq 0$, we can define the $p$-Kazhdan--Lusztig polynomials of type $(W,P) = (S_{n+m},S_m \times S_n)$ as follows.  
For $\la,  \mu \in \mptn$ we set
$$
{^p}n_{\la,\mu}(q)= 
\begin{cases}
q^{\deg(\underline{\mu} \la)}		&\text{if $ \underline{\mu} \la $ is oriented}\\
0						&\text{otherwise.}
\end{cases}
$$
We refer to \cite[Theorem 7.3]{compan2} and \cite[Theorem A]{compan} for a justification of this definition and to \cite{MR2918294} for the origins of this combinatorics.

  \begin{figure}[ht!]
  $$   \begin{tikzpicture} [scale=0.85]
		
		
		\path (4,1) coordinate (origin); 
		\path (origin)--++(0.5,0.5) coordinate (origin2);  
	 	\draw(origin2)--++(0:6); 
		\foreach \i in {1,2,3,4,5,...,10,11}
		{
			\path (origin2)--++(0:0.5*\i) coordinate (a\i); 
			\path (origin2)--++(0:0.5*\i)--++(-90:0.00) coordinate (c\i); 
			  }
		
		\foreach \i in {1,2,3,4,5,...,19}
		{
			\path (origin2)--++(0:0.25*\i) --++(-90:0.5) coordinate (b\i); 
			\path (origin2)--++(0:0.25*\i) --++(-90:0.9) coordinate (d\i); 
		}
		\path(a1) --++(90:0.12) node  {  $  \down   $} ;
		\path(a3) --++(90:0.12) node  {  $  \down   $} ;
		\path(a2) --++(-90:0.15) node  {  $  \up   $} ;
		\path(a4) --++(-90:0.15) node  {  $  \up   $} ;
		\path(a5) --++(-90:0.15) node  {  $  \up  $} ;
		\path(a6) --++(90:0.12) node  {  $  \down  $} ;
		\path(a7) --++(90:0.12) node  {  $  \down  $} ;
		\path(a8) --++(-90:0.15) node  {  $  \up  $} ;
		\path(a9) --++(-90:0.15) node  {  $  \up  $} ;
 \path(a10) --++(90:0.12) node  {  $  \down  $} ; 
 \path(a11) --++(90:0.12) node  {  $  \down  $} ; 		
		
		\draw[    thick](c2) to [out=-90,in=0] (b3) to [out=180,in=-90] (c1); 
		\draw[    thick](c4) to [out=-90,in=0] (b7) to [out=180,in=-90] (c3);


		\draw[    thick](c8) to [out=-90,in=0] (b15) to [out=180,in=-90] (c7); 
 	
			\draw[    thick](c9) to [out=-90,in=0] (d15) to [out=180,in=-90] (c6);

			\path(a1) --++(-90:1.25) --++(180:0.5)	coordinate (c1);
		\draw[    thick](c5) --++(90:-0.25) to [out=-90,in=0] (c1);

	\path(a11) --++(-90:0.7) --++(0:0.5)	coordinate (c1);
		
				\draw[    thick](c11) --++(90:-0.2) to [out=-90,in=180] (c1);  

			\path(a11) --++(-90:1.25) --++(0:0.5)	coordinate (c1);
		
				\draw[    thick](c10) --++(90:-0.4) to [out=-90,in=180] (c1);

	\end{tikzpicture}\qquad\quad
	 \begin{tikzpicture} [scale=0.85]
		
		
		\path (4,1) coordinate (origin); 
		\path (origin)--++(0.5,0.5) coordinate (origin2);  
	 	\draw(origin2)--++(0:6); 
		\foreach \i in {1,2,3,4,5,...,10,11}
		{
			\path (origin2)--++(0:0.5*\i) coordinate (a\i); 
			\path (origin2)--++(0:0.5*\i)--++(-90:0.00) coordinate (c\i); 
			  }
		
		\foreach \i in {1,2,3,4,5,...,19}
		{
			\path (origin2)--++(0:0.25*\i) --++(-90:0.5) coordinate (b\i); 
			\path (origin2)--++(0:0.25*\i) --++(-90:0.9) coordinate (d\i); 
		}
		
		\draw[    thick](c2) to [out=-90,in=0] (b3) to [out=180,in=-90] (c1); 
		\draw[    thick](c4) to [out=-90,in=0] (b7) to [out=180,in=-90] (c3);


		\draw[    thick](c8) to [out=-90,in=0] (b15) to [out=180,in=-90] (c7); 
 	
			\draw[    thick](c9) to [out=-90,in=0] (d15) to [out=180,in=-90] (c6);

			\path(a1) --++(-90:1.25) --++(180:0.5)	coordinate (c1);
		\draw[    thick](c5) --++(90:-0.25) to [out=-90,in=0] (c1);

	\path(a11) --++(-90:0.7) --++(0:0.5)	coordinate (c1);
		
				\draw[    thick](c11) --++(90:-0.2) to [out=-90,in=180] (c1);  

			\path(a11) --++(-90:1.25) --++(0:0.5)	coordinate (c1);
		
				\draw[    thick](c10) --++(90:-0.4) to [out=-90,in=180] (c1);

	\end{tikzpicture}$$
\caption{The construction of the cup diagram $\underline{\la}$ for $\la=(5,4,2^2) \in \mathscr{P}_{5,6}$. }
\label{figure4}
	\end{figure}
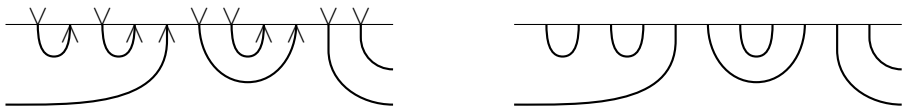
  
 It is clear that for a fixed $\mu\in \mptn$, the diagram $\underline{\mu}\lambda$ is oriented if and only if the weight $\lambda$ is obtained from the weight $\mu$ by swapping the labels on some of the pairs of vertices connected by a cup in $\underline{\mu}$. Moreover, in this case the degree of $\underline{\mu}\la$ is precisely the number of such swapped pairs. 

     We define the {\sf defect}  of  $\la  \in \mathscr{P}_{m,n}$,
  to be 
$d(\la)=d - m \in \ZZ_{\leq 0}$ if $(d,d-1,d-2,\dots, 1)\subseteq \la$ but
  $(d+1,d,d-1,\dots, 1)\not\subseteq \la$. 
  In other words, the defect records the  largest staircase partition sitting inside of $\la$ (relative to the largest staircase partition $(m,m-1,m-2,\dots,1)\in \mptn$).  
  We say that a partition $\la\in \mathscr{P}_{m,n}$ is {\sf regular} if $d(\la)=0$.
  We let $\mathscr{R}_{m,n}\subseteq  \mathscr{P}_{m,n}$ denote the subset of regular partitions. 
  It follows from definitions that \(\la \in \mptn\) has defect 
  \(d(\la)\leq 0\) if and only if \(\underline{\la} \la\) has precisely  
    \(|d(\la)|\)
   vertices labelled by   $\up$  connected to south-westerly rays.
 For example 
 the partitions $\la=(5,4,2^2)$ has defect $-1$, as seen in \cref{figure4}.

     \section{ Dyck combinatorics}
\label{dyckgens}
We have defined the $p$-Kazhdan--Lusztig polynomials via
 counting of certain 
oriented   diagrams.  
 For the purposes of this paper, we require richer combinatorial objects which {\em refine} the diagrammatic construction: these are provided by tilings by Dyck paths.  
  
 Let us start with a simple example to see how these Dyck paths come from the  oriented   diagrams. Consider the partitions $\mu = (5^3,4, 1)$ and $\la = (4^2,3,1^2)$. The oriented   diagrams $  \underline{\mu} \mu $ and $\underline{\la}\la$  are illustrated in   \cref{example-2}. We see that $\la$ is obtained from $\mu$ by swapping the labels of the vertices of one cup, $p \in \underline{\mu}$. 
The partition $\la$ is obtained from the partition $\mu$ by removing a corresponding Dyck path $P \subseteq  \mu $.
More generally, if $\la, \mu\in \mptn$ with $\underline{\mu}\la$ oriented of degree $k$, then we will see that the partition $\la$ is obtained from the partition $\mu$ by removing $k$ Dyck paths.

\begin{figure}[ht!]
$$
\begin{tikzpicture}[scale=0.6]

 \path(0,0) coordinate (origin2);
     \path(0,0)--++(135:2.5)--++(-135:2.5)
     --++(-45:0.25)--++(45:0.25)
          --++(-90:0.18)  node[gray]  {$\up$};

 \foreach \i in {5,6,7,8,9}
{    \path(0,0)--++(135:2.5-0.5*\i)--++(-135:2.5-0.5*\i)
     --++(-45:0.25)--++(45:0.25)
          --++(90:0.16)  node[gray]  {$\down$}; }

\foreach \i in {1,2,3,4}
{    \path(0,0)--++(135:2.5-0.5*\i)--++(-135:2.5-0.5*\i)
     --++(-45:0.25)--++(45:0.25)
          --++(-90:0.18)  node[gray]  {$\up$};
 
}

     \path(0,0)--++(135:4)--++(45:6)
     --++(-45:0.25)--++(45:0.25)
          --++(-90:0.18)  node {$\up$};
          
              \path(0,0)--++(135:4)--++(45:6)
     --++(-45:0.25)--++(45:0.25)
          --++(45:0.5*2)--++(-45:0.5*2)
          --++(-90:0.18)  node[gray] {$\up$};

                  \path(0,0)--++(135:4)--++(45:6)
     --++(-45:0.25)--++(45:0.25)
          --++(45:0.5*1)--++(-45:0.5*1)
     --++(90:0.16)  node[gray] {$\down$};

                  \path(0,0)--++(135:4)--++(45:6)
     --++(-45:0.25)--++(45:0.25)
 --++(-135:0.5*3)          --++(135:0.5*3)
     --++(90:0.16)  node[gray] {$\down$};
     
         \path(0,0)--++(135:4)--++(45:6)
     --++(-45:0.25)--++(45:0.25)
 --++(-135:0.5*5)          --++(135:0.5*5)
     --++(90:0.16)  node[gray] {$\down$};

         \path(0,0)--++(135:4)--++(45:6)
     --++(-45:0.25)--++(45:0.25)
 --++(-135:0.5*6)          --++(135:0.5*6)
     --++(90:0.16)  node[gray] {$\down$};

         \path(0,0)--++(135:4)--++(45:6)
     --++(-45:0.25)--++(45:0.25)
 --++(-135:0.5*7)          --++(135:0.5*7)
     --++(90:0.16)  node {$\down$};

                   \path(0,0)--++(135:4)--++(45:6)
     --++(-45:0.25)--++(45:0.25)
          --++(-135:0.5*2)--++(135:0.5*2)
          --++(-90:0.18)  node[gray] {$\up$};
          
               \path(0,0)--++(135:4)--++(45:6)
     --++(-45:0.25)--++(45:0.25)
          --++(-135:0.5*1)--++(135:0.5*1)
          --++(-90:0.18)  node[gray] {$\up$};
          
     \path(0,0)--++(135:4)--++(45:6)
     --++(-45:0.25)--++(45:0.25)
          --++(-135:0.5*4)--++(135:0.5*4)
          --++(-90:0.18)  node[gray] {$\up$};

     \draw[very thick] (0,0)--++(135:5)--++(45:3)--++(-45:1)--++(45:1)--++(-45:3)
     --++(45:1)--++(-45:1)--(0,0);

         \path(0,0)--++(135:4) coordinate (high);
  
           \fill[magenta,opacity=0.3] (high)--++(135:1)--++(45:3)--++(-45:1)
         --++(45:1) --++(-45:3
         )
          --++(-135:1)--++(135:2)--++(-135:1) --++(135:1)          --++(-135:2);

   \path(0,0)--++(45:2.75)--++(-45:2.75) coordinate (bottomR);
   \path(0,0)--++(135:2.75)--++(-135:2.75) coordinate (bottomL);
  \draw[very thick] (bottomL)--(bottomR);
    \draw[very thick,densely dotted] (bottomL)--++(180:0.5);

       \path(0,0)
       --++(45:2.75)--++(-45:2.75)--++(135:5)--++(45:5) coordinate (topR);
   \path(0,0)--++(135:2.75)--++(-135:2.75)--++(45:5)--++(135:5) coordinate (topL);
  \draw[very thick] (topL)--(topR);
    \draw[very thick,densely dotted] (topL)--++(180:0.5);
     \draw[very thick,densely dotted] (topR)--++( 0:0.5);


     \path(0,0)--++(135:4)--++(45:6)
     --++(-45:0.25)--++(45:0.25)
          --++(-90:0.18)  node {$\up$};
          
              \path(0,0)--++(135:4)--++(45:6)
     --++(-45:0.25)--++(45:0.25)
          --++(45:0.5*2)--++(-45:0.5*2)
          --++(-90:0.18)  node[gray] {$\up$};

                  \path(0,0)--++(135:4)--++(45:6)
     --++(-45:0.25)--++(45:0.25)
          --++(45:0.5*1)--++(-45:0.5*1)
     --++(90:0.16)  node[gray] {$\down$};

                  \path(0,0)--++(135:4)--++(45:6)
     --++(-45:0.25)--++(45:0.25)
 --++(-135:0.5*3)          --++(135:0.5*3)
     --++(90:0.16)  node[gray] {$\down$};
     
         \path(0,0)--++(135:4)--++(45:6)
     --++(-45:0.25)--++(45:0.25)
 --++(-135:0.5*5)          --++(135:0.5*5)
     --++(90:0.16)  node[gray] {$\down$};

         \path(0,0)--++(135:4)--++(45:6)
     --++(-45:0.25)--++(45:0.25)
 --++(-135:0.5*6)          --++(135:0.5*6)
     --++(90:0.16)  node[gray] {$\down$};

         \path(0,0)--++(135:4)--++(45:6)
     --++(-45:0.25)--++(45:0.25)
 --++(-135:0.5*7)          --++(135:0.5*7)
     --++(90:0.16)  node {$\down$};

                   \path(0,0)--++(135:4)--++(45:6)
     --++(-45:0.25)--++(45:0.25)
          --++(-135:0.5*2)--++(135:0.5*2)
          --++(-90:0.18)  node[gray] {$\up$};
          
               \path(0,0)--++(135:4)--++(45:6)
     --++(-45:0.25)--++(45:0.25)
          --++(-135:0.5*1)--++(135:0.5*1)
          --++(-90:0.18)  node[gray] {$\up$};
          
     \path(0,0)--++(135:4)--++(45:6)
     --++(-45:0.25)--++(45:0.25)
          --++(-135:0.5*4)--++(135:0.5*4)
          --++(-90:0.18)  node[gray] {$\up$};

 \clip(bottomL)--(bottomR)--(topR)--(topL);

 \path(0,0) coordinate (start);
   
     \path(start)--++(45:0.5) coordinate (X1);
     \path(start)--++(135:0.5) coordinate (X2);     
     \draw[very thick,gray!70] (X1) to [out=135,in= 45] (X2);
  \path(X1)--++(45:0.5)--++(135:0.5) coordinate (X3);
     \path(X2)--++(45:0.5)--++(135:0.5) coordinate (X4);
     \draw[very thick,gray!70] (X3) to [out=-135,in= -45] (X4);; 
 
     \draw[very thick,gray!70] (X1)--++(-90:7);
     \draw[very thick,gray!70] (X2)--++(-90:7); 
 
 \path(start)--++(45:1) coordinate (start);
   
     \path(start)--++(45:0.5) coordinate (X1);
     \path(start)--++(135:0.5) coordinate (X2);     
     \draw[very thick,gray!70] (X1) to [out=135,in= 45] (X2);
  \path(X1)--++(45:0.5)--++(135:0.5) coordinate (X3) ;
     \path(X2)--++(45:0.5)--++(135:0.5) coordinate (X4);
     \draw[very thick,gray!70] (X3) to [out=-135,in= -45] (X4);; 
 
    \draw[very thick,gray!70] (X1)--++(-90:7);

 \path(start)--++(45:1) coordinate (start);
   
     \path(start)--++(45:0.5) coordinate (X1);
     \path(start)--++(135:0.5) coordinate (X2);     
     \draw[very thick,gray!70] (X1) to [out=135,in= 45] (X2);
  \path(X1)--++(45:0.5)--++(135:0.5) coordinate (X3) ;
     \path(X2)--++(45:0.5)--++(135:0.5) coordinate (X4);
     \draw[very thick,gray!70] (X3) to [out=-135,in= -45] (X4);; 
 
    \draw[very thick,gray!70] (X1)--++(-90:7);

 \path(start)--++(45:1) coordinate (start);
   
     \path(start)--++(45:0.5) coordinate (X1);
     \path(start)--++(135:0.5) coordinate (X2);     
     \draw[very thick,gray!70] (X1) to [out=135,in= 45] (X2);
  \path(X1)--++(45:0.5)--++(135:0.5) coordinate (X3) ;
     \path(X2)--++(45:0.5)--++(135:0.5) coordinate (X4);
     \draw[very thick,gray!70] (X3) to [out=-135,in= -45] (X4);; 
 
    \draw[very thick,gray!70] (X1)--++(-90:7);

 \path(start)--++(45:1) coordinate (start);
   
     \path(start)--++(45:0.5) coordinate (X1);
     \path(start)--++(135:0.5) coordinate (X2);     
     \draw[very thick,gray!70] (X1) to [out=135,in= 45] (X2);
  \path(X1)--++(45:0.5)--++(135:0.5) coordinate (X3) ;
     \path(X2)--++(45:0.5)--++(135:0.5) coordinate (X4);
     \draw[very thick,gray!70] (X3) to [out=-135,in= -45] (X4);; 
 
    \draw[very thick,gray!70] (X1)--++(-90:7);

    \draw[very thick,gray!70] (X3)--++(90:7); 

    \draw[very thick,gray!70] (X4)--++(90:7);


 \path(0,0) coordinate (start);
   
     \path(start)--++(45:0.5) coordinate (X1);
     \path(start)--++(135:0.5) coordinate (X2);     
     \draw[very thick,gray!70] (X1) to [out=135,in= 45] (X2);
  \path(X1)--++(45:0.5)--++(135:0.5) coordinate (X3);
     \path(X2)--++(45:0.5)--++(135:0.5) coordinate (X4);
     \draw[very thick,gray!70] (X3) to [out=-135,in= -45] (X4);; 
 
     \draw[very thick,gray!70] (X1)--++(-90:7);
     \draw[very thick,gray!70] (X2)--++(-90:7);

 \path(start)--++(135:1) coordinate (start);
   
     \path(start)--++(45:0.5) coordinate (X1);
     \path(start)--++(135:0.5) coordinate (X2);     
     \draw[very thick,gray!70] (X1) to [out=135,in= 45] (X2);
  \path(X1)--++(45:0.5)--++(135:0.5) coordinate (X3) ;
     \path(X2)--++(45:0.5)--++(135:0.5) coordinate (X4);
     \draw[very thick,gray!70] (X3) to [out=-135,in= -45] (X4);; 
 
    \draw[very thick,gray!70] (X2)--++(-90:7);

 \path(start)--++(135:1) coordinate (start);
   
     \path(start)--++(45:0.5) coordinate (X1);
     \path(start)--++(135:0.5) coordinate (X2);     
     \draw[very thick,gray!70] (X1) to [out=135,in= 45] (X2);
  \path(X1)--++(45:0.5)--++(135:0.5) coordinate (X3) ;
     \path(X2)--++(45:0.5)--++(135:0.5) coordinate (X4);
     \draw[very thick,gray!70] (X3) to [out=-135,in= -45] (X4);; 
 
    \draw[very thick,gray!70] (X2)--++(-90:7);

 \path(start)--++(135:1) coordinate (start);
   
     \path(start)--++(45:0.5) coordinate (X1);
     \path(start)--++(135:0.5) coordinate (X2);     
     \draw[very thick,gray!70] (X1) to [out=135,in= 45] (X2);
  \path(X1)--++(45:0.5)--++(135:0.5) coordinate (X3) ;
     \path(X2)--++(45:0.5)--++(135:0.5) coordinate (X4);
     \draw[very thick,gray!70] (X3) to [out=-135,in= -45] (X4);; 
 
    \draw[very thick,gray!70] (X2)--++(-90:7);

 \path(start)--++(135:1) coordinate (start);
   
     \path(start)--++(45:0.5) coordinate (X1);
     \path(start)--++(135:0.5) coordinate (X2);     
     \draw[very thick,gray!70] (X1) to [out=135,in= 45] (X2);
  \path(X1)--++(45:0.5)--++(135:0.5) coordinate (X3) ;
     \path(X2)--++(45:0.5)--++(135:0.5) coordinate (X4);
     \draw[very thick,black] (X3) to [out=-135,in= -45] (X4);; 
 
    \draw[very thick,gray!70] (X2)--++(-90:7);
    \draw[very thick,black] (X4)--++(90:7);


 \path(0,0)--++(45:1)--++(135:1) coordinate (start);
   
     \path(start)--++(45:0.5) coordinate (X1);
     \path(start)--++(135:0.5) coordinate (X2);     
     \draw[very thick,gray!70] (X1) to [out=135,in= 45] (X2);
  \path(X1)--++(45:0.5)--++(135:0.5) coordinate (X3) ;
     \path(X2)--++(45:0.5)--++(135:0.5) coordinate (X4);
     \draw[very thick,gray!70] (X3) to [out=-135,in= -45] (X4);;

 \path(start)--++(45:1)  coordinate (start);
   
     \path(start)--++(45:0.5) coordinate (X1);
     \path(start)--++(135:0.5) coordinate (X2);     
     \draw[very thick,gray!70] (X1) to [out=135,in= 45] (X2);
  \path(X1)--++(45:0.5)--++(135:0.5) coordinate (X3) ;
     \path(X2)--++(45:0.5)--++(135:0.5) coordinate (X4);
     \draw[very thick,gray!70] (X3) to [out=-135,in= -45] (X4);;

 \path(start)--++(45:1)  coordinate (start);
   
     \path(start)--++(45:0.5) coordinate (X1);
     \path(start)--++(135:0.5) coordinate (X2);     
     \draw[very thick,gray!70] (X1) to [out=135,in= 45] (X2);
  \path(X1)--++(45:0.5)--++(135:0.5) coordinate (X3) ;
     \path(X2)--++(45:0.5)--++(135:0.5) coordinate (X4);
     \draw[very thick,black] (X3) to [out=-135,in= -45] (X4);;

    \draw[very thick,black] (X3)--++(90:7);


 \path(0,0)--++(45:1)--++(135:1) coordinate (start);
   
     \path(start)--++(45:0.5) coordinate (X1);
     \path(start)--++(135:0.5) coordinate (X2);     
     \draw[very thick,gray!70] (X1) to [out=135,in= 45] (X2);
  \path(X1)--++(45:0.5)--++(135:0.5) coordinate (X3) ;
     \path(X2)--++(45:0.5)--++(135:0.5) coordinate (X4);
     \draw[very thick,gray!70] (X3) to [out=-135,in= -45] (X4);;

 \path(start)--++(135:1)  coordinate (start);
   
     \path(start)--++(45:0.5) coordinate (X1);
     \path(start)--++(135:0.5) coordinate (X2);     
     \draw[very thick,gray!70] (X1) to [out=135,in= 45] (X2);
  \path(X1)--++(45:0.5)--++(135:0.5) coordinate (X3) ;
     \path(X2)--++(45:0.5)--++(135:0.5) coordinate (X4);
     \draw[very thick,gray!70] (X3) to [out=-135,in= -45] (X4);; 
 
%

 \path(start)--++(135:1)  coordinate (start);
   
     \path(start)--++(45:0.5) coordinate (X1);
     \path(start)--++(135:0.5) coordinate (X2);     
     \draw[very thick,gray!70] (X1) to [out=135,in= 45] (X2);
  \path(X1)--++(45:0.5)--++(135:0.5) coordinate (X3) ;
     \path(X2)--++(45:0.5)--++(135:0.5) coordinate (X4);
     \draw[very thick,black] (X3) to [out=-135,in= -45] (X4);;

 \path(start)--++(135:1)  coordinate (start2);
   
     \path(start2)--++(45:0.5) coordinate (X1);
     \path(start2)--++(135:0.5) coordinate (X2);     
     \draw[very thick,black ] (X1) to [out=135,in= 45] (X2);
  \path(X1)--++(45:0.5)--++(135:0.5) coordinate (X3) ;
     \path(X2)--++(45:0.5)--++(135:0.5) coordinate (X4);
     \draw[very thick ,gray!70] (X3) to [out=-135,in= -45] (X4);; 
     \draw[very thick,gray!70] (X4)--++(90:7);

 \path(start)--++(-45:1)--++(45:1)  coordinate (start);
   
     \path(start)--++(45:0.5) coordinate (X1);
     \path(start)--++(135:0.5) coordinate (X2);     
     \draw[very thick,gray!70] (X1) to [out=135,in= 45] (X2);
  \path(X1)--++(45:0.5)--++(135:0.5) coordinate (X3) ;
     \path(X2)--++(45:0.5)--++(135:0.5) coordinate (X4);
     \draw[very thick,black] (X3) to [out=-135,in= -45] (X4);;

 \path(start)--++(135:1)  coordinate (start2);
   
     \path(start2)--++(45:0.5) coordinate (X1);
     \path(start2)--++(135:0.5) coordinate (X2);     
     \draw[very thick,black ] (X1) to [out=135,in= 45] (X2);
  \path(X1)--++(45:0.5)--++(135:0.5) coordinate (X3) ;
     \path(X2)--++(45:0.5)--++(135:0.5) coordinate (X4);
     \draw[very thick ,gray!70] (X3) to [out=-135,in= -45] (X4);;

 \path(start)--++(45:1)  coordinate (start2);
   
     \path(start2)--++(45:0.5) coordinate (X1);
     \path(start2)--++(135:0.5) coordinate (X2);     
     \draw[very thick,black ] (X1) to [out=135,in= 45] (X2);
  \path(X1)--++(45:0.5)--++(135:0.5) coordinate (X3) ;
     \path(X2)--++(45:0.5)--++(135:0.5) coordinate (X4);
     \draw[very thick ,gray!70] (X3) to [out=-135,in= -45] (X4);; 
     \draw[very thick,gray!70] (X3)--++(90:7); 

 \path(start2)--++(135:1)  coordinate (start3);
   
     \path(start3)--++(45:0.5) coordinate (X1);
     \path(start3)--++(135:0.5) coordinate (X2);     
     \draw[very thick,gray!70 ] (X1) to [out=135,in= 45] (X2);
  \path(X1)--++(45:0.5)--++(135:0.5) coordinate (X3) ;
     \path(X2)--++(45:0.5)--++(135:0.5) coordinate (X4);
     \draw[very thick ,gray!70] (X3) to [out=-135,in= -45] (X4);; 
     \draw[very thick,gray!70] (X4)--++(90:7); 
     \draw[very thick,gray!70] (X3)--++(90:7);

 \path(start3)--++(135:1)--++(-135:1)  coordinate (start4);
   
     \path(start4)--++(45:0.5) coordinate (X1);
     \path(start4)--++(135:0.5) coordinate (X2);     
     \draw[very thick,gray!70 ] (X1) to [out=135,in= 45] (X2);
  \path(X1)--++(45:0.5)--++(135:0.5) coordinate (X3) ;
     \path(X2)--++(45:0.5)--++(135:0.5) coordinate (X4);
     \draw[very thick ,gray!70] (X3) to [out=-135,in= -45] (X4);; 
     \draw[very thick,gray!70] (X4)--++(90:7); 
     \draw[very thick,gray!70] (X3)--++(90:7);

\clip(0,0)--++(135:5)--++(45:3)--++(-45:1)
   --++(45:1)--++(-45:3)--++(45:1)
   --++(-45:1) 
 --(0,0);

       \foreach \i in {0,1,2,3,4,5,6,7,8,9,10,11,12,13}
{
\path (origin2)--++(45:0.5*\i) coordinate (c\i); 
\path (origin2)--++(135:0.5*\i)  coordinate (d\i); 
  }

   \foreach \i in {0,1,2,3,4,5,6,7,8,9,10,11,12,13}
{
\path (origin2)--++(45:1*\i) coordinate (c\i); 
\path (origin2)--++(135:1*\i)  coordinate (d\i); 
\draw[thick,densely dotted] (c\i)--++(135:14);
\draw[thick,densely dotted] (d\i)--++(45:14);
  }

\end{tikzpicture} \qquad\qquad 
\begin{tikzpicture}[scale=0.6]

     \path(0,0)--++(135:2.5)--++(-135:2.5)
     --++(-45:0.25)--++(45:0.25)
          --++(-90:0.18)  node   {$\up$};

 \foreach \i in {5,6,7,8}
{    \path(0,0)--++(135:2.5-0.5*\i)--++(-135:2.5-0.5*\i)
     --++(-45:0.25)--++(45:0.25)
          --++(90:0.16)  node[gray]  {$\down$}; }

 \path(0,0)--++(135:2.5-0.5*9)--++(-135:2.5-0.5*9)
     --++(-45:0.25)--++(45:0.25)
          --++(90:0.16)  node   {$\down$};

\foreach \i in {1,2,3,4}
{    \path(0,0)--++(135:2.5-0.5*\i)--++(-135:2.5-0.5*\i)
     --++(-45:0.25)--++(45:0.25)
          --++(-90:0.18)  node[gray]  {$\up$};
 
}

     \path(0,0)--++(135:4)--++(45:6)
     --++(-45:0.25)--++(45:0.25)
      --++(90:0.16)  node {$\down$};
          
              \path(0,0)--++(135:4)--++(45:6)
     --++(-45:0.25)--++(45:0.25)
          --++(45:0.5*2)--++(-45:0.5*2)
          --++(-90:0.18)  node[gray] {$\up$};

                  \path(0,0)--++(135:4)--++(45:6)
     --++(-45:0.25)--++(45:0.25)
          --++(45:0.5*1)--++(-45:0.5*1)
     --++(90:0.16)  node[gray] {$\down$};

                  \path(0,0)--++(135:4)--++(45:6)
     --++(-45:0.25)--++(45:0.25)
 --++(-135:0.5*3)          --++(135:0.5*3)
     --++(90:0.16)  node[gray] {$\down$};
     
         \path(0,0)--++(135:4)--++(45:6)
     --++(-45:0.25)--++(45:0.25)
 --++(-135:0.5*5)          --++(135:0.5*5)
     --++(90:0.16)  node[gray] {$\down$};

         \path(0,0)--++(135:4)--++(45:6)
     --++(-45:0.25)--++(45:0.25)
 --++(-135:0.5*6)          --++(135:0.5*6)
     --++(90:0.16)  node[gray] {$\down$};

         \path(0,0)--++(135:4)--++(45:6)
     --++(-45:0.25)--++(45:0.25)
 --++(-135:0.5*7)          --++(135:0.5*7)
     --++(-90:0.18)  node {$\up$}--++(90:0.18) coordinate (hereitis);

\path(hereitis)--++(-45:5)--++(-135:5) coordinate (hereitsisnt);
     \draw[very thick] (hereitis)--(hereitsisnt);

                   \path(0,0)--++(135:4)--++(45:6)
     --++(-45:0.25)--++(45:0.25)
          --++(-135:0.5*2)--++(135:0.5*2)
          --++(-90:0.18)  node[gray] {$\up$};
          
               \path(0,0)--++(135:4)--++(45:6)
     --++(-45:0.25)--++(45:0.25)
          --++(-135:0.5*1)--++(135:0.5*1)
          --++(-90:0.18)  node[gray] {$\up$};
          
     \path(0,0)--++(135:4)--++(45:6)
     --++(-45:0.25)--++(45:0.25)
          --++(-135:0.5*4)--++(135:0.5*4)
          --++(-90:0.18)  node[gray] {$\up$};

    \path (0,0)--++(135:4.2)--++(45:1.2) coordinate (start);


    \fill[magenta,opacity=0.3,rounded corners](start)--++(-135:1)--++(135:0.6)--++(45:3-0.4)
    --++(-45:1)    --++(45:1)    --++(-45:3-0.4)
        --++(-135:1-0.4)    --++(135:2)
            --++(-135:1)    --++(135:1)    --++(-135:1);

     \draw[very thick] (0,0)--++(135:4)--++(45:2)--++(-45:1)--++(45:1)--++(-45:2)
     --++(45:2)--++(-45:1)--(0,0);

   \path(0,0)--++(45:2.75)--++(-45:2.75) coordinate (bottomR);
   \path(0,0)--++(135:2.75)--++(-135:2.75) coordinate (bottomL);
  \draw[very thick] (bottomL)--(bottomR);
    \draw[very thick,densely dotted] (bottomL)--++(180:0.5);

       \path(0,0)
       --++(45:2.75)--++(-45:2.75)--++(135:5)--++(45:5) coordinate (topR);
   \path(0,0)--++(135:2.75)--++(-135:2.75)--++(45:5)--++(135:5) coordinate (topL);
  \draw[very thick] (topL)--(topR);
    \draw[very thick,densely dotted] (topL)--++(180:0.5);
     \draw[very thick,densely dotted] (topR)--++( 0:0.5);

 \clip(bottomL)--(bottomR)--(topR)--(topL);

 \path(0,0) coordinate (start);
   
     \path(start)--++(45:0.5) coordinate (X1);
     \path(start)--++(135:0.5) coordinate (X2);     
     \draw[very thick,gray!70] (X1) to [out=135,in= 45] (X2);
  \path(X1)--++(45:0.5)--++(135:0.5) coordinate (X3);
     \path(X2)--++(45:0.5)--++(135:0.5) coordinate (X4);
     \draw[very thick,gray!70] (X3) to [out=-135,in= -45] (X4);; 
 
     \draw[very thick,gray!70] (X1)--++(-90:7);
     \draw[very thick,gray!70] (X2)--++(-90:7); 
 
 \path(start)--++(45:1) coordinate (start);
   
     \path(start)--++(45:0.5) coordinate (X1);
     \path(start)--++(135:0.5) coordinate (X2);     
     \draw[very thick,gray!70] (X1) to [out=135,in= 45] (X2);
  \path(X1)--++(45:0.5)--++(135:0.5) coordinate (X3) ;
     \path(X2)--++(45:0.5)--++(135:0.5) coordinate (X4);
     \draw[very thick,gray!70] (X3) to [out=-135,in= -45] (X4);; 
 
    \draw[very thick,gray!70] (X1)--++(-90:7);

 \path(start)--++(45:1) coordinate (start);
   
     \path(start)--++(45:0.5) coordinate (X1);
     \path(start)--++(135:0.5) coordinate (X2);     
     \draw[very thick,gray!70] (X1) to [out=135,in= 45] (X2);
  \path(X1)--++(45:0.5)--++(135:0.5) coordinate (X3) ;
     \path(X2)--++(45:0.5)--++(135:0.5) coordinate (X4);
     \draw[very thick,gray!70] (X3) to [out=-135,in= -45] (X4);; 
 
    \draw[very thick,gray!70] (X1)--++(-90:7);

 \path(start)--++(45:1) coordinate (start);
   
     \path(start)--++(45:0.5) coordinate (X1);
     \path(start)--++(135:0.5) coordinate (X2);     
     \draw[very thick,gray!70] (X1) to [out=135,in= 45] (X2);
  \path(X1)--++(45:0.5)--++(135:0.5) coordinate (X3) ;
     \path(X2)--++(45:0.5)--++(135:0.5) coordinate (X4);
     \draw[very thick ] (X3) to [out=-135,in= -45] (X4);;

    \draw[very thick,gray!70 ] (X1)--++(-90:7);


    \draw[very thick  ] (X4)--++(90:7);

 \path(start)--++(45:1) coordinate (start);
   
     \path(start)--++(45:0.5) coordinate (X1);
     \path(start)--++(135:0.5) coordinate (X2);     
     \draw[very thick ] (X1) to [out=135,in= 45] (X2);
  \path(X1)--++(45:0.5)--++(135:0.5) coordinate (X3) ;
     \path(X2)--++(45:0.5)--++(135:0.5) coordinate (X4);
     \draw[very thick,gray!70] (X3) to [out=-135,in= -45] (X4);;

    \draw[very thick ] (X1)--++(-90:7);

    \draw[very thick,gray!70] (X3)--++(90:7); 

    \draw[very thick,gray!70] (X4)--++(90:7);


 \path(0,0) coordinate (start);
   
     \path(start)--++(45:0.5) coordinate (X1);
     \path(start)--++(135:0.5) coordinate (X2);     
     \draw[very thick,gray!70] (X1) to [out=135,in= 45] (X2);
  \path(X1)--++(45:0.5)--++(135:0.5) coordinate (X3);
     \path(X2)--++(45:0.5)--++(135:0.5) coordinate (X4);
     \draw[very thick,gray!70] (X3) to [out=-135,in= -45] (X4);; 
 
     \draw[very thick,gray!70] (X1)--++(-90:7);
     \draw[very thick,gray!70] (X2)--++(-90:7);

 \path(start)--++(135:1) coordinate (start);
   
     \path(start)--++(45:0.5) coordinate (X1);
     \path(start)--++(135:0.5) coordinate (X2);     
     \draw[very thick,gray!70] (X1) to [out=135,in= 45] (X2);
  \path(X1)--++(45:0.5)--++(135:0.5) coordinate (X3) ;
     \path(X2)--++(45:0.5)--++(135:0.5) coordinate (X4);
     \draw[very thick,gray!70] (X3) to [out=-135,in= -45] (X4);; 
 
    \draw[very thick,gray!70] (X2)--++(-90:7);

 \path(start)--++(135:1) coordinate (start);
   
     \path(start)--++(45:0.5) coordinate (X1);
     \path(start)--++(135:0.5) coordinate (X2);     
     \draw[very thick,gray!70] (X1) to [out=135,in= 45] (X2);
  \path(X1)--++(45:0.5)--++(135:0.5) coordinate (X3) ;
     \path(X2)--++(45:0.5)--++(135:0.5) coordinate (X4);
     \draw[very thick,gray!70] (X3) to [out=-135,in= -45] (X4);; 
 
    \draw[very thick,gray!70] (X2)--++(-90:7);

 \path(start)--++(135:1) coordinate (start);
   
     \path(start)--++(45:0.5) coordinate (X1);
     \path(start)--++(135:0.5) coordinate (X2);     
     \draw[very thick,gray!70] (X1) to [out=135,in= 45] (X2);
  \path(X1)--++(45:0.5)--++(135:0.5) coordinate (X3) ;
     \path(X2)--++(45:0.5)--++(135:0.5) coordinate (X4);
     \draw[very thick,gray!70] (X3) to [out=-135,in= -45] (X4);; 
 
    \draw[very thick,gray!70] (X2)--++(-90:7);
    \draw[very thick,gray!70] (X4)--++(90:7);


 \path(0,0)--++(45:1)--++(135:1) coordinate (start);
   
     \path(start)--++(45:0.5) coordinate (X1);
     \path(start)--++(135:0.5) coordinate (X2);     
     \draw[very thick,gray!70] (X1) to [out=135,in= 45] (X2);
  \path(X1)--++(45:0.5)--++(135:0.5) coordinate (X3) ;
     \path(X2)--++(45:0.5)--++(135:0.5) coordinate (X4);
     \draw[very thick,gray!70] (X3) to [out=-135,in= -45] (X4);;

 \path(start)--++(45:1)  coordinate (start);
   
     \path(start)--++(45:0.5) coordinate (X1);
     \path(start)--++(135:0.5) coordinate (X2);     
     \draw[very thick,gray!70] (X1) to [out=135,in= 45] (X2);
  \path(X1)--++(45:0.5)--++(135:0.5) coordinate (X3) ;
     \path(X2)--++(45:0.5)--++(135:0.5) coordinate (X4);
     \draw[very thick,gray!70] (X3) to [out=-135,in= -45] (X4);;

    \draw[very thick,gray!70] (X3)--++(90:7);


 \path(0,0)--++(45:1)--++(135:1) coordinate (start);
   
     \path(start)--++(45:0.5) coordinate (X1);
     \path(start)--++(135:0.5) coordinate (X2);     
     \draw[very thick,gray!70] (X1) to [out=135,in= 45] (X2);
  \path(X1)--++(45:0.5)--++(135:0.5) coordinate (X3) ;
     \path(X2)--++(45:0.5)--++(135:0.5) coordinate (X4);
     \draw[very thick,gray!70] (X3) to [out=-135,in= -45] (X4);;

 \path(start)--++(135:1)  coordinate (start);
   
     \path(start)--++(45:0.5) coordinate (X1);
     \path(start)--++(135:0.5) coordinate (X2);     
     \draw[very thick,gray!70] (X1) to [out=135,in= 45] (X2);
  \path(X1)--++(45:0.5)--++(135:0.5) coordinate (X3) ;
     \path(X2)--++(45:0.5)--++(135:0.5) coordinate (X4);
     \draw[very thick,gray!70] (X3) to [out=-135,in= -45] (X4);;

 \path(start)--++(135:1)  coordinate (start2);
   
     \path(start2)--++(45:0.5) coordinate (X1);
     \path(start2)--++(135:0.5) coordinate (X2);     
     \draw[very thick,gray!70 ] (X1) to [out=135,in= 45] (X2);
  \path(X1)--++(45:0.5)--++(135:0.5) coordinate (X3) ;
     \path(X2)--++(45:0.5)--++(135:0.5) coordinate (X4);
     \draw[very thick ,gray!70] (X3) to [out=-135,in= -45] (X4);; 
     \draw[very thick,gray!70] (X4)--++(90:7); 

     \draw[very thick,gray!70] (X3)--++(90:7);

 \path(start)--++(-45:1)--++(45:1)  coordinate (start);
   
     \path(start)--++(45:0.5) coordinate (X1);
     \path(start)--++(135:0.5) coordinate (X2);     
     \draw[very thick,gray!70] (X1) to [out=135,in= 45] (X2);
  \path(X1)--++(45:0.5)--++(135:0.5) coordinate (X3) ;
     \path(X2)--++(45:0.5)--++(135:0.5) coordinate (X4);
     \draw[very thick,gray!70] (X3) to [out=-135,in= -45] (X4);;

 \path(start)--++(135:1)  coordinate (start2);
   
     \path(start2)--++(45:0.5) coordinate (X1);
     \path(start2)--++(135:0.5) coordinate (X2);     
     \draw[very thick,gray!70 ] (X1) to [out=135,in= 45] (X2);
  \path(X1)--++(45:0.5)--++(135:0.5) coordinate (X3) ;
     \path(X2)--++(45:0.5)--++(135:0.5) coordinate (X4);
     \draw[very thick ,gray!70] (X3) to [out=-135,in= -45] (X4);; 
     \draw[very thick,gray!70] (X4)--++(90:7); 
      \draw[very thick,gray!70] (X3)--++(90:7);

\clip(0,0)--++(135:4)--++(45:2)--++(-45:1)
   --++(45:1)--++(-45:2)--++(45:2)
   --++(-45:1) 
 --(0,0);

       \foreach \i in {0,1,2,3,4,5,6,7,8,9,10,11,12,13}
{
\path (origin2)--++(45:0.5*\i) coordinate (c\i); 
\path (origin2)--++(135:0.5*\i)  coordinate (d\i); 
  }

   \foreach \i in {0,1,2,3,4,5,6,7,8,9,10,11,12,13}
{
\path (origin2)--++(45:1*\i) coordinate (c\i); 
\path (origin2)--++(135:1*\i)  coordinate (d\i); 
\draw[thick,densely dotted] (c\i)--++(135:14);
\draw[thick,densely dotted] (d\i)--++(45:14);
  }

\end{tikzpicture} 
   $$
   \vspace{-1cm}
   
   \caption{On the left we picture the partition/cup diagram for $(5^3,4,1)$ and we highlight the arc, $p$, and the corresponding 
removable    Dyck  path $P$. On the right we have the
   partition/cup diagram obtained by removing $P$.      }
   \label{example-2}
  \end{figure}
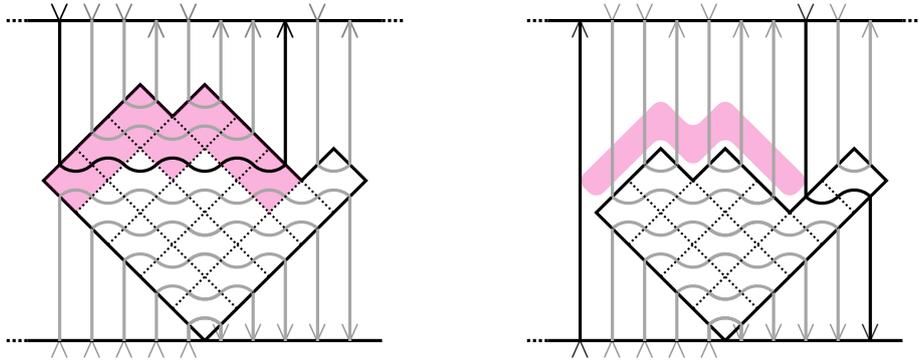

    \subsection{Dyck paths}
 
 We define a path on the $m\times n$ tiled rectangle to    
be a finite non-empty set $P$
 of  tiles  that are ordered $[r_1, c_1], \ldots , [r_s, c_s]$ for some $s\geq 1$ such that 
 for each $1\leq i\leq s-1$ we have $[r_{i+1}, c_{i+1}] = [r_i+1, c_i]$ or $[r_i, c_i -1]$. Note that the set $\underline{\sf cont}(P)$ of contents of the tiles in a path $P$ form an interval of integers.
 We say that $P$ is a {\sf Dyck path} if
 $$\min \{ r_i+c_i-1 \, : \, 1\leq i\leq s\} = r_1+c_1-1 = r_s+c_s-1,$$that is, the minimal height of the path is achieved at the start and end of the path.
 We will write 
 \begin{align*}
 {\sf first}(P) = {\sf cont}([r_1, c_1]) \qquad\textup{ and } \qquad {\sf last}(P) = {\sf cont}([r_s, c_s]). 
 \end{align*}
  We designate the {\sf height} \(h(P)\) and {\sf breadth} \(b(P)\) as:
 \begin{align*}
{\sf ht}(P) = r_1+c_1-1-m = r_s+c_s-1 -m\qquad \textup{ and }
  \qquad b(P) =  \tfrac{1}{2}(|P|+1)
 \end{align*}
 so that \(h(P)\) records the vertical position of the lowest nodes in \(P\) and \(b(P)\) records the horizontal distance covered by \(P\).

\begin{defn} Let $P$ and $Q$ be Dyck paths.
\begin{itemize}[leftmargin=*]
\item We say that $P$ and $Q$ are {\sf adjacent} if and only if the multiset given by the disjoint union ${\sf cont}(P) \sqcup {\sf cont}(Q)$ is an interval.
\item We say that $P$ and $Q$ are {\sf distant} if and only if $$\min \{  |{\sf cont}[r,c] - {\sf cont}[x,y]| \, : \, [r,c]\in P, [x,y]\in Q\} \geq 2.$$
\item We say that $P$ {\sf covers} $Q$ and write $Q\prec P$ if and only if $${\sf first}(Q) >{\sf first}(P) \,\,  \mbox{and} \,\,  {\sf last}(Q) < {\sf last} (P).$$
\end{itemize}
 \end{defn}
 
 Examples of such Dyck paths $P$ and $Q$ are given in Figure \ref{adjdistcover}.

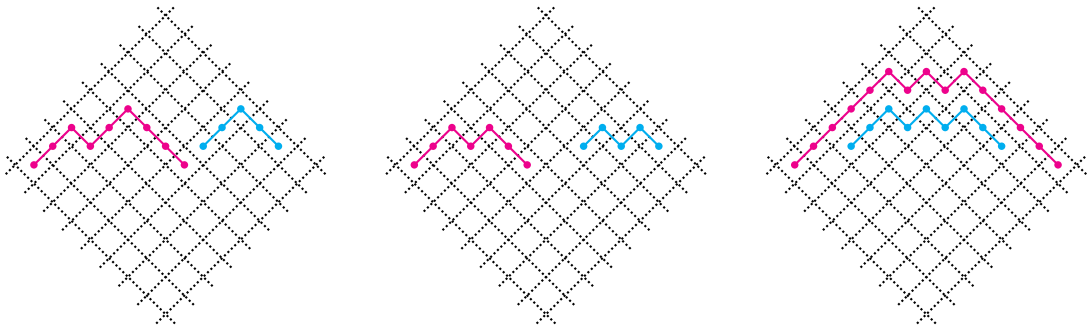
\begin{figure}[ht!]
$$   \begin{tikzpicture}[scale=0.35]
 \path(0,0) coordinate (origin2);
  \path(0,0)--++(135:2) coordinate (origin3);

     \foreach \i in {0,1,2,3,4,5,6,7,8}
{
\path (origin3)--++(45:0.5*\i) coordinate (c\i); 
\path (origin3)--++(135:0.5*\i)  coordinate (d\i); 
  }

   \foreach \i in {0,1,2,3,4,5,6,7,8}
{
\path (origin3)--++(45:1*\i) coordinate (c\i); 
\path (c\i)--++(-45:0.5) coordinate (c\i); 
\path (origin3)--++(135:1*\i)  coordinate (d\i); 
\path (d\i)--++(-135:0.5) coordinate (d\i); 
\draw[thick,densely dotted] (c\i)--++(135:9);
\draw[thick,densely dotted] (d\i)--++(45:9);
  }

\path(origin3)--++(45:-0.5)--++(135:7.5) coordinate (X);

\path(X)--++(45:1) coordinate (X) ;
\fill[magenta](X) circle (4pt);
\draw[ thick, magenta](X)--++(45:1) coordinate (X) ;
\fill[magenta](X) circle (4pt);
\draw[ thick, magenta](X)--++(45:1) coordinate (X) ;
\fill[magenta](X) circle (4pt);
\draw[ thick, magenta](X)--++(-45:1) coordinate (X) ;
\fill[magenta](X) circle (4pt);
\draw[ thick, magenta](X)--++(45:1) coordinate (X) ;
 \fill[magenta](X) circle (4pt);
\draw[ thick, magenta](X)--++(45:1) coordinate (X) ;

\fill[magenta](X) circle (4pt);
\draw[ thick, magenta](X)--++(-45:1) coordinate (X) ;
\fill[magenta](X) circle (4pt);
\draw[ thick, magenta](X)--++(-45:1) coordinate (X) ;
\fill[magenta](X) circle (4pt);
\draw[ thick, magenta](X)--++(-45:1) coordinate (X) ;
\fill[magenta](X) circle (4pt);

\path (X)--++(45:1) coordinate (X) ; 
\fill[cyan](X) circle (4pt);
\draw[ thick, cyan](X)--++(45:1) coordinate (X) ;
\fill[cyan](X) circle (4pt);
\draw[ thick, cyan](X)--++(45:1) coordinate (X) ;
\fill[cyan](X) circle (4pt);
\draw[ thick, cyan](X)--++(-45:1) coordinate (X) ;
\fill[cyan](X) circle (4pt);
\draw[ thick, cyan](X)--++(-45:1) coordinate (X) ;
\fill[cyan](X) circle (4pt);
\end{tikzpicture} 
\qquad
   \begin{tikzpicture}[scale=0.35]
 \path(0,0) coordinate (origin2);
  \path(0,0)--++(135:2) coordinate (origin3);

     \foreach \i in {0,1,2,3,4,5,6,7,8}
{
\path (origin3)--++(45:0.5*\i) coordinate (c\i); 
\path (origin3)--++(135:0.5*\i)  coordinate (d\i); 
  }

   \foreach \i in {0,1,2,3,4,5,6,7,8}
{
\path (origin3)--++(45:1*\i) coordinate (c\i); 
\path (c\i)--++(-45:0.5) coordinate (c\i); 
\path (origin3)--++(135:1*\i)  coordinate (d\i); 
\path (d\i)--++(-135:0.5) coordinate (d\i); 
\draw[thick,densely dotted] (c\i)--++(135:9);
\draw[thick,densely dotted] (d\i)--++(45:9);
  }

\path(origin3)--++(45:-0.5)--++(135:7.5) coordinate (X);

\path(X)--++(45:1) coordinate (X) ;
\fill[magenta](X) circle (4pt);
\draw[ thick, magenta](X)--++(45:1) coordinate (X) ;
\fill[magenta](X) circle (4pt);
\draw[ thick, magenta](X)--++(45:1) coordinate (X) ;
\fill[magenta](X) circle (4pt);
\draw[ thick, magenta](X)--++(-45:1) coordinate (X) ;
\fill[magenta](X) circle (4pt);
\draw[ thick, magenta](X)--++(45:1) coordinate (X) ;

\fill[magenta](X) circle (4pt);
\draw[ thick, magenta](X)--++(-45:1) coordinate (X) ;
\fill[magenta](X) circle (4pt);
\draw[ thick, magenta](X)--++(-45:1) coordinate (X) ;
\fill[magenta](X) circle (4pt);
\path(X)--++(-45:1) coordinate (X) ;

\path (X)--++(45:2) coordinate (X) ; 
\fill[cyan](X) circle (4pt);
\draw[ thick, cyan](X)--++(45:1) coordinate (X) ;
\fill[cyan](X) circle (4pt);
\draw[ thick, cyan](X)--++(-45:1) coordinate (X) ;
\fill[cyan](X) circle (4pt);
\draw[ thick, cyan](X)--++(45:1) coordinate (X) ;
\fill[cyan](X) circle (4pt);
\draw[ thick, cyan](X)--++(-45:1) coordinate (X) ;
\fill[cyan](X) circle (4pt);
\end{tikzpicture}
\qquad
   \begin{tikzpicture}[scale=0.35]
 \path(0,0) coordinate (origin2);
  \path(0,0)--++(135:2) coordinate (origin3);

     \foreach \i in {0,1,2,3,4,5,6,7,8}
{
\path (origin3)--++(45:0.5*\i) coordinate (c\i); 
\path (origin3)--++(135:0.5*\i)  coordinate (d\i); 
  }

   \foreach \i in {0,1,2,3,4,5,6,7,8}
{
\path (origin3)--++(45:1*\i) coordinate (c\i); 
\path (c\i)--++(-45:0.5) coordinate (c\i); 
\path (origin3)--++(135:1*\i)  coordinate (d\i); 
\path (d\i)--++(-135:0.5) coordinate (d\i); 
\draw[thick,densely dotted] (c\i)--++(135:9);
\draw[thick,densely dotted] (d\i)--++(45:9);
  }

\path(origin3)--++(45:-0.5)--++(135:7.5) coordinate (X);

\path(X)--++(45:1) coordinate (X) ;
\fill[magenta](X) circle (4pt);
\draw[ thick, magenta](X)--++(45:1) coordinate (X) ;
\fill[magenta](X) circle (4pt);
\draw[ thick, magenta](X)--++(45:1) coordinate (X) ;
\fill[magenta](X) circle (4pt);
\draw[ thick, magenta](X)--++(45:1) coordinate (X) ;

\fill[magenta](X) circle (4pt);
\draw[ thick, magenta](X)--++(45:1) coordinate (X) ;
\fill[magenta](X) circle (4pt);
\draw[ thick, magenta](X)--++(45:1) coordinate (X) ;
\fill[magenta](X) circle (4pt);
\draw[ thick, magenta](X)--++(-45:1) coordinate (X) ;
\fill[magenta](X) circle (4pt);
 \draw[ thick, magenta](X)--++(45:1) coordinate (X) ;
\fill[magenta](X) circle (4pt);
\draw[ thick, magenta](X)--++(-45:1) coordinate (X) ;
\fill[magenta](X) circle (4pt);
\draw[ thick, magenta](X)--++(45:1) coordinate (X) ;
\fill[magenta](X) circle (4pt);

 \draw[ thick, magenta](X)--++(-45:1) coordinate (X) ;
\fill[magenta](X) circle (4pt);
 \draw[ thick, magenta](X)--++(-45:1) coordinate (X) ;
\fill[magenta](X) circle (4pt);
 \draw[ thick, magenta](X)--++(-45:1) coordinate (X) ;
\fill[magenta](X) circle (4pt);
 \draw[ thick, magenta](X)--++(-45:1) coordinate (X) ;
\fill[magenta](X) circle (4pt);
 \draw[ thick, magenta](X)--++(-45:1) coordinate (X) ;
\fill[magenta](X) circle (4pt);

\path (X)--++(-135:1)--++(135:2) coordinate (X) ; 
\fill[cyan](X) circle (4pt);
\draw[ thick, cyan](X)--++(135:1) coordinate (X) ;
\fill[cyan](X) circle (4pt);
\draw[ thick, cyan](X)--++(135:1) coordinate (X) ;
\fill[cyan](X) circle (4pt);
 
 \draw[ thick, cyan](X)--++(-135:1) coordinate (X) ;
\fill[cyan](X) circle (4pt);
\draw[ thick, cyan](X)--++(135:1) coordinate (X) ;
\fill[cyan](X) circle (4pt);
\draw[ thick, cyan](X)--++(-135:1) coordinate (X) ;
\fill[cyan](X) circle (4pt);
\draw[ thick, cyan](X)--++(135:1) coordinate (X) ;
\fill[cyan](X) circle (4pt);
\draw[ thick, cyan](X)--++(-135:1) coordinate (X) ;
\fill[cyan](X) circle (4pt);
\draw[ thick, cyan](X)--++(-135:1) coordinate (X) ;
\fill[cyan](X) circle (4pt);

 \end{tikzpicture}
 $$
\caption{Examples of $\color{magenta}P$ and $\color{cyan}Q$ adjacent, distant, and ${\color{cyan}Q }\prec \color{magenta}  P$ respectively.}
\label{adjdistcover}
\end{figure}

Now we fix a partition $\mu\in \mptn$.   
Throughout the paper, we will identify all Dyck paths $P$ having the same content interval $\underline{\sf cont}(P)$. 
There are a few of places where we will need to fix a particular representative for a Dyck path $P$ and in that case we will use subscripts, such as $P_b$.

\begin{defn}
Let $\mu \in \mptn$ and $P$ be a Dyck path. We say that $P$ is a {\sf removable  Dyck path} from $\mu$ if 
 there is a representative $P_{b}$ of $P$   such that $\la  :=  \mu\setminus P_b\in \mptn$.
 In this case we will write $\la = \mu-P$, and \({\sf ht}^\mu(P):={\sf ht}(P_b)\). 
 We define the set ${\rm DRem}(\mu)$ to be the set of all removable Dyck paths from $\mu$.
 
 We say that $P$ is an {\sf addable Dyck path} of $\mu$ if there is a representative $P_b$ of $P$ such that $\la := \mu \sqcup P_b \in \mptn$. In this case we will write $\la = \mu + P$  and  ${\sf ht}^\la(P):=h(P_b)$. 
 We define the set ${\rm DAdd}(\mu)$ to be the set of all addable Dyck paths of $\mu$. 
 
   We let ${\rm DRem}_k(\mu)$ denote the set of all removable Dyck paths of $\mu$ of height 
  $k$ and similarly ${\rm DRem}_{\geq k}(\mu)$, ${\rm DRem}_{\leq k}(\mu)$  (and we define similarly ${\rm DAdd}_k(\mu)$ et cetera).
  \end{defn}
 
 
A removable Dyck path \(P \in {\rm DRem}(\mu)\) corresponds to an anti-clockwise arc in the diagram \(\underline{\mu} \mu\) with endpoints at \({\sf first}(P)\) and \({\sf last}(P)\). In particular, if \(P_b\) is the removable representative of \(P \in {\rm DRem}(\mu)\), then its height \({\sf ht}^\mu(P)\) is given by
\begin{align*}
{\sf ht}^\mu(P) = \#\{\down \textup{ to the left of } \first(P) \textup{ in } \mu\} - \#\{\up \textup{ to the left of } \first(P) \textup{ in } \mu\}. 
\end{align*}


%

\begin{defn}
Let $\mu\in \mptn$ and $P,Q\in {\rm DRem}(\mu)$. We say that $P$ and $Q$ {\sf commute} if $P\in {\rm DRem}(\mu -Q)$ and $Q\in {\rm DRem}(\mu - P)$.
\end{defn}

%

  We wish to consider the effect of adding two Dyck paths, or subtracting one from the other.

\begin{defn}
Let $P\in {\rm DRem}(\mu)$ and $Q\in {\rm DRem}(\mu - P)$ be adjacent. 
We define the {\sf merge} of $P$ and $Q$, denoted 
 $\langle P\cup Q \rangle_\mu$, if it exists, to be  the smallest removable Dyck path of $\mu$ containing $P\cup Q$. 
\end{defn}

\begin{defn}
 Let $P,Q\in {\rm DRem}(\mu)$   be adjacent, with \(P \prec Q\).
 Then we define the {\sf  split} of $Q$ by $P$, denoted \({\sf split}_Q(P)\), to be the Dyck tiling 
 $Q\setminus P = Q^1\sqcup Q^2$ where $Q^1, Q^2 \in {\rm DRem}(\mu - P)$. 
 
 \end{defn}


 \begin{defn}\label{Dyckpair}
 Let $\lambda\subseteq  \mu\in \mptn$. A {\sf Dyck tiling} of the skew partition $\mu\setminus \la$ is a set $\{P^1,\dotsc,P^k\}$ of Dyck paths such that
 $$\mu\setminus \la = \bigsqcup_{i=1}^k P^i $$
 and for each $i\neq j$ we have either $P^i$ covers $P^j$ (or vice versa), or $P^i$ and $P^j$ are distant. 
 We call $(\la,\mu)$ a {\sf Dyck pair of degree $k$} if $\mu \setminus \la$ has a Dyck tiling with $k$ Dyck paths.
 \end{defn}
 
 The Dyck tiling \(\{P^1, \ldots, P^k\}\) for \(\mu \backslash \la\) (if it exists) is unique, though there are generally distinct choices of representatives \(\{P^1_b, \ldots, P^k_b\}\) for \(\{P^1, \ldots, P^k\}\) (we will consider these choices in the next subsection). We will freely associate the skew partition \(\mu \backslash \la\) with its Dyck tiling, and abuse notation in using \(\mu \backslash \la\) to refer to the Dyck tiling as well.

\subsection{Dyck tableaux}
A {\sf good move}   is a map  
$ \underline{\mu}\la\to \underline{\nu} \la$
 such that $\deg( \underline{\nu}\la)=\deg( \underline{\mu}\la)-1$ and is of one of 
  the following four possible forms.
   The first  case to consider is when 
$\nu=\mu-P$ where $P \in {\rm DRem}(\mu)$ and $ d(\nu)=d(\mu)-1$,
in which case we can break the arc $p$ into 2 strands as follows: 
\begin{align} \tag{G1}\label{gm1}
 \begin{minipage}{2.2cm}  \begin{tikzpicture} [scale=0.6 ]
\draw[densely dotted](0,0)--++(180:0.25);
\draw[densely dotted](2.75,0)--++(0:0.25);
\draw(0,0)--++(0:1);
\draw[densely dotted](1,0)--++(0:0.75);
\draw(1.75,0)--++(0:1);
\path(2.25,0) --++(90:0.12) node  {  $  \down   $} ;
\path(0.5,0) --++(-90:0.15) node  {  $  \up   $} ;
\draw(0.5,0)--++(-90:0.2)   to [out=-90,in=180] (1+0.25,-0.85) --(1.75-0.25,-0.85)   to [out=0,in=-90]  (2.25,-0.2)--(2.25,0);
\draw[densely dotted](1.375,0) circle (1.625cm);
\draw[white,thick,densely dotted](1+0.125,-0.85) --(1.75-0.125,-0.85);
 	\end{tikzpicture}	
	\end{minipage}
	 \xrightarrow { \ \ \ }\;
	 \begin{minipage}{2.3cm}  \begin{tikzpicture} [scale=0.6 ]
\draw[densely dotted](0,0)--++(180:0.25);
\draw[densely dotted](2.75,0)--++(0:0.25);
\draw(0,0)--++(0:1);
\draw[densely dotted](1,0)--++(0:0.75);
\draw(1.75,0)--++(0:1);
\path(2.25,0) --++(90:0.12) node  {  $  \down   $} ;
\path(0.5,0) --++(-90:0.15) node  {  $  \up   $} ;
\draw[densely dotted](1.375,0) circle (1.625cm);
\clip(1.375,0) circle (1.625cm);
\draw(0.5,0)--++(-90:0.2) to [out=-90,in=15] (0,-0.8);
\draw(2.25,0)--++(-90:0.2) to [out=-90,in=180-15] (2.75,-0.8);
 	\end{tikzpicture}	
	\end{minipage} 
\end{align}
The second   case is that 
$\nu=\mu-P$ where $P \in {\rm DRem}(\mu)$ and
$d(\mu)=d(\nu)$.  
 We first suppose that    there does not exist $P\prec Q\in {\rm DRem}  (\la)$, in which case we can do precisely one of  the following good moves
\begin{align} \tag{G2}\label{gm2}
 \begin{minipage}{2.6cm}  \begin{tikzpicture} [scale=0.6,xscale=-1 ]
\draw[densely dotted](0,0)--++(180:0.25);
\draw[densely dotted](3.5,0)--++(0:0.25);
\draw(0,0)--++(0:1);
\draw[densely dotted](1,0)--++(0:0.75);
\draw(1.75,0)--++(0:1.75);
\path(0.5,0) --++(90:0.12) node  {  $  \down   $} ;
\path(2.25 ,0) --++(-90:0.15) node  {  $  \up   $} ;
\path(3,0) --++(90:0.12) node  {  $  \down   $} ;
\draw(0.5,0)--++(-90:0.2)   to [out=-90,in=180] (1+0.25,-0.75) --(1.75-0.25,-0.75)   to [out=0,in=-90]  (2.25,-0.2)--(2.25,0);
\draw(3,0)--++(-90:0.2) to [out=-90,in=5] (0.15,-1.2);
\draw[densely dotted](1.75,0) circle (2cm);
\draw[white,thick,densely dotted](1+0.125,-0.75) --(1.75-0.125,-0.75);
 	\end{tikzpicture}	
	\end{minipage}
	  \;\xrightarrow { \ \ \ }\;\;
\begin{minipage}{2.6cm}  \begin{tikzpicture} [scale=0.6,xscale=-1 ]
\draw[densely dotted](0,0)--++(180:0.25);
\draw[densely dotted](3.5,0)--++(0:0.25);
\draw(0,0)--++(0:1);
\draw[densely dotted](1,0)--++(0:0.75);
\draw(1.75,0)--++(0:1.75);
\path(0.5,0) --++(90:0.12) node  {  $  \down   $} ;
\path(2.25 ,0) --++(-90:0.15) node  {  $  \up   $} ;
\path(3,0) --++(90:0.12) node  {  $  \down   $} ;
\draw(2.25,0)--++(-90:0.2) to [out=-90,in=180]  (2.25+0.75/2,-0.8) 
to [out=0,in=-90]   (3,-0.2)--(3,0);
\draw(0.5,0)--++(-90:0.2) to [out=-90,in=0] (-0.05,-0.8);
\draw[densely dotted](1.75,0) circle (2cm);
\draw[white,thick,densely dotted](1.35,-1.2)  --++(0:0.5);
 	\end{tikzpicture}	
	\end{minipage}
	\qquad
	\qquad
	 \begin{minipage}{2.6cm}  \begin{tikzpicture} [scale=0.6 ]
\draw[densely dotted](0,0)--++(180:0.25);
\draw[densely dotted](3.5,0)--++(0:0.25);
\draw(0,0)--++(0:1);
\draw[densely dotted](1,0)--++(0:0.75);
\draw(1.75,0)--++(0:1.75);
\path(2.25,0) --++(90:0.12) node  {  $  \down   $} ;
\path(0.5,0) --++(-90:0.15) node  {  $  \up   $} ;
\path(3,0) --++(-90:0.15) node  {  $  \up   $} ;
\draw(0.5,0)--++(-90:0.2)   to [out=-90,in=180] (1+0.25,-0.75) --(1.75-0.25,-0.75)   to [out=0,in=-90]  (2.25,-0.2)--(2.25,0);
\draw(3,0)--++(-90:0.2) to [out=-90,in=5] (0.15,-1.2);
\draw[densely dotted](1.75,0) circle (2cm);
\draw[white,thick,densely dotted](1+0.125,-0.75) --(1.75-0.125,-0.75);
 	\end{tikzpicture}	
	\end{minipage}
	  \;\xrightarrow { \ \ \ }\;\;
\begin{minipage}{2.6cm}  \begin{tikzpicture} [scale=0.6 ]
\draw[densely dotted](0,0)--++(180:0.25);
\draw[densely dotted](3.5,0)--++(0:0.25);
\draw(0,0)--++(0:1);
\draw[densely dotted](1,0)--++(0:0.75);
\draw(1.75,0)--++(0:1.75);
\path(2.25,0) --++(90:0.12) node  {  $  \down   $} ;
\path(0.5,0) --++(-90:0.15) node  {  $  \up   $} ;
\path(3,0) --++(-90:0.15) node  {  $  \up   $} ;
\draw(2.25,0)--++(-90:0.2) to [out=-90,in=180]  (2.25+0.75/2,-0.8) 
to [out=0,in=-90]   (3,-0.2)--(3,0);
\draw(0.5,0)--++(-90:0.2) to [out=-90,in=0] (-0.015,-1);
\draw[densely dotted](1.75,0) circle (2cm);
\draw[white,thick,densely dotted](1.35,-1.2)  --++(0:0.5);
 	\end{tikzpicture}	
	\end{minipage}
\end{align}We now suppose that  there does  exist  $P\prec Q\in {\rm DRem}  (\la)$ and suppose $b(Q)$ is minimal with respect to this property.
We can further assume that  there is no $P\prec R\prec Q$  such that $R \in \mu /\la$. 
In which case we   do  the following good move 
\begin{align} \tag{G3}\label{gm3}
	 \begin{minipage}{2.6cm}  \begin{tikzpicture} [scale=0.6 ]
\path(0,0)coordinate (X);
\draw[densely dotted](X)--++(0:0.2 ) coordinate (X);
 \draw(X)--++(0:0.6)coordinate (X);
 \draw[densely dotted](X)--++(0:0.4 ) coordinate (X);
 \draw(X)--++(0:0.6)coordinate (X);
\draw[densely dotted](X)--++(0:0.4 ) coordinate (X);
 \draw(X)--++(0:0.6)coordinate (X);
\draw[densely dotted](X)--++(0:0.4 ) coordinate (X);
 \draw(X)--++(0:0.6)coordinate (X);
\draw[densely dotted](X)--++(0:0.2 ) coordinate (X);
  \path(0,0) --++(0:0.5)--++(90:0.12)  node  {  $  \down   $} ;
\path(0,0)--++(0:0.5+1) --++(-90:0.15) node  {  $  \up   $} ;
 \path(0,0) --++(0:0.5+2)--++(90:0.12)  node  {  $  \down   $} ;
\path(0,0)--++(0:0.5+3) --++(-90:0.15) node  {  $  \up   $} ;
 \draw(1.5,0)--++(-90:0.4)   to  [out=-90,in=180] (2,-1) to [out=0,in=-90]  (2.5,-0.4)--(2.5,0);
 \draw(0.5,0)--++(-90:0.4)   to  [out=-90,in=180] (2,-1.3) to [out=0,in=-90]  (3.5,-0.4)--(3.5,0);
 \draw[densely dotted](2,0) circle (2cm);
 \draw[white,thick,densely dotted] (1.8,-1)--(2.2,-1);
  \draw[white,thick,densely dotted] (1.6,-1.3)--(2.4,-1.3);
 	\end{tikzpicture}	
	\end{minipage}
	\xrightarrow { \ \ \ }\;	 \begin{minipage}{2.6cm}  \begin{tikzpicture} [scale=0.6 ]
\path(0,0)coordinate (X);
\draw[densely dotted](X)--++(0:0.2 ) coordinate (X);
 \draw(X)--++(0:0.6)coordinate (X);
 \draw[densely dotted](X)--++(0:0.4 ) coordinate (X);
 \draw(X)--++(0:0.6)coordinate (X);
\draw[densely dotted](X)--++(0:0.4 ) coordinate (X);
 \draw(X)--++(0:0.6)coordinate (X);
\draw[densely dotted](X)--++(0:0.4 ) coordinate (X);
 \draw(X)--++(0:0.6)coordinate (X);
\draw[densely dotted](X)--++(0:0.2 ) coordinate (X);
  \path(0,0) --++(0:0.5)--++(90:0.12)  node  {  $  \down   $} ;
\path(0,0)--++(0:0.5+1) --++(-90:0.15) node  {  $  \up   $} ;
 \path(0,0) --++(0:0.5+2)--++(90:0.12)  node  {  $  \down   $} ;
\path(0,0)--++(0:0.5+3) --++(-90:0.15) node  {  $  \up   $} ;
  \draw(0.5,0)--++(-90:0.4)   to  [out=-90,in=180] (1,-1) to [out=0,in=-90]  (1.5,-0.4)--(1.5,0);
  \draw(2.5,0)--++(-90:0.4)   to  [out=-90,in=180] (3,-1) to [out=0,in=-90]  (3.5,-0.4)--(3.5,0);
 \draw[densely dotted](2,0) circle (2cm);
  \draw[white,thick,densely dotted] (0.8,-1)--(1.2,-1);
    \draw[white,thick,densely dotted] (2.8,-1)--(3.2,-1);
	\end{tikzpicture}	
	\end{minipage}
\end{align}

\begin{prop}
For any Dyck pair $(\la,\mu)$ of degree $k$, there clearly exists a sequence of good moves
 $$\underline{\mu}\la=\underline{\mu_k}\la
 \to \underline{\mu_{k-1}}\la
 \to \dots \to 
 \underline{\mu_0}\la
 =\underline{\la}\la.$$
 \end{prop}
 \begin{proof}
We simply note that the only {\em apparent} restriction 
 for this process is that the \eqref{gm3} good move requires
there not exist  $P\prec R\prec Q$; 
 however, 
if there does exist such an $R$ 
then we can consider the pair $(R,Q)$ first.
\end{proof}

 \begin{defn}
 Let \((\la,\mu)\) be a Dyck pair of degree \(k\). A {\sf Dyck tableau} for \((\la, \mu)\) is any sequence of good moves 
$$\underline{\mu}\la=\underline{\mu_k}\la
 \to \underline{\mu_{k-1}}\la
 \to \dots \to 
 \underline{\mu_0}\la
 =\underline{\la}\la.$$
 \end{defn}

\subsection{Height of paths in a tiling}

Let \((\la,\mu)\) be a Dyck pair. Let \(P\) be a clockwise arc in \(\underline{\mu} \la\).
We define the set of Dyck paths  {\sf supported} by  $P\in \mu/\la$
 (and denote this set by  \({\sf supp}^\mu_\la(P)\))
    as follows. Extend a line downward from the lowest point of the arc \(P\) in \(\underline{\mu} \la\) which terminates as soon as it has intersected equal numbers of clockwise arcs (including \(P)\) as anti-clockwise arcs/half-arcs, or else never terminates. Then \({\sf supp}^\mu_\la(P)\) is the set of clockwise arcs in \( \underline{\mu}\la \) intersected by this line.  
 We define the {\sf height} of   a Dyck path $P\in \mu/\la$ as follows 
\begin{align}\label{realheight}
{\sf ht}^\mu_\la(P) = \min\{{\sf ht}^\mu(P), {\sf ht}^\mu(Q) -1 \mid Q \in {\sf supp}^\mu_\la(P)\backslash P \}.
\end{align}
Note that if \(P \in {\rm DRem}(\mu)\) then  \({\sf ht}^\mu(P) = {\sf ht}^\mu_{\mu - P}(P)\). 

  \begin{prop}
Let \((P^1, \ldots, P^k)\) be the unique ordering of the clockwise arcs in \(\underline{\mu} \lambda\) such that 
\begin{align*}
{\sf ht}^\mu_\la(P^i) < {\sf ht}^\mu_\la(P^{i+1})  \qquad \textup{or}  \qquad {\sf ht}^\mu_\la(P^i) = {\sf ht}^\mu_\la(P^{i+1}) \textup{ and } \last(P^i) < \first(P^{i+1}).
\end{align*}
Then, setting $\mu_i=\mu_{i+1}-P^{i+1}$, we have that 
\begin{align}\label{kfjghsdjkghdjfk}\underline{\mu}\la=\underline{\mu_k}\la
 \to \underline{\mu_{k-1}}\la
 \to \dots \to 
 \underline{\mu_0}\la
 =\underline{\la}\la\end{align}is a Dyck tableau. 
\end{prop}
\begin{proof}
Let 
$P\prec Q\in {\rm DRem}  (\la)$  and assume $b(Q)$ is minimal with respect to this property. 
If 
  $P\prec R\prec Q$ 
for 
  $R \in \mu /\la$ then 
  $R \in {\sf supp}^\mu_\la(P)$ and so 
  ${\sf ht}^\mu_\la(P)<
 {\sf ht}^\mu_\la(R)$.  Therefore we can always remove any maximal height Dyck path using a good move.
\end{proof}

%

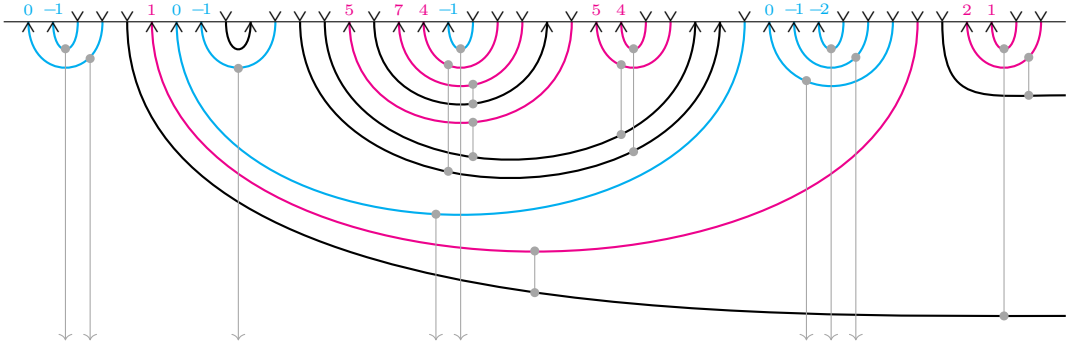
\begin{figure}[h]
\begin{align*}
 \begin{tikzpicture} [scale=0.325]
\draw[thick, cyan] (1,0) .. controls ++(0,-2.5) and ++(0,-2.5) .. (4,0);
\draw[thick, cyan] (2,0) .. controls ++(0,-1.5) and ++(0,-1.5) .. (3,0);
\draw[thick, cyan] (8,0) .. controls ++(0,-2.5) and ++(0,-2.5) .. (11,0);
\draw[thick, black](9,0) .. controls ++(0,-1.5) and ++(0,-1.5) .. (10,0);
\draw[thick, magenta] (14,0) .. controls ++(0,-5.5) and ++(0,-5.5) .. (23,0);
\draw[thick, black] (15,0) .. controls ++(0,-4.5) and ++(0,-4.5) .. (22,0);
\draw[thick, magenta] (16,0) .. controls ++(0,-3.5) and ++(0,-3.5) .. (21,0);
\draw[thick, magenta] (17,0) .. controls ++(0,-2.5) and ++(0,-2.5) .. (20,0);
\draw[thick, cyan] (18,0) .. controls ++(0,-1.5) and ++(0,-1.5) .. (19,0);
\draw[thick, magenta] (24,0) .. controls ++(0,-2.5) and ++(0,-2.5) .. (27,0);
\draw[thick, magenta] (25,0) .. controls ++(0,-1.5) and ++(0,-1.5) .. (26,0);
\draw[thick, cyan] (31,0) .. controls ++(0,-3.5) and ++(0,-3.5) .. (36,0);
\draw[thick, cyan] (32,0) .. controls ++(0,-2.5) and ++(0,-2.5) .. (35,0);
\draw[thick, cyan] (33,0) .. controls ++(0,-1.5) and ++(0,-1.5) .. (34,0);
\draw[thick, magenta] (39,0) .. controls ++(0,-2.5) and ++(0,-2.5) .. (42,0);
\draw[thick, magenta] (40,0) .. controls ++(0,-1.5) and ++(0,-1.5) .. (41,0);
%
\draw[thick, black] (12,0) .. controls ++(0,-8.5) and ++(0,-8.5) .. (29,0);
\draw[thick, black] (13,0) .. controls ++(0,-7.5) and ++(0,-7.5) .. (28,0);
\draw[thick, cyan] (7,0) .. controls ++(0,-10.5) and ++(0,-10.5) .. (30,0);
\draw[thick, magenta] (6,0) .. controls ++(0,-12.5) and ++(0,-12.5) .. (37,0);
\draw[thick, black] (5,0) .. controls ++(0,-12.5) and ++(-12.5,0) .. (43,-12);
\draw[thick, black] (38,0) .. controls ++(0,-3.5) and ++(-3.5,0) .. (43,-3);
\node at (1,-0.22) {${\scriptscriptstyle \boldsymbol{\wedge}}$}; 
\node at (2,-0.22) {${\scriptscriptstyle \boldsymbol{\wedge}}$}; 
\node at (6,-0.22) {${\scriptscriptstyle \boldsymbol{\wedge}}$}; 
\node at (7,-0.22) {${\scriptscriptstyle \boldsymbol{\wedge}}$}; 
\node at (8,-0.22) {${\scriptscriptstyle \boldsymbol{\wedge}}$}; 
\node at (10,-0.22) {${\scriptscriptstyle \boldsymbol{\wedge}}$}; 
\node at (14,-0.22) {${\scriptscriptstyle \boldsymbol{\wedge}}$}; 
\node at (16,-0.22) {${\scriptscriptstyle \boldsymbol{\wedge}}$}; 
\node at (17,-0.22) {${\scriptscriptstyle \boldsymbol{\wedge}}$}; 
\node at (18,-0.22) {${\scriptscriptstyle \boldsymbol{\wedge}}$}; 
\node at (22,-0.22) {${\scriptscriptstyle \boldsymbol{\wedge}}$}; 
\node at (24,-0.22) {${\scriptscriptstyle \boldsymbol{\wedge}}$}; 
\node at (25,-0.22) {${\scriptscriptstyle \boldsymbol{\wedge}}$}; 
\node at (28,-0.22) {${\scriptscriptstyle \boldsymbol{\wedge}}$}; 
\node at (29,-0.22) {${\scriptscriptstyle \boldsymbol{\wedge}}$}; 
\node at (31,-0.22) {${\scriptscriptstyle \boldsymbol{\wedge}}$}; 
\node at (32,-0.22) {${\scriptscriptstyle \boldsymbol{\wedge}}$}; 
\node at (33,-0.22) {${\scriptscriptstyle \boldsymbol{\wedge}}$}; 
\node at (39,-0.22) {${\scriptscriptstyle \boldsymbol{\wedge}}$}; 
\node at (40,-0.22) {${\scriptscriptstyle \boldsymbol{\wedge}}$}; 
\node at (3,0.20) {${\scriptscriptstyle \boldsymbol{\vee}}$}; 
\node at (4,0.20) {${\scriptscriptstyle \boldsymbol{\vee}}$}; 
\node at (5,0.20) {${\scriptscriptstyle \boldsymbol{\vee}}$}; 
\node at (9,0.20) {${\scriptscriptstyle \boldsymbol{\vee}}$}; 
\node at (11,0.20) {${\scriptscriptstyle \boldsymbol{\vee}}$}; 
\node at (12,0.20) {${\scriptscriptstyle \boldsymbol{\vee}}$}; 
\node at (13,0.20) {${\scriptscriptstyle \boldsymbol{\vee}}$}; 
\node at (15,0.20) {${\scriptscriptstyle \boldsymbol{\vee}}$}; 
\node at (19,0.20) {${\scriptscriptstyle \boldsymbol{\vee}}$}; 
\node at (20,0.20) {${\scriptscriptstyle \boldsymbol{\vee}}$}; 
\node at (21,0.20) {${\scriptscriptstyle \boldsymbol{\vee}}$}; 
\node at (23,0.20) {${\scriptscriptstyle \boldsymbol{\vee}}$}; 
\node at (26,0.20) {${\scriptscriptstyle \boldsymbol{\vee}}$}; 
\node at (27,0.20) {${\scriptscriptstyle \boldsymbol{\vee}}$}; 
\node at (30,0.20) {${\scriptscriptstyle \boldsymbol{\vee}}$}; 
\node at (34,0.20) {${\scriptscriptstyle \boldsymbol{\vee}}$}; 
\node at (35,0.20) {${\scriptscriptstyle \boldsymbol{\vee}}$}; 
\node at (36,0.20) {${\scriptscriptstyle \boldsymbol{\vee}}$}; 
\node at (37,0.20) {${\scriptscriptstyle \boldsymbol{\vee}}$}; 
\node at (38,0.20) {${\scriptscriptstyle \boldsymbol{\vee}}$}; 
\node at (41,0.20) {${\scriptscriptstyle \boldsymbol{\vee}}$}; 
\node at (42,0.20) {${\scriptscriptstyle \boldsymbol{\vee}}$}; 
\node[cyan] at (1,0.5) {${\scriptscriptstyle 0}$}; 
\node[cyan] at (2,0.5) {${\scriptscriptstyle -\hspace{-0.5mm}1}$};
\node[magenta] at (6,0.5) {${\scriptscriptstyle 1}$}; 
\node[cyan] at (7,0.5) {${\scriptscriptstyle 0}$}; 
\node[cyan] at (8,0.5) {${\scriptscriptstyle -\hspace{-0.5mm}1}$};  
\node[magenta] at (14,0.5) {${\scriptscriptstyle 5}$}; 
\node[magenta] at (16,0.5) {${\scriptscriptstyle 7}$}; 
\node[magenta] at (17,0.5) {${\scriptscriptstyle 4}$};
\node[cyan] at (18,0.5)  {${\scriptscriptstyle -\hspace{-0.5mm}1}$};  
\node[magenta] at (24,0.5) {${\scriptscriptstyle 5}$}; 
\node[magenta] at (25,0.5) {${\scriptscriptstyle 4}$};
\node[cyan] at (31,0.5) {${\scriptscriptstyle 0}$};
\node[cyan] at (32,0.5)  {${\scriptscriptstyle -\hspace{-0.5mm}1}$};  
\node[cyan] at (33,0.5)  {${\scriptscriptstyle -\hspace{-0.5mm}2}$}; 
\node[magenta] at (39,0.5) {${\scriptscriptstyle 2}$};
\node[magenta] at (40,0.5) {${\scriptscriptstyle 1}$}; 
 \draw[] (0,0)--(43,0);
   \draw[gray!70, thick, fill=gray!70]  (2.5,-1.1) circle (4pt);
\draw[mark=*,gray!70, ->] (2.5,-1)--(2.5,-13);
   \draw[gray!70, thick, fill=gray!70]  (3.5,-1.5) circle (4pt);
\draw[mark=*,gray!70, ->] (3.5,-1.5)--(3.5,-13);
   \draw[gray!70, thick, fill=gray!70]  (21.5,-9.35) circle (4pt);
      \draw[gray!70, thick, fill=gray!70]  (21.5,-11.05) circle (4pt);
\draw[gray!70]  (21.5,-9.35)-- (21.5,-11.05);
   \draw[gray!70, thick, fill=gray!70]  (18.5,-1.1) circle (4pt);
\draw[mark=*,gray!70, ->] (18.5,-1.1)--(18.5,-13);
   \draw[gray!70, thick, fill=gray!70]  (19,-2.55) circle (4pt);
      \draw[gray!70, thick, fill=gray!70]  (19,-3.35) circle (4pt);
      \draw[gray!70] (19,-2.55)--(19,-3.35) ;
   \draw[gray!70, thick, fill=gray!70]  (19,-4.1) circle (4pt);
      \draw[gray!70, thick, fill=gray!70]  (19,-5.5) circle (4pt);
      \draw[gray!70] (19,-4.1)--(19,-5.5) ;
   \draw[gray!70, thick, fill=gray!70]  (18,-1.75) circle (4pt);
      \draw[gray!70, thick, fill=gray!70]  (18,-6.1) circle (4pt);
        \draw[gray!70] (18,-1.75)--(18,-6.1);
   \draw[gray!70, thick, fill=gray!70]  (9.5,-1.9) circle (4pt);
     \draw[mark=*,gray!70, ->]  (9.5,-1.9)-- (9.5,-13);      
   \draw[gray!70, thick, fill=gray!70]  (25.5,-1.1) circle (4pt);
      \draw[gray!70, thick, fill=gray!70]  (25.5,-5.3) circle (4pt);
        \draw[gray!70]  (25.5,-1.1)--(25.5,-5.3) ; 
   \draw[gray!70, thick, fill=gray!70]  (25,-1.75) circle (4pt);
      \draw[gray!70, thick, fill=gray!70]  (25,-4.6) circle (4pt);
        \draw[gray!70] (25,-1.75)--(25,-4.6) ;  
   \draw[gray!70, thick, fill=gray!70]  (17.5,-7.85) circle (4pt);
    \draw[mark=*,gray!70, ->]   (17.5,-7.85)--(17.5,-13) ;     
   \draw[gray!70, thick, fill=gray!70]  (33.5,-1.1) circle (4pt);
    \draw[mark=*,gray!70, ->]  (33.5,-1.1) --(33.5,-13)  ;         
   \draw[gray!70, thick, fill=gray!70]  (34.5,-1.45) circle (4pt);
    \draw[mark=*,gray!70, ->]  (34.5,-1.45)--(34.5,-13)  ;         
   \draw[gray!70, thick, fill=gray!70]  (32.5,-2.4) circle (4pt);
    \draw[mark=*,gray!70, ->]  (32.5,-2.4)--(32.5,-13)  ;  
     \draw[gray!70, thick, fill=gray!70]  (40.5,-1.1) circle (4pt);
      \draw[gray!70, thick, fill=gray!70]  (40.5,-12) circle (4pt);
 \draw[gray!70]  (40.5,-1.1)--(40.5,-12)  ;        
 \draw[gray!70, thick, fill=gray!70]  (41.5,-1.45) circle (4pt);
      \draw[gray!70, thick, fill=gray!70]  (41.5,-3) circle (4pt);
\draw[gray!70]  (41.5,-1.45)--(41.5,-3) ;        
\end{tikzpicture}
\end{align*}
\begin{center}
\end{center}
\caption{The arc diagram of \(\underline{\mu} \la\), for the partitions \(\mu = (20^2,18^5,16^7,11^2,6^3,2^3)\) and \(\la = (18^3,15,14^3,13,10^3,9,7^2,5,2^5)\). For each coloured clockwise arc \(P\), the support set \({\sf supp}^\mu_\la(P)\) is indicated by the set of coloured arcs intersected by the grey line emanating downward from \(P\). 
We record the height   \({\sf ht}^\mu_\la(P)\) (in the sense of \cref{realheight})   at the upward vertex of \(P\) (and we colour pink/blue if this height is positive/non-positive). Compare with \cref{BigDyck}. }
\label{figureSatSets}
\end{figure}

\begin{figure}[ht!]
  \begin{tikzpicture}[scale=0.275]
  \path(0,0)--++(135:2) coordinate (hhhh);
 \draw[very thick] (hhhh)--++(135:20)--++(45:22)--++(-45:20)--++(-135:22);
 \clip(hhhh)--++(135:20)--++(45:22)--++(-45:20)--++(-135:22);
\path(0,0) coordinate (origin2);
  \path(0,0)--++(135:2) coordinate (origin3);

     \foreach \i in {0,1,2,3,4,5,6,7,8,...,22}
{
\path (origin3)--++(45:0.5*\i) coordinate (c\i); 
\path (origin3)--++(135:0.5*\i)  coordinate (d\i); 
  }

   \foreach \i in {0,1,2,3,4,5,6,7,8,...,22}
{
\path (origin3)--++(45:1*\i) coordinate (c\i); 
\path (c\i)--++(-45:0.5) coordinate (c\i); 
\path (origin3)--++(135:1*\i)  coordinate (d\i); 
\path (d\i)--++(-135:0.5) coordinate (d\i); 
\draw[thick,densely dotted] (c\i)--++(135:24);
\draw[thick,densely dotted] (d\i)--++(45:24);
  }

\path(0,0)--++(135:2) coordinate (hhhh);

\fill[opacity=0.2](hhhh)--++(135:18)--++(45:3) --++(-45:3) --++(45:1)--++(-45:1) --++(45:3) --++(-45:1) --++(45:1)--++(-45:3) --++(45:3)
--++(-45:1) --++(45:1)--++(-45:2) --++(45:2)--++(-45:2) --++(45:1)  --++(-45:3) --++(45:5)--++(-45:2)       ;

\path(origin3)--++(45:-0.5)--++(135:19.5) coordinate (X)coordinate (start);
\path (start)--++(45:1)--++(-45:1) coordinate (X) ;

\fill[cyan](X) circle (4pt);

\path (X)--++(135:1)  coordinate (X) ; 

\fill[cyan](X) circle (4pt);
\draw[ thick, cyan](X)--++(45:1) coordinate (X) ;
\fill[cyan](X) circle (4pt);
\draw[ thick, cyan](X)--++(-45:1) coordinate (X) ;
\fill[cyan](X) circle (4pt);

 \path (X)--++(-45:3) --++(45:2) coordinate (start) ;  \path (start) coordinate (X) ; 
\fill[cyan](X) circle (4pt);
\draw[ thick, cyan](X)--++(45:1) coordinate (X) ;
\fill[cyan](X) circle (4pt);
\draw[ thick, cyan](X)--++(-45:1) coordinate (X) ;
\fill[cyan](X) circle (4pt);

 \path (start)--++(135:1)  coordinate (start) ;  \path (start) coordinate (X) ; 
\fill[cyan](X) circle (4pt);
\draw[ thick, cyan](X)--++(45:1) coordinate (X) ;
\fill[cyan](X) circle (4pt);
\draw[ thick, cyan](X)--++(45:1) coordinate (X) ;
\fill[cyan](X) circle (4pt);
\draw[ thick, cyan](X)--++(-45:1) coordinate (X) ;
\fill[cyan](X) circle (4pt);
\draw[ thick, cyan](X)--++(-45:1) coordinate (X) ;
\fill[cyan](X) circle (4pt);
\draw[ thick, cyan](X)--++(45:1) coordinate (X) ;
\fill[cyan](X) circle (4pt);
\draw[ thick, cyan](X)--++(45:1) coordinate (X) ;

\fill[cyan](X) circle (4pt);
\draw[ thick, cyan](X)--++(-45:1) coordinate (X) ;
\fill[cyan](X) circle (4pt);
\draw[ thick, cyan](X)--++(45:1) coordinate (X) ;

\fill[cyan](X) circle (4pt);
\draw[ thick, cyan](X)--++(-45:1) coordinate (X) ;
\fill[cyan](X) circle (4pt);
\draw[ thick, cyan](X)--++(-45:1) coordinate (X) ;
\fill[cyan](X) circle (4pt);
\draw[ thick, cyan](X)--++(45:1) coordinate (X) ;

\fill[cyan](X) circle (4pt);
\draw[ thick, cyan](X)--++(-45:1) coordinate (X) ;
\fill[cyan](X) circle (4pt);
\draw[ thick, cyan](X)--++(45:1) coordinate (X) ;
\fill[cyan](X) circle (4pt);
\draw[ thick, cyan](X)--++(45:1) coordinate (X) ;
\fill[cyan](X) circle (4pt);
\draw[ thick, cyan](X)--++(-45:1) coordinate (X) ;

\fill[cyan](X) circle (4pt);
\draw[ thick, cyan](X)--++(45:1) coordinate (X) ;
\fill[cyan](X) circle (4pt);
\draw[ thick, cyan](X)--++(-45:1) coordinate (X) ;
\fill[cyan](X) circle (4pt);
\draw[ thick, cyan](X)--++(-45:1) coordinate (X) ;

\fill[cyan](X) circle (4pt);
\draw[ thick, cyan](X)--++(45:1) coordinate (X) ;
\fill[cyan](X) circle (4pt);
\draw[ thick, cyan](X)--++(45:1) coordinate (X) ;
\fill[cyan](X) circle (4pt);
\draw[ thick, cyan](X)--++(-45:1) coordinate (X) ;
\fill[cyan](X) circle (4pt);
\draw[ thick, cyan](X)--++(-45:1) coordinate (X) ;
\fill[cyan](X) circle (4pt);

 \path (start)--++(135:1)  coordinate (start) ;  \path (start) coordinate (X) ; 
\fill[magenta](X) circle (4pt);
\draw[ thick, magenta](X)--++(45:1) coordinate (X) ;
\fill[magenta](X) circle (4pt);
\draw[ thick, magenta](X)--++(45:1) coordinate (X) ;
\fill[magenta](X) circle (4pt);
\draw[ thick, magenta](X)--++(45:1) coordinate (X) ;
\fill[magenta](X) circle (4pt);
\draw[ thick, magenta](X)--++(-45:1) coordinate (X) ;
\fill[magenta](X) circle (4pt);
\draw[ thick, magenta](X)--++(-45:1) coordinate (X) ;
\fill[magenta](X) circle (4pt);
\draw[ thick, magenta](X)--++(45:1) coordinate (X) ;
\fill[magenta](X) circle (4pt);
\draw[ thick, magenta](X)--++(45:1) coordinate (X) ;

\fill[magenta](X) circle (4pt);
\draw[ thick, magenta](X)--++(-45:1) coordinate (X) ;
\fill[magenta](X) circle (4pt);
\draw[ thick, magenta](X)--++(45:1) coordinate (X) ;

\fill[magenta](X) circle (4pt);
\draw[ thick, magenta](X)--++(-45:1) coordinate (X) ;
\fill[magenta](X) circle (4pt);
\draw[ thick, magenta](X)--++(-45:1) coordinate (X) ;
\fill[magenta](X) circle (4pt);
\draw[ thick, magenta](X)--++(45:1) coordinate (X) ;

\fill[magenta](X) circle (4pt);
\draw[ thick, magenta](X)--++(-45:1) coordinate (X) ;
\fill[magenta](X) circle (4pt);
\draw[ thick, magenta](X)--++(45:1) coordinate (X) ;
\fill[magenta](X) circle (4pt);
\draw[ thick, magenta](X)--++(45:1) coordinate (X) ;
\fill[magenta](X) circle (4pt);
\draw[ thick, magenta](X)--++(-45:1) coordinate (X) ;

\fill[magenta](X) circle (4pt);
\draw[ thick, magenta](X)--++(45:1) coordinate (X) ;
\fill[magenta](X) circle (4pt);
\draw[ thick, magenta](X)--++(-45:1) coordinate (X) ;
\fill[magenta](X) circle (4pt);
\draw[ thick, magenta](X)--++(-45:1) coordinate (X) ;

\fill[magenta](X) circle (4pt);
\draw[ thick, magenta](X)--++(45:1) coordinate (X) ;
\fill[magenta](X) circle (4pt);
\draw[ thick, magenta](X)--++(45:1) coordinate (X) ;
\fill[magenta](X) circle (4pt);
\draw[ thick, magenta](X)--++(-45:1) coordinate (X) ;
\fill[magenta](X) circle (4pt);
\draw[ thick, magenta](X)--++(-45:1) coordinate (X) ;
\fill[magenta](X) circle (4pt);
\draw[ thick, magenta](X)--++(-45:1) coordinate (X) ;
\fill[magenta](X) circle (4pt);
 
 \draw[ thick, magenta](X)--++(45:1) coordinate (X) ;
\fill[magenta](X) circle (4pt);
\draw[ thick, magenta](X)--++(45:1) coordinate (X) ;
\fill[magenta](X) circle (4pt);
\draw[ thick, magenta](X)--++(45:1) coordinate (X) ;
\fill[magenta](X) circle (4pt);

 \fill[magenta](X) circle (4pt);
\draw[ thick, magenta](X)--++(-45:1) coordinate (X) ;
\fill[magenta](X) circle (4pt);
\draw[ thick, magenta](X)--++(-45:1) coordinate (X) ;
\fill[magenta](X) circle (4pt);
\draw[ thick, magenta](X)--++(-45:1) coordinate (X) ;
\fill[magenta](X) circle (4pt);

 \path (start)--++(45:6)--++(-45:2)  coordinate (start) ;  \path (start) coordinate (X) ; 
\fill[magenta](X) circle (4pt);
\draw[ thick, magenta](X)--++(45:1) coordinate (X) ;
\fill[magenta](X) circle (4pt);
\draw[ thick, magenta](X)--++(-45:1) coordinate (X) ;
\fill[magenta](X) circle (4pt);
\draw[ thick, magenta](X)--++(45:1) coordinate (X) ;
\fill[magenta](X) circle (4pt);
\draw[ thick, magenta](X)--++(45:1) coordinate (X) ;
\fill[magenta](X) circle (4pt);
\draw[ thick, magenta](X)--++(-45:1) coordinate (X) ;
\fill[magenta](X) circle (4pt);
\draw[ thick, magenta](X)--++(-45:1) coordinate (X) ;

\fill[magenta](X) circle (4pt);
\draw[ thick, magenta](X)--++(45:1) coordinate (X) ;
\fill[magenta](X) circle (4pt);
\draw[ thick, magenta](X)--++(-45:1) coordinate (X) ;
\fill[magenta](X) circle (4pt);

 \path (start)--++(45:2)  coordinate (start) ;  \path (start) coordinate (X) ; 
\fill[magenta](X) circle (4pt);
\draw[ thick, magenta](X)--++(45:1) coordinate (X) ;
\fill[magenta](X) circle (4pt);
\draw[ thick, magenta](X)--++(45:1) coordinate (X) ;
 \fill[magenta](X) circle (4pt);
\draw[ thick, magenta](X)--++(-45:1) coordinate (X) ;
\fill[magenta](X) circle (4pt);
\draw[ thick, magenta](X)--++(-45:1) coordinate (X) ;
\fill[magenta](X) circle (4pt);

  \path (start)--++(-45:2)--++(-135:1)  coordinate (start) ;  \path (start) coordinate (X) ; 
\fill[magenta](X) circle (4pt);
\draw[ thick, magenta](X)--++(45:1) coordinate (X) ;
\fill[magenta](X) circle (4pt);
\draw[ thick, magenta](X)--++(-45:1) coordinate (X) ;
\fill[magenta](X) circle (4pt);

  \path (start)--++(-45:4)--++(45:4)  coordinate (start) ;  \path (start) coordinate (X) ; 
\fill[magenta](X) circle (4pt);
  \path (start)--++(135:1) coordinate (start) ;  \path (start) coordinate (X) ; 
\fill[magenta](X) circle (4pt);
\draw[ thick, magenta](X)--++(45:1) coordinate (X) ;
\fill[magenta](X) circle (4pt);
\draw[ thick, magenta](X)--++(-45:1) coordinate (X) ;
\fill[magenta](X) circle (4pt);

  \path (start)--++(-45:8)--++(45:1)  coordinate (start) ;  \path (start) coordinate (X) ; 
\fill[cyan](X) circle (4pt);
  \path (start)--++(135:1) coordinate (start) ;  \path (start) coordinate (X) ; 
\fill[cyan](X) circle (4pt);
\draw[ thick, cyan](X)--++(45:1) coordinate (X) ;
\fill[cyan](X) circle (4pt);
\draw[ thick, cyan](X)--++(-45:1) coordinate (X) ;
\fill[cyan](X) circle (4pt);

  \path (start)--++(135:1) coordinate (start) ;  \path (start) coordinate (X) ; 
\fill[cyan](X) circle (4pt);
\draw[ thick, cyan](X)--++(45:1) coordinate (X) ;
\fill[cyan](X) circle (4pt);
\draw[ thick, cyan](X)--++(45:1) coordinate (X) ;
\fill[cyan](X) circle (4pt);
\draw[ thick, cyan](X)--++(-45:1) coordinate (X) ;
\fill[cyan](X) circle (4pt);
\draw[ thick, cyan](X)--++(-45:1) coordinate (X) ;
\fill[cyan](X) circle (4pt);

  \path (start)--++(45:5) --++(-45:4) coordinate (start) ;  \path (start) coordinate (X) ; 
\fill[magenta](X) circle (4pt);
  \path (start)--++(135:1) coordinate (start) ;  \path (start) coordinate (X) ; 
\fill[magenta](X) circle (4pt);
\draw[ thick, magenta](X)--++(45:1) coordinate (X) ;
\fill[magenta](X) circle (4pt);
\draw[ thick, magenta](X)--++(-45:1) coordinate (X) ;
\fill[magenta](X) circle (4pt);

  \path (0,0)--++(135:12.5) --++(45:8.5) coordinate (start) ;  \path (start) coordinate (X) ; 

  \fill[cyan](X) circle (4pt);

\end{tikzpicture} 
\caption{A Dyck tiling for the Dyck pair $(\la,\mu)$ as in \cref{figureSatSets} chosen so as to highlight the heights. The canonical tableau is the tableau which adds these specific Dyck paths in order from bottom-to-top and left-to-right. 
The Dyck paths of strictly positive height are coloured
{\color{magenta}pink} and all others are coloured {\color{cyan}blue}.
Compare with \cref{figureSatSets}.}
\label{BigDyck}
\end{figure}
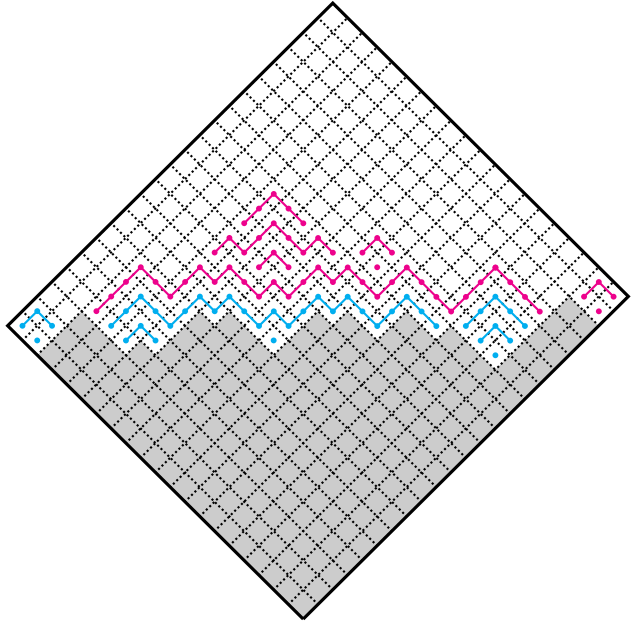

\begin{defn}
We refer to the tableau of \eqref{kfjghsdjkghdjfk} as 
  the {\sf canonical Dyck tableau}. 
 For $k \in \ZZ$ we define 
 $$
( \mu\setminus \la)_k = \{ P \mid P \in  \mu\setminus \la \text { and } {\sf ht}^\mu_\la(P)=k\}
 $$and we define $( \mu\setminus \la)_{\leq k }$ and $( \mu\setminus \la)_{\geq k }$ in a similar fashion. 
See \cref{BigDyck} for an example. \end{defn}

\begin{rmk}
The canonical Dyck tiling corresponds to the unique choice of representatives \(\{P_b^1, \ldots, P_b^k\}\) for the Dyck tiles \(\{P^1, \ldots, P^k\}\) 
with the following property: for all \(i,j\) with \(P^i \prec P^j\) we have that  \(P^i_b\) is below \(P^j_b\) only if there is no other choice of representatives \(\{P_{b'}^1, \ldots, P_{b'}^k\}\) with \(P^i_b\) above \(P^j_b\).

One may arrive at the canonical tiling for \(\mu \backslash \la\) from any choice of representatives for the Dyck tiling of \(\mu \backslash \la\) by iteratively commuting shorter Dyck tiles upwards past longer Dyck tiles whenever possible.
The acceptable arrangements for Dyck paths $\color{cyan}P\color{black} \prec \color{magenta}Q$  in a canonical tiling are captured 
in \cref{goodbad}.
\end{rmk}

\begin{figure}[ht!]
\begin{align*}
\begin{tikzpicture}[scale=0.35]
 \path(0,0) coordinate (origin2);
  \path(0,0)--++(135:2) coordinate (origin3);
     \foreach \i in {0,1,2,3,4}
{
\path (origin3)--++(45:0.5*\i) coordinate (c\i); 
\path (origin3)--++(135:0.5*\i)  coordinate (d\i); 
  }
   \foreach \i in {0,1,2,3,4}
{
\path (origin3)--++(45:1*\i) coordinate (c\i); 
\path (c\i)--++(-45:0.5) coordinate (c\i); 
\path (origin3)--++(135:1*\i)  coordinate (d\i); 
\path (d\i)--++(-135:0.5) coordinate (d\i); 
\draw[thick,densely dotted] (c\i)--++(135:5);
\draw[thick,densely dotted] (d\i)--++(45:5);
  }
\path(origin3)--++(45:-0.5)--++(135:3.5) coordinate (X);
%
\path(X)--++(45:1) coordinate (X) ;
\fill[magenta](X) circle (4pt);
\draw[ thick, magenta](X)--++(45:1) coordinate (X) ;
\fill[magenta](X) circle (4pt);
\draw[ thick, magenta](X)--++(45:1) coordinate (X) ;
\fill[magenta](X) circle (4pt);
\draw[ thick, magenta](X)--++(45:1) coordinate (X) ;
\fill[magenta](X) circle (4pt);
\draw[ thick, magenta](X)--++(-45:1) coordinate (X) ;
 \fill[magenta](X) circle (4pt);
\draw[ thick, magenta](X)--++(-45:1) coordinate (X) ;
\fill[magenta](X) circle (4pt);
\draw[ thick, magenta](X)--++(-45:1) coordinate (X) ;
\fill[magenta](X) circle (4pt);
 \path (X)--++(45:-3)--++(-45:-2) coordinate (X) ; 
\fill[cyan](X) circle (4pt);
\draw[ thick, cyan](X)--++(45:1) coordinate (X) ;
\fill[cyan](X) circle (4pt);
\draw[ thick, cyan](X)--++(45:1) coordinate (X) ;
\fill[cyan](X) circle (4pt);
\draw[ thick, cyan](X)--++(-45:1) coordinate (X) ;
\fill[cyan](X) circle (4pt);
\draw[ thick, cyan](X)--++(-45:1) coordinate (X) ;
\fill[cyan](X) circle (4pt);
 \path (X)--++(45:-4)--++(-45:-5) coordinate (X) ; 
\end{tikzpicture} 
\qquad
\begin{tikzpicture}[scale=0.35]
 \path(0,0) coordinate (origin2);
  \path(0,0)--++(135:2) coordinate (origin3);
     \foreach \i in {0,1,2,3,4}
{
\path (origin3)--++(45:0.5*\i) coordinate (c\i); 
\path (origin3)--++(135:0.5*\i)  coordinate (d\i); 
  }
   \foreach \i in {0,1,2,3,4}
{
\path (origin3)--++(45:1*\i) coordinate (c\i); 
\path (c\i)--++(-45:0.5) coordinate (c\i); 
\path (origin3)--++(135:1*\i)  coordinate (d\i); 
\path (d\i)--++(-135:0.5) coordinate (d\i); 
\draw[thick,densely dotted] (c\i)--++(135:5);
\draw[thick,densely dotted] (d\i)--++(45:5);
  }
\path(origin3)--++(45:-0.5)--++(135:3.5) coordinate (X);
%
\path(X)--++(45:1) coordinate (X) ;
\fill[magenta](X) circle (4pt);
\draw[ thick, magenta](X)--++(45:1) coordinate (X) ;
\fill[magenta](X) circle (4pt);
\draw[ thick, magenta](X)--++(-45:1) coordinate (X) ;
\fill[magenta](X) circle (4pt);
\draw[ thick, magenta](X)--++(45:1) coordinate (X) ;
\fill[magenta](X) circle (4pt);
\draw[ thick, magenta](X)--++(-45:1) coordinate (X) ;
 \fill[magenta](X) circle (4pt);
\draw[ thick, magenta](X)--++(45:1) coordinate (X) ;
\fill[magenta](X) circle (4pt);
\draw[ thick, magenta](X)--++(-45:1) coordinate (X) ;
\fill[magenta](X) circle (4pt);
 \path (X)--++(45:-1)--++(-45:-3) coordinate (X) ; 
\fill[cyan](X) circle (4pt);
\draw[ thick, cyan](X)--++(45:1) coordinate (X) ;
\fill[cyan](X) circle (4pt);
\draw[ thick, cyan](X)--++(-45:1) coordinate (X) ;
\fill[cyan](X) circle (4pt);
\end{tikzpicture} 
\begin{tikzpicture}[scale=0.35]
 \path(0,0) coordinate (origin2);
  \path(0,0)--++(135:2) coordinate (origin3);
     \foreach \i in {0,1,2,3,4}
{
\path (origin3)--++(45:0.5*\i) coordinate (c\i); 
\path (origin3)--++(135:0.5*\i)  coordinate (d\i); 
  }
   \foreach \i in {0,1,2,3,4}
{
\path (origin3)--++(45:1*\i) coordinate (c\i); 
\path (c\i)--++(-45:0.5) coordinate (c\i); 
\path (origin3)--++(135:1*\i)  coordinate (d\i); 
\path (d\i)--++(-135:0.5) coordinate (d\i); 
\draw[thick,densely dotted] (c\i)--++(135:5);
\draw[thick,densely dotted] (d\i)--++(45:5);
  }
\path(origin3)--++(45:-0.5)--++(135:3.5) coordinate (X);
%
\path(X)--++(45:1) coordinate (X) ;
\fill[magenta](X) circle (4pt);
\draw[ thick, magenta](X)--++(45:1) coordinate (X) ;
\fill[magenta](X) circle (4pt);
\draw[ thick, magenta](X)--++(45:1) coordinate (X) ;
\fill[magenta](X) circle (4pt);
\draw[ thick, magenta](X)--++(45:1) coordinate (X) ;
\fill[magenta](X) circle (4pt);
\draw[ thick, magenta](X)--++(-45:1) coordinate (X) ;
 \fill[magenta](X) circle (4pt);
\draw[ thick, magenta](X)--++(-45:1) coordinate (X) ;
\fill[magenta](X) circle (4pt);
\draw[ thick, magenta](X)--++(-45:1) coordinate (X) ;
\fill[magenta](X) circle (4pt);
 \path (X)--++(45:-2)--++(-45:-2) coordinate (X) ; 
\fill[cyan](X) circle (4pt);
\draw[ thick, cyan](X)--++(45:1) coordinate (X) ;
\fill[cyan](X) circle (4pt);
\draw[ thick, cyan](X)--++(-45:1) coordinate (X) ;
\fill[cyan](X) circle (4pt);
 \path (X)--++(45:-4)--++(-45:-4) coordinate (X) ; 
 \end{tikzpicture} 
\end{align*}
\caption{The first two arrangements of Dyck paths are good and the final arrangement is bad.  
This is because the small blue Dyck path in the third diagram can be commuted upwards to obtained the second diagram (thus turning the bad arrangement into a good one).
See    \cref{figureSatSets}  for an  example.}
\label{goodbad}
\end{figure}
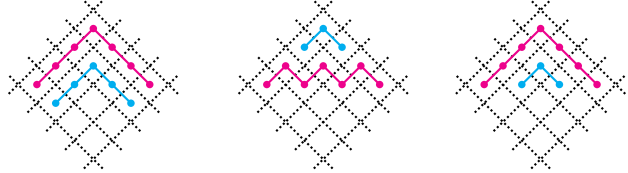

\section{The regularisation map} \label{regulatiaasaiofgsaf}
 
We now define a regularisation map from non-regular partitions to regular partitions. 
For any Dyck pair $(\alpha,\mu)$, we  
 then  construct a sequence  from ${\sf reg}(\alpha)$ to 
 $\mu$ given by adding and splitting Dyck paths. This is the combinatorial shadow of a representation theoretic   result: 
  in  \cref{arcalgebras} we will use this map to show that  the   $H^m_n$-cell-module $S_{m,n}(\alpha)$ has simple head
 $D_{m,n}({\sf reg}(\alpha))$.
  
 \begin{defn}
 Given $\alpha \in \mptn $ with $d(\alpha)=d<0$, we define the regularisation of $\alpha$, denoted $ {\sf reg}(\alpha)$,  as follows.  
We define a sequence of partitions  
\begin{align}\label{Step1}
\alpha= {\sf reg}_{d}(\alpha) \subset  {\sf reg}_{d+1}(\alpha) \subset \dots \subset  {\sf reg}_{-1}(\alpha) \subset  {\sf reg}_{0}(\alpha) = {\sf reg}(\alpha)
\end{align}such that $P^{k}= {\sf reg}_{k}(\alpha) \setminus  {\sf reg}_{k-1}(\alpha)$ is the maximal breadth addable Dyck path  of height $k$.  
In other words, $P^{k}$ is the Dyck path obtained by connecting all the nodes $[r,c]\not \in {\sf reg}_{k-1}(\alpha) $ such that $r+c-1-m = k$.  

From the arc diagram point of view,  $ \underline{{\sf reg}(\alpha)}\alpha$ is the diagram obtained from $\underline{\alpha}\alpha$ by 
applying (the reverse of) as many  $(G1)$ good moves as possible, thus replacing all $|d(\alpha)|$ south-westerly strands 
(and therefore a corresponding  $|d(\alpha)|$ south-easterly strands) with cups.
\end{defn}


 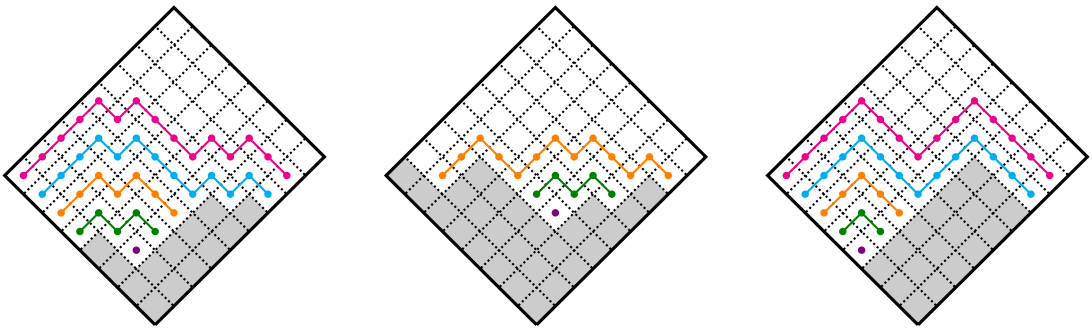
\begin{figure}[ht!]
 $$   \begin{tikzpicture}[scale=0.35]
  \path(0,0)--++(135:2) coordinate (hhhh);
 \draw[very thick] (hhhh)--++(135:8)--++(45:9)--++(-45:8)--++(-135:9);
 \clip(hhhh)--++(135:8)--++(45:9)--++(-45:8)--++(-135:9);
\path(0,0) coordinate (origin2);
  \path(0,0)--++(135:2) coordinate (origin3);

     \foreach \i in {0,1,2,3,4,5,6,7,8}
{
\path (origin3)--++(45:0.5*\i) coordinate (c\i); 
\path (origin3)--++(135:0.5*\i)  coordinate (d\i); 
  }

   \foreach \i in {0,1,2,3,4,5,6,7,8}
{
\path (origin3)--++(45:1*\i) coordinate (c\i); 
\path (c\i)--++(-45:0.5) coordinate (c\i); 
\path (origin3)--++(135:1*\i)  coordinate (d\i); 
\path (d\i)--++(-135:0.5) coordinate (d\i); 
\draw[thick,densely dotted] (c\i)--++(135:10);
\draw[thick,densely dotted] (d\i)--++(45:10);
  }

\path(0,0)--++(135:2) coordinate (hhhh);

\fill[opacity=0.2](hhhh)--++(135:4)--++(45:1)--++(-45:2)--++(45:4)--++(-45:1)--++(45:1)--++(-45:1);

\path(origin3)--++(45:-0.5)--++(135:7.5) coordinate (X)coordinate (start);

\path(X)--++(45:1) coordinate (X) ;
\fill[magenta](X) circle (4pt);
\draw[ thick, magenta](X)--++(45:1) coordinate (X) ;
\fill[magenta](X) circle (4pt);
\draw[ thick, magenta](X)--++(45:1) coordinate (X) ;
\fill[magenta](X) circle (4pt);
\draw[ thick, magenta](X)--++(45:1) coordinate (X) ;
\fill[magenta](X) circle (4pt);
\draw[ thick, magenta](X)--++(45:1) coordinate (X) ;
 \fill[magenta](X) circle (4pt);
\draw[ thick, magenta](X)--++(-45:1) coordinate (X) ;

\fill[magenta](X) circle (4pt);
\draw[ thick, magenta](X)--++(45:1) coordinate (X) ;
\fill[magenta](X) circle (4pt);
\draw[ thick, magenta](X)--++(-45:1) coordinate (X) ;
\fill[magenta](X) circle (4pt);
\draw[ thick, magenta](X)--++(-45:1) coordinate (X) ;
\fill[magenta](X) circle (4pt);

\draw[ thick, magenta](X)--++(-45:1) coordinate (X) ;
\fill[magenta](X) circle (4pt);
\draw[ thick, magenta](X)--++(45:1) coordinate (X) ;
\fill[magenta](X) circle (4pt);

\draw[ thick, magenta](X)--++(-45:1) coordinate (X) ;
\fill[magenta](X) circle (4pt);
\draw[ thick, magenta](X)--++(45:1) coordinate (X) ;
\fill[magenta](X) circle (4pt);
\draw[ thick, magenta](X)--++(-45:1) coordinate (X) ;
\fill[magenta](X) circle (4pt);
\draw[ thick, magenta](X)--++(-45:1) coordinate (X) ;
\fill[magenta](X) circle (4pt);

\path (start)--++(45:1)--++(-45:1) coordinate (X) ; 
\fill[cyan](X) circle (4pt);
\draw[ thick, cyan](X)--++(45:1) coordinate (X) ;
\fill[cyan](X) circle (4pt);
\draw[ thick, cyan](X)--++(45:1) coordinate (X) ;
\fill[cyan](X) circle (4pt);
\draw[ thick, cyan](X)--++(45:1) coordinate (X) ;
\fill[cyan](X) circle (4pt);
\draw[ thick, cyan](X)--++(-45:1) coordinate (X) ;
\fill[cyan](X) circle (4pt);
\draw[ thick, cyan](X)--++(45:1) coordinate (X) ;
\fill[cyan](X) circle (4pt);
\draw[ thick, cyan](X)--++(-45:1) coordinate (X) ;
\fill[cyan](X) circle (4pt);
\draw[ thick, cyan](X)--++(-45:1) coordinate (X) ;
\fill[cyan](X) circle (4pt);
\draw[ thick, cyan](X)--++(-45:1) coordinate (X) ;
\fill[cyan](X) circle (4pt);

\draw[ thick, cyan](X)--++(45:1) coordinate (X) ;
\fill[cyan](X) circle (4pt);
\draw[ thick, cyan](X)--++(-45:1) coordinate (X) ;
\fill[cyan](X) circle (4pt);
\draw[ thick, cyan](X)--++(45:1) coordinate (X) ;
\fill[cyan](X) circle (4pt);
\draw[ thick, cyan](X)--++(-45:1) coordinate (X) ;
\fill[cyan](X) circle (4pt);

\path (start)--++(45:1)--++(-45:2) coordinate (X) ; 
\fill[orange](X) circle (4pt);
\draw[ thick, orange](X)--++(45:1) coordinate (X) ;
\fill[orange](X) circle (4pt);
\draw[ thick, orange](X)--++(45:1) coordinate (X) ;
\fill[orange](X) circle (4pt);
\draw[ thick, orange](X)--++(-45:1) coordinate (X) ;
\fill[orange](X) circle (4pt);
\draw[ thick, orange](X)--++(45:1) coordinate (X) ;
\fill[orange](X) circle (4pt);
\draw[ thick, orange](X)--++(-45:1) coordinate (X) ;
\fill[orange](X) circle (4pt);
\draw[ thick, orange](X)--++(-45:1) coordinate (X) ;
\fill[orange](X) circle (4pt);

\path (start)--++(45:1)--++(-45:3) coordinate (X) ; 
\fill[darkgreen](X) circle (4pt);
\draw[ thick, darkgreen](X)--++(45:1) coordinate (X) ;
\fill[darkgreen](X) circle (4pt);
\draw[ thick, darkgreen](X)--++(-45:1) coordinate (X) ;
 \fill[darkgreen](X) circle (4pt);
\draw[ thick, darkgreen](X)--++(45:1) coordinate (X) ;
\fill[darkgreen](X) circle (4pt);
\draw[ thick, darkgreen](X)--++(-45:1) coordinate (X) ;
\fill[darkgreen](X) circle (4pt);

\path (start)--++(45:2)--++(-45:5) coordinate (X) ; 
\fill[violet](X) circle (4pt);

\end{tikzpicture} 
\qquad
  \begin{tikzpicture}[scale=0.35]
  \path(0,0)--++(135:2) coordinate (hhhh);
 \draw[very thick] (hhhh)--++(135:8)--++(45:9)--++(-45:8)--++(-135:9);
 \clip(hhhh)--++(135:8)--++(45:9)--++(-45:8)--++(-135:9);
  
 \path(0,0) coordinate (origin2);
  \path(0,0)--++(135:2) coordinate (origin3);

     \foreach \i in {0,1,2,3,4,5,6,7,8}
{
\path (origin3)--++(45:0.5*\i) coordinate (c\i); 
\path (origin3)--++(135:0.5*\i)  coordinate (d\i); 
  }

   \foreach \i in {0,1,2,3,4,5,6,7,8}
{
\path (origin3)--++(45:1*\i) coordinate (c\i); 
\path (c\i)--++(-45:0.5) coordinate (c\i); 
\path (origin3)--++(135:1*\i)  coordinate (d\i); 
\path (d\i)--++(-135:0.5) coordinate (d\i); 
\draw[thick,densely dotted] (c\i)--++(135:10);
\draw[thick,densely dotted] (d\i)--++(45:10);
  }

\path(0,0)--++(135:2) coordinate (hhhh);

\fill[opacity=0.2](hhhh)
--++(135:8)--++(45:1)--++(-45:2)--++(45:2)--++(-45:4)--++(45:2)--++(-45:1)--++(45:2)--++(-45:1)--++(45:2);

\path(origin3)--++(45:-0.5)--++(135:7.5) coordinate (X)coordinate (start);

\path (start)--++(45:2)--++(-45:1) coordinate (X) ; 
\fill[orange](X) circle (4pt);
\draw[ thick, orange](X)--++(45:1) coordinate (X) ;
\fill[orange](X) circle (4pt);
\draw[ thick, orange](X)--++(45:1) coordinate (X) ;
\fill[orange](X) circle (4pt);
\draw[ thick, orange](X)--++(-45:1) coordinate (X) ;
\fill[orange](X) circle (4pt);
\draw[ thick, orange](X)--++(-45:1) coordinate (X) ;
\fill[orange](X) circle (4pt);
\draw[ thick, orange](X)--++(45:1) coordinate (X) ;
\fill[orange](X) circle (4pt);
\draw[ thick, orange](X)--++(45:1) coordinate (X) ;
\fill[orange](X) circle (4pt);
\draw[ thick, orange](X)--++(-45:1) coordinate (X) ;
\fill[orange](X) circle (4pt);
\draw[ thick, orange](X)--++(45:1) coordinate (X) ;
\fill[orange](X) circle (4pt);

\draw[ thick, orange](X)--++(-45:1) coordinate (X) ;
\fill[orange](X) circle (4pt);
\draw[ thick, orange](X)--++(-45:1) coordinate (X) ;
\fill[orange](X) circle (4pt);
 \draw[ thick, orange](X)--++(45:1) coordinate (X) ;
\fill[orange](X) circle (4pt);

\draw[ thick, orange](X)--++(-45:1) coordinate (X) ;
\fill[orange](X) circle (4pt);

\path (start)--++(45:4)--++(-45:4) coordinate (X) ; 
\fill[darkgreen](X) circle (4pt);
\draw[ thick, darkgreen](X)--++(45:1) coordinate (X) ;
\fill[darkgreen](X) circle (4pt);
\draw[ thick, darkgreen](X)--++(-45:1) coordinate (X) ;
 \fill[darkgreen](X) circle (4pt);
 \draw[ thick, darkgreen](X)--++(45:1) coordinate (X) ;
\fill[darkgreen](X) circle (4pt);
\draw[ thick, darkgreen](X)--++(-45:1) coordinate (X) ;
 \fill[darkgreen](X) circle (4pt);

\path (start)--++(45:4)--++(-45:5) coordinate (X) ; 
\fill[violet](X) circle (4pt);

\end{tikzpicture} 
\qquad
 \begin{tikzpicture}[scale=0.35]
  \path(0,0)--++(135:2) coordinate (hhhh);
 \draw[very thick] (hhhh)--++(135:8)--++(45:9)--++(-45:8)--++(-135:9);
 \clip(hhhh)--++(135:8)--++(45:9)--++(-45:8)--++(-135:9);
\path(0,0) coordinate (origin2);
  \path(0,0)--++(135:2) coordinate (origin3);

     \foreach \i in {0,1,2,3,4,5,6,7,8}
{
\path (origin3)--++(45:0.5*\i) coordinate (c\i); 
\path (origin3)--++(135:0.5*\i)  coordinate (d\i); 
  }

   \foreach \i in {0,1,2,3,4,5,6,7,8}
{
\path (origin3)--++(45:1*\i) coordinate (c\i); 
\path (c\i)--++(-45:0.5) coordinate (c\i); 
\path (origin3)--++(135:1*\i)  coordinate (d\i); 
\path (d\i)--++(-135:0.5) coordinate (d\i); 
\draw[thick,densely dotted] (c\i)--++(135:10);
\draw[thick,densely dotted] (d\i)--++(45:10);
  }

\path(0,0)--++(135:2) coordinate (hhhh);

\fill[opacity=0.2](hhhh)--++(135:3)--++(45:6)--++(-45:3);

\path(origin3)--++(45:-0.5)--++(135:7.5) coordinate (X)coordinate (start);

\path(X)--++(45:1) coordinate (X) ;
\fill[magenta](X) circle (4pt);
\draw[ thick, magenta](X)--++(45:1) coordinate (X) ;
\fill[magenta](X) circle (4pt);
\draw[ thick, magenta](X)--++(45:1) coordinate (X) ;
\fill[magenta](X) circle (4pt);
\draw[ thick, magenta](X)--++(45:1) coordinate (X) ;
\fill[magenta](X) circle (4pt);
\draw[ thick, magenta](X)--++(45:1) coordinate (X) ;
 \fill[magenta](X) circle (4pt);
\draw[ thick, magenta](X)--++(-45:1) coordinate (X) ;

\fill[magenta](X) circle (4pt);
\draw[ thick, magenta](X)--++(-45:1) coordinate (X) ;
\fill[magenta](X) circle (4pt);
\draw[ thick, magenta](X)--++(-45:1) coordinate (X) ;
\fill[magenta](X) circle (4pt);
\draw[ thick, magenta](X)--++(45:1) coordinate (X) ;
\fill[magenta](X) circle (4pt);

\draw[ thick, magenta](X)--++(45:1) coordinate (X) ;
\fill[magenta](X) circle (4pt);
\draw[ thick, magenta](X)--++(45:1) coordinate (X) ;
\fill[magenta](X) circle (4pt);

\draw[ thick, magenta](X)--++(-45:1) coordinate (X) ;
\fill[magenta](X) circle (4pt);
\draw[ thick, magenta](X)--++(-45:1) coordinate (X) ;
\fill[magenta](X) circle (4pt);
\draw[ thick, magenta](X)--++(-45:1) coordinate (X) ;
\fill[magenta](X) circle (4pt);
\draw[ thick, magenta](X)--++(-45:1) coordinate (X) ;
\fill[magenta](X) circle (4pt);

\path (start)--++(45:1)--++(-45:1) coordinate (X) ; 
\fill[cyan](X) circle (4pt);
\draw[ thick, cyan](X)--++(45:1) coordinate (X) ;
\fill[cyan](X) circle (4pt);
\draw[ thick, cyan](X)--++(45:1) coordinate (X) ;
\fill[cyan](X) circle (4pt);
\draw[ thick, cyan](X)--++(45:1) coordinate (X) ;
\fill[cyan](X) circle (4pt);
\draw[ thick, cyan](X)--++(-45:1) coordinate (X) ;
\fill[cyan](X) circle (4pt);
\draw[ thick, cyan](X)--++(-45:1) coordinate (X) ;
\fill[cyan](X) circle (4pt);
\draw[ thick, cyan](X)--++(-45:1) coordinate (X) ;
\fill[cyan](X) circle (4pt);
\draw[ thick, cyan](X)--++(45:1) coordinate (X) ;
\fill[cyan](X) circle (4pt);
\draw[ thick, cyan](X)--++(45:1) coordinate (X) ;
\fill[cyan](X) circle (4pt);

\draw[ thick, cyan](X)--++(45:1) coordinate (X) ;
\fill[cyan](X) circle (4pt);
\draw[ thick, cyan](X)--++(-45:1) coordinate (X) ;
\fill[cyan](X) circle (4pt);
\draw[ thick, cyan](X)--++(-45:1) coordinate (X) ;
\fill[cyan](X) circle (4pt);
\draw[ thick, cyan](X)--++(-45:1) coordinate (X) ;
\fill[cyan](X) circle (4pt);

\path (start)--++(45:1)--++(-45:2) coordinate (X) ; 
\fill[orange](X) circle (4pt);
\draw[ thick, orange](X)--++(45:1) coordinate (X) ;
\fill[orange](X) circle (4pt);
\draw[ thick, orange](X)--++(45:1) coordinate (X) ;
\fill[orange](X) circle (4pt);
\draw[ thick, orange](X)--++(-45:1) coordinate (X) ;
\fill[orange](X) circle (4pt);
\draw[ thick, orange](X)--++(-45:1) coordinate (X) ;
\fill[orange](X) circle (4pt);

\path (start)--++(45:1)--++(-45:3) coordinate (X) ; 
\fill[darkgreen](X) circle (4pt);
\draw[ thick, darkgreen](X)--++(45:1) coordinate (X) ;
\fill[darkgreen](X) circle (4pt);
\draw[ thick, darkgreen](X)--++(-45:1) coordinate (X) ;
 \fill[darkgreen](X) circle (4pt);

\path (start)--++(45:1)--++(-45:4) coordinate (X) ; 
\fill[violet](X) circle (4pt);

\end{tikzpicture} $$

\caption{Three partitions of negative  defect  and their regularisations.}
\label{exls_reg}
\end{figure}

\begin{rmk}\label{deg_reg_ht}
The regularisation process produces the canonical tiling
 $${\sf reg}(\alpha) \setminus \alpha = \bigsqcup\limits_{d(\alpha)+1\leq k\leq 0}  P^{k} $$of ${\sf reg}(\alpha) \setminus \alpha$ and with  ${\sf ht}(P^k)=k\leq 0$  for all $d(\alpha)+1\leq k\leq 0$. 
 Moreover, this is the unique  tiling of  ${\sf reg}(\alpha) \setminus \alpha$ and is of degree $|d(\alpha)|$. 
\end{rmk}

\begin{exl}\label{ex:seq_reg}
Consider the partition $\alpha=(8,6^2,2^2,1^2)$ whose regularisation is the one in the centre of   \cref{exls_reg}. We have that $d(\alpha)=-3$ and the sequence of partitions produced by the regularisation is: $$\alpha={\sf reg}_{-3}(\alpha) \subset {\sf reg}_{-2}(\alpha)\subset {\sf reg}_{-1}(\alpha)\subset {\sf reg}_{0}(\alpha)= {\sf reg}(\alpha)$$
where ${\sf reg}_{-2}(\alpha)={\sf reg}_{-3}(\alpha) + \color{violet}{P^{-2}}$, ${\sf reg}_{-1}(\alpha)={\sf reg}_{-2}(\alpha) + \color{darkgreen}{P^{-1}}$, and ${\sf reg}_{0}(\alpha)={\sf reg}_{-1}(\alpha) + \color{orange}{P^{0}}$. Each $P^{k}$ for $-2 \leq k \leq 0$ corresponds to the the Dyck path of the same colour in the Figure \ref{seq_reg_exl}.  The corresponding sequence of oriented diagrams is pictured in \cref{orient1}
\end{exl}

\begin{figure}[ht!]
$$
 \begin{minipage}{3.4cm}  \begin{tikzpicture} [scale=0.4]
 		\clip(0.75,-2) rectangle (9.25,1);	
	\path (0,0) coordinate (origin); 
		\path (origin)--++(0.5,0.5) coordinate (origin2);  
\path(origin2)--++(0:0.25) 	 coordinate (origin3); 	\draw(origin3)--++(0:8.5); 
		\foreach \i in {1,2,3,4,5,...,17}
		{
			\path (origin2)--++(0:0.5*\i) coordinate (a\i); 
 			\path (origin2)--++(0:0.5*\i)--++(-90:0.00) coordinate (c\i); 
			  }
		
		\foreach \i in {1,2,3,4,5,6}
		{
	\path   (0.75,0) --++(-90:\i*0.5 )coordinate (L\i); 			
	\path   (9.25,0) --++(-90:\i*0.5-0.4)coordinate (R\i); 			
			}
		
		\foreach \i in {1,2,3,4,5,...,35}
		{
			\path (origin2)--++(0:0.25*\i) --++(-90:0.5) coordinate (b\i); 
			\path (origin2)--++(0:0.25*\i) --++(-90:0.9) coordinate (d\i); 
		}
			\foreach \i in {1,4,5,10,11,13,14,16,17}
	{		\path(a\i) --++(90:0.14) node  {\scalefont{0.7}$  \down   $} ;}
		 	
			\foreach \i in {2,3,6,7,8,9,12,15}
	{		\path(a\i) --++(90:-0.18) node  {\scalefont{0.7}$  \up   $} ;}

		\draw[    thick](c2) to [out=-90,in=0] (b3) to [out=180,in=-90] (c1); 
  	
		\draw[    thick](c6) to [out=-90,in=0] (b11) to [out=180,in=-90] (c5); 
 	
			\draw[    thick](c7) to [out=-90,in=0] (d11) to [out=180,in=-90] (c4);

 		\draw[    thick](c3) --++(90:-0.25) to [out=-90,in=0] (L1);  
 		\draw[    thick](c8) --++(90:-0.25) to [out=-90,in=0] (L2);  
 		\draw[    thick](c9) --++(90:-0.25) to [out=-90,in=0] (L3);  
  		 
		  		\draw[    thick](c10) --++(90:-0.25) to [out=-90,in=180] (R4);  

 		\draw[    thick](c12) to [out=-90,in=0] (b23) to [out=180,in=-90] (c11);

		  		\draw[    thick](c13) --++(90:-0.25) to [out=-90,in=180] (R3);  
 		\draw[    thick](c15) to [out=-90,in=0] (b29) to [out=180,in=-90] (c14); 

				\draw[    thick](c16) --++(90:-0.25) to [out=-90,in=180] (R2);   \draw[    thick](c17) --++(90:-0.25) to [out=-90,in=180] (R1);  
	\end{tikzpicture}  
	\end{minipage}
\;{\color{violet} \longrightarrow} \;   
 \begin{minipage}{3.4cm}  \begin{tikzpicture} [scale=0.4]
 		\clip(0.75,-2) rectangle (9.25,1);	
	\path (0,0) coordinate (origin); 
		\path (origin)--++(0.5,0.5) coordinate (origin2);  
\path(origin2)--++(0:0.25) 	 coordinate (origin3); 	\draw(origin3)--++(0:8.5); 
		\foreach \i in {1,2,3,4,5,...,17}
		{
			\path (origin2)--++(0:0.5*\i) coordinate (a\i); 
 			\path (origin2)--++(0:0.5*\i)--++(-90:0.00) coordinate (c\i); 
			  }
		
		\foreach \i in {1,2,3,4,5,6}
		{
	\path   (0.75,0) --++(-90:\i*0.5 )coordinate (L\i); 			
	\path   (9.25,0) --++(-90:\i*0.5-0.4)coordinate (R\i); 			
			}

		\draw[    thick](c2) to [out=-90,in=0] (b3) to [out=180,in=-90] (c1); 
  	
		\draw[    thick](c6) to [out=-90,in=0] (b11) to [out=180,in=-90] (c5); 
 	
			\draw[    thick](c7) to [out=-90,in=0] (d11) to [out=180,in=-90] (c4);

 		\draw[    thick](c3) --++(90:-0.25) to [out=-90,in=0] (L1);  
 		\draw[    thick](c8) --++(90:-0.25) to [out=-90,in=0] (L2);  
  		  		\draw[   violet,  thick](c10) to [out=-90,in=0] (b19) to [out=180,in=-90] (c9); 


 		\draw[    thick](c12) to [out=-90,in=0] (b23) to [out=180,in=-90] (c11);

		  		\draw[    thick](c13) --++(90:-0.25) to [out=-90,in=180] (R3);  
 		\draw[    thick](c15) to [out=-90,in=0] (b29) to [out=180,in=-90] (c14); 

				\draw[    thick](c16) --++(90:-0.25) to [out=-90,in=180] (R2);   \draw[    thick](c17) --++(90:-0.25) to [out=-90,in=180] (R1);  
	\foreach \i in {1,2,3,4,5,...,35}
		{
			\path (origin2)--++(0:0.25*\i) --++(-90:0.5) coordinate (b\i); 
			\path (origin2)--++(0:0.25*\i) --++(-90:0.9) coordinate (d\i); 
		}
			\foreach \i in {1,4,5,10,11,13,14,16,17}
	{		\path(a\i) --++(90:0.14) node  {\scalefont{0.7}$  \down   $} ;}
		 	
			\foreach \i in {2,3,6,7,8,9,12,15}
	{		\path(a\i) --++(90:-0.18) node  {\scalefont{0.7}$  \up   $} ;}
	
	\end{tikzpicture}  
	\end{minipage}
	\;{\color{darkgreen} \longrightarrow} \;   
 \begin{minipage}{3.4cm}  \begin{tikzpicture} [scale=0.4]
 		\clip(0.75,-2) rectangle (9.25,1);	
	\path (0,0) coordinate (origin); 
		\path (origin)--++(0.5,0.5) coordinate (origin2);  
\path(origin2)--++(0:0.25) 	 coordinate (origin3); 	\draw(origin3)--++(0:8.5); 
		\foreach \i in {1,2,3,4,5,...,17}
		{
			\path (origin2)--++(0:0.5*\i) coordinate (a\i); 
 			\path (origin2)--++(0:0.5*\i)--++(-90:0.00) coordinate (c\i); 
			  }
		
		\foreach \i in {1,2,3,4,5,6}
		{
	\path   (0.75,0) --++(-90:\i*0.5 )coordinate (L\i); 			
	\path   (9.25,0) --++(-90:\i*0.5-0.4)coordinate (R\i); 			
			}
		
		\foreach \i in {1,2,3,4,5,...,35}
		{
			\path (origin2)--++(0:0.25*\i) --++(-90:0.5) coordinate (b\i); 
			\path (origin2)--++(0:0.25*\i) --++(-90:0.9) coordinate (d\i); 
		}
			\foreach \i in {1,4,5,10,11,13,14,16,17}
	{		\path(a\i) --++(90:0.14) node  {\scalefont{0.7}$  \down   $} ;}
		 	
			\foreach \i in {2,3,6,7,8,9,12,15}
	{		\path(a\i) --++(90:-0.18) node  {\scalefont{0.7}$  \up   $} ;}

		\draw[    thick](c2) to [out=-90,in=0] (b3) to [out=180,in=-90] (c1); 
  	
		\draw[    thick](c6) to [out=-90,in=0] (b11) to [out=180,in=-90] (c5); 
 	
			\draw[    thick](c7) to [out=-90,in=0] (d11) to [out=180,in=-90] (c4);

 		\draw[    thick](c3) --++(90:-0.25) to [out=-90,in=0] (L1);  
  		  		\draw[  violet,  thick](c10) to [out=-90,in=0] (b19) to [out=180,in=-90] (c9); 

\path(d21)--++(90:-0.25) coordinate (d22);
	\draw[   darkgreen, thick](c13) to [out=-90,in=0] (d22) to [out=180,in=-90] (c8);


 		\draw[     thick](c12) to [out=-90,in=0] (b23) to [out=180,in=-90] (c11);

 		\draw[    thick](c15) to [out=-90,in=0] (b29) to [out=180,in=-90] (c14); 

				\draw[    thick](c16) --++(90:-0.25) to [out=-90,in=180] (R2);   \draw[    thick](c17) --++(90:-0.25) to [out=-90,in=180] (R1);  
	\foreach \i in {1,2,3,4,5,...,35}
		{
			\path (origin2)--++(0:0.25*\i) --++(-90:0.5) coordinate (b\i); 
			\path (origin2)--++(0:0.25*\i) --++(-90:0.9) coordinate (d\i); 
		}
			\foreach \i in {1,4,5,10,11,13,14,16,17}
	{		\path(a\i) --++(90:0.14) node  {\scalefont{0.7}$  \down   $} ;}
		 	
			\foreach \i in {2,3,6,7,8,9,12,15}
	{		\path(a\i) --++(90:-0.18) node  {\scalefont{0.7}$  \up   $} ;}
	
	\end{tikzpicture}  
	\end{minipage}
		\;{\color{orange} \longrightarrow} \;   
 \begin{minipage}{3.4cm}  \begin{tikzpicture} [scale=0.4]
 		\clip(0.75,-2) rectangle (9.25,1);	
	\path (0,0) coordinate (origin); 
		\path (origin)--++(0.5,0.5) coordinate (origin2);  
\path(origin2)--++(0:0.25) 	 coordinate (origin3); 	\draw(origin3)--++(0:8.5); 
		\foreach \i in {1,2,3,4,5,...,17}
		{
			\path (origin2)--++(0:0.5*\i) coordinate (a\i); 
 			\path (origin2)--++(0:0.5*\i)--++(-90:0.00) coordinate (c\i); 
			  }
		
		\foreach \i in {1,2,3,4,5,6}
		{
	\path   (0.75,0) --++(-90:\i*0.5 )coordinate (L\i); 			
	\path   (9.25,0) --++(-90:\i*0.5-0.4)coordinate (R\i); 			
			}
		
		\foreach \i in {1,2,3,4,5,...,35}
		{
			\path (origin2)--++(0:0.25*\i) --++(-90:0.5) coordinate (b\i); 
			\path (origin2)--++(0:0.25*\i) --++(-90:0.9) coordinate (d\i); 
		}
			\foreach \i in {1,4,5,10,11,13,14,16,17}
	{		\path(a\i) --++(90:0.14) node  {\scalefont{0.7}$  \down   $} ;}
		 	
			\foreach \i in {2,3,6,7,8,9,12,15}
	{		\path(a\i) --++(90:-0.18) node  {\scalefont{0.7}$  \up   $} ;}

		\draw[    thick](c2) to [out=-90,in=0] (b3) to [out=180,in=-90] (c1); 
  	
		\draw[    thick](c6) to [out=-90,in=0] (b11) to [out=180,in=-90] (c5); 
 	
			\draw[    thick](c7) to [out=-90,in=0] (d11) to [out=180,in=-90] (c4);

  		  		\draw[  violet,  thick](c10) to [out=-90,in=0] (b19) to [out=180,in=-90] (c9); 

\path(d21)--++(90:-0.25) coordinate (d22);
	\draw[ darkgreen,   thick](c13) to [out=-90,in=0] (d22) to [out=180,in=-90] (c8);


 		\draw[    thick](c12) to [out=-90,in=0] (b23) to [out=180,in=-90] (c11);

\path(b19)--++(90:-0.95) coordinate (b19);

 		\draw[  orange,  thick](c16) to [out=-90,in=0] (b19) to [out=180,in=-90] (c3);

 		\draw[    thick](c15) to [out=-90,in=0] (b29) to [out=180,in=-90] (c14); 

 		\draw[    thick](c17) --++(90:-0.25) to [out=-90,in=180] (R1);  
	\foreach \i in {1,2,3,4,5,...,35}
		{
			\path (origin2)--++(0:0.25*\i) --++(-90:0.5) coordinate (b\i); 
			\path (origin2)--++(0:0.25*\i) --++(-90:0.9) coordinate (d\i); 
		}
			\foreach \i in {1,4,5,10,11,13,14,16,17}
	{		\path(a\i) --++(90:0.14) node  {\scalefont{0.7}$  \down   $} ;}
		 	
			\foreach \i in {2,3,6,7,8,9,12,15}
	{		\path(a\i) --++(90:-0.18) node  {\scalefont{0.7}$  \up   $} ;}
	
	\end{tikzpicture}  
	\end{minipage}
	$$ 
	\caption{On the left we picture $\underline{\alpha}\alpha$ and on the right 
	$\underline{{\sf reg}(\alpha)}\alpha$ for $\alpha=(8,6^2,2^2,1^2)$. Reading the diagrams from right to left we are doing good move \eqref{gm1} three times.}
	\label{orient1}
	\end{figure}
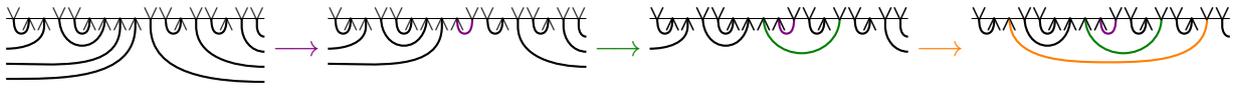
 
\begin{figure}[ht!]
 $$
\begin{minipage}{3.2cm}
\begin{tikzpicture}[scale=0.25]
 \path(0,0)--++(135:2) coordinate (hhhh);
 \draw[very thick] (hhhh)--++(135:8)--++(45:9)--++(-45:8)--++(-135:9);
 \clip(hhhh)--++(135:8)--++(45:9)--++(-45:8)--++(-135:9);
  
 \path(0,0) coordinate (origin2);
  \path(0,0)--++(135:2) coordinate (origin3);

     \foreach \i in {0,1,2,3,4,5,6,7,8}
{
\path (origin3)--++(45:0.5*\i) coordinate (c\i); 
\path (origin3)--++(135:0.5*\i)  coordinate (d\i); 
  }

   \foreach \i in {0,1,2,3,4,5,6,7,8}
{
\path (origin3)--++(45:1*\i) coordinate (c\i); 
\path (c\i)--++(-45:0.5) coordinate (c\i); 
\path (origin3)--++(135:1*\i)  coordinate (d\i); 
\path (d\i)--++(-135:0.5) coordinate (d\i); 
\draw[thick,densely dotted] (c\i)--++(135:10);
\draw[thick,densely dotted] (d\i)--++(45:10);
  }

\path(0,0)--++(135:2) coordinate (hhhh);

\fill[opacity=0.2](hhhh)
--++(135:8)--++(45:1)--++(-45:2)--++(45:2)--++(-45:4)--++(45:2)--++(-45:1)--++(45:2)--++(-45:1)--++(45:2);

\path(origin3)--++(45:-0.5)--++(135:7.5) coordinate (X)coordinate (start);


\end{tikzpicture}\end{minipage}
{\color{violet} \longrightarrow} \; \;
\begin{minipage}{3.2cm}
\begin{tikzpicture}[scale=0.25]
 \path(0,0)--++(135:2) coordinate (hhhh);
 \draw[very thick] (hhhh)--++(135:8)--++(45:9)--++(-45:8)--++(-135:9);
 \clip(hhhh)--++(135:8)--++(45:9)--++(-45:8)--++(-135:9);
  
 \path(0,0) coordinate (origin2);
  \path(0,0)--++(135:2) coordinate (origin3);

     \foreach \i in {0,1,2,3,4,5,6,7,8}
{
\path (origin3)--++(45:0.5*\i) coordinate (c\i); 
\path (origin3)--++(135:0.5*\i)  coordinate (d\i); 
  }

   \foreach \i in {0,1,2,3,4,5,6,7,8}
{
\path (origin3)--++(45:1*\i) coordinate (c\i); 
\path (c\i)--++(-45:0.5) coordinate (c\i); 
\path (origin3)--++(135:1*\i)  coordinate (d\i); 
\path (d\i)--++(-135:0.5) coordinate (d\i); 
\draw[thick,densely dotted] (c\i)--++(135:10);
\draw[thick,densely dotted] (d\i)--++(45:10);
  }

\path(0,0)--++(135:2) coordinate (hhhh);

\fill[opacity=0.2](hhhh)
--++(135:8)--++(45:1)--++(-45:2)--++(45:2)--++(-45:4)--++(45:2)--++(-45:1)--++(45:2)--++(-45:1)--++(45:2);

\path(origin3)--++(45:-0.5)--++(135:7.5) coordinate (X)coordinate (start);

\path (start)--++(45:4)--++(-45:5) coordinate (X) ; 
\fill[violet](X) circle (4pt);

\end{tikzpicture}\end{minipage}
{\color{darkgreen} \longrightarrow}\; \;
\begin{minipage}{3.2cm}
\begin{tikzpicture}[scale=0.25]
 \path(0,0)--++(135:2) coordinate (hhhh);
 \draw[very thick] (hhhh)--++(135:8)--++(45:9)--++(-45:8)--++(-135:9);
 \clip(hhhh)--++(135:8)--++(45:9)--++(-45:8)--++(-135:9);
  
 \path(0,0) coordinate (origin2);
  \path(0,0)--++(135:2) coordinate (origin3);

     \foreach \i in {0,1,2,3,4,5,6,7,8}
{
\path (origin3)--++(45:0.5*\i) coordinate (c\i); 
\path (origin3)--++(135:0.5*\i)  coordinate (d\i); 
  }

   \foreach \i in {0,1,2,3,4,5,6,7,8}
{
\path (origin3)--++(45:1*\i) coordinate (c\i); 
\path (c\i)--++(-45:0.5) coordinate (c\i); 
\path (origin3)--++(135:1*\i)  coordinate (d\i); 
\path (d\i)--++(-135:0.5) coordinate (d\i); 
\draw[thick,densely dotted] (c\i)--++(135:10);
\draw[thick,densely dotted] (d\i)--++(45:10);
  }

\path(0,0)--++(135:2) coordinate (hhhh);

\fill[opacity=0.2](hhhh)
--++(135:8)--++(45:1)--++(-45:2)--++(45:2)--++(-45:4)--++(45:2)--++(-45:1)--++(45:2)--++(-45:1)--++(45:2);

\path(origin3)--++(45:-0.5)--++(135:7.5) coordinate (X)coordinate (start);

\path (start)--++(45:4)--++(-45:4) coordinate (X) ; 
\fill[darkgreen](X) circle (4pt);
\draw[ thick, darkgreen](X)--++(45:1) coordinate (X) ;
\fill[darkgreen](X) circle (4pt);
\draw[ thick, darkgreen](X)--++(-45:1) coordinate (X) ;
 \fill[darkgreen](X) circle (4pt);
 \draw[ thick, darkgreen](X)--++(45:1) coordinate (X) ;
\fill[darkgreen](X) circle (4pt);
\draw[ thick, darkgreen](X)--++(-45:1) coordinate (X) ;
 \fill[darkgreen](X) circle (4pt);

\path (start)--++(45:4)--++(-45:5) coordinate (X) ; 
\fill[violet](X) circle (4pt);

\end{tikzpicture}
\end{minipage}
{\color{orange} \longrightarrow}\; \;
\begin{minipage}{3cm}\begin{tikzpicture}[scale=0.25]
 \path(0,0)--++(135:2) coordinate (hhhh);
 \draw[very thick] (hhhh)--++(135:8)--++(45:9)--++(-45:8)--++(-135:9);
 \clip(hhhh)--++(135:8)--++(45:9)--++(-45:8)--++(-135:9);
  
 \path(0,0) coordinate (origin2);
  \path(0,0)--++(135:2) coordinate (origin3);

     \foreach \i in {0,1,2,3,4,5,6,7,8}
{
\path (origin3)--++(45:0.5*\i) coordinate (c\i); 
\path (origin3)--++(135:0.5*\i)  coordinate (d\i); 
  }

   \foreach \i in {0,1,2,3,4,5,6,7,8}
{
\path (origin3)--++(45:1*\i) coordinate (c\i); 
\path (c\i)--++(-45:0.5) coordinate (c\i); 
\path (origin3)--++(135:1*\i)  coordinate (d\i); 
\path (d\i)--++(-135:0.5) coordinate (d\i); 
\draw[thick,densely dotted] (c\i)--++(135:10);
\draw[thick,densely dotted] (d\i)--++(45:10);
  }

\path(0,0)--++(135:2) coordinate (hhhh);

\fill[opacity=0.2](hhhh)
--++(135:8)--++(45:1)--++(-45:2)--++(45:2)--++(-45:4)--++(45:2)--++(-45:1)--++(45:2)--++(-45:1)--++(45:2);

\path(origin3)--++(45:-0.5)--++(135:7.5) coordinate (X)coordinate (start);

\path (start)--++(45:2)--++(-45:1) coordinate (X) ; 
\fill[orange](X) circle (4pt);
\draw[ thick, orange](X)--++(45:1) coordinate (X) ;
\fill[orange](X) circle (4pt);
\draw[ thick, orange](X)--++(45:1) coordinate (X) ;
\fill[orange](X) circle (4pt);
\draw[ thick, orange](X)--++(-45:1) coordinate (X) ;
\fill[orange](X) circle (4pt);
\draw[ thick, orange](X)--++(-45:1) coordinate (X) ;
\fill[orange](X) circle (4pt);
\draw[ thick, orange](X)--++(45:1) coordinate (X) ;
\fill[orange](X) circle (4pt);
\draw[ thick, orange](X)--++(45:1) coordinate (X) ;
\fill[orange](X) circle (4pt);
\draw[ thick, orange](X)--++(-45:1) coordinate (X) ;
\fill[orange](X) circle (4pt);
\draw[ thick, orange](X)--++(45:1) coordinate (X) ;
\fill[orange](X) circle (4pt);

\draw[ thick, orange](X)--++(-45:1) coordinate (X) ;
\fill[orange](X) circle (4pt);
\draw[ thick, orange](X)--++(-45:1) coordinate (X) ;
\fill[orange](X) circle (4pt);
 \draw[ thick, orange](X)--++(45:1) coordinate (X) ;
\fill[orange](X) circle (4pt);

\draw[ thick, orange](X)--++(-45:1) coordinate (X) ;
\fill[orange](X) circle (4pt);

\path (start)--++(45:4)--++(-45:4) coordinate (X) ; 
\fill[darkgreen](X) circle (4pt);
\draw[ thick, darkgreen](X)--++(45:1) coordinate (X) ;
\fill[darkgreen](X) circle (4pt);
\draw[ thick, darkgreen](X)--++(-45:1) coordinate (X) ;
 \fill[darkgreen](X) circle (4pt);
 \draw[ thick, darkgreen](X)--++(45:1) coordinate (X) ;
\fill[darkgreen](X) circle (4pt);
\draw[ thick, darkgreen](X)--++(-45:1) coordinate (X) ;
 \fill[darkgreen](X) circle (4pt);

\path (start)--++(45:4)--++(-45:5) coordinate (X) ; 
\fill[violet](X) circle (4pt);

\end{tikzpicture}\end{minipage}
$$
\caption{The sequence of partitions produced by the regularisation of $\alpha=(8,6^2,2^2,1^2)$.}
\label{seq_reg_exl}
\end{figure}
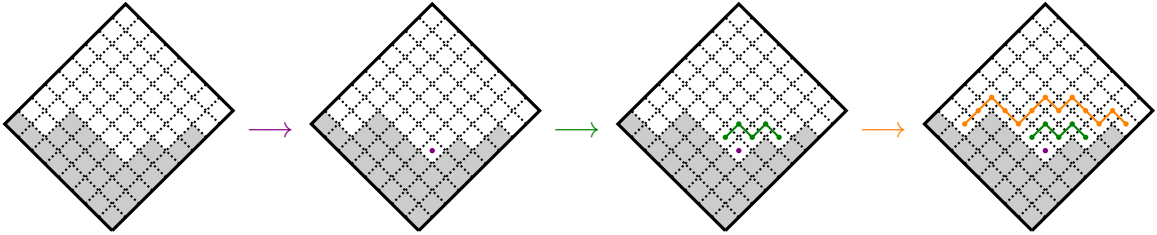

\begin{prop}\label{RegularizePairProp}
Let \(\alpha \in \mptn, \mu    \in \mathscr{R}_{m,n}\). Assume \(\mu    \backslash \alpha = (\mu    \backslash \alpha)_{\leq 0}\) is a Dyck pair of degree \(k\). 
Then \((\mu  , {\sf reg}(\alpha)
 )\) is a Dyck pair of degree \(k +d(\alpha)\).
\end{prop}

%
%
%
%
%
%
%

\begin{proof}
 Consider the diagrams $\underline{\alpha}\alpha$ and $\underline{\mu }\alpha$.  
 We let $\mathcal{A}'$ denote the set of arcs and strands common to both  diagrams  $\underline{\alpha}\alpha$ and $\underline{\mu }\alpha$ and we set  $\color{cyan}\mathcal{A}$ to the 
set of arcs in $\underline{\mu }\alpha$ not belonging to  $\mathcal{A}'$.  
Our assumption that $(\mu \backslash	\alpha)_{\leq 0}= \mu  \backslash	\alpha $ implies that $\underline{\alpha}\alpha$ is obtained from $\underline{\mu}\alpha $ by a sequence of good moves of the form \eqref{gm1} and \eqref{gm2}.  In particular, no end point of an arc in $\color{cyan}\mathcal A$ lies within a region above an arc or strand from $\mathcal A'$.

We let 
$\color{darkgreen}\mathcal{B}$  denote the set of   arcs  in 
 $\underline{\alpha}\alpha$ and which are    not in $\mathcal{A}'$.   
We let  $\color{magenta}\mathcal{C}$
 denote the set of strands in $\underline{\alpha}\alpha$ which do not belong to $\mathcal{A}'$.  
Since $\mu  $ is of defect zero, 
the end points 
 of the arcs 
  in
  $\color{cyan}\mathcal{A}$ are in bijection with the end points of the arcs and strands in  ${\color{darkgreen}\mathcal{B}}\sqcup\color{magenta}\mathcal{C}$.  
We have that   $\color{magenta}\mathcal{C}$ consists of $2|d(\alpha)|$ strands (half south-easterly and half south-westerly) and that $\color{darkgreen}\mathcal{B}$  consists solely of anti-clockwise oriented cups.

We observe that $\underline{{\sf reg}(\alpha)}$ is obtained from $\underline{\alpha}$ by replacing all the strands in $\color{magenta}\mathcal{C}$ with   cups.  
Therefore ${{\sf reg}(\alpha)}$ is obtained from 
$\alpha$ by flipping the labels of the vertices at the end points 
 of the  strands  from  $\color{magenta}\mathcal{C}$.
On the other hand, $\mu  $ is obtained from ${\sf reg}(\alpha)$ by flipping all 
the labels of the vertices 
at the end points 
 of the arcs   in
 $\color{darkgreen}\mathcal{B}$. 

Therefore $\underline{{\sf reg}(\alpha)}{\mu  }		 $ is an oriented diagram of degree equal to the number of arcs in $\color{darkgreen}\mathcal{B}$ (which is equal to \(k + d(\alpha)\)) as required.
%
%
%
%
\end{proof}

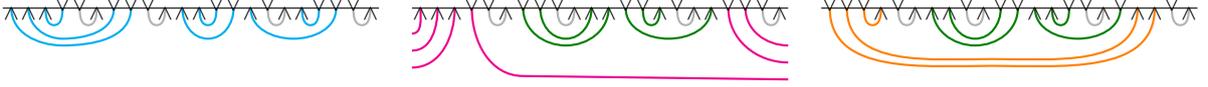
\begin{figure}[h]
 $$
 \begin{minipage}{5cm}  \begin{tikzpicture} [scale=0.45]
 \clip(0.75,-2) rectangle (11.75,1);	
	\path (0,0) coordinate (origin); 
		\path (origin)--++(0.5,0.5) coordinate (origin2);  
\path(origin2)--++(0:0.25) 	 coordinate (origin3); 	\draw(origin3)--++(0:11); 
		\foreach \i in {1,2,3,4,5,...,24}
		{
			\path (origin2)--++(0:0.5*\i) coordinate (a\i); 
 			\path (origin2)--++(0:0.5*\i)--++(-90:0.00) coordinate (c\i); 
			  }
		
		\foreach \i in {1,2,3,4,5,6}
		{
	\path   (0.75,0) --++(-90:\i*0.5 )coordinate (L\i); 			
	\path   (11.25,0) --++(-90:\i*0.5-0.4)coordinate (R\i); 			
			}
	\foreach \i in {1,2,3,4,5,...,80}
		{
			\path (origin2)--++(0:0.25*\i) --++(-90:0.5) coordinate (b\i); 
			\path (origin2)--++(0:0.25*\i) --++(-90:0.9) coordinate (d\i); 
		}

		\path(b8)--++(-90:0.6) coordinate (X);
		\draw[  cyan,  thick](c8) to [out=-90,in=0] (X) to [out=180,in=-90] (c1); 

		\path(b8)--++(-90:0.4) coordinate (X);	
\draw[  cyan,  thick](c7) to [out=-90,in=0] (X) to [out=180,in=-90] (c2); 		

\draw[  cyan,  thick](c4) to [out=-90,in=0] (b7) to [out=180,in=-90] (c3); 		

 \draw[ cyan,  thick](c13) to [out=-90,in=0] (b25) to [out=180,in=-90] (c12); 		

		\path(b25)--++(-90:0.4) coordinate (X);	
 \draw[ cyan,  thick](c14) to [out=-90,in=0] (X) to [out=180,in=-90] (c11); 		

 \draw[  cyan,  thick](c19) to [out=-90,in=0] (b37) to [out=180,in=-90] (c18); 		

		\path(b35)--++(-90:0.4) coordinate (X);	
 \draw[  cyan,  thick](c20) to [out=-90,in=0] (X) to [out=180,in=-90] (c15);

\draw[  gray!60,  thick](c6) to [out=-90,in=0] (b11) to [out=180,in=-90] (c5); 		
 \draw[  gray!60,  thick](c10) to [out=-90,in=0] (b19) to [out=180,in=-90] (c9); 		
 \draw[  gray!60,  thick](c17) to [out=-90,in=0] (b33) to [out=180,in=-90] (c16); 		
 \draw[  gray!60,  thick](c22) to [out=-90,in=0] (b43) to [out=180,in=-90] (c21); 		
 		\foreach \i in  {4,5,7,8,9,13,14,16,19,20,21}
	{		\path(a\i) --++(90:0.12) node  {\scalefont{0.7}$  \down   $} ;}
		 	
			\foreach \i in {1,2,3,6,10,11,12,15,17,18,22} 	{		\path(a\i) --++(90:-0.16) node  {\scalefont{0.7}$  \up   $} ;}
		\end{tikzpicture}  
	\end{minipage}
	\quad
	 \begin{minipage}{5cm}  \begin{tikzpicture} [scale=0.45]
 \clip(0.75,-2) rectangle (11.75,1);	
	\path (0,0) coordinate (origin); 
		\path (origin)--++(0.5,0.5) coordinate (origin2);  
\path(origin2)--++(0:0.25) 	 coordinate (origin3); 	\draw(origin3)--++(0:11); 
		\foreach \i in {1,2,3,4,5,...,23}
		{
			\path (origin2)--++(0:0.5*\i) coordinate (a\i); 
 			\path (origin2)--++(0:0.5*\i)--++(-90:0.00) coordinate (c\i); 
			  }
		
		\foreach \i in {1,2,3,4,5,6}
		{
	\path   (0.75,0) --++(-90:\i*0.5-0.25 )coordinate (L\i); 			
	\path   (11.75,0) --++(-90:\i*0.5-0.4)coordinate (R\i); 			
			}
	\foreach \i in {1,2,3,4,5,...,80}
		{
			\path (origin2)--++(0:0.25*\i) --++(-90:0.5) coordinate (b\i); 
			\path (origin2)--++(0:0.25*\i) --++(-90:0.9) coordinate (d\i); 
		}

 	\draw[  magenta,  thick](c1) to [out=-90,in=0]  (L1);
	 	\draw[   magenta,  thick](c2) to [out=-90,in=0]  (L2);
 	\draw[   magenta,  thick](c3) to [out=-90,in=0]  (L3);		%

	\draw[  magenta,  thick](c4) to [out=-90,in=180]  (4,-1.5)-- (R4);
 \draw[  magenta,  thick](c20) to [out=-90,in=180]  (R2);--++(0:0.1);
  \draw[  magenta,  thick](c19) to [out=-90,in=180]  (R3);--++(0:0.1);

 \path(b19)--++(-90:0.4) coordinate (X);	
  \draw[ darkgreen,  thick](c11) to [out=-90,in=0] (X) to [out=180,in=-90] (c8); 		
   \path(b19)--++(-90:0.6) coordinate (X);	
  \draw[ darkgreen,  thick](c12) to [out=-90,in=0] (X) to [out=180,in=-90] (c7); 		
%

  \draw[  darkgreen,  thick](c15) to [out=-90,in=0] (b29) to [out=180,in=-90] (c14); 		
   \path(b31)--++(-90:0.4) coordinate (X);	

  \draw[  darkgreen,  thick](c18) to [out=-90,in=0] (X) to [out=180,in=-90] (c13); 		
%

\draw[  gray!60,  thick](c6) to [out=-90,in=0] (b11) to [out=180,in=-90] (c5); 		
 \draw[  gray!60,  thick](c10) to [out=-90,in=0] (b19) to [out=180,in=-90] (c9); 		
 \draw[  gray!60,  thick](c17) to [out=-90,in=0] (b33) to [out=180,in=-90] (c16); 		
 \draw[  gray!60,  thick](c22) to [out=-90,in=0] (b43) to [out=180,in=-90] (c21);

	 			\foreach \i in  {4,5,7,8,9,13,14,16,19,20,21}
	{		\path(a\i) --++(90:0.12) node  {\scalefont{0.7}$  \down   $} ;}
		 	
			\foreach \i in {1,2,3,6,10,11,12,15,17,18,22} 	{		\path(a\i) --++(90:-0.16) node  {\scalefont{0.7}$  \up   $} ;}		\end{tikzpicture}  \end{minipage}
\quad	 \begin{minipage}{5cm}  \begin{tikzpicture} [scale=0.45]
 \clip(0.75,-2) rectangle (11.75,1);	
	\path (0,0) coordinate (origin); 
		\path (origin)--++(0.5,0.5) coordinate (origin2);  
\path(origin2)--++(0:0.25) 	 coordinate (origin3); 	\draw(origin3)--++(0:11); 
		\foreach \i in {1,2,3,4,5,...,23}
		{
			\path (origin2)--++(0:0.5*\i) coordinate (a\i); 
 			\path (origin2)--++(0:0.5*\i)--++(-90:0.00) coordinate (c\i); 
			  }
		
		\foreach \i in {1,2,3,4,5,6}
		{
	\path   (0.75,0) --++(-90:\i*0.5-0.25 )coordinate (L\i); 			
	\path   (11.75,0) --++(-90:\i*0.5-0.4)coordinate (R\i); 			
			}
	\foreach \i in {1,2,3,4,5,...,80}
		{
			\path (origin2)--++(0:0.25*\i) --++(-90:0.5) coordinate (b\i); 
			\path (origin2)--++(0:0.25*\i) --++(-90:0.9) coordinate (d\i); 
		}

 	\draw[  orange,  thick](c4) to [out=-90,in=0] (b7) to [out=180,in=-90] (c3); 

 \path(b21)--++(-90:1) coordinate (X);	
	\draw[  orange,  thick](c19) to [out=-90,in=0] (X) to [out=180,in=-90] (c2); 
	 \path(b21)--++(-90:1.2) coordinate (X);	
	\draw[  orange,  thick](c20) to [out=-90,in=0] (X) to [out=180,in=-90] (c1); 
%
%

 \path(b19)--++(-90:0.4) coordinate (X);	
  \draw[ darkgreen,  thick](c11) to [out=-90,in=0] (X) to [out=180,in=-90] (c8); 		
   \path(b19)--++(-90:0.6) coordinate (X);	
  \draw[ darkgreen,  thick](c12) to [out=-90,in=0] (X) to [out=180,in=-90] (c7); 		
 
  \draw[  darkgreen,  thick](c15) to [out=-90,in=0] (b29) to [out=180,in=-90] (c14); 		
   \path(b31)--++(-90:0.4) coordinate (X);	

  \draw[  darkgreen,  thick](c18) to [out=-90,in=0] (X) to [out=180,in=-90] (c13);

\draw[  gray!60,  thick](c6) to [out=-90,in=0] (b11) to [out=180,in=-90] (c5); 		
 \draw[  gray!60,  thick](c10) to [out=-90,in=0] (b19) to [out=180,in=-90] (c9); 		
 \draw[  gray!60,  thick](c17) to [out=-90,in=0] (b33) to [out=180,in=-90] (c16); 		
 \draw[  gray!60,  thick](c22) to [out=-90,in=0] (b43) to [out=180,in=-90] (c21);

 \foreach \i in  {1,2,3,11,12,15,18,5,9,16,21}
	{		\path(a\i) --++(90:0.12) node  {\scalefont{0.7}$  \down   $} ;}
 \foreach \i in {4,7,8,13,14,20,6,10,17,22,19} 	{		\path(a\i) --++(90:-0.16) node  {\scalefont{0.7}$  \up   $} ;}		\end{tikzpicture}  \end{minipage}		
$$
\caption{
From left to right we picture
$\underline{\mu}\alpha$, 
$\underline{\alpha}\alpha$, 
$\underline{{\sf reg}(\alpha)}\mu$. The sets $\color{cyan}\mathcal{A}$, 
 $\color{darkgreen}\mathcal{B}$, and  $\color{magenta}\mathcal{C}$ are coloured.
}
\label{RegProcFig}
\end{figure}

Instead of adding Dyck paths to get from  $\la$
to 
  \( {\sf reg}(\alpha)\), we can remove Dyck paths to get from   \(  {\sf reg}(\alpha)\) to $\la$.  In fact, since $\mu/\alpha$ is a Dyck tiling, these removals can 
  be thought of as splitting Dyck paths of $ {\sf reg}(\alpha)/\alpha$.  In particular, we have the following.

\begin{cor}\label{cor_pure_fav}
Let $(\mu\setminus\alpha)_{\leq 0}$ be a Dyck tiling with $\alpha\notin\mathscr{R}_{m,n}$. Then there is a sequence of partitions of the form
\begin{align}\label{Step2}
 {\sf reg}(\alpha)
 ={\sf split}_{d(\alpha)}(\mu\setminus\alpha) \supseteq {\sf split}_{d(\alpha)+1}(\mu\setminus \alpha)
\dots 
  \supseteq {\sf  split}_{0}(\mu\setminus \alpha)=   
\alpha\sqcup   ( \mu\setminus \alpha)_{\leq 0}
\end{align}
with 
$
({\sf split}_{k}(\mu\setminus\alpha) )\setminus ( {\sf split}_
{k-1}(\mu\setminus \alpha))
=
 R_1^k\sqcup  R_2^k \sqcup\dots \sqcup  R_K^k
$
a disjoint union of commuting Dyck paths such that 
 $$P^k - R_1^k- R_2^k- \dots -R_K^k= (\mu\setminus\alpha)_k.$$
In particular, $R_1^k, R_2^k, \ldots, R_K^k$ are ordered from left to right.  \end{cor}

  \begin{rmk}\label{useful-forsplit}
  It's worth noting that if $R_i^k$ and $R_j^{\ell} $ are such that $k<\ell$, then   either $R_i^k$ and $R_j^{\ell}$ commute  or  $R_i^{k} \succ R_j^{\ell}$.  
  
  \end{rmk}

\begin{exl}\label{ex:split_seq}
Consider the Dyck tiling $\mu\setminus\alpha$  on the left of    \cref{fig:fav_seq_add} where $\alpha$ is the partition of Example \ref{ex:seq_reg}.
  We know the regularisation of $\alpha$ from Example \ref{ex:seq_reg}. Following Corollary \ref{cor_pure_fav}, we have that our ``favourite path'' from ${\sf reg}(\alpha)\setminus\alpha$ to $(\mu\setminus\alpha)_{\leq0}$ is depicted in \cref{favpath}. With the notation introduced in Corollary \ref{cor_pure_fav}, we have that
$${\sf reg}(\alpha)={\sf split}_{-3}(\mu\setminus\alpha)\supseteq {\sf split}_{-1}(\mu\setminus\alpha)\supseteq {\sf split}_{0}(\mu\setminus\alpha)=\alpha\sqcup (\mu\setminus\alpha)_{\leq 0}$$and we have that 
\begin{itemize}[leftmargin=*]
    \item ${\sf reg}(\alpha) = \alpha \sqcup {\color{violet}P^{-2}}\sqcup {\color{darkgreen}P^{-1}}\sqcup {\color{orange}P^{0}}$;
    \item there are no splitting at height $-2$, so ${\sf split}_{-2}(\mu\setminus\alpha) = {\sf split}_{-3}(\mu\setminus\alpha)$;
    \item ${\sf split}_{-1}(\mu\setminus\alpha) \setminus {\sf split}_{-2}(\mu\setminus\alpha) = R^{-1}$ and so $(\mu\setminus\alpha)_{-1}={\color{darkgreen}P^{-1}}-{\color{darkgreen}R^{-1}}$;
    \item  ${\sf split}_{0}(\mu\setminus\alpha) \setminus {\sf split}_{-1}(\mu\setminus\alpha) = R_1^{0} \sqcup R_2^{0}$ and so  $(\mu\setminus\alpha)_{0}={\color{orange}P^{0}}-{\color{orange}R_1^{0}}-{\color{orange}R_2^0}$.
\end{itemize}

\begin{figure}[ht!]
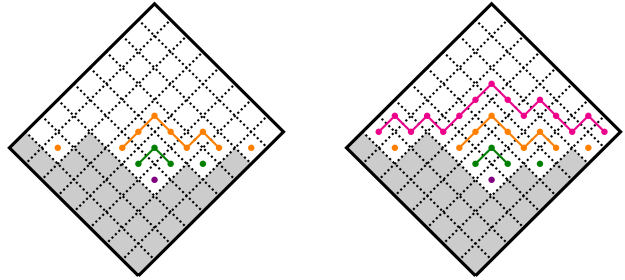

   
 $ 

$
\caption{The Dyck tiling of 
$(\mu\setminus\alpha)_{\leq 0}$  and 
$\mu\setminus\alpha$ with $\mu=(8^3,7,6^3,4,2)$ and $\alpha=(8,6^2,2^2,1^2)$}
\label{fig:fav_seq_add}
\end{figure}

\begin{figure}[ht!]
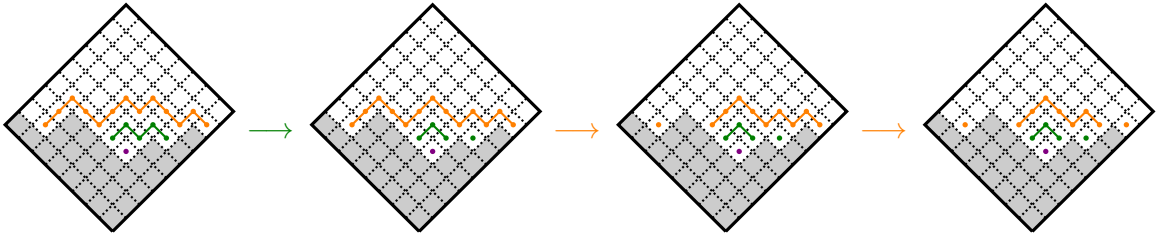

$$
\begin{minipage}{3.2cm}
%
\end{minipage}
$$
\caption{Our favourite path as in \cref{ex:split_seq}.}
\label{favpath}
\end{figure}

\end{exl}

\begin{rmk}\label{cor_not_pure_fav}
Let $\mu\setminus\alpha$ be a Dyck tiling with $\alpha\notin\mathscr{R}_{m,n}$. Then there is a sequence of partitions of the form
\begin{align}\label{Step3}
 ( \mu\setminus \alpha)_{\leq 0}={\sf add}_0(\mu\setminus \alpha) \subseteq {\sf add}_{1}(\mu\setminus \alpha)
\dots  \subseteq {\sf add}_{k-1}(\mu\setminus \alpha)  
  \subseteq {\sf add}_{k}(\mu\setminus \alpha)=   \mu
\end{align}
   where 
   $${\sf add}_k(\mu\setminus \alpha)\setminus 
   {\sf add}_{k-1}(\mu\setminus \alpha)=
    A_1^k\sqcup  A_2^k \sqcup\dots \sqcup  A_K^k
$$is a union of addable Dyck paths of height $k>0$.
\end{rmk}

\begin{exl}\label{yetmore}
Consider the Dyck tiling $\mu\setminus \alpha= (8^3,7,6^3,4,2) \setminus \alpha=(8,6^2,2^2,1^2)$ depicted in \cref{fig:fav_seq_add}.
 Notice that the Dyck tiling of $\mu\setminus\alpha$ has a Dyck path of height $1$ and that $(\mu \setminus \alpha)_{\leq 0}$ is the Dyck tiling consider in Example \ref{ex:split_seq}.  
\end{exl}

 \begin{algorithm}\label{PURE1}
 Let $(\alpha,\mu)$ be a Dyck pair. 
We construct a canonical sequence of add/splits from $\alpha \to \mu$ by first adding Dyck paths  as in  
\eqref{Step1}, then splitting Dyck paths as in \eqref{Step2},   then adding Dyck paths as in \eqref{Step3}.
 \end{algorithm}
  
An example of the three ``big steps" in this algorithm is given in \cref{seq_reg_exl2}, the smaller steps are given in \cref{seq_reg_exl,favpath}.

\begin{figure}[ht!]
 $$
\begin{minipage}{3.2cm}
\begin{tikzpicture}[scale=0.25]
 \path(0,0)--++(135:2) coordinate (hhhh);
 \draw[very thick] (hhhh)--++(135:8)--++(45:9)--++(-45:8)--++(-135:9);
 \clip(hhhh)--++(135:8)--++(45:9)--++(-45:8)--++(-135:9);
  
 \path(0,0) coordinate (origin2);
  \path(0,0)--++(135:2) coordinate (origin3);

     \foreach \i in {0,1,2,3,4,5,6,7,8}
{
\path (origin3)--++(45:0.5*\i) coordinate (c\i); 
\path (origin3)--++(135:0.5*\i)  coordinate (d\i); 
  }

   \foreach \i in {0,1,2,3,4,5,6,7,8}
{
\path (origin3)--++(45:1*\i) coordinate (c\i); 
\path (c\i)--++(-45:0.5) coordinate (c\i); 
\path (origin3)--++(135:1*\i)  coordinate (d\i); 
\path (d\i)--++(-135:0.5) coordinate (d\i); 
\draw[thick,densely dotted] (c\i)--++(135:10);
\draw[thick,densely dotted] (d\i)--++(45:10);
  }

\path(0,0)--++(135:2) coordinate (hhhh);

\fill[opacity=0.2](hhhh)
--++(135:8)--++(45:1)--++(-45:2)--++(45:2)--++(-45:4)--++(45:2)--++(-45:1)--++(45:2)--++(-45:1)--++(45:2);

\path(origin3)--++(45:-0.5)--++(135:7.5) coordinate (X)coordinate (start);


\end{tikzpicture}\end{minipage} 
{  \longrightarrow}\; \;
\begin{minipage}{3.2cm}\begin{tikzpicture}[scale=0.25]
 \path(0,0)--++(135:2) coordinate (hhhh);
 \draw[very thick] (hhhh)--++(135:8)--++(45:9)--++(-45:8)--++(-135:9);
 \clip(hhhh)--++(135:8)--++(45:9)--++(-45:8)--++(-135:9);
  
 \path(0,0) coordinate (origin2);
  \path(0,0)--++(135:2) coordinate (origin3);

     \foreach \i in {0,1,2,3,4,5,6,7,8}
{
\path (origin3)--++(45:0.5*\i) coordinate (c\i); 
\path (origin3)--++(135:0.5*\i)  coordinate (d\i); 
  }

   \foreach \i in {0,1,2,3,4,5,6,7,8}
{
\path (origin3)--++(45:1*\i) coordinate (c\i); 
\path (c\i)--++(-45:0.5) coordinate (c\i); 
\path (origin3)--++(135:1*\i)  coordinate (d\i); 
\path (d\i)--++(-135:0.5) coordinate (d\i); 
\draw[thick,densely dotted] (c\i)--++(135:10);
\draw[thick,densely dotted] (d\i)--++(45:10);
  }

\path(0,0)--++(135:2) coordinate (hhhh);

\fill[opacity=0.2](hhhh)
--++(135:8)--++(45:1)--++(-45:2)--++(45:2)--++(-45:4)--++(45:2)--++(-45:1)--++(45:2)--++(-45:1)--++(45:2);

\path(origin3)--++(45:-0.5)--++(135:7.5) coordinate (X)coordinate (start);

\path (start)--++(45:2)--++(-45:1) coordinate (X) ; 
\fill[orange](X) circle (4pt);
\draw[ thick, orange](X)--++(45:1) coordinate (X) ;
\fill[orange](X) circle (4pt);
\draw[ thick, orange](X)--++(45:1) coordinate (X) ;
\fill[orange](X) circle (4pt);
\draw[ thick, orange](X)--++(-45:1) coordinate (X) ;
\fill[orange](X) circle (4pt);
\draw[ thick, orange](X)--++(-45:1) coordinate (X) ;
\fill[orange](X) circle (4pt);
\draw[ thick, orange](X)--++(45:1) coordinate (X) ;
\fill[orange](X) circle (4pt);
\draw[ thick, orange](X)--++(45:1) coordinate (X) ;
\fill[orange](X) circle (4pt);
\draw[ thick, orange](X)--++(-45:1) coordinate (X) ;
\fill[orange](X) circle (4pt);
\draw[ thick, orange](X)--++(45:1) coordinate (X) ;
\fill[orange](X) circle (4pt);

\draw[ thick, orange](X)--++(-45:1) coordinate (X) ;
\fill[orange](X) circle (4pt);
\draw[ thick, orange](X)--++(-45:1) coordinate (X) ;
\fill[orange](X) circle (4pt);
 \draw[ thick, orange](X)--++(45:1) coordinate (X) ;
\fill[orange](X) circle (4pt);

\draw[ thick, orange](X)--++(-45:1) coordinate (X) ;
\fill[orange](X) circle (4pt);

\path (start)--++(45:4)--++(-45:4) coordinate (X) ; 
\fill[darkgreen](X) circle (4pt);
\draw[ thick, darkgreen](X)--++(45:1) coordinate (X) ;
\fill[darkgreen](X) circle (4pt);
\draw[ thick, darkgreen](X)--++(-45:1) coordinate (X) ;
 \fill[darkgreen](X) circle (4pt);
 \draw[ thick, darkgreen](X)--++(45:1) coordinate (X) ;
\fill[darkgreen](X) circle (4pt);
\draw[ thick, darkgreen](X)--++(-45:1) coordinate (X) ;
 \fill[darkgreen](X) circle (4pt);

\path (start)--++(45:4)--++(-45:5) coordinate (X) ; 
\fill[violet](X) circle (4pt);

\end{tikzpicture}\end{minipage}
{\longrightarrow} \;\;
\begin{minipage}{3.2cm}
\begin{tikzpicture}[scale=0.25]
 \path(0,0)--++(135:2) coordinate (hhhh);
 \draw[very thick] (hhhh)--++(135:8)--++(45:9)--++(-45:8)--++(-135:9);
 \clip(hhhh)--++(135:8)--++(45:9)--++(-45:8)--++(-135:9);
  
 \path(0,0) coordinate (origin2);
  \path(0,0)--++(135:2) coordinate (origin3);

\path (0,0)--++(135:8)--++(45:2) coordinate (X) ; 

\fill[opacity=0.2,orange] (X)--++(135:1)--++(45:2)--++(-45:2)--++(-135:1)--++(135:1);;

\path(X)--++(45:0.5)--++(135:0.5) node {\scalefont{0.6}$\boldsymbol\times$} coordinate(X);
\path(X) --++(45:1) node {\scalefont{0.6}$\boldsymbol\times$} coordinate(X);
\path(X) --++(-45:1) node {\scalefont{0.6}$\boldsymbol\times$} coordinate(X);

\path(0,0)--++(135:2) coordinate (hhhh);

\path (0,0)--++(135:5)--++(45:6) coordinate (X) ; 

\fill[opacity=0.2,darkgreen] (X)--++(135:1)--++(45:1)--++(-45:1);

\path(X)--++(45:0.5)--++(135:0.5) node {\scalefont{0.6}$\boldsymbol\times$} coordinate(X);

\path (0,0)--++(135:3)--++(45:7) coordinate (X) ; 

\fill[opacity=0.2,orange] (X)--++(135:1)--++(45:1)--++(-45:1);

\path(X)--++(45:0.5)--++(135:0.5) node {\scalefont{0.6}$\boldsymbol\times$} coordinate(X);

     \foreach \i in {0,1,2,3,4,5,6,7,8}
{
\path (origin3)--++(45:0.5*\i) coordinate (c\i); 
\path (origin3)--++(135:0.5*\i)  coordinate (d\i); 
  }

   \foreach \i in {0,1,2,3,4,5,6,7,8}
{
\path (origin3)--++(45:1*\i) coordinate (c\i); 
\path (c\i)--++(-45:0.5) coordinate (c\i); 
\path (origin3)--++(135:1*\i)  coordinate (d\i); 
\path (d\i)--++(-135:0.5) coordinate (d\i); 
\draw[thick,densely dotted] (c\i)--++(135:10);
\draw[thick,densely dotted] (d\i)--++(45:10);
  }

\path(0,0)--++(135:2) coordinate (hhhh);

\fill[opacity=0.2](hhhh)
--++(135:8)--++(45:1)--++(-45:2)--++(45:2)--++(-45:4)--++(45:2)--++(-45:1)--++(45:2)--++(-45:1)--++(45:2);

\path(origin3)--++(45:-0.5)--++(135:7.5) coordinate (X)coordinate (start);

%
%
%
%
%

\path (start)--++(45:2)--++(-45:1) coordinate (X) ; 
\fill[orange](X) circle (4pt);

\path (start)--++(45:4)--++(-45:3) coordinate (X) ; 
\fill[orange](X) circle (4pt);
\draw[ thick, orange](X)--++(45:1) coordinate (X) ;
\fill[orange](X) circle (4pt);
\draw[ thick, orange](X)--++(45:1) coordinate (X) ;
\fill[orange](X) circle (4pt);
\draw[ thick, orange](X)--++(-45:1) coordinate (X) ;
\fill[orange](X) circle (4pt);
\draw[ thick, orange](X)--++(-45:1) coordinate (X) ;
\fill[orange](X) circle (4pt);

\draw[ thick, orange](X)--++(45:1) coordinate (X) ;
\fill[orange](X) circle (4pt);
\draw[ thick, orange](X)--++(-45:1) coordinate (X) ;
\fill[orange](X) circle (4pt);


\path (start)--++(45:8)--++(-45:7) coordinate (X) ; 
\fill[orange](X) circle (4pt);

%
%
%

\path (start)--++(45:4)--++(-45:4) coordinate (X) ; 
\fill[darkgreen](X) circle (4pt);
\draw[ thick, darkgreen](X)--++(45:1) coordinate (X) ;
\fill[darkgreen](X) circle (4pt);
\draw[ thick, darkgreen](X)--++(-45:1) coordinate (X) ;
 \fill[darkgreen](X) circle (4pt);
\path (start)--++(45:6)--++(-45:6) coordinate (X) ;
 \fill[darkgreen](X) circle (4pt);

\path (start)--++(45:4)--++(-45:5) coordinate (X) ; 
\fill[violet](X) circle (4pt);

\end{tikzpicture}
\end{minipage}
 {\longrightarrow}\;\;
\begin{minipage}{3.2cm}
\begin{tikzpicture}[scale=0.25]
 \path(0,0)--++(135:2) coordinate (hhhh);
 \draw[very thick] (hhhh)--++(135:8)--++(45:9)--++(-45:8)--++(-135:9);
 \clip(hhhh)--++(135:8)--++(45:9)--++(-45:8)--++(-135:9);
  
 \path(0,0) coordinate (origin2);
  \path(0,0)--++(135:2) coordinate (origin3);

     \foreach \i in {0,1,2,3,4,5,6,7,8}
{
\path (origin3)--++(45:0.5*\i) coordinate (c\i); 
\path (origin3)--++(135:0.5*\i)  coordinate (d\i); 
  }

   \foreach \i in {0,1,2,3,4,5,6,7,8}
{
\path (origin3)--++(45:1*\i) coordinate (c\i); 
\path (c\i)--++(-45:0.5) coordinate (c\i); 
\path (origin3)--++(135:1*\i)  coordinate (d\i); 
\path (d\i)--++(-135:0.5) coordinate (d\i); 
\draw[thick,densely dotted] (c\i)--++(135:10);
\draw[thick,densely dotted] (d\i)--++(45:10);
  }

\path(0,0)--++(135:2) coordinate (hhhh);

\fill[opacity=0.2](hhhh)
--++(135:8)--++(45:1)--++(-45:2)--++(45:2)--++(-45:4)--++(45:2)--++(-45:1)--++(45:2)--++(-45:1)--++(45:2);

\path(origin3)--++(45:-0.5)--++(135:7.5) coordinate (X)coordinate (start);

\path(X)--++(45:2) coordinate (X) ;
\fill[magenta](X) circle (4pt);
\draw[ thick, magenta](X)--++(45:1) coordinate (X) ;
\fill[magenta](X) circle (4pt);
\draw[ thick, magenta](X)--++(-45:1) coordinate (X) ;
\fill[magenta](X) circle (4pt);
\draw[ thick, magenta](X)--++(45:1) coordinate (X) ;
 \fill[magenta](X) circle (4pt);
\draw[ thick, magenta](X)--++(-45:1) coordinate (X) ;

\fill[magenta](X) circle (4pt);
\draw[ thick, magenta](X)--++(45:1) coordinate (X) ;
\fill[magenta](X) circle (4pt);
\draw[ thick, magenta](X)--++(45:1) coordinate (X) ;
\fill[magenta](X) circle (4pt);
\draw[ thick, magenta](X)--++(45:1) coordinate (X) ;
\fill[magenta](X) circle (4pt);

\draw[ thick, magenta](X)--++(-45:1) coordinate (X) ;
\fill[magenta](X) circle (4pt);
\draw[ thick, magenta](X)--++(-45:1) coordinate (X) ;
\fill[magenta](X) circle (4pt);

\draw[ thick, magenta](X)--++(45:1) coordinate (X) ;
\fill[magenta](X) circle (4pt);
\draw[ thick, magenta](X)--++(-45:1) coordinate (X) ;
\fill[magenta](X) circle (4pt);
\draw[ thick, magenta](X)--++(-45:1) coordinate (X) ;
\fill[magenta](X) circle (4pt);
\draw[ thick, magenta](X)--++(45:1) coordinate (X) ;
\fill[magenta](X) circle (4pt);
\draw[ thick, magenta](X)--++(-45:1) coordinate (X) ;
\fill[magenta](X) circle (4pt);

\path (start)--++(45:2)--++(-45:1) coordinate (X) ; 
\fill[orange](X) circle (4pt);

\path (start)--++(45:4)--++(-45:3) coordinate (X) ; 
\fill[orange](X) circle (4pt);
\draw[ thick, orange](X)--++(45:1) coordinate (X) ;
\fill[orange](X) circle (4pt);
\draw[ thick, orange](X)--++(45:1) coordinate (X) ;
\fill[orange](X) circle (4pt);
\draw[ thick, orange](X)--++(-45:1) coordinate (X) ;
\fill[orange](X) circle (4pt);
\draw[ thick, orange](X)--++(-45:1) coordinate (X) ;
\fill[orange](X) circle (4pt);

\draw[ thick, orange](X)--++(45:1) coordinate (X) ;
\fill[orange](X) circle (4pt);
\draw[ thick, orange](X)--++(-45:1) coordinate (X) ;
\fill[orange](X) circle (4pt);


\path (start)--++(45:8)--++(-45:7) coordinate (X) ; 
\fill[orange](X) circle (4pt);

%
%
%

\path (start)--++(45:4)--++(-45:4) coordinate (X) ; 
\fill[darkgreen](X) circle (4pt);
\draw[ thick, darkgreen](X)--++(45:1) coordinate (X) ;
\fill[darkgreen](X) circle (4pt);
\draw[ thick, darkgreen](X)--++(-45:1) coordinate (X) ;
 \fill[darkgreen](X) circle (4pt);
\path (start)--++(45:6)--++(-45:6) coordinate (X) ;
 \fill[darkgreen](X) circle (4pt);

\path (start)--++(45:4)--++(-45:5) coordinate (X) ; 
\fill[violet](X) circle (4pt);

\end{tikzpicture}
\end{minipage}
$$
\caption{The sequence $\alpha \to {\sf reg}(\alpha) \to \alpha\sqcup (\mu\setminus\alpha)_{\leq 0} \to \mu\setminus\alpha$ where the steps are regularise, split, add as in \eqref{Step1},  \eqref{Step2}, and  \eqref{Step3}. In the third diagram the boxes marked with an $\times$ are the ones deleted in the splitting step.}
\label{seq_reg_exl2}
\end{figure}
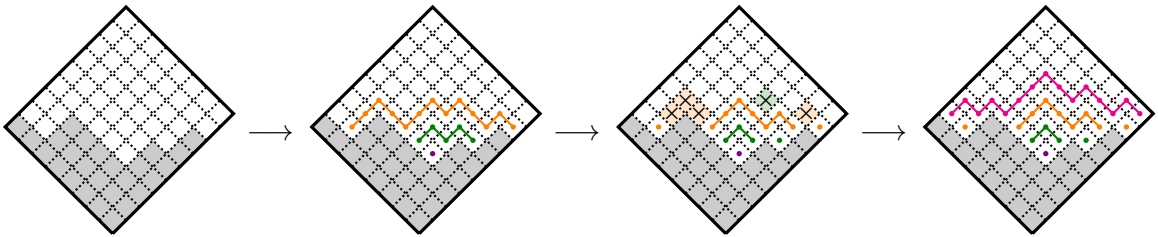

    \color{black}

    \section{The  (extended) Khovanov arc algebras, cellularity,  and the Schur functor}\label{arcalgebras}

    We are now ready to introduce the algebras of interest in this paper and recall the facts we will require about their representation theory.
The new result of this section is that {\em all} cell modules of Khovanov arc algebras have simple heads (this makes use of the combinatorics developed in \cref{regulatiaasaiofgsaf}).  
This allows us to give a new construction of the cellular basis  of the Khovanov arc algebra; this will provide the backbone of our proof of Theorem A from the introduction. 
    
       \subsection{The original definitions of the (extended) Khovanov arc algebras}

\label{Khovanov  arc algebras}

 We now recall the definition of  the extended Khovanov arc algebras studied in  \cite{MR2600694,MR2781018,MR2955190,MR2881300}. 
 We define $K ^{m}_{ n }$ to be the algebra spanned by diagrams 
$$
\{
\underline{\la}
\mu \overline{\nu}
\mid \la,\mu,\nu \in \mptn \text{ such that }
\mu\overline{\nu},  \underline{\la}
\mu \text{  are oriented}\}
$$
 with the multiplication defined as follows.
First set 
 $$(\underline{\la}
\mu \overline{\nu})(\underline{\alpha}
\beta \overline{\gamma}) = 0 \quad \mbox{unless $\nu = \alpha$}.$$ 
To compute $(\underline{\la}
\mu \overline{\nu})(\underline{\nu}
\beta \overline{\gamma})$ place $(\underline{\la}
\mu \overline{\nu})$ under $(\underline{\nu}
\beta \overline{\gamma})$ then follow the \lq surgery' procedure.
This surgery combines two circles into one or splits one circle into two  using the following rules for re-orientation (where we use the notation
$1=\text{anti-clockwise circle}$, $x=\text{clockwise circle}$, $y=\text{oriented strand}$).  We have the splitting rules 
$$1 \mapsto 1 \otimes x + x \otimes 1,
\quad
 x \mapsto x \otimes x,
 \quad
 y \mapsto x \otimes y. 
 $$ 
 and the merging rules 
\begin{align*}
1 \otimes 1 \mapsto 1, 
\quad
1 \otimes x \mapsto x,
\quad
 x \otimes 1 \mapsto x,
\quad
x \otimes x \mapsto 0,
\quad
1 \otimes y  \mapsto y,\quad
x \otimes y \mapsto 0,
\end{align*}
\begin{align*}
y \otimes y  \mapsto
\left\{
\begin{array}{ll}
y\otimes y&\text{if both strands are propagating,  
one   is
 }\\
&\text{$\up$-oriented and  the other   is $\down$-oriented;}\\
0&\text{otherwise.} \end{array}
\right.
\end{align*}


\begin{eg}\label{brace-surgery}
We have the following product of Khovanov diagrams

$$   \begin{minipage}{3.1cm}
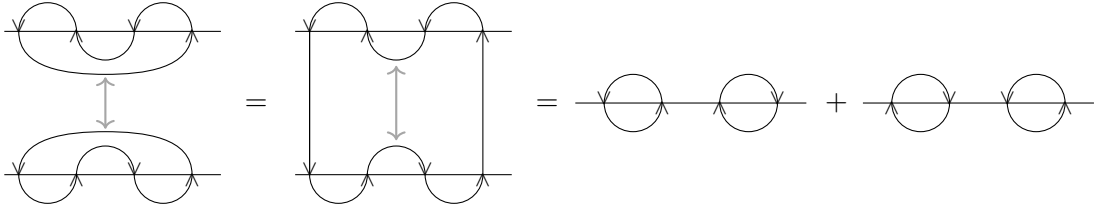
 \begin{tikzpicture}  [scale=0.76]

 \draw(-0.25,0)--++(0:3.75) coordinate (X);

   \draw[thick](0 , 0.09) node {$\scriptstyle\down$};        
   \draw(2,0.09) node {$\scriptstyle\down$};
    \draw[thick](1 ,-0.09) node {$\scriptstyle\up$};
   \draw[thick](3 ,-0.09) node {$\scriptstyle\up$};

    \draw
    (0,0)
      to [out=90,in=180] (0.5,0.5)  to [out=0,in=90] (1,0)
        (1,0)
 to [out=-90,in=180] (1.5, -0.5) to [out=0,in= -90] (2,0)
  to [out=90,in=180] (2.5,0.5)  to [out=0,in=90] (3,0)  
   to [out=-90,in=0] (1.5, -0.75) to [out=180,in= -90] (0,0)
  ;

\draw[gray!70,<->,thick](1.5,-0.8)-- (1.5,-2.5+0.8);

   \draw(-0.25,-2.5)--++(0:3.75) coordinate (X);

   \draw[thick](0 , 0.09-2.5) node {$\scriptstyle\down$};        
   \draw(2,0.09-2.5) node {$\scriptstyle\down$};

   \draw[thick](1 ,-0.09-2.5) node {$\scriptstyle\up$};
   \draw[thick](3 ,-0.09-2.5) node {$\scriptstyle\up$};

    \draw
    (0,-2.5)
      to [out=-90,in=180] (0.5,-2.5-0.5)  to [out=0,in=-90] (1,-2.5-0)
        (1,-2.5-0)
 to [out=90,in=180] (1.5, -2.5+0.5) to [out=0,in= 90] (2,-2.5-0)
  to [out=-90,in=180] (2.5,-2.5-0.5)  to [out=0,in=-90] (3,-2.5-0)  
   to [out=90,in=0] (1.5, -2.5+0.75) to [out=180,in= 90] (0,-2.5-0)
  ;

 \end{tikzpicture} \end{minipage}
 = \; \;
 \begin{minipage}{3.1cm}
 \begin{tikzpicture}  [scale=0.76]

 \draw(-0.25,0)--++(0:3.75) coordinate (X);

   \draw[thick](0 , 0.09) node {$\scriptstyle\down$};        
   \draw(2,0.09) node {$\scriptstyle\down$};
    \draw[thick](1 ,-0.09) node {$\scriptstyle\up$};
   \draw[thick](3 ,-0.09) node {$\scriptstyle\up$};

    \draw
    (0,0)  --++(-90:2.5)  --++( 90:2.5)
      to [out=90,in=180] (0.5,0.5)  to [out=0,in=90] (1,0)
        (1,0)
 to [out=-90,in=180] (1.5, -0.5) to [out=0,in= -90] (2,0)
  to [out=90,in=180] (2.5,0.5)  to [out=0,in=90] (3,0)  
 
   ;

\draw[gray!70,<->,thick](1.5,-0.6)-- (1.5,-2.5+0.6);

   \draw(-0.25,-2.5)--++(0:3.75) coordinate (X);

   \draw[thick](0 , 0.09-2.5) node {$\scriptstyle\down$};        
   \draw(2,0.09-2.5) node {$\scriptstyle\down$};

   \draw[thick](1 ,-0.09-2.5) node {$\scriptstyle\up$};
   \draw[thick](3 ,-0.09-2.5) node {$\scriptstyle\up$};

    \draw
    (0,-2.5)
      to [out=-90,in=180] (0.5,-2.5-0.5)  to [out=0,in=-90] (1,-2.5-0)
        (1,-2.5-0)
 to [out=90,in=180] (1.5, -2.5+0.5) to [out=0,in= 90] (2,-2.5-0)
  to [out=-90,in=180] (2.5,-2.5-0.5)  to [out=0,in=-90] (3,-2.5-0)--++(90:2.5)  
   ;

 \end{tikzpicture} \end{minipage}
 =\;
 \begin{minipage}{3.1cm}
 \begin{tikzpicture}  [scale=0.76]

 \draw(-0.5,0)--++(0:4) coordinate (X);

   \draw[thick](0 , 0.09) node {$\scriptstyle\down$};        
    \draw[thick](1 ,-0.09) node {$\scriptstyle\up$};

      \draw[thick](2 ,-0.09) node {$\scriptstyle\up$};  
   \draw(3,0.09) node {$\scriptstyle\down$};

    \draw
    (0,0)
      to [out=90,in=180] (0.5,0.5)  to [out=0,in=90] (1,0)
     to [out=-90,in=0] (0.5,-0.5)  to [out=180,in=-90] (0,0)
         
    ;

   \draw
    (2+0,0)
      to [out=90,in=180] (2+0.5,0.5)  to [out=0,in=90] (2+1,0)
     to [out=-90,in=0] (2+0.5,-0.5)  to [out=180,in=-90] (2+0,0)
         
    ;

 \end{tikzpicture} \end{minipage} 
 \;+\;
  \begin{minipage}{3.4cm}
 \begin{tikzpicture}  [scale=0.76]

 \draw(-0.5,0)--++(0:4) coordinate (X);

   \draw[thick](1 , 0.09) node {$\scriptstyle\down$};        
    \draw[thick](0 ,-0.09) node {$\scriptstyle\up$};

      \draw[thick](3 ,-0.09) node {$\scriptstyle\up$};  
   \draw(2,0.09) node {$\scriptstyle\down$};

    \draw
    (0,0)
      to [out=90,in=180] (0.5,0.5)  to [out=0,in=90] (1,0)
     to [out=-90,in=0] (0.5,-0.5)  to [out=180,in=-90] (0,0)
         
    ;

   \draw
    (2+0,0)
      to [out=90,in=180] (2+0.5,0.5)  to [out=0,in=90] (2+1,0)
     to [out=-90,in=0] (2+0.5,-0.5)  to [out=180,in=-90] (2+0,0)
         
    ;

 \end{tikzpicture} \end{minipage}  $$ 
 where we highlight with arrows the pair of arcs on which we are about to perform surgery.  The first equality follows from   the merging rule for $1\otimes 1 \mapsto 1$ and the second equality follows from  the merging rule 
$ 1 \mapsto 1 \otimes x + x \otimes 1$. 
 \end{eg}

\begin{eg}\label{brace-surgery2}
We have the following product of Khovanov diagrams

$$   \begin{minipage}{3.1cm}
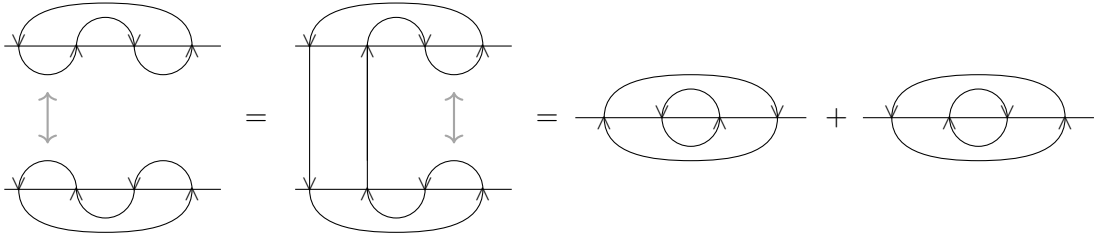
 \begin{tikzpicture}  [scale=0.76]

 \draw(-0.25,0)--++(0:3.75) coordinate (X);

   \draw[thick](0 , 0.09) node {$\scriptstyle\down$};        
   \draw(2,0.09) node {$\scriptstyle\down$};
    \draw[thick](1 ,-0.09) node {$\scriptstyle\up$};
   \draw[thick](3 ,-0.09) node {$\scriptstyle\up$};

    \draw
    (0,0)
      to [out=-90,in=180] (0.5,-0.5)  to [out=0,in=-90] (1,0)
        (1,0)
 to [out=90,in=180] (1.5, 0.5) to [out=0,in= 90] (2,0)
  to [out=-90,in=180] (2.5,-0.5)  to [out=0,in=-90] (3,0)  
   to [out=90,in=0] (1.5, 0.75) to [out=180,in= 90] (0,0)
  ;

\draw[gray!70,<->,thick](0.5,-0.8)-- (0.5,-2.5+0.8);

   \draw(-0.25,-2.5)--++(0:3.75) coordinate (X);

   \draw[thick](0 , 0.09-2.5) node {$\scriptstyle\down$};        
   \draw(2,0.09-2.5) node {$\scriptstyle\down$};

   \draw[thick](1 ,-0.09-2.5) node {$\scriptstyle\up$};
   \draw[thick](3 ,-0.09-2.5) node {$\scriptstyle\up$};

    \draw
    (0,-2.5)
      to [out=90,in=180] (0.5,-2.5+0.5)  to [out=0,in=90] (1,-2.5-0)
        (1,-2.5-0)
 to [out=-90,in=180] (1.5, -2.5-0.5) to [out=0,in= -90] (2,-2.5-0)
  to [out=90,in=180] (2.5,-2.5+0.5)  to [out=0,in=90] (3,-2.5-0)  
   to [out=-90,in=0] (1.5, -2.5-0.75) to [out=180,in= -90] (0,-2.5-0)
  ;

 \end{tikzpicture} \end{minipage}
 = \; \;
%
%
%
%
%
%
%
%
%
%
%
%
%
%
%
%
%
%
%
%
%
%
%
%
%
%
%
%
%
%
%
%
%
%
%
%
 \begin{minipage}{3.1cm}
 \begin{tikzpicture}  [scale=0.76]

 \draw(-0.25,0)--++(0:3.75) coordinate (X);

   \draw[thick](0 , 0.09) node {$\scriptstyle\down$};        
   \draw(2,0.09) node {$\scriptstyle\down$};
    \draw[thick](1 ,-0.09) node {$\scriptstyle\up$};
   \draw[thick](3 ,-0.09) node {$\scriptstyle\up$};

    \draw
 (1,-2)--(1,0)
        (1,0)
 to [out=90,in=180] (1.5, 0.5) to [out=0,in= 90] (2,0)
  to [out=-90,in=180] (2.5,-0.5)  to [out=0,in=-90] (3,0)  
   to [out=90,in=0] (1.5, 0.75) to [out=180,in= 90] (0,0)
  ;

\draw[gray!70,<->,thick](2.5,-0.8)-- (2.5,-2.5+0.8);

   \draw(-0.25,-2.5)--++(0:3.75) coordinate (X);

   \draw[thick](0 , 0.09-2.5) node {$\scriptstyle\down$};        
   \draw(2,0.09-2.5) node {$\scriptstyle\down$};

   \draw[thick](1 ,-0.09-2.5) node {$\scriptstyle\up$};
   \draw[thick](3 ,-0.09-2.5) node {$\scriptstyle\up$};

    \draw
    (0,0)--(0,-2.5);
   
    \draw
    (1,0)--
        (1,-2.5-0)
 to [out=-90,in=180] (1.5, -2.5-0.5) to [out=0,in= -90] (2,-2.5-0)
  to [out=90,in=180] (2.5,-2.5+0.5)  to [out=0,in=90] (3,-2.5-0)  
   to [out=-90,in=0] (1.5, -2.5-0.75) to [out=180,in= -90] (0,-2.5-0)
  ;

 \end{tikzpicture} \end{minipage}
  =\;
 \begin{minipage}{3.1cm}
 \begin{tikzpicture}  [scale=0.76]

 \draw(-0.5,0)--++(0:4) coordinate (X);

   \draw[thick](1 , 0.09) node {$\scriptstyle\down$};        
    \draw[thick](0 ,-0.09) node {$\scriptstyle\up$};

      \draw[thick](2 ,-0.09) node {$\scriptstyle\up$};  
   \draw(3,0.09) node {$\scriptstyle\down$};

    \draw
    (0+1,0)
      to [out=90,in=180] (0.5+1,0.5)  to [out=0,in=90] (1+1,0)
     to [out=-90,in=0] (0.5+1,-0.5)  to [out=180,in=-90] (0+1,0)
         
    ;

   \draw
    (0 ,0)
      to [out=90,in=180] (0.5+1,0.75)  to [out=0,in=90] (3,0)
     to [out=-90,in=0] (0.5+1,-0.75)  to [out=180,in=-90] (0,0)
         
    ;

%

 \end{tikzpicture} \end{minipage} 
 \;+\;
  \begin{minipage}{3.4cm}
 \begin{tikzpicture}  [scale=0.76]

 \draw(-0.5,0)--++(0:4) coordinate (X);

   \draw[thick](0 , 0.09) node {$\scriptstyle\down$};        
    \draw[thick](1 ,-0.09) node {$\scriptstyle\up$};

      \draw[thick](3 ,-0.09) node {$\scriptstyle\up$};  
   \draw(2,0.09) node {$\scriptstyle\down$};

  \draw
    (0+1,0)
      to [out=90,in=180] (0.5+1,0.5)  to [out=0,in=90] (1+1,0)
     to [out=-90,in=0] (0.5+1,-0.5)  to [out=180,in=-90] (0+1,0)
         
    ;

   \draw
    (0 ,0)
      to [out=90,in=180] (0.5+1,0.75)  to [out=0,in=90] (3,0)
     to [out=-90,in=0] (0.5+1,-0.75)  to [out=180,in=-90] (0,0)
         
    ;
   
%
%
%
%
%
%
%

 \end{tikzpicture} \end{minipage}  $$ 
 where we highlight with arrows the pair of arcs on which we are about to perform surgery. 
This is similar to \cref{brace-surgery}. 
 \end{eg}

    \subsection{Quiver presentations and cellularity of the extended Khovanov arc algebra}
 In   \cite{compan4} we  proved that the {\em extended} Khovanov arc algebras   are isomorphic to the basic algebras of the category algebras of Hecke categories of type $( {S}_{m+n},  {S}_m\times {S}_n)$, see also \cite{Bowstrophazdev}.  This isomorphism was constructed in a similar spirit to that of \cite{cell4us2,MR4381198}.  
 These isomorphisms allowed us to prove an integral presentation of the extended Khovanov arc algebras  
 in   \cite[Theorem B]{compan4}, which we now recall.

\begin{defn}\label{presentation}
The algebra $K^{m}_{ n }$ is the  associative $\Bbbk$-algebra generated by the elements 
\begin{equation}\label{geners}
\{D^\la_\mu,
D_\la^\mu \mid 	
\text{$\la, \mu\in \mptn$ with $\la = \mu - P$ for some $P\in {\rm DRem}(\mu)$} 
	\}\cup\{ {\sf 1}_\mu \mid \mu \in \mptn \}		
	\end{equation}
	subject to the  following relations and their duals. 

	\smallskip\noindent
{\bf The idempotent   
relations:} 
For all $\la,\mu \in \mptn$, we have that 
\begin{equation}\label{rel1}
{\sf 1}_\mu{\sf 1}_\la =\delta_{\la,\mu}{\sf 1}_\la \qquad 
\qquad {\sf 1}_\la D^\la_\mu {\sf 1}_\mu = D^\la_\mu.
\end{equation} 

\smallskip\noindent
	{\bf The 
	self-dual relation: } 
	Let  $P\in {\rm DRem}(\mu)$ and $\la = \mu - P$. Then we   have 
\begin{equation}\label{selfdualrel}
D_\mu^{\la} D_{\la}^\mu
= (-1)^{b(P)-1}\Bigg(
2
\!\! \sum_{   \begin{subarray}{c} Q\in {\rm DRem}(\la) \\ P\prec Q \end{subarray}}
\!\!
(-1) ^{b(Q) } D^{\la}_{\la- Q } D^{\la - Q  }_{\la} + 
\!\!\sum_{  \begin{subarray}{c} Q\in {\rm DRem}(\la) \\ Q \, \text{adj.}\, P \end{subarray}}
\!\!
 (-1)^{b(Q)} D^{\la}_{\la- Q } D^{\la - Q  }_{\la}\Bigg) 
\end{equation}  
where we abbreviate ``adjacent to" simply as ``adj."  

\smallskip\noindent
{\bf The commuting relations:} 
Let $P,Q\in {\rm DRem}(\mu)$ which commute. Then we have 
\begin{equation}\label{commuting}
D^{\mu -P-Q}_{\mu-P}D^{\mu-P}_\mu = D^{\mu-P-Q}_{\mu -Q}D^{\mu -Q}_\mu \qquad
D^{\mu -P}_\mu D^\mu_{\mu - Q} = D^{\mu - P}_{\mu - P - Q}D^{\mu - P - Q}_{\mu - Q}.
\end{equation}

\smallskip\noindent
{\bf The non-commuting relation:}   
Let $P,Q\in {\rm DRem}(\mu)$ with $P\prec Q$ which do not commute. Then  $Q\setminus P = Q^1\sqcup Q^2$ where $Q^1, Q^2 \in {\rm DRem}(\mu - P)$ and we have 
\begin{equation}\label{noncommuting}
D^{\mu -Q}_\mu D^\mu_{\mu-P} = D^{\mu - Q}_{\mu - P - Q^1}D^{\mu - P - Q^1}_{\mu - P} = D^{\mu - Q}_{\mu - P - Q^2}D^{\mu - P - Q^2}_{\mu - P}
\end{equation}

\smallskip\noindent
{\bf The adjacent relation:}  
Let $P\in {\rm DRem}(\mu)$ and $Q\in {\rm DRem}(\mu - P)$ be adjacent. Recall that 
 $\langle P\cup Q \rangle_\mu$, if it exists,  denotes the smallest removable Dyck path of $\mu$ containing $P\cup Q$. Then we have 
\begin{equation}\label{adjacent-OG}
D^{\mu - P - Q}_{\mu - P}D^{\mu - P}_\mu = \left\{ \begin{array}{ll} (-1)^{b(\langle P\cup Q\rangle_\mu) - b(Q)} D^{\mu - P - Q}_{\mu - \langle P\cup Q\rangle_\mu}D^{\mu - \langle P\cup Q\rangle_\mu}_\mu & \mbox{if $\langle P\cup Q\rangle_\mu$ exists} \\ 0 & \mbox{otherwise} \end{array} \right.
\end{equation}
  \end{defn}

 \begin{defn}
 Given a Dyck tiling 
 $\mu\setminus \la = \bigsqcup_{i=1}^k P^i $ 
we define an associated diagram
$$D^\la_\mu = D^\la _{\la +P^1}D^{\la +P^1}_{\la +P^1+P^2}\dots D^{\la +P^1+\dots+P^{k-1}}_{\la +P^1+\dots+P^{k-1}+P^k} $$
which is independent of the ordering in which the Dyck paths are added.
\end{defn}

 The following construction is due to Libedinsky--Williamson \cite{antiLW}, but is couched in our language in \cite[Theorem 5.4]{compan4}.
 
 \begin{thm} [{\cite[Theorem 5.4]{compan4} and \cite{antiLW}}]\label{heere ris the basus}
  The  algebra   
$ K^m_n  $  is a graded cellular (in fact quasi-hereditary) algebra with  graded cellular basis given by
\begin{equation}\label{basis}
\{D_\la^\mu D^\la_\nu \mid  \la,\mu,\nu\in \mptn \,\, \mbox{with}\,\,		(\la,\mu), (\la,\nu) \text{ Dyck pairs}  \}
\end{equation}
with 
$${\rm deg}( D_\la^\mu D^\la_\nu) = {\rm deg}(\la, \mu) + {\rm deg}(\la, \nu),$$
with respect to the involution $*$ and the partial order on $\mptn$ given by inclusion.
\end{thm}

\begin{defn}
For $\la \in  \mptn$, we define the {\sf standard module} $\Delta_{m,n}(\la)$ to be the right $K^m_n$-module
$$
\Delta_{m,n}(\la)= {\sf 1}_\la K^m_n/ (K^m_n(\textstyle \sum _{\mu \subset \la}{\sf 1}_\mu) K^m_n).
$$
This module has basis $\{D^\la_\nu \mid (\la,\nu) \text{ is a Dyck pair}\}$ with action given by right concatenation (modulo  the ideal spanned by diagrams which factor through  $\mu\subset\la$). 
  \end{defn}

\begin{prop}
A complete and irredundant set of non-isomorphic simple $K^m_n$-modules
is given by 
$$
\{L_{m,n}(\la)\mid 
L_{m,n}(\la)= \Delta_{m,n}(\la)/{\rm rad}(\Delta_{m,n}(\la))
 , \text { for }\la \in \mathscr{P}_{m,n} \}.
$$
\end{prop}

  We now describe the full submodule structure of the  standard modules.
  As $K^m_n$ is positively graded, the grading provides a submodule filtration of $\Delta_{m,n}(\la)$. Decompose ${\rm DP}(\la)$ as
 $${\rm DP}(\la) = \bigsqcup_{k \geq 0} {\rm DP}_k(\la) \quad \text{where} \quad  {\rm DP}_k(\la) = \{\mu\in {\rm DP}(\la) \, : \, \deg (\la, \mu)=k\}.$$

 \begin{thm}\label{submodule}
Let $\la \in \mptn$. The Alperin diagram of the standard module $\Delta_{m,n}(\la)$ has vertex set labelled by the set 
 $\{ L_{m,n}(\mu)\, : \,  \mu\in {\rm DP}(\la)\}$  and edges
$$L_{m,n}(\mu) \longrightarrow L_{m,n}(\nu)$$
whenever $\mu \in {\rm DP}_k(\la)$, $\nu \in {\rm DP}_{k+1}(\la)$ for some $k\geq 0$ and $\nu = \mu \pm P$ for some $P\in {\rm DAdd}(\mu)$ or $P\in {\rm DRem}(\mu)$ respectively. 
 \end{thm}
 
 An example is depicted in \cref{submodules}, below.

\begin{figure}[ht!]
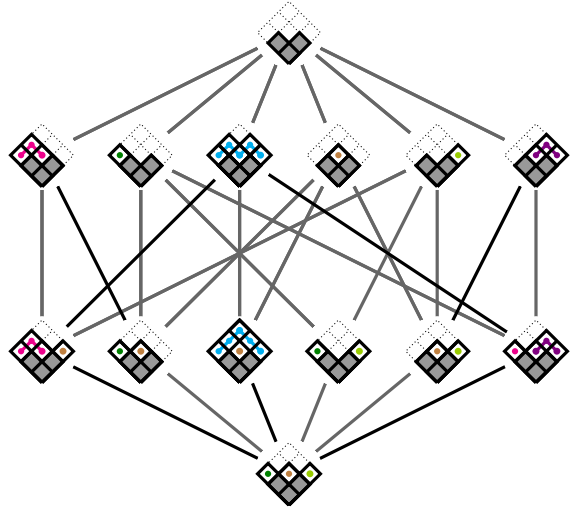


$$

 $$

 \caption{The full submodule lattice of the Verma module $\Delta_{3,3}(2,1)$ for  
 $K_{3,3} $ and $\Bbbk$ any field. 
 We represent each simple module by the corresponding partition (in Russian notation) and highlight the $3\times 3$ rectangle in which the partition exists. 
 This module has simple head $L_{3,3}(2,1)$ and simple socle $L_{3,3}(3,2,1)$. 
  Each edge connects a pair of partitions which differ by adding or removing a single Dyck path. }
 \label{submodules}
 
 \end{figure}

  \subsection{The Schur functor } 
   We set 
    $
    \mathscr{R}_{m,n}  =
\{\la\in  \mathscr{P}_{m,n}\mid \la \text{ has  defect zero}
\}. $
We define the {\sf Schur idempotent} to be the element
$$
e = \sum _{\la \in \mathscr{R}_{m,n}}{\sf 1}_\la \in K^m_n.
$$
The main object of study in this paper will be the subalgebras 
 $$e  K^m_ne\subseteq K^m_n$$
  for $m,n\in \NN$.
We have the following immediate results:

 \begin{cor} \label{heere ris the basus2}
  The  algebra   
$e K^m_n e $  is a graded cellular  algebra with  graded cellular basis given by
\begin{equation}\label{basis}
\{D_\la^\mu D^\la_\nu \mid  \la \in \mptn ,\mu,\nu\in \mathscr{R}_{m,n} \,\, \mbox{with}\,\,		(\la,\mu), (\la,\nu) \text{ Dyck pairs}  \}
\end{equation}
with 
 ${\rm deg}( D_\la^\mu D^\la_\nu) = {\rm deg}(\la, \mu) + {\rm deg}(\la, \nu),$ 
with respect to the involution $*$ and the partial order on $\mptn$ given by inclusion.
\end{cor}
  
  \begin{proof}
 This follows from \cref{heere ris the basus} and   the definition of these algebras by idempotent truncation.
  \end{proof}

\begin{defn}
For $\la \in  \mptn$ we define the {\sf Specht module} $S_{m,n}(\la)$ to be the right $H^m_n$-module
$$
S_{m,n}(\la)= \Delta_{m,n}(\la) e 
$$
this module has basis $\{D^\la_\nu \mid (\la,\nu) \text{ is a Dyck pair}, \nu \in \mathscr{R}_{m,n}\}$ with action given by right concatenation (modulo  the ideal spanned by diagrams which factor through  $\mu\subset\la$).
\end{defn}

\begin{prop}
A complete and irredundant set of non-isomorphic simple $H^m_n$-modules
is given by 
$$
\{D_{m,n}(\la)\mid D_{m,n}(\la)= L_{m,n}(\la) e , \text { for }\la \in \mathscr{R}_{m,n}\subset \mptn\}.
$$
\end{prop}

\begin{proof}
By  \cref{heere ris the basus} (which implies that  $ \mathscr{P}_{m,n}$ labels the simple $K^m_n$-modules) 
we need only show that 
$ 
L_{m,n}(\la) e =0
$ if and only if $\la \not \in \mathscr{R}_{m,n}$. 
This claim follows immediately from the fact that the algebra  $K^m_n$ is basic (and so all simple modules are 1-dimensional and generated by weight idempotents) and the definition of $e\in  K^m_n$ as the sum over all weight idempotents from $ \mathscr{R}_{m,n}$.
\end{proof}

 \begin{prop}\label{submodule2}
Let $\la \in \mptn$. The Alperin diagram of the Specht module $S_{m,n}(\la)$ has vertex set labelled by the set 
 $\{ D_{m,n}(\mu)\, : \,  \mu\in {\rm DP}(\la)\cap  \mathscr{R}_{m,n}\}$  and edges
$$D_{m,n}(\mu) \longrightarrow D_{m,n}(\nu)$$
whenever $\mu \in {\rm DP}_k(\la)$, $\nu \in {\rm DP}_{k+1}(\la)$ for some $k\geq 0$ and $\nu = \mu \pm P$ for some $P\in {\rm DAdd}(\mu)$ or $P\in {\rm DRem}(\mu)$ respectively. 
 \end{prop}

   \begin{proof}
 This follows from \cref{submodule} and   the definition of these algebras by idempotent truncation.
  \end{proof}
An example is depicted in \cref{submodules3}, below.

\begin{figure}[ht!]
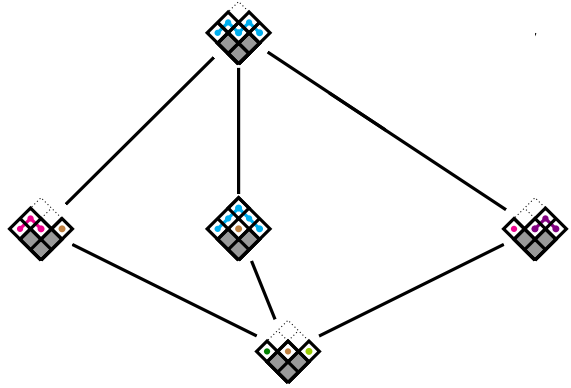

 $$

 $$

 \caption{The full submodule lattice of the Specht module $S_{3,3}(2,1)$ for  
 $eK_{3,3} e$ and $\Bbbk$ any field. 
This is obtained from the diagram in \cref{submodules} by simply deleting all composition factors labelled by 
$\la \not\in \mathscr{R}_{m,n}$ (and edges from these vertices). }
 \label{submodules3}
 
 \end{figure}

Thus we have complete submodule structure of the Specht modules, for free.  In particular, we have the following:
  
\begin{prop}
The Specht module  $S_{m,n}(\alpha)$ has   simple head isomorphic to $D({\sf reg}(\alpha))
\langle d(\alpha) \rangle $.

\end{prop} 
\begin{proof}
This follows from \cref{submodule2} and \cref{PURE1}.
\end{proof}

\begin{prop}
 The  algebra   
$e K^m_n e $  is a graded cellular  algebra with  graded cellular basis given by
\begin{equation}\label{basis22}
\{D_{{\sf reg}(\alpha)}^\mu  
D^{{\sf reg}(\alpha)}_\alpha D_{{\sf reg}(\alpha)}^\alpha 
D^{{\sf reg}(\alpha)}_\nu \mid  \alpha \in \mptn ,\mu,\nu\in \mathscr{R}_{m,n} \,\, \mbox{with}\,\,		(\alpha,\mu), (\alpha,\nu) \text{ Dyck pairs}  \}
\end{equation}
with 
 ${\rm deg}( D_\alpha^\mu D^\alpha_\nu) = {\rm deg}(\alpha, \mu) + {\rm deg}(\alpha, \nu),$ 
with respect to the involution $*$ and the partial order on $\mptn$ given by inclusion.

\end{prop} 
\begin{proof}
This follows from \cref{submodule2} and \cref{PURE1}.
\end{proof}

%
%

 \section{A symmetric algebra defined via  Dyck combinatorics}
 
We now define an abstract algebra, $\mathcal{A}_{m,n}$ via Dyck-combinatorial 
generators and relations.  The main result of this paper will be that this algebra is isomorphic
as a $\ZZ$-graded $\Bbbk$-algebra 
 to $e K^m_ne$.  We will prove this result in two steps over the next two chapters: we will first provide a spanning set of $\mathcal{A}_{m,n}$  in terms of pairs of Dyck paths; we will then construct a surjective $\Bbbk$-algebra homomorphism $\mathcal{A}_{m,n}\to H^m_n$.  Putting these two facts together, we will deduce that the spanning set is in fact a basis and the homomorphism is injective, as required.
 However, before we can get on with the task of defining $\mathcal{A}_{m,n}$, we first require some additional combinatorics.

\begin{defn}

We define the {\sf regular  quiver} $Q_{m,n}$ with    vertex set 
$\{\mathbbm 1_\la \mid \la \in   \mathscr{R}_{m,n} \}$  and 
  arrows   
  \begin{itemize}
\item  $\mathbb D^\la_\mu: \la   \to  \mu$ 
and 
   $\mathbb D^\mu_\la:\mu \to \la$ for every $\la=\mu - P$ with $P\in {\rm DRem}_{>0}(\mu)$; 

\item for $m=n$ we have additional    ``loops" of degree 2, 
$\mathbb L^\la_\la :\la\to \la$ for every $\la\in \mathscr{R}_{m,m} $.
\end{itemize}
\end{defn}

An example is depicted in \cref{loopydyck} below.

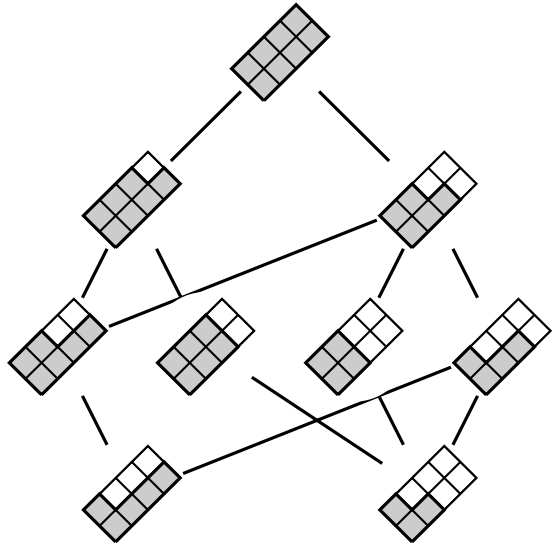
\begin{figure}[ht!]

$$\begin{tikzpicture}[scale=0.78]

 \draw[very thick] 
 (0,0)--(-2.5,-2.5)
 (0,0)--(2.5,-2.5)
 (-3.75,-5)--(-2.5,-2.5)
  (-1.25,-5)--(-2.5,-2.5)
    (1.25,-5)--(2.5,-2.5)
        (-3.75,-5)--(2.5,-2.5)
            (2.5+1.25,-5)--(2.5,-2.5)
         (-2.5,-7.5)--(-3.75,-5)          (-2.5,-7.5)--(3.75,-5)

                  (2.5,-7.5)--(-1.25,-5) 
                  (2.5,-7.5)--(1.25,-5)                            (2.5,-7.5)--(3.75,-5) 
         ; 
 \draw(0,0) node 
 {
$ \begin{tikzpicture}[scale=0.3]
 
 \path  (0,0) coordinate (Y);
 \path  (Y)--++(135:1)--++(45:2)coordinate (X);
\fill [white] (X) circle(69pt);
\begin{scope}
 \draw[very thick,fill=gray!40] (Y)--++(135:2)--++(45:4)--++(-45:2) --++(-135:4)--(0,0); 
  \clip (Y)--++(135:2)--++(45:4)--++(-45:2) --++(-135:4)--(0,0); 
--++(135:2)--++(45:4)--++(-45:2) --++(-135:4)--(0,0); 
 
 		\foreach \i in {0,1,2,3,4,5,6,7,8,9,10,11,12}
		{
			\path (Y)--++(45:1*\i) coordinate (c\i); 
			\path (Y)--++(135:1*\i)  coordinate (d\i); 
			\draw[thick, ] (c\i)--++(135:14);
			\draw[thick, ] (d\i)--++(45:14);
		}
	 	\end{scope}

 \end{tikzpicture}$
 };

 \draw(-2.5,-2.5) node 
 {
$ \begin{tikzpicture}[scale=0.3]
 
 \path  (0,0) coordinate (Y);
 \path  (Y)--++(135:1)--++(45:2)coordinate (X);
\fill [white] (X) circle(69pt);
\begin{scope}
 \draw[ thick,fill=white] (Y)--++(135:2)--++(45:4)--++(-45:2) --++(-135:4)--(0,0); 

 \draw[very thick,fill=gray!40] (Y)--++(135:2)--++(45:3)--++(-45:1) --++(45:1) --++(-45:1) --++(-135:4)--(0,0); 
  \clip (Y)--++(135:2)--++(45:4)--++(-45:2) --++(-135:4)--(0,0); 
--++(135:2)--++(45:3)--++(-45:1) --++(45:1) --++(-45:1) --++(-135:4)--(0,0); 
 
 		\foreach \i in {0,1,2,3,4,5,6,7,8,9,10,11,12}
		{
			\path (Y)--++(45:1*\i) coordinate (c\i); 
			\path (Y)--++(135:1*\i)  coordinate (d\i); 
			\draw[thick, ] (c\i)--++(135:14);
			\draw[thick, ] (d\i)--++(45:14);
		}
	 	\end{scope}

 \end{tikzpicture}$
 };

 \draw(2.5,-2.5) node 
 {
$ \begin{tikzpicture}[scale=0.3]
 
 \path  (0,0) coordinate (Y);
 \path  (Y)--++(135:1)--++(45:2)coordinate (X);
\fill [white] (X) circle(69pt);
\begin{scope}
 \draw[ thick,fill=white] (Y)--++(135:2)--++(45:4)--++(-45:2) --++(-135:4)--(0,0); 

 \draw[very thick,fill=gray!40] (Y)--++(135:2)--++(45:2)--++(-45:1) --++(45:1) --++(-45:1) --++(-135:3)--(0,0); 
  \clip (Y)--++(135:2)--++(45:4)--++(-45:2) --++(-135:4)--(0,0); 
--++(135:2)--++(45:4)--++(-45:2) --++(-135:4)--(0,0); 
--++(135:2)--++(45:4)--++(-45:2) --++(-135:4)--(0,0); 
 
 		\foreach \i in {0,1,2,3,4,5,6,7,8,9,10,11,12}
		{
			\path (Y)--++(45:1*\i) coordinate (c\i); 
			\path (Y)--++(135:1*\i)  coordinate (d\i); 
			\draw[thick, ] (c\i)--++(135:14);
			\draw[thick, ] (d\i)--++(45:14);
		}
	 	\end{scope}

 \end{tikzpicture}$
 };

 \draw(-2.5-1.25,-5) node 
 {
$ \begin{tikzpicture}[scale=0.3]
 
 \path  (0,0) coordinate (Y);
 \path  (Y)--++(135:1)--++(45:2)coordinate (X);
\fill [white] (X) circle(69pt);
\begin{scope}\draw[ thick,fill=white] (Y)--++(135:2)--++(45:4)--++(-45:2) --++(-135:4)--(0,0); 

 \draw[very thick,fill=gray!40]  (Y)--++(135:2)--++(45:2)--++(-45:1) --++(45:2) --++(-45:1) --++(-135:4)--(0,0); 
 \clip  (Y)--++(135:2)--++(45:4)--++(-45:2)   --++(-135:4)--(0,0); 
 
 		\foreach \i in {0,1,2,3,4,5,6,7,8,9,10,11,12}
		{
			\path (Y)--++(45:1*\i) coordinate (c\i); 
			\path (Y)--++(135:1*\i)  coordinate (d\i); 
			\draw[thick, ] (c\i)--++(135:14);
			\draw[thick, ] (d\i)--++(45:14);
		}
	 	\end{scope}

 \end{tikzpicture}$
 };

 \draw(-2.5+1.25,-5) node 
 {
$ \begin{tikzpicture}[scale=0.3]
 
 \path  (0,0) coordinate (Y);
 \path  (Y)--++(135:1)--++(45:2)coordinate (X);
\fill [white] (X) circle(69pt);
\begin{scope}
\draw[ thick,fill=white] (Y)--++(135:2)--++(45:4)--++(-45:2) --++(-135:4)--(0,0); 

 \draw[very thick,fill=gray!40] (Y)--++(135:2)--++(45:3)--++(-45:2)    --++(-135:3)--(0,0); 
 \clip  (Y)--++(135:2)--++(45:4)--++(-45:2)    --++(-135:4)--(0,0); 
 
 		\foreach \i in {0,1,2,3,4,5,6,7,8,9,10,11,12}
		{
			\path (Y)--++(45:1*\i) coordinate (c\i); 
			\path (Y)--++(135:1*\i)  coordinate (d\i); 
			\draw[thick, ] (c\i)--++(135:14);
			\draw[thick, ] (d\i)--++(45:14);
		}
	 	\end{scope}

 \end{tikzpicture}$
 };

 \draw(1.25,-5) node 
 {
$ \begin{tikzpicture}[scale=0.3]
 
 \path  (0,0) coordinate (Y);
 \path  (Y)--++(135:1)--++(45:2)coordinate (X);
\fill [white] (X) circle(69pt);
\begin{scope}
\draw[ thick,fill=white] (Y)--++(135:2)--++(45:4)--++(-45:2) --++(-135:4)--(0,0); 

 \draw[very thick,fill=gray!40] (Y)--++(135:2)--++(45:2)--++(-45:2)    --++(-135:2)--(0,0); 
  \clip (Y)--++(135:2)--++(45:4)--++(-45:2) --++(-135:4)--(0,0); 
--++(135:2)--++(45:4)--++(-45:2) --++(-135:4)--(0,0); 
--++(135:2)--++(45:2)--++(-45:2)    --++(-135:2)--(0,0); 
 
 		\foreach \i in {0,1,2,3,4,5,6,7,8,9,10,11,12}
		{
			\path (Y)--++(45:1*\i) coordinate (c\i); 
			\path (Y)--++(135:1*\i)  coordinate (d\i); 
			\draw[thick, ] (c\i)--++(135:14);
			\draw[thick, ] (d\i)--++(45:14);
		}
	 	\end{scope}

 \end{tikzpicture}$
 };

 \draw(2.5+1.25,-5) node 
 {
$ \begin{tikzpicture}[scale=0.3]
 
 \path  (0,0) coordinate (Y);
 \path  (Y)--++(135:1)--++(45:2)coordinate (X);
\fill [white] (X) circle(69pt);
\begin{scope}\draw[ thick,fill=white] (Y)--++(135:2)--++(45:4)--++(-45:2) --++(-135:4)--(0,0); 

 \draw[very thick,fill=gray!40](Y)--++(135:2)--++(45:1)--++(-45:1) --++(45:2) --++(-45:1) --++(-135:3)--(0,0); 
  \clip (Y)--++(135:2)--++(45:4)--++(-45:2) --++(-135:4)--(0,0); 
--++(135:2)--++(45:1)--++(-45:1) --++(45:2) --++(-45:1) --++(-135:3)--(0,0); 
 
 		\foreach \i in {0,1,2,3,4,5,6,7,8,9,10,11,12}
		{
			\path (Y)--++(45:1*\i) coordinate (c\i); 
			\path (Y)--++(135:1*\i)  coordinate (d\i); 
			\draw[thick, ] (c\i)--++(135:14);
			\draw[thick, ] (d\i)--++(45:14);
		}
	 	\end{scope}

 \end{tikzpicture}$
 };

 \draw(-2.5,-5-2.5) node 
 {
$ \begin{tikzpicture}[scale=0.3]
 
 \path  (0,0) coordinate (Y);
 \path  (Y)--++(135:1)--++(45:2)coordinate (X);
\fill [white] (X) circle(69pt);
\begin{scope}
\draw[ thick,fill=white] (Y)--++(135:2)--++(45:4)--++(-45:2) --++(-135:4)--(0,0); 

 \draw[very thick,fill=gray!40] (Y)--++(135:2)--++(45:1)--++(-45:1) --++(45:3) --++(-45:1) --++(-135:4)--(0,0); 
  \clip (Y)--++(135:2)--++(45:4)--++(-45:2) --++(-135:4)--(0,0); 
--++(135:2)--++(45:1)--++(-45:1) --++(45:3) --++(-45:1) --++(-135:4)--(0,0); 
 
 		\foreach \i in {0,1,2,3,4,5,6,7,8,9,10,11,12}
		{
			\path (Y)--++(45:1*\i) coordinate (c\i); 
			\path (Y)--++(135:1*\i)  coordinate (d\i); 
			\draw[thick, ] (c\i)--++(135:14);
			\draw[thick, ] (d\i)--++(45:14);
		}
	 	\end{scope}

 \end{tikzpicture}$
 };

 \draw(2.5,-5-2.5) node 
 {
$ \begin{tikzpicture}[scale=0.3]
 
 \path  (0,0) coordinate (Y);
 \path  (Y)--++(135:1)--++(45:2)coordinate (X);
\fill [white] (X) circle(69pt);
\begin{scope}\draw[ thick,fill=white] (Y)--++(135:2)--++(45:4)--++(-45:2) --++(-135:4)--(0,0); 

 \draw[very thick,fill=gray!40](Y)--++(135:2)--++(45:1)--++(-45:1) --++(45:1) --++(-45:1) --++(-135:2)--(0,0); 
  \clip (Y)--++(135:2)--++(45:4)--++(-45:2) --++(-135:4)--(0,0); 
--++(135:2)--++(45:4)--++(-45:2) --++(-135:4)--(0,0); 
--++(135:2)--++(45:1)--++(-45:1) --++(45:1) --++(-45:1) --++(-135:2)--(0,0); 
 
 		\foreach \i in {0,1,2,3,4,5,6,7,8,9,10,11,12}
		{
			\path (Y)--++(45:1*\i) coordinate (c\i); 
			\path (Y)--++(135:1*\i)  coordinate (d\i); 
			\draw[thick, ] (c\i)--++(135:14);
			\draw[thick, ] (d\i)--++(45:14);
		}
	 	\end{scope}

 \end{tikzpicture}$
 };

 \end{tikzpicture} $$
\caption{The regular  quiver   $ {Q}_{2,4}$ is obtained by ``doubling up" the above graph.  }
\label{loopydyck} 
\end{figure}

\begin{defn}
Given $P, Q$ two Dyck paths, we say that   $P$ is {\sf (right) dominated by}  $Q$, and write $P\lhd  Q$ if 
 ${\sf last}(P) < {\sf first}(Q)$.  
Given $P \in {\rm Rem}_0(\la)$, we let ${\sf rt}(P)$ denote the element of 
$\{Q \mid P\vartriangleleft Q, Q \in {\rm Add}_1(\la) \}$ which is of maximal breadth.
\end{defn}

We define certain generalised loop elements of the  path algebra $\Bbbk Q_{m,n}$ as follows.
We first associate the loop generator to a canonical Dyck path as follows: 
\begin{defn}\label{L=lincombo}
For $\la \in \mathscr{R}_{m,n}$ and $P \in {\rm DRem}(\la)$, we define an element $\mathbb{L}^\la_\la(-P)$ as a product of the generators of the path algebra as follows
\begin{align*} 
\mathbb{L}^\la_\la(-P)=
\begin{cases}
(-1)^{b(P)}	\mathbb D^\la _{\la-P}\mathbb D_\la ^{\la-P}			&\text{if }\level(P)>0
\\[5pt]
(-1)^{b({\sf rt}(P))+1}	\mathbb D^\la _{\la+{\sf rt}(P)}\mathbb D_\la ^{\la+{\sf rt}(P)}			&\text{if }\level(P)=0 \text { and }m<n
\\[5pt]
- \mathbb L^\la_\la +(-1)^{b({\sf rt}(P))+1}	\mathbb D^\la _{\la+{\sf rt}(P)}\mathbb D_\la ^{\la+{\sf rt}(P)}			 &\text{if }\level(P)=0 \text { and }m=n   
\\
 \mathbb L^\la_\la 		
	 &\text{if } \level(P)=0, \text { $m=n$, ${\sf last}(P)=m-1$   }   
\end{cases} 
\end{align*}
\end{defn}

%
%
%
%
%

\begin{defn}\label{generatortheorem2}
We define the symmetric Dyck algebra, $\mathcal{A}_{m,n}$,  to be the $\ZZ$-graded $\Bbbk$-algebra  given by the path algebra 
$\Bbbk Q_{m,n}$ modulo the following relations and their duals: 

	\smallskip\noindent
{\bf The idempotent   
relations:} 
For all $\la,\mu \in \restr$, we have that 
\begin{equation}\label{rel1}
\mathbbm 1_\mu{\mathbbm 1}_\la =\delta_{\la,\mu}{\mathbbm 1}_\la \qquad 
\qquad {\mathbbm 1}_\la \mathbb  D^\la_\mu {\mathbbm 1}_\mu = \mathbb  D^\la_\mu
\qquad {\mathbbm 1}_\la \mathbb L^\la_\la{\mathbbm 1}_\la=\mathbb L^\la_\la 
\end{equation} 
(where the final relation is only defined for the $m=n$ case).

\smallskip\noindent
	{\bf The 
	self-dual relation: } 
	Let $\la \in \restr$ and  $P\in {\rm DAdd}(\la)$. Then we   have 
\begin{equation}\label{rel2}
\mathbb  D_{\la+P}^{\la} \mathbb  D_{\la}^{\la+P}
= (-1)^{b(P)-1}\Bigg(
2
\!\! \sum_{   \begin{subarray}{c} Q\in {\rm DRem}(\la) \\ P\prec Q \end{subarray}}
\!\!
\mathbb{L}^{\la}_\la  ( - Q ) + 
\!\!\sum_{  \begin{subarray}{c} Q\in {\rm DRem}(\la) \\ Q \, \text{adj.}\, P \end{subarray}}
\!\!
\mathbb{L}^{\la}_\la  ( - Q )\Bigg) 
\end{equation}  
where the elements on the righthand-side are linear combinations of the generators, as in \cref{L=lincombo}.

\smallskip\noindent
{\bf The commuting relations:} 
Let $P,Q\in {\rm DRem}(\la)$ which commute. Then we have 
\begin{equation}\label{rel3}
\mathbb  D^{\la -P-Q}_{\la-P}\mathbb  D^{\la-P}_\la = \mathbb  D^{\la-P-Q}_{\la -Q}\mathbb  D^{\la -Q}_\la \qquad
\mathbb  D^{\la -P}_\la \mathbb  D^\la_{\la - Q} = \mathbb  D^{\la - P}_{\la - P - Q}\mathbb  D^{\la - P - Q}_{\la - Q}.
\end{equation}

\smallskip\noindent
{\bf The non-commuting relation:}   
Let $P,Q\in {\rm DRem}(\mu)$ with $P\prec Q$ which do not commute. Then  $Q\setminus P = Q^1\sqcup Q^2$ where $Q^1, Q^2 \in {\rm DRem}(\mu - P)$ and we have 
\begin{equation}\label{rel4}
\mathbb  D^{\mu -Q}_\mu \mathbb  D^\mu_{\mu-P} = \mathbb  D^{\mu - Q}_{\mu - P - Q^1}\mathbb  D^{\mu - P - Q^1}_{\mu - P} = \mathbb  D^{\mu - Q}_{\mu - P - Q^2}\mathbb  D^{\mu - P - Q^2}_{\mu - P}
\end{equation}

\smallskip\noindent
{\bf The adjacent relation:}  
Let $P\in {\rm DRem}(\mu)$ and $Q\in {\rm DRem}(\mu - P)$ be adjacent. Recall that 
 $\langle P\cup Q \rangle_\mu$, if it exists,  denotes the smallest removable Dyck path of $\mu$ containing $P\cup Q$. Then we have 
\begin{equation}\label{adjacent}
\mathbb   D^{\mu - P - Q}_{\mu - P}\mathbb   D^{\mu - P}_\mu = \left\{ \begin{array}{ll} (-1)^{b(\langle P\cup Q\rangle_\mu) - b(Q)} \mathbb   D^{\mu - P - Q}_{\mu - \langle P\cup Q\rangle_\mu}\mathbb   D^{\mu - \langle P\cup Q\rangle_\mu}_\mu & \mbox{if $\langle P\cup Q\rangle_\mu$ exists} \\ 0 & \mbox{otherwise} \end{array} \right.
\end{equation}

\smallskip\noindent
{\bf The cubic relation:}    
Let $P \in {\rm DAdd}_1(\mu)$ be such ${\sf last}(P)$ is maximal with respect to this property. Then 
\begin{equation}\label{cubic}
\mathbb   D_{\mu}^{\mu+P}
\mathbb  D^{\mu}_{\mu+P}
\mathbb  D^{\mu}_{\mu+P}
=
\begin{cases}(-1)^{b(P)+1}
2 \mathbb{L}^{\mu+P}_{\mu+P}   \mathbb
  D^{\mu+P}_{\mu}			&\text{if $m=n$}\\
0								 	&\text{if $m<n$}.
\end{cases}
\end{equation}

\smallskip\noindent
{\bf The additional $m=n$ relations:}   
We have the loop-nilpotency and loop-commutation relations: namely, for all $\la,\mu \in \mathscr{R}_{m,m}$ the following holds
\begin{equation}\label{loop-relation}
(\mathbb L^\mu_\mu)^2=0 \qquad \mathbb D^\la_\mu \mathbb L^\mu_\mu =  \mathbb L^\la_\la  \mathbb D^\la_\mu .
\end{equation}

\end{defn}

%
%
%

  \section{The isomorphism theorem}\label{Therome}

%
%
%
%
%
%
%
 In this section, we prove the main result of this paper: that the algebras $\mathcal{A}_{m,n}$ and 
 $K^m_n$ are isomorphic as $\ZZ$-graded $\Bbbk$-algebras. 
We first fix some notation as follows:
 $$
 \mathbb D (-Q) := 
 \sum_{
 \begin{subarray}c
  {\la\in  \mathscr{R}_{m,n}}
  \\
 Q \in {\rm DRem}(\la)
 \end{subarray}
 }\!\!\!
 {\mathbbm 1}_\la\mathbb  D^\la_{\la-Q}
 \qquad
 \mathbb D (+Q) := 
\sum_{
 \begin{subarray}c
  {\la\in  \mathscr{R}_{m,n}}
  \\
Q \in {\rm DAdd}(\la)
 \end{subarray}
 }\!\!\!
 {\mathbbm 1}_\la \mathbb D^\la_{\la+Q}
 \qquad  
 \mathbb L (-Q) := 
\sum_{
 \begin{subarray}c
  {\la\in  \mathscr{R}_{m,n}}
  \\
Q \in {\rm DRem}(\la)
 \end{subarray}
 } \!\!\!
  {\mathbbm 1}_\la \mathbb L^\la_{\la}(-Q).
 $$

 \subsection{Implied relations in the symmetric Dyck path algebra}
 In this section we go through the not inconsiderable effort of establishing 
 a number of lemmas and propositions regarding the products of elements from 
 $\mathcal{A}_{m,n}$.
 
 \begin{lem} 
  \label{h0loop_rect}
Let $m<n$. For $Q \in {\rm DRem}_0(\la)$ and $ P \in {\rm DRem}_{>0}(\la)$ we have that 
 either $P$ and $Q$ commute or 
  $Q\setminus P = Q^1\sqcup Q^2$ where $Q^1, Q^2 \in {\rm DRem}(\la - P)$. 
In the former case we have that 
 $$
 \mathbb L^\la_\la(-Q  )
  \mathbb D^\la_{\la-P }
  = 
  \mathbb D^\la_{\la-P }
   \mathbb L^{\la-P }_{\la-P }(-Q  )
 $$
 and in the latter case we have that 
\begin{align}\label{noncomlemmerrob1}
    \mathbb D^\la_{\la-P }
   \mathbb L^{\la-P }_{\la-P }(-Q^1)  = 
   \mathbb L^\la_\la(-Q  )
  \mathbb D^\la_{\la-P }
  = 
  \mathbb D^\la_{\la-P }
   \mathbb L^{\la-P }_{\la-P }(-Q^2).
\end{align}
 \end{lem} 
 \begin{proof}
We recall from \cref{L=lincombo}   that $$\mathbb L^\la_\la(-Q )=(-1)^{b({\sf rt}(Q))}\mathbb D^{\la}_{\la + {\sf rt}(Q)} \mathbb D_{\la}^{\la + {\sf rt}(Q)}.$$
If $P,Q$ commute, then $P, {\sf rt}(Q)$ commute as well. This is straightforward if $P$ is to the left of $Q$. If $P$ is to the right of $Q$, then 
our assumptions $Q \in {\rm DRem}_0(\la)$ and $ P \in {\rm DRem}_{>0}(\la)$   
together imply that 
${\sf first}(P)\geq{\sf last}(Q)+2$. 
By definition  $ {\sf rt}(Q) \in {\rm DAdd}_1(\la)$  and so 
 ${\sf first}({\sf rt}(Q))\leq {\sf first}(P)-2$ 
 and either 
  ${\sf last }({\sf rt}(Q))\leq {\sf first}(P)-2$  or   ${\sf last }({\sf rt}(Q))\geq {\sf last}(P)+2$; in either case, $P$ 
  and ${\sf rt}(Q)$ commute. 
 Examples of these  two cases are  depicted  in  \cref{fig:PQcommute}.
 Therefore we have that 
\begin{align*}
    \mathbb L^\la_\la(-Q ) \mathbb D^\la_{\la-P } 
    &
    = (-1)^{b({\sf rt}(Q))}{\mathbbm 1}_\la 	 \mathbb D ( + {\sf rt}(Q)  )  \mathbb D( -  {\sf rt}(Q) )   \mathbb D(-P )\\
    &= (-1)^{b({\sf rt}(Q))}{\mathbbm 1}_\la 	 \mathbb D ( + {\sf rt}(Q)  )    \mathbb D(-P )	 \mathbb D( -  {\sf rt}(Q) )		\\
    &= (-1)^{b({\sf rt}(Q))}	{\mathbbm 1}_\la   \mathbb D(-P )	 \mathbb D ( + {\sf rt}(Q)  )    \mathbb D( -  {\sf rt}(Q) )		\\
    &= \mathbb D^\la_{\la-P }   \mathbb L^{\la-P }_{\la-P }(-Q ).
\end{align*}
where only the second and third  equalities   follows from applying the commuting relation \eqref{rel3} to the Dyck paths $P$ and ${\sf rt}(Q)$;   the first and fourth  equalities follow by \cref{L=lincombo}.

\begin{figure}[ht!]
%
 
 $$ \begin{tikzpicture}[scale=0.35]
  \path(0,0)--++(135:2) coordinate (hhhh);
    \path(hhhh)--++(45:1) coordinate (hhhh1);
  
 \draw[very thick] (hhhh1)--++(135:5)--++(45:6)--++(-45:5)--++(-135:6);
 \clip(hhhh1)--++(135:5)--++(45:6)--++(-45:5)--++(-135:6);
\path(0,0) coordinate (origin2);
  \path(0,0)--++(135:2) coordinate (origin3);

      \foreach \i in {0,1,2,3,4,5,6,7,8}
{
\path (origin3)--++(45:0.5*\i) coordinate (c\i); 
\path (origin3)--++(135:0.5*\i)  coordinate (d\i); 
  }

   \foreach \i in {0,1,2,3,4,5,6,7,8}
{
\path (origin3)--++(45:1*\i) coordinate (c\i); 
\path (c\i)--++(-45:0.5) coordinate (c\i); 
\path (origin3)--++(135:1*\i)  coordinate (d\i); 
\path (d\i)--++(-135:0.5) coordinate (d\i); 
\draw[thick,densely dotted] (c\i)--++(135:10);
\draw[thick,densely dotted] (d\i)--++(45:10);
  }

\path(0,0)--++(135:2) coordinate (hhhh);

\fill[opacity=0.2](hhhh)--++(135:5)--++(45:1)--++(-45:1) coordinate (JJ)--++(45:1)--++(-45:1)--++(45:2)--++(-45:1) coordinate (JJ1)--++(45:1)--++(-45:1)--++(45:2)--++(-45:1);

\fill[opacity=0.4,magenta] 
  (JJ)
--++(135:1)
 --++(45:2)
--++(-45:2)--++(-135:1)--++(135:1)
;

\fill[opacity=0.4,cyan] 
  (JJ1)
--++(135:1)
 --++(45:2)
--++(-45:2)--++(-135:1)--++(135:1)
;

\path(origin3)--++(45:-0.5)--++(135:7.5) coordinate (X)coordinate (start);

\path (start)--++(135:1)coordinate (start);
 
\path (start)--++(45:1)--++(-45:3) coordinate (X) ; 
 
\path(X)--++(45:1) coordinate (X) ;
 
\path (X)--++(-45:1) coordinate (X) ;
\path(X)--++(45:1) coordinate (X) ;
\path(X)--++(-45:1) coordinate (X) ;

\path(X)--++(45:1) coordinate (X) ;

 \fill[darkgreen](X) circle (4pt);
 

\end{tikzpicture} \qquad
  \begin{tikzpicture}[scale=0.35]
%
\path(0,0)--++(135:2) coordinate (hhhh);
    \path(hhhh)--++(45:1) coordinate (hhhh1);
  
 \draw[very thick] (hhhh1)--++(135:5)--++(45:6)--++(-45:5)--++(-135:6);
 \clip(hhhh1)--++(135:5)--++(45:6)--++(-45:5)--++(-135:6);
\path(0,0) coordinate (origin2);
  \path(0,0)--++(135:2) coordinate (origin3);

      \foreach \i in {0,1,2,3,4,5,6,7,8}
{
\path (origin3)--++(45:0.5*\i) coordinate (c\i); 
\path (origin3)--++(135:0.5*\i)  coordinate (d\i); 
  }

   \foreach \i in {0,1,2,3,4,5,6,7,8}
{
\path (origin3)--++(45:1*\i) coordinate (c\i); 
\path (c\i)--++(-45:0.5) coordinate (c\i); 
\path (origin3)--++(135:1*\i)  coordinate (d\i); 
\path (d\i)--++(-135:0.5) coordinate (d\i); 
\draw[thick,densely dotted] (c\i)--++(135:10);
\draw[thick,densely dotted] (d\i)--++(45:10);
  }

\path(0,0)--++(135:2) coordinate (hhhh);

\fill[opacity=0.2](hhhh)--++(135:5)--++(45:1)--++(-45:1) coordinate (JJ)--++(45:1)--++(-45:1)--++(45:2)--++(-45:1) coordinate (JJ1)--++(45:1)--++(-45:1)--++(45:1)--++(-45:1);

\fill[opacity=0.4,magenta] 
  (JJ)
--++(135:1)
 --++(45:2)
--++(-45:2)--++(-135:1)--++(135:1)
;

\fill[opacity=0.4,cyan] 
  (JJ1)
--++(135:1)
 --++(45:2)
--++(-45:2)--++(-135:1)--++(135:1)
;

\path(origin3)--++(45:-0.5)--++(135:7.5) coordinate (X)coordinate (start);

\path (start)--++(135:1)coordinate (start);
 
\path (start)--++(45:1)--++(-45:3) coordinate (X) ; 
 
\path(X)--++(45:1) coordinate (X) ;
 
\path (X)--++(-45:1) coordinate (X) ;
\path(X)--++(45:1) coordinate (X) ;
\path(X)--++(-45:1) coordinate (X) ;

\path(X)--++(45:1) coordinate (X) ;

 \fill[darkgreen](X) circle (4pt);
  \draw[ thick, darkgreen](X)--++(45:1) coordinate (X) ;
\fill[darkgreen](X) circle (4pt);
\draw[ thick, darkgreen](X)--++(45:1) coordinate (X) ;
\fill[darkgreen](X) circle (4pt);
\draw[ thick, darkgreen](X)--++(45:1) coordinate (X) ;
\fill[darkgreen](X) circle (4pt);
\draw[ thick, darkgreen](X)--++(-45:1) coordinate (X) ;
\fill[darkgreen](X) circle (4pt);
\draw[ thick, darkgreen](X)--++(-45:1) coordinate (X) ;
\fill[darkgreen](X) circle (4pt);
\draw[ thick, darkgreen](X)--++(-45:1) coordinate (X) ;
\fill[darkgreen](X) circle (4pt);

\end{tikzpicture} 
$$
    \caption{Examples of commuting paths  $\color{cyan}P$,  $  \color{magenta}Q$
     and $ \color{darkgreen} {\sf rt}({ Q})$ 
    such that    ${ \color{magenta}Q} \in {\rm DRem}_0(\la)$ and $ {\color{cyan}P} \in {\rm DRem}_{>0}(\la)$ for $\la= (5^2,3^3,1)$ and $\la= (5^2,3^3)$ respectively.
    }
    \label{fig:PQcommute}
\end{figure}
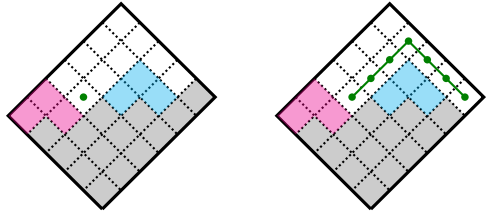

   We first consider the latter equality in \cref{noncomlemmerrob1}.  
We now suppose that  $Q \setminus P = Q^1\sqcup Q^2$ where $Q^1, Q^2 \in {\rm DRem}(\la - P)$ and without loss of generality, we assume that  
$Q^1$ is to the left of $P$ and $Q^2$ is to the right of $P$. 
In which case,   ${\sf rt}(Q)={\sf rt}(Q^2)  $ and this Dyck path commutes with $P$.
 We have that 
\begin{align*}
 \mathbb D^\la_{\la-P }\mathbb L^{\la-P}_{\la-P}(-Q  ) 
     &=  (-1)^{b({\sf rt}(Q ))}
          {\mathbbm 1}_\la  \mathbb D (-P )\mathbb D(+{\sf rt}(Q ) )\mathbb D(-{\sf rt}(Q ) )\\
     &= (-1)^{b({\sf rt}(Q ))}
          {\mathbbm 1}_\la \mathbb D(+{\sf rt}(Q ) )   \mathbb D (-P ) \mathbb D(-{\sf rt}(Q ) )\\
     &=(-1)^{b({\sf rt}(Q ))}
          {\mathbbm 1}_\la \mathbb D(+{\sf rt}(Q ) )     \mathbb D(-{\sf rt}(Q ) )	\mathbb D (-P )	\\
     &= \mathbb L^{\la}_{\la}(-Q ) \mathbb D^\la_{\la-P }.
\end{align*}
where the first and final equalities follow from \cref{L=lincombo} and our observation that  ${\sf rt}(Q)={\sf rt}(Q^2)  $;
  the second and third equalities follow from  applying the commuting relation \eqref{rel3} to the Dyck paths $P$ and ${\sf rt}(Q)$.

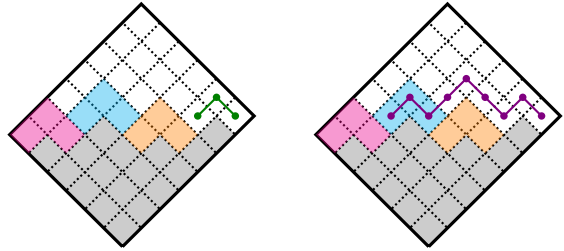
\begin{figure}[ht!]

$$ \begin{tikzpicture}[scale=0.35]
  \path(0,0)--++(135:2) coordinate (hhhh);
 \draw[very thick] (hhhh)--++(135:6)--++(45:7)--++(-45:6)--++(-135:7);
 \clip (hhhh)--++(135:6)--++(45:7)--++(-45:6)--++(-135:7);
\path(0,0) coordinate (origin2);
  \path(0,0)--++(135:2) coordinate (origin3);

      \foreach \i in {0,1,2,3,4,5,6,7,8}
{
\path (origin3)--++(45:0.5*\i) coordinate (c\i); 
\path (origin3)--++(135:0.5*\i)  coordinate (d\i); 
  }

   \foreach \i in {0,1,2,3,4,5,6,7,8}
{
\path (origin3)--++(45:1*\i) coordinate (c\i); 
\path (c\i)--++(-45:0.5) coordinate (c\i); 
\path (origin3)--++(135:1*\i)  coordinate (d\i); 
\path (d\i)--++(-135:0.5) coordinate (d\i); 
\draw[thick,densely dotted] (c\i)--++(135:10);
\draw[thick,densely dotted] (d\i)--++(45:10);
  }

\path(0,0)--++(135:2) coordinate (hhhh);

\fill[opacity=0.2](hhhh)
--++(135:5) coordinate (JJ)
--++(45:1)
--++(-45:1)--++(45:2)
--++(-45:2)--++(45:1)
--++(-45:1)--++(45:2)--++(-45:1);

\fill[opacity=0.4,magenta](hhhh)
  (JJ)
--++(135:1)
 --++(45:2)
--++(-45:2)--++(-135:1)--++(135:1)
;

\path(JJ)--++(45:2)--++(-45:1) coordinate (JJ);

\fill[opacity=0.4,cyan](hhhh)
  (JJ)
--++(135:1)
 --++(45:2)
--++(-45:2)--++(-135:1)--++(135:1)
;

\path(JJ)--++(45:1)--++(-45:2) coordinate (JJ);

\fill[opacity=0.4,orange](hhhh)
  (JJ)
--++(135:1)
 --++(45:2)
--++(-45:2)--++(-135:1)--++(135:1)
;

\path(origin3)--++(45:-0.5)--++(135:7.5) coordinate (X)coordinate (start);

\path (start)--++(135:1)coordinate (start);
 
\path (start)--++(45:5)--++(-45:6) coordinate (X) ; 
 
\path(X)--++(45:1) coordinate (X) ;
 
\path (X)--++(-45:1) coordinate (X) ;
 \fill[darkgreen](X) circle (4pt);
\draw[ thick, darkgreen](X)--++(45:1) coordinate (X) ;
\fill[darkgreen](X) circle (4pt);
\draw[ thick, darkgreen](X)--++(-45:1) coordinate (X) ;
\fill[darkgreen](X) circle (4pt);

\end{tikzpicture}
\qquad 
\begin{tikzpicture}[scale=0.35]
  \path(0,0)--++(135:2) coordinate (hhhh);
 \draw[very thick] (hhhh)--++(135:6)--++(45:7)--++(-45:6)--++(-135:7);
 \clip (hhhh)--++(135:6)--++(45:7)--++(-45:6)--++(-135:7);
\path(0,0) coordinate (origin2);
  \path(0,0)--++(135:2) coordinate (origin3);

      \foreach \i in {0,1,2,3,4,5,6,7,8}
{
\path (origin3)--++(45:0.5*\i) coordinate (c\i); 
\path (origin3)--++(135:0.5*\i)  coordinate (d\i); 
  }

   \foreach \i in {0,1,2,3,4,5,6,7,8}
{
\path (origin3)--++(45:1*\i) coordinate (c\i); 
\path (c\i)--++(-45:0.5) coordinate (c\i); 
\path (origin3)--++(135:1*\i)  coordinate (d\i); 
\path (d\i)--++(-135:0.5) coordinate (d\i); 
\draw[thick,densely dotted] (c\i)--++(135:10);
\draw[thick,densely dotted] (d\i)--++(45:10);
  }

\path(0,0)--++(135:2) coordinate (hhhh);

\fill[opacity=0.2](hhhh)
--++(135:5) coordinate (JJ)
--++(45:1)
--++(-45:1)--++(45:2)
--++(-45:2)--++(45:1)
--++(-45:1)--++(45:2)--++(-45:1);

\fill[opacity=0.4,magenta](hhhh)
  (JJ)
--++(135:1)
 --++(45:2)
--++(-45:2)--++(-135:1)--++(135:1)
;

\path(JJ)--++(45:2)--++(-45:1) coordinate (JJ);

\fill[opacity=0.4,cyan](hhhh)
  (JJ)
--++(135:1)
 --++(45:2)
--++(-45:2)--++(-135:1)--++(135:1)
;

\path(JJ)--++(45:1)--++(-45:2) coordinate (JJ);

\fill[opacity=0.4,orange](hhhh)
  (JJ)
--++(135:1)
 --++(45:2)
--++(-45:2)--++(-135:1)--++(135:1)
;

\path(origin3)--++(45:-0.5)--++(135:7.5) coordinate (X)coordinate (start);

\path (start)--++(135:1)coordinate (start);
 
\path (start)--++(45:2)--++(-45:3) coordinate (X) ; 
 
\path(X)--++(45:1) coordinate (X) ;
 
\path (X)--++(-45:1) coordinate (X) ;
 \fill[violet](X) circle (4pt);
\draw[ thick, violet](X)--++(45:1) coordinate (X) ;
\fill[violet](X) circle (4pt);
\draw[ thick, violet](X)--++(-45:1) coordinate (X) ;
\fill[violet](X) circle (4pt);
 
\draw[ thick, violet](X)--++(45:1) coordinate (X) ;
\fill[violet](X) circle (4pt);
\draw[ thick, violet](X)--++(45:1) coordinate (X) ;
\fill[violet](X) circle (4pt);
\draw[ thick, violet](X)--++(-45:1) coordinate (X) ;
\fill[violet](X) circle (4pt);
\draw[ thick, violet](X)--++(-45:1) coordinate (X) ;
\fill[violet](X) circle (4pt);

\draw[ thick, violet](X)--++(45:1) coordinate (X) ;
\fill[violet](X) circle (4pt);
\draw[ thick, violet](X)--++(-45:1) coordinate (X) ;
\fill[violet](X) circle (4pt);

\end{tikzpicture}$$

    \caption{
    An example  of 
        $Q, {\color{cyan} P} \in {\rm DRem}_1(\la)$ for $\la=(6^2,5^2,3,1)$ 
  such that 
    $Q-{\color{cyan} P}={\color{magenta} Q^1} \sqcup {\color{orange} Q^2}$. 
        On the left we depict ${\color{darkgreen}{\sf rt}(Q)}= {\sf rt}( {\color{orange} Q^2})$. 
       On the right we depict ${\color{violet}{\sf rt}(  Q^1 )}$ which we note is the smallest Dyck path
       in ${\rm DAdd}(\la-P)$ 
        containing 
       both 
$        {\sf rt}( {\color{orange} Q^2})$ and ${\color{cyan} P}$.  
       }
    \label{fig:PsplitsQ}
\end{figure}

We now consider the former equality in \cref{noncomlemmerrob1}.  
We have that ${\sf rt}(Q^1)$ is the smallest Dyck path in 
 ${\rm Add}_1(\la-P)$ such that  $P \prec {\sf rt}(Q^1)$ and ${\sf rt}(Q) \prec {\sf rt}(Q^1)$.  
 We further note that $Q^2 \prec {\sf rt}(Q^1)$ and that 
  $b(  {\sf rt}(Q^1))=b(  {\sf rt}(Q)) + b(Q^2)+b(P)$. 
See \cref{fig:PsplitsQ} for   examples. 
 We have that 
 \begin{align*}
     \mathbb D^\la_{\la-P }\mathbb L^{\la-P}_{\la-P}(-Q^1 ) 
     &=
       (-1)^{b({\sf rt}(Q^1))} {\mathbbm 1}_\la \mathbb D(-P) \mathbb D (+{\sf rt}(Q^1) ) 
       \mathbb D(-{\sf rt}(Q^1)  ) \\
      &=  
      (-1)^{2b({\sf rt}(Q^1))-b({\sf rt}(Q^2))}
      {\mathbbm 1}_\la 
      \mathbb D (+{\sf rt}(Q^2)  )
  \mathbb D (+ Q^2   )
  \mathbb D (- {\sf rt}(Q^1   ))
 \\
     &= (-1)^{b({\sf rt}(Q^2))}
     {\mathbbm 1}_\la \mathbb D (+{\sf rt}(Q^2) ) 
     \mathbb D (-{\sf rt}(Q^2) )
     \mathbb D (-P) 
     \\
      &= \mathbb L^{\la}_{\la}(-Q ) \mathbb D^\la_{\la-P }.
\end{align*}
where the second equality follows from the adjacency relation  \eqref{adjacent} 
 applied to 
 ${\sf rt}(Q^2) \in {\rm Add}(\la)$ and 
 $Q^2 \in {\rm Add}(\la+{\sf rt}(Q^1))$; 
  the third equality follows from   relation  \eqref{rel4} applied to 
  the non-commuting pair $Q^2 \prec {\sf rt}(Q^1)$; the first and fourth equalities follow from  \cref{L=lincombo}.  
\end{proof}

 \begin{lem}
  \label{h0loop_sq}
 Let $m=n$. For $Q \in {\rm DRem}_0(\la)$ and $ P \in {\rm DRem}_{>0}(\la)$ we have that 
 either $P$ and $Q$ commute or 
  $Q\setminus P = Q^1\sqcup Q^2$ where $Q^1, Q^2 \in {\rm DRem}(\la - P)$. 
In the former case we have that 
\begin{align}\label{sdagkhdfhsdfjgdfhjkghdjksgfhsajkgfhjdsgfhdegfjdhfjdhf2}
 \mathbb L^\la_\la(-Q  )
  \mathbb D^\la_{\la-P }
  = 
  \mathbb D^\la_{\la-P }
   \mathbb L^{\la-P }_{\la-P }(-Q  )
\end{align}
 and in the latter case we have that 
\begin{align}\label{sdagkhdfhsdfjgdfhjkghdjksgfhsajkgfhjdsgfhdegfjdhfjdhf}
  \mathbb D^\la_{\la-P }
   \mathbb L^{\la-P }_{\la-P }(-Q^1)  = 
 \mathbb L^\la_\la(-Q  )
  \mathbb D^\la_{\la-P }
  = 
  \mathbb D^\la_{\la-P }
   \mathbb L^{\la-P }_{\la-P }(-Q^2).
\end{align} \end{lem}

\begin{proof} We   first consider the case that  ${\sf last}(Q)=m-1$, 
  that is, the case 
that  $\mathbb L^\la_\la(-Q )=\mathbb L^\la_\la$.  We first consider the subcase in which 
  $P$ and $Q$ commute. 
  We have that  $\mathbb L^{\la-P}_{\la-P}(-Q )=\mathbb L^{\la-P}_{\la-P}$ and by   relation  \eqref{loop-relation} it follows that
$$\mathbb L^\la_\la(-Q  ) \mathbb D^\la_{\la-P }
  = \mathbb L^\la_\la \mathbb D^\la_{\la-P }
  = \mathbb D^\la_{\la-P } \mathbb L^{\la-P}_{\la-P} 
  =\mathbb D^\la_{\la-P }
   \mathbb L^{\la-P }_{\la-P }(-Q  ). 
 $$ 
 Now we consider the subcase in which $P \in {\rm DRem}(\la)$ is such that   $Q \setminus P = Q^1\sqcup Q^2$ where $Q^1, Q^2 \in {\rm DRem}(\la - P)$.
Without loss of generality, we suppose that $Q^1$ is to the left of $Q^2$ and therefore
 ${\sf last}(Q^2)=m-1$. 
Thus, $\mathbb L^{\la-P}_{\la-P}(-Q^2 )=\mathbb L^{\la-P}_{\la-P}$ and by  relation  \eqref{loop-relation} it follows that
$$\mathbb L^\la_\la(-Q  ) \mathbb D^\la_{\la-P }
  = \mathbb L^\la_\la \mathbb D^\la_{\la-P }
  = \mathbb D^\la_{\la-P } \mathbb L^{\la-P}_{\la-P} 
  =\mathbb D^\la_{\la-P }
   \mathbb L^{\la-P }_{\la-P }(-Q ^2).
 $$ Now, we observe that our assumptions on $P$ and $Q$ imply that ${\sf rt}(Q^1)=P$.  
Thus by \cref{L=lincombo} we have that 
$$\mathbb L^{\la-P}_{\la-P}(-Q^1)=-\mathbb L^{\la-P}_{\la-P}-(-1)^{b(P)}\mathbb D_\la^{\la-P }\mathbb D_{\la-P}^{\la },$$ 
which we input    into the lefthand-side of \cref{sdagkhdfhsdfjgdfhjkghdjksgfhsajkgfhjdsgfhdegfjdhfjdhf} and hence obtain 
\begin{align*}
     \mathbb D^\la_{\la-P }\mathbb L^{\la-P}_{\la-P}(-Q^1 ) 
     &= 
     {\mathbbm 1}_\la  \mathbb D (-P )(-\mathbb L^{\la-P}_{\la-P}-(-1)^{b(P)}
     \mathbb D (+P)     \mathbb D (-P)) \\
   &= 
   -  {\mathbbm 1}_\la  \mathbb D (-P ) \mathbb L^{\la-P}_{\la-P}
    -  2(-1)^{2b(P)+1}{\mathbbm 1}_\la
      \mathbb L^\la_\la  \mathbb D (-P) 
    \\
   &= 
   -  {\mathbbm 1}_\la    \mathbb L^{\la }_{\la }\mathbb D (-P )
    +  2 {\mathbbm 1}_\la
      \mathbb L^\la_\la  \mathbb D (-P) 
\\
     &=  \mathbb L^{\la}_{\la}\mathbb D^\la_{\la-P }\\
     &=  \mathbb L^{\la}_{\la}(-Q )\mathbb D^\la_{\la-P }.
\end{align*}
where the first and final equalities follow  from \cref{L=lincombo}; the second equality follows from 
relation \eqref{cubic}; the third follows from applying relation \eqref{loop-relation} to the lefthand term and tidying-up the signs for the righthand term; the fourth equality is trivial.

It remain to consider the case in which   ${\sf last}(Q)<m-1$. 
By   \cref{L=lincombo} we can express these loops in terms of ${\sf rt}(Q)$ and the case already considered above, as follows:
$$\mathbb L^\la_\la(-Q )=-\mathbb L^{\la}_{\la}-(-1)^{b({\sf rt}(Q))}\mathbb D^{\la}_{\la + {\sf rt}(Q)} \mathbb D_{\la}^{\la + {\sf rt}(Q)}.$$
Notice that $ {\sf rt}(Q) \in {\rm DAdd}_1(\la)$ and $ P \in {\rm DRem}_{>0}(\la)$ and so they cannot be adjacent. Thus,
the pair of Dyck path ${\sf rt}(Q)$ and $ P$ commute. We have that 
\begin{align*}
       \mathbb L^\la_\la(-Q )\mathbb D^\la_{\la-P }
       &={\mathbbm 1}_\la 
       \big(-\mathbb L^{\la}_{\la}-(-1)^{b({\sf rt}(Q))}\mathbb D (+ {\sf rt}(Q) ) \mathbb D(- {\sf rt}(Q) )\big ) \mathbb D ( -P  )\\
      &=-{\mathbbm 1}_\la 
        \mathbb D ( -P  )\mathbb L^{\la-P}_{\la-P}
        -(-1)^{b({\sf rt}(Q))}{\mathbbm 1}_\la\mathbb D (+ {\sf rt}(Q) ) \mathbb D(- {\sf rt}(Q) ) \mathbb D ( -P  )    
       \\
   &=-{\mathbbm 1}_\la 
        \mathbb D ( -P  )\mathbb L^{\la-P}_{\la-P}
        -(-1)^{b({\sf rt}(Q))}{\mathbbm 1}_\la\mathbb   D ( -P  )    \mathbb  D (+ {\sf rt}(Q) ) \mathbb D(- {\sf rt}(Q) )
      \\
   &=-{\mathbbm 1}_\la 
        \mathbb D ( -P  )
        \big(\mathbb L^{\la-P}_{\la-P}
        -(-1)^{b({\sf rt}(Q))}  \mathbb  D (+ {\sf rt}(Q) ) \mathbb D(- {\sf rt}(Q) )\big)
 \end{align*}
where the first equality follows from \cref{L=lincombo}; the second follows from the loop-commutation relation \eqref{loop-relation}; the third follows from the commuting relation applied to the Dyck paths ${\sf rt}(Q)$ and $ P$; the fourth equality follows  from re-bracketing.
Finally, we observe that if $P\prec Q$, then  ${\sf rt}(Q^2)={\sf rt} (Q)$ for $Q^2\in {\rm DRem}( \la-P)$; 
whereas if $P$ and $Q$ commute we   have that $Q \in {\rm DRem}( \la-P)$ (and so we do not need to rewrite anything).  
Therefore we conclude that 
  \begin{align*}
       \mathbb L^\la_\la(-Q )\mathbb D^\la_{\la-P }
       &= \begin{cases}
         {\mathbbm 1}_\la  \mathbb D(-P)\mathbb L (-Q ) & \text{if }Q,P \text{ commute;}\\
        {\mathbbm 1}_\la     \mathbb D(-P) \mathbb L (-Q^2) & \text{otherwise} 
       \end{cases}.
\end{align*}
We hence deduce that \cref{sdagkhdfhsdfjgdfhjkghdjksgfhsajkgfhjdsgfhdegfjdhfjdhf2} and the righthand equality of \cref{sdagkhdfhsdfjgdfhjkghdjksgfhsajkgfhjdsgfhdegfjdhfjdhf} both hold.

It  remains to verify the lefthand equality in \cref{sdagkhdfhsdfjgdfhjkghdjksgfhsajkgfhjdsgfhdegfjdhfjdhf}. 
We have that ${\sf rt}(Q^1)$ is the smallest Dyck path in 
 ${\rm Add}_1(\la-P)$ such that  $P \prec {\sf rt}(Q^1)$ and ${\sf rt}(Q) \prec {\sf rt}(Q^1)$.  
 We further note that $Q^2 \prec {\sf rt}(Q^1)$ and that 
  $b(  {\sf rt}(Q^1))=b(  {\sf rt}(Q)) + b(Q^2)+b(P)$. 
See \cref{fig:sq_PQsplit_noright} for   examples. 
We have that 
\begin{align*}
   \mathbb D^\la_{\la-P }\mathbb L^{\la-P}_{\la-P}(-Q^1)
     &= 
    {\mathbbm 1}_\la \mathbb D (-P )
    \big(-\mathbb L^{\la-P}_{\la-P}-(-1)^{b({\sf rt}(Q^1))}
    \mathbb D (+ {\sf rt}(Q^1))	 \mathbb D (- {\sf rt}(Q^1)
    \big)		\\
      &= -\mathbb L^{\la}_{\la}\mathbb D^\la_{\la-P }
     -(-1)^{b({\sf rt}(Q^1))}
      \mathbb D (-P )  \mathbb D (+ {\sf rt}(Q^1))	 \mathbb D (- {\sf rt}(Q^1)%
\\
     &= -\mathbb L^{\la}_{\la}\mathbb D^\la_{\la-P }
 -  (-1)^{2b({\sf rt}(Q^1))-b({\sf rt}(Q ))}
      {\mathbbm 1}_\la 
      \mathbb D (+{\sf rt}(Q )  )
   \mathbb D (+ Q^2   )
   \mathbb D (- {\sf rt}(Q^1   ))
 \\
      &= -\mathbb L^{\la}_{\la}\mathbb D^\la_{\la-P }
- (-1)^{b({\sf rt}(Q ))}
     {\mathbbm 1}_\la \mathbb D (+{\sf rt}(Q ) ) 
     \mathbb D (-{\sf rt}(Q ) )
     \mathbb D (-P) 
     \\
          &= (-\mathbb L^{\la}_{\la}-(-1)^{b({\sf rt}(Q))} \mathbb D (+{\sf rt}(Q ) ) 
     \mathbb D (-{\sf rt}(Q ) )\mathbb D_{\la-P}^{\la} \\
     &=\mathbb L^\la_\la(-Q )\mathbb D^\la_{\la-P }.
\end{align*}
where the 
first and final equalities follow by  \cref{L=lincombo}
and the second and third equalities follows from relations  \eqref{adjacent} and \eqref{rel4} (exactly as in the $m\neq n$ case).
\end{proof}

\begin{figure}[ht!]

$$ \begin{tikzpicture}[scale=0.35]
  \path(0,0)--++(135:2) coordinate (hhhh);
 \draw[very thick] (hhhh)--++(135:6)--++(45:6)--++(-45:6)--++(-135:6);
 \clip (hhhh)--++(135:6)--++(45:6)--++(-45:6)--++(-135:6);
\path(0,0) coordinate (origin2);
  \path(0,0)--++(135:2) coordinate (origin3);

      \foreach \i in {0,1,2,3,4,5,6,7,8}
{
\path (origin3)--++(45:0.5*\i) coordinate (c\i); 
\path (origin3)--++(135:0.5*\i)  coordinate (d\i); 
  }

   \foreach \i in {0,1,2,3,4,5,6,7,8}
{
\path (origin3)--++(45:1*\i) coordinate (c\i); 
\path (c\i)--++(-45:0.5) coordinate (c\i); 
\path (origin3)--++(135:1*\i)  coordinate (d\i); 
\path (d\i)--++(-135:0.5) coordinate (d\i); 
\draw[thick,densely dotted] (c\i)--++(135:10);
\draw[thick,densely dotted] (d\i)--++(45:10);
  }

\path(0,0)--++(135:2) coordinate (hhhh);

\fill[opacity=0.2](hhhh)
--++(135:5) coordinate (JJ)
--++(45:1)
--++(-45:1)--++(45:2)
--++(-45:2)--++(45:1)
--++(-45:1)--++(45:2)--++(-45:1);

\fill[opacity=0.4,magenta](hhhh)
  (JJ)
--++(135:1)
 --++(45:2)
--++(-45:2)--++(-135:1)--++(135:1)
;

\path(JJ)--++(45:2)--++(-45:1) coordinate (JJ);

\fill[opacity=0.4,cyan](hhhh)
  (JJ)
--++(135:1)
 --++(45:2)
--++(-45:2)--++(-135:1)--++(135:1)
;

\path(JJ)--++(45:1)--++(-45:2) coordinate (JJ);

\fill[opacity=0.4,orange](hhhh)
  (JJ)
--++(135:1)
 --++(45:2)
--++(-45:2)--++(-135:1)--++(135:1)
;

\path(origin3)--++(45:-0.5)--++(135:7.5) coordinate (X)coordinate (start);

\path (start)--++(135:1)coordinate (start);
 
\path (start)--++(45:5)--++(-45:6) coordinate (X) ; 
 
\path(X)--++(45:1) coordinate (X) ;
 
\path (X)--++(-45:1) coordinate (X) ;
 \fill[darkgreen](X) circle (4pt);
%

\end{tikzpicture}
\qquad 
\begin{tikzpicture}[scale=0.35]
  \path(0,0)--++(135:2) coordinate (hhhh);
  \draw[very thick] (hhhh)--++(135:6)--++(45:6)--++(-45:6)--++(-135:6);
 \clip (hhhh)--++(135:6)--++(45:6)--++(-45:6)--++(-135:6);\path(0,0) coordinate (origin2);
  \path(0,0)--++(135:2) coordinate (origin3);

      \foreach \i in {0,1,2,3,4,5,6,7,8}
{
\path (origin3)--++(45:0.5*\i) coordinate (c\i); 
\path (origin3)--++(135:0.5*\i)  coordinate (d\i); 
  }

   \foreach \i in {0,1,2,3,4,5,6,7,8}
{
\path (origin3)--++(45:1*\i) coordinate (c\i); 
\path (c\i)--++(-45:0.5) coordinate (c\i); 
\path (origin3)--++(135:1*\i)  coordinate (d\i); 
\path (d\i)--++(-135:0.5) coordinate (d\i); 
\draw[thick,densely dotted] (c\i)--++(135:10);
\draw[thick,densely dotted] (d\i)--++(45:10);
  }

\path(0,0)--++(135:2) coordinate (hhhh);

\fill[opacity=0.2](hhhh)
--++(135:5) coordinate (JJ)
--++(45:1)
--++(-45:1)--++(45:2)
--++(-45:2)--++(45:1)
--++(-45:1)--++(45:2)--++(-45:1);

\fill[opacity=0.4,magenta](hhhh)
  (JJ)
--++(135:1)
 --++(45:2)
--++(-45:2)--++(-135:1)--++(135:1)
;

\path(JJ)--++(45:2)--++(-45:1) coordinate (JJ);

\fill[opacity=0.4,cyan](hhhh)
  (JJ)
--++(135:1)
 --++(45:2)
--++(-45:2)--++(-135:1)--++(135:1)
;

\path(JJ)--++(45:1)--++(-45:2) coordinate (JJ);

\fill[opacity=0.4,orange](hhhh)
  (JJ)
--++(135:1)
 --++(45:2)
--++(-45:2)--++(-135:1)--++(135:1)
;

\path(origin3)--++(45:-0.5)--++(135:7.5) coordinate (X)coordinate (start);

\path (start)--++(135:1)coordinate (start);
 
\path (start)--++(45:2)--++(-45:3) coordinate (X) ; 
 
\path(X)--++(45:1) coordinate (X) ;
 
\path (X)--++(-45:1) coordinate (X) ;
 \fill[violet](X) circle (4pt);
\draw[ thick, violet](X)--++(45:1) coordinate (X) ;
\fill[violet](X) circle (4pt);
\draw[ thick, violet](X)--++(-45:1) coordinate (X) ;
\fill[violet](X) circle (4pt);
 
\draw[ thick, violet](X)--++(45:1) coordinate (X) ;
\fill[violet](X) circle (4pt);
\draw[ thick, violet](X)--++(45:1) coordinate (X) ;
\fill[violet](X) circle (4pt);
\draw[ thick, violet](X)--++(-45:1) coordinate (X) ;
\fill[violet](X) circle (4pt);
\draw[ thick, violet](X)--++(-45:1) coordinate (X) ;
\fill[violet](X) circle (4pt);

%

\end{tikzpicture}$$

    \caption{An example  of 
        $Q, {\color{cyan} P} \in {\rm DRem}_1(\la)$ for $\la=(6^2,5^2,3,1)$ 
  such that 
    $Q-{\color{cyan} P}={\color{magenta} Q^1} \sqcup {\color{orange} Q^2}$. 
        On the left we depict ${\color{darkgreen}{\sf rt}(Q)}= {\sf rt}( {\color{orange} Q^2})$. 
       On the right we depict ${\color{violet}{\sf rt}(  Q^1 )}$ which we note is the smallest Dyck path
       in ${\rm DAdd}(\la-P)$ 
        containing 
       both 
$        {\sf rt}( {\color{orange} Q^2})$ and ${\color{cyan} P}$.  
       }
    \label{fig:sq_PQsplit_noright}
\end{figure}
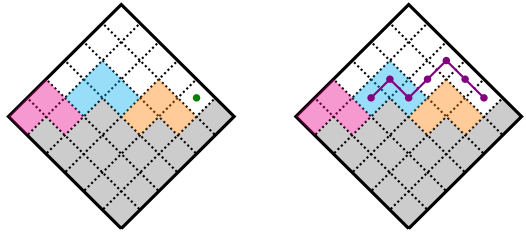

We observe that the previous two lemmas covered the $m=n$ and $m\neq n$ cases separately.  
Both proofs were very similar, but it was the 
$m=n$ case that was more intricate. In what follows, we prove the results for the $m=n$ case and leave adapting these arguments to
 the easier $m\neq n$  case as an exercise for the reader.

 \begin{prop}
 \label{lemmar3}
Let  $Q \in {\rm DRem}_0(\la)$ and $ P \in {\rm DAdd}_{>0}(\la)$. 
If   $P$ and $Q$ commute,   
  we have that 
\begin{align}\label{uuuuuuuu1}
 \mathbb L^\la_\la(-Q  )
  \mathbb D^\la_{\la+P }
  &= 
  \mathbb D^\la_{\la+P }
   \mathbb L^{\la+P }_{\la+P }(-Q  ).
\intertext{
If   $P$ and $Q$ are adjacent,    we have that
}
\label{uuuuuuuu2}
 \mathbb L^\la_\la(-Q  )
  \mathbb D^\la_{\la+P }
  &= 
  \mathbb D^\la_{\la+P }
   \mathbb L^{\la+P }_{\la+P }(-\langle P\cup Q\rangle _{\la+P})    
\end{align} \end{prop}

\begin{proof}
We focus on the $m=n$ case (the $m\neq n$ case is similar, but easier).  
As in the proof of \cref{h0loop_sq}, we must separate out according to the  two distinct cases 
for $ \mathbb L^\la_\la(-Q  )$ in  \cref{L=lincombo}.

\smallskip
\noindent 
\textbf{Case 1. } We first suppose that 
 ${\sf last}(Q)=m-1$   
  where, by  \cref{L=lincombo}, we have that 
 $\mathbb L^\la_\la(-Q  )=\mathbb L^\la_\la.$ 
 We first consider \eqref{uuuuuuuu1}, in which  
 $P$ and $Q$ are commuting Dyck paths.
  In this case  $Q \in {\rm DRem}(\la-P)$. 
  Then
\begin{align*}
    \mathbb L^\la_\la(-Q  ) \mathbb D^{\la}_{\la+P}&=\mathbb L^\la_\la \mathbb D^{\la}_{\la+P}  = \mathbb D^{\la}_{\la+P}\mathbb L^{\la+P}_{\la+P}    = \mathbb D^{\la}_{\la+P}\mathbb L^{\la+P}_{\la+P}(-Q).
\end{align*}
by the loop-commutation relation \eqref{loop-relation} and  \cref{L=lincombo}. 

Now, suppose that $P,Q$ are adjacent. In which case, $\langle P\cup Q\rangle _{\la+P}$ is the rightmost removable Dyck path of $\la+P$ (that is ${\sf last}(\langle P\cup Q\rangle _{\la+P})=m-1$). 
  We hence have that 
\begin{align*}
    \mathbb L^\la_\la(-Q  ) \mathbb D^{\la}_{\la+P}&=\mathbb L^\la_\la \mathbb D^{\la}_{\la+P} = 
    \mathbb D^{\la}_{\la+P}\mathbb L^{\la+P}_{\la+P}  
      = \mathbb D^{\la}_{\la+P}\mathbb L^{\la+P}_{\la+P}(-\langle P\cup Q\rangle _{\la+P})
\end{align*}
by the loop-commutation relation \eqref{loop-relation}.

\smallskip\noindent
\textbf{Case 2. } For the remainder of the proof, we can assume that 
 ${\sf last}(Q)<m-1$ 
  and so  
$$\mathbb L^\la_\la(-Q  )=-\mathbb L^{\la}_{\la}-(-1)^{b({\sf rt}(Q))}
{\mathbbm 1}_\la \mathbb D ( + {\sf rt}(Q) ) \mathbb D ( - {\sf rt}(Q) ) $$ by   \cref{L=lincombo}.
Then
\begin{align}
    \mathbb L^\la_\la(-Q  ) \mathbb D^{\la}_{\la+P}&=
 {\mathbbm 1}_\la   (-\mathbb L^{\la}_{\la}-(-1)^{b({\sf rt}(Q))}
  \mathbb D ( + {\sf rt}(Q) ) \mathbb D ( - {\sf rt}(Q) )) \mathbb D(+P) & \notag\\
    &
    \label{latter term222}
     = -{\mathbbm 1}_\la\mathbb D(+P)\mathbb L^{\la+P}_{\la+P} 
    -(-1)^{b({\sf rt}(Q))}
{\mathbbm 1}_\la \mathbb D ( + {\sf rt}(Q) ) \mathbb D ( - {\sf rt}(Q) )  \mathbb D(+P)
\end{align}
where the second equality follows from the loop-commutation relation  \eqref{loop-relation}. We now need to consider   the commuting and adjacent case separately (and focus  on the latter term on the righthand-side of \eqref{latter term222}). 

\smallskip
\noindent 
{\bf Case 2: Commuting subcase. }
The first case is that in which   $P,Q$ commute. 
 We must refine this further into two subcases: that in which $P$ and ${\sf rt}(Q)$ are a commuting pair of Dyck paths, and that in which $P$ and ${\sf rt}(Q)$ do not commute.

We first suppose that $P$ and $ {\sf rt}(Q)$ commute (as in the leftmost example in  \cref{right_comm}).
In this case the latter term on the righthand-side of \eqref{latter term222}  is as follows
\begin{align}\label{eqq1}
{\mathbbm 1}_\la \mathbb D ( + {\sf rt}(Q) ) \mathbb D ( - {\sf rt}(Q) )  \mathbb D(+P)     &
= 
{\mathbbm 1}_\la  \mathbb D(+P)  
 \mathbb D ( + {\sf rt}(Q) ) \mathbb D ( - {\sf rt}(Q) ) 
\end{align}
by two applications of the commuting relation \eqref{rel3}. 
Substituting  \eqref{eqq1} into  \eqref{latter term222} we obtain 
\begin{align*}
\mathbb L^\la_\la(-Q  )
\mathbb  D^{\la}_{\la+P}&
   = -{\mathbbm 1}_\la\mathbb D(+P)\mathbb L^{\la+P}_{\la+P} 
        -(-1)^{b({\sf rt}(Q))}
{\mathbbm 1}_\la  \mathbb D(+P)\mathbb D ( + {\sf rt}(Q) ) \mathbb D ( - {\sf rt}(Q) )  
\\
&= 
  {\mathbbm 1}_\la\mathbb D(+P)
 (-\mathbb L^{\la+P}_{\la+P} 
        -(-1)^{b({\sf rt}(Q))}
  {\mathbbm 1}_\la\mathbb D ( + {\sf rt}(Q) ) \mathbb D ( - {\sf rt}(Q) )  )
\\
&=  \mathbb D^{\la}_{\la+P} \mathbb L^{\la+P}_{\la+P}(-Q)  
\end{align*}
 as required.

%
%

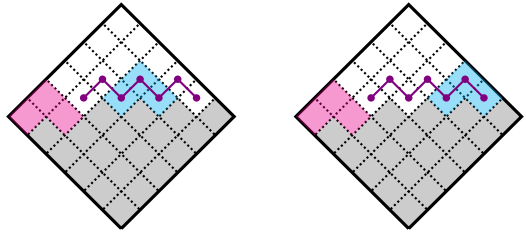
\begin{figure}[ht!]

$$ \begin{tikzpicture}[scale=0.35]
  \path(0,0)--++(135:2) coordinate (hhhh);
 \draw[very thick] (hhhh)--++(135:6)--++(45:6)--++(-45:6)--++(-135:6);
 \clip (hhhh)--++(135:6)--++(45:6)--++(-45:6)--++(-135:6);
\path(0,0) coordinate (origin2);
  \path(0,0)--++(135:2) coordinate (origin3);

      \foreach \i in {0,1,2,3,4,5,6,7,8}
{
\path (origin3)--++(45:0.5*\i) coordinate (c\i); 
\path (origin3)--++(135:0.5*\i)  coordinate (d\i); 
  }

   \foreach \i in {0,1,2,3,4,5,6,7,8}
{
\path (origin3)--++(45:1*\i) coordinate (c\i); 
\path (c\i)--++(-45:0.5) coordinate (c\i); 
\path (origin3)--++(135:1*\i)  coordinate (d\i); 
\path (d\i)--++(-135:0.5) coordinate (d\i); 
\draw[thick,densely dotted] (c\i)--++(135:10);
\draw[thick,densely dotted] (d\i)--++(45:10);
  }

\path(0,0)--++(135:2) coordinate (hhhh);

\fill[opacity=0.2](hhhh)
--++(135:5) coordinate (JJ)
--++(45:1)
--++(-45:1)--++(45:2)
--++(-45:1)--++(45:1)
--++(-45:1)--++(45:1)--++(-45:1)--++(45:1)--++(-45:1);

\fill[opacity=0.4,magenta](hhhh)
  (JJ)
--++(135:1)
 --++(45:2)
--++(-45:2)--++(-135:1)--++(135:1)
;

\path(JJ)--++(45:2)--++(-45:1) coordinate (JJ);
%
%

\path(JJ)--++(45:1)--++(-45:1) coordinate (JJ);

\fill[opacity=0.4,cyan](hhhh)
  (JJ)
--++(135:1)
 --++(45:2)
--++(-45:2)--++(-135:1)--++(135:1)
;

\path(origin3)--++(45:-0.5)--++(135:7.5) coordinate (X)coordinate (start);

\path (start)--++(135:1)coordinate (start);
 
\path (start)--++(45:5)--++(-45:6) coordinate (X) ; 
 
%

\path (start)--++(45:2)--++(-45:3) coordinate (X) ; 
 
\path(X)--++(45:1) coordinate (X) ;
 
\path (X)--++(-45:1) coordinate (X) ;
 \fill[violet](X) circle (4pt);
\draw[ thick, violet](X)--++(45:1) coordinate (X) ;
\fill[violet](X) circle (4pt);
\draw[ thick, violet](X)--++(-45:1) coordinate (X) ;
\fill[violet](X) circle (4pt);
 
\draw[ thick, violet](X)--++(45:1) coordinate (X) ;
\fill[violet](X) circle (4pt);
\draw[ thick, violet](X)--++(-45:1) coordinate (X) ;
\fill[violet](X) circle (4pt);
\draw[ thick, violet](X)--++(45:1) coordinate (X) ;
\fill[violet](X) circle (4pt);
 \draw[ thick, violet](X)--++(-45:1) coordinate (X) ;
\fill[violet](X) circle (4pt);

\end{tikzpicture}
\qquad \begin{tikzpicture}[scale=0.35]
  \path(0,0)--++(135:2) coordinate (hhhh);
 \draw[very thick] (hhhh)--++(135:6)--++(45:6)--++(-45:6)--++(-135:6);
 \clip (hhhh)--++(135:6)--++(45:6)--++(-45:6)--++(-135:6);
\path(0,0) coordinate (origin2);
  \path(0,0)--++(135:2) coordinate (origin3);

      \foreach \i in {0,1,2,3,4,5,6,7,8}
{
\path (origin3)--++(45:0.5*\i) coordinate (c\i); 
\path (origin3)--++(135:0.5*\i)  coordinate (d\i); 
  }

   \foreach \i in {0,1,2,3,4,5,6,7,8}
{
\path (origin3)--++(45:1*\i) coordinate (c\i); 
\path (c\i)--++(-45:0.5) coordinate (c\i); 
\path (origin3)--++(135:1*\i)  coordinate (d\i); 
\path (d\i)--++(-135:0.5) coordinate (d\i); 
\draw[thick,densely dotted] (c\i)--++(135:10);
\draw[thick,densely dotted] (d\i)--++(45:10);
  }

\path(0,0)--++(135:2) coordinate (hhhh);

\fill[opacity=0.2](hhhh)
--++(135:5) coordinate (JJ)
--++(45:1)
--++(-45:1)--++(45:2)
--++(-45:1)--++(45:1)
--++(-45:1)--++(45:1)--++(-45:1)--++(45:1)--++(-45:1);

\fill[opacity=0.4,magenta](hhhh)
  (JJ)
--++(135:1)
 --++(45:2)
--++(-45:2)--++(-135:1)--++(135:1)
;

\path(JJ)--++(45:2)--++(-45:1) coordinate (JJ);
%
%

\path(JJ)--++(45:2)--++(-45:2) coordinate (JJ);

\fill[opacity=0.4,cyan](hhhh)
  (JJ)
--++(135:1)
 --++(45:2)
--++(-45:2)--++(-135:1)--++(135:1)
;

\path(origin3)--++(45:-0.5)--++(135:7.5) coordinate (X)coordinate (start);

\path (start)--++(135:1)coordinate (start);
 
\path (start)--++(45:5)--++(-45:6) coordinate (X) ; 
 
%

\path (start)--++(45:2)--++(-45:3) coordinate (X) ; 
 
\path(X)--++(45:1) coordinate (X) ;
 
\path (X)--++(-45:1) coordinate (X) ;
 \fill[violet](X) circle (4pt);
\draw[ thick, violet](X)--++(45:1) coordinate (X) ;
\fill[violet](X) circle (4pt);
\draw[ thick, violet](X)--++(-45:1) coordinate (X) ;
\fill[violet](X) circle (4pt);
 
\draw[ thick, violet](X)--++(45:1) coordinate (X) ;
\fill[violet](X) circle (4pt);
\draw[ thick, violet](X)--++(-45:1) coordinate (X) ;
\fill[violet](X) circle (4pt);
\draw[ thick, violet](X)--++(45:1) coordinate (X) ;
\fill[violet](X) circle (4pt);
 \draw[ thick, violet](X)--++(-45:1) coordinate (X) ;
\fill[violet](X) circle (4pt);

\end{tikzpicture}
$$
\caption{Examples of commuting Dyck paths ${\color{magenta}Q}\in {\rm DRem}_0(\la)$ and 
${\color{cyan}P}\in {\rm DAdd}_1(\la)$ for $\la=(6^2,4,3,2,1)$. 
In the former case ${\color{violet}{\sf rt}(Q)}\in {\rm DAdd}_1(\la)$  and ${\color{cyan}P}\in {\rm DAdd}_1(\la)$ commute, in the latter case they do not.
 }
\label{right_comm}
\end{figure}

We  continue with our assumption that $P$ and $Q$
  commute, but   now suppose that $P,  {\sf rt}(Q)$ do not commute. 
This implies that ${\sf last}(P)={\sf last}({\sf rt}(Q))$ and such that 
$P\prec {\sf rt}(Q)$ as illustrated in the rightmost diagram in \cref{right_comm}. 
In which case, $P \in {\rm DAdd}_1(\la)$ and   there exists
 $S\in {\rm DRem}_2(\la+{\sf rt}(Q))$  such that
  $${\sf split}_{S}({\sf rt}(Q))=T\sqcup P$$
  for some $T \in {\rm DAdd}_1(\la)$. 
We have that 
\begin{align}\label{eqq2}
       \mathbb D ( + {\sf rt}(Q) ) \mathbb D ( - {\sf rt}(Q) )  \mathbb D(+P) 
    &=
    (-1)^{ b({\sf rt}(Q))-b(T)}
 {\mathbbm 1}_\la \mathbb D ( + {\sf rt}(Q) ) \mathbb D ( - S )  \mathbb D(-T) 
 \notag\\
    &=
      (-1)^{ b({\sf rt}(Q))-b(T)}
 {\mathbbm 1}_\la     \mathbb D(+P)   \mathbb D(+T)   \mathbb D(-T) 
\end{align}
where the first equality follows from applying relation \eqref{adjacent}
to the   pair of adjacent Dyck paths 
$S, T$ 
and the second equality follows by applying relation \eqref{rel4}
to the   pair of non-commuting Dyck paths 
$S\prec {\sf rt}(Q)$. 
Substituting    \eqref{eqq2} into \eqref{latter term222} we obtain the following 
\begin{align*}
\mathbb L^\la_\la(-Q  )   \mathbb D^{\la}_{\la+P}&=  
-\mathbb D^{\la}_{\la+P}\mathbb L^{\la+P}_{\la+P} - 
 (-1)^{ 2b({\sf rt}(Q))-b(T)}
   {\mathbbm 1}_\la    \mathbb D(+P)   \mathbb D(+T)   \mathbb D(-T) 
   \\   
      &=    {\mathbbm 1}_\la    \mathbb D(+P)   (-\mathbb L^{\la+P}_{\la+P} - 
       (-1)^{b(T)}
    \mathbb D(+T)   \mathbb D(-T) )\\
&=  \mathbb D^{\la}_{\la+P} \mathbb L^{\la+P}_{\la+P}(-Q) 
\end{align*}
 where the final equality follows   since $T={\sf rt}(Q)$ in $\la+P$.

\smallskip
\noindent 
{\bf Case 2: Adjacent subcase. }
For the remainder of the proof we suppose that $P,Q$ are a pair of   adjacent Dyck paths (in particular $P \in {\rm DAdd}_1(\la)$). 
There are three subcases to consider:
$(i)$ the Dyck paths $P$ and ${\sf rt}(Q)$ do not commute
$(ii)$ the Dyck paths $P$ and ${\sf rt}(Q)$ do  commute
$(iii)$ the case  $P={\sf rt}(Q)$.  
These   are depicted in \cref{3cases-river}.

Subcase $(i)$. We first assume that  the pair $P$ and $  {\sf rt}(Q)$ do not commute.
Or equivalently,   that ${\sf first}(P)= {\sf last}(Q)+1$ and $b(P)$ is not    maximal with respect to this property  (when $b(P)$  maximal   we are in  case $(iii)$). 
An example of this is   depicted in the leftmost diagram in \cref{3cases-river}. 
In this case, there exists $S\in {\rm DRem}_2(\la+{\sf rt}(Q))$  
 such that $${\sf split}_{S}({\sf rt}(Q))=T\sqcup P $$ 
 for some $T \in {\rm DAdd}_1(\la)$. 
We have that 
\begin{align}\label{eqq3}
  {\mathbbm 1}_\la   \mathbb D ( + {\sf rt}(Q) ) \mathbb D ( - {\sf rt}(Q) )  \mathbb D(+P) 
    &=(-1)^{ b({\sf rt}(Q))-b(T)}  {\mathbbm 1}_\la
    \mathbb D(+ {\sf rt}(Q)) \mathbb D (-S) \mathbb D(-T) 
     \notag\\ 
    &=
    (-1)^{ b({\sf rt}(Q))-b(T)}  {\mathbbm 1}_\la
    \mathbb D(+P) \mathbb D(+T) \mathbb D(-T) 
\end{align}
where the first equality follows by applying  relation  \eqref{adjacent} 
to the adjacent pair $S,T$ 
 and the second follows by  relation \eqref{rel4} to the non-commuting pair $S\prec {\sf rt}(Q)$.
Substituting \eqref{eqq3} into  \eqref{latter term222} we obtain the following 
\begin{align*}
\mathbb L^\la_\la(-Q  ) 
\mathbb 
D^{\la}_{\la+P}
&
=
  - {\mathbbm 1}_\la \mathbb D(+P)  \mathbb L^{\la+P}_{\la+P} -  (-1)^{2  b({\sf rt}(Q))-b(T)  }	{\mathbbm 1}_\la \mathbb D (+ P)  \mathbb D (+T)  \mathbb D (-T) 
  \\
&=  \mathbb D^{\la}_{\la+P}(-\mathbb L^{\la+P}_{\la+P} -  (-1)^{b(T)} \mathbb D (+T)  \mathbb D (-T)  )							\\
&=  \mathbb D^{\la}_{\la+P} \mathbb L^{\la+P}_{\la+P}(-\langle P\sqcup Q\rangle_{\la+P})  
\end{align*}
where the final equality follows since  $T={\sf rt}(\langle P\sqcup Q\rangle_{\la+P})$.

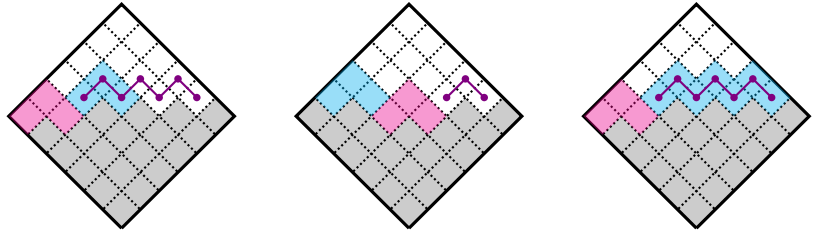
\begin{figure}[ht!]

$$
 \begin{tikzpicture}[scale=0.35]
  \path(0,0)--++(135:2) coordinate (hhhh);
 \draw[very thick] (hhhh)--++(135:6)--++(45:6)--++(-45:6)--++(-135:6);
 \clip (hhhh)--++(135:6)--++(45:6)--++(-45:6)--++(-135:6);
\path(0,0) coordinate (origin2);
  \path(0,0)--++(135:2) coordinate (origin3);

      \foreach \i in {0,1,2,3,4,5,6,7,8}
{
\path (origin3)--++(45:0.5*\i) coordinate (c\i); 
\path (origin3)--++(135:0.5*\i)  coordinate (d\i); 
  }

   \foreach \i in {0,1,2,3,4,5,6,7,8}
{
\path (origin3)--++(45:1*\i) coordinate (c\i); 
\path (c\i)--++(-45:0.5) coordinate (c\i); 
\path (origin3)--++(135:1*\i)  coordinate (d\i); 
\path (d\i)--++(-135:0.5) coordinate (d\i); 
\draw[thick,densely dotted] (c\i)--++(135:10);
\draw[thick,densely dotted] (d\i)--++(45:10);
  }

\path(0,0)--++(135:2) coordinate (hhhh);

\fill[opacity=0.2](hhhh)
--++(135:5)  coordinate (JJ)
--++(45:1)
--++(-45:1)--++(45:2)
--++(-45:1) --++(45:1)
--++(-45:1)--++(45:1)--++(-45:1)--++(45:1)--++(-45:1);

\fill[opacity=0.4,magenta](hhhh)
  (JJ)
--++(135:1)
 --++(45:2)
--++(-45:2)--++(-135:1)--++(135:1)
;


\path(JJ)--++(45:2)--++(-45:1) coordinate (JJ);

\fill[opacity=0.4,cyan](hhhh)
  (JJ)
--++(135:1)
 --++(45:2)
--++(-45:2)--++(-135:1)--++(135:1)
;

\path(origin3)--++(45:-0.5)--++(135:7.5) coordinate (X)coordinate (start);

\path (start)--++(135:1)coordinate (start);
 
\path (start)--++(45:5)--++(-45:6) coordinate (X) ;

\path (start)--++(45:2)--++(-45:3) coordinate (X) ; 
 
\path(X)--++(45:1) coordinate (X) ;
 
\path (X)--++(-45:1) coordinate (X) ;
 \fill[violet](X) circle (4pt);
\draw[ thick, violet](X)--++(45:1) coordinate (X) ;
\fill[violet](X) circle (4pt);
\draw[ thick, violet](X)--++(-45:1) coordinate (X) ;
\fill[violet](X) circle (4pt);
 
\draw[ thick, violet](X)--++(45:1) coordinate (X) ;
\fill[violet](X) circle (4pt);
\draw[ thick, violet](X)--++(-45:1) coordinate (X) ;
\fill[violet](X) circle (4pt);
\draw[ thick, violet](X)--++(45:1) coordinate (X) ;
\fill[violet](X) circle (4pt);
 \draw[ thick, violet](X)--++(-45:1) coordinate (X) ;
\fill[violet](X) circle (4pt);

\end{tikzpicture}
\qquad
 \begin{tikzpicture}[scale=0.35]
  \path(0,0)--++(135:2) coordinate (hhhh);
 \draw[very thick] (hhhh)--++(135:6)--++(45:6)--++(-45:6)--++(-135:6);
 \clip (hhhh)--++(135:6)--++(45:6)--++(-45:6)--++(-135:6);
\path(0,0) coordinate (origin2);
  \path(0,0)--++(135:2) coordinate (origin3);

      \foreach \i in {0,1,2,3,4,5,6,7,8}
{
\path (origin3)--++(45:0.5*\i) coordinate (c\i); 
\path (origin3)--++(135:0.5*\i)  coordinate (d\i); 
  }

   \foreach \i in {0,1,2,3,4,5,6,7,8}
{
\path (origin3)--++(45:1*\i) coordinate (c\i); 
\path (c\i)--++(-45:0.5) coordinate (c\i); 
\path (origin3)--++(135:1*\i)  coordinate (d\i); 
\path (d\i)--++(-135:0.5) coordinate (d\i); 
\draw[thick,densely dotted] (c\i)--++(135:10);
\draw[thick,densely dotted] (d\i)--++(45:10);
  }

\path(0,0)--++(135:2) coordinate (hhhh);

\fill[opacity=0.2](hhhh)
--++(135:6)  
--++(45:1)
--++(-45:1)--++(45:1)
--++(-45:2)coordinate (JJ)--++(45:1)
--++(-45:1)--++(45:2)--++(-45:1)--++(45:1)--++(-45:1);

\fill[opacity=0.4,magenta](hhhh)
  (JJ)
--++(135:1)
 --++(45:2)
--++(-45:2)--++(-135:1)--++(135:1)
;

\path(JJ)--++(135:3)--++(-135:2) coordinate (JJ);

\path(JJ)--++(45:1)--++(-45:1) coordinate (JJ);

\fill[opacity=0.4,cyan](hhhh)
  (JJ)
--++(135:1)
 --++(45:2)
--++(-45:2)--++(-135:1)--++(135:1)
;

\path(origin3)--++(45:-0.5)--++(135:7.5) coordinate (X)coordinate (start);

\path (start)--++(135:1)coordinate (start);
 
\path (start)--++(45:5)--++(-45:6) coordinate (X) ; 
  
\path (start)--++(45:4)--++(-45:5) coordinate (X) ; 
 
\path(X)--++(45:1) coordinate (X) ;
 
\path (X)--++(-45:1) coordinate (X) ;
 \fill[violet](X) circle (4pt);
\draw[ thick, violet](X)--++(45:1) coordinate (X) ;
\fill[violet](X) circle (4pt);
\draw[ thick, violet](X)--++(-45:1) coordinate (X) ;
\fill[violet](X) circle (4pt);

\end{tikzpicture}
\qquad
 \begin{tikzpicture}[scale=0.35]
  \path(0,0)--++(135:2) coordinate (hhhh);
 \draw[very thick] (hhhh)--++(135:6)--++(45:6)--++(-45:6)--++(-135:6);
 \clip (hhhh)--++(135:6)--++(45:6)--++(-45:6)--++(-135:6);
\path(0,0) coordinate (origin2);
  \path(0,0)--++(135:2) coordinate (origin3);

      \foreach \i in {0,1,2,3,4,5,6,7,8}
{
\path (origin3)--++(45:0.5*\i) coordinate (c\i); 
\path (origin3)--++(135:0.5*\i)  coordinate (d\i); 
  }

   \foreach \i in {0,1,2,3,4,5,6,7,8}
{
\path (origin3)--++(45:1*\i) coordinate (c\i); 
\path (c\i)--++(-45:0.5) coordinate (c\i); 
\path (origin3)--++(135:1*\i)  coordinate (d\i); 
\path (d\i)--++(-135:0.5) coordinate (d\i); 
\draw[thick,densely dotted] (c\i)--++(135:10);
\draw[thick,densely dotted] (d\i)--++(45:10);
  }

\path(0,0)--++(135:2) coordinate (hhhh);

\fill[opacity=0.2](hhhh)
--++(135:5)  coordinate (JJ)
--++(45:1)
--++(-45:1)--++(45:2)
--++(-45:1) --++(45:1)
--++(-45:1)--++(45:1)--++(-45:1)--++(45:1)--++(-45:1);

\fill[opacity=0.4,magenta](hhhh)
  (JJ)
--++(135:1)
 --++(45:2)
--++(-45:2)--++(-135:1)--++(135:1)
;


\path(JJ)--++(45:2)--++(-45:1) coordinate (JJ);

\fill[opacity=0.4,cyan](hhhh)
  (JJ)
--++(135:1)
 --++(45:2)
--++(-45:1) --++(45:1)--++(-45:1) --++(45:1)
--++(-45:2)--++(-135:1)--++(135:1)--++(-135:1)--++(135:1)--++(-135:1)--++(135:1)
;

\path(origin3)--++(45:-0.5)--++(135:7.5) coordinate (X)coordinate (start);

\path (start)--++(135:1)coordinate (start);
 
\path (start)--++(45:5)--++(-45:6) coordinate (X) ;

\path (start)--++(45:2)--++(-45:3) coordinate (X) ; 
 
\path(X)--++(45:1) coordinate (X) ;
 
\path (X)--++(-45:1) coordinate (X) ;
 \fill[violet](X) circle (4pt);
\draw[ thick, violet](X)--++(45:1) coordinate (X) ;
\fill[violet](X) circle (4pt);
\draw[ thick, violet](X)--++(-45:1) coordinate (X) ;
\fill[violet](X) circle (4pt);
 
\draw[ thick, violet](X)--++(45:1) coordinate (X) ;
\fill[violet](X) circle (4pt);
\draw[ thick, violet](X)--++(-45:1) coordinate (X) ;
\fill[violet](X) circle (4pt);
\draw[ thick, violet](X)--++(45:1) coordinate (X) ;
\fill[violet](X) circle (4pt);
 \draw[ thick, violet](X)--++(-45:1) coordinate (X) ;
\fill[violet](X) circle (4pt);

\end{tikzpicture}
$$
\caption{Examples of adjacent paths  ${\color{cyan}P}\in {\rm DAdd}_1(\la)$,  ${  \color{magenta}Q}
\in {\rm DRem}_0(\la)$
     and ${ \color{violet} {\sf rt}({ Q})}\in {\rm DAdd}_1(\la)$.
     The three subcases depicted from left-to-right are 
     $(i)$ 
 ${\color{cyan}P}\prec   { \color{violet} {\sf rt}({ Q})}$ are non-commuting 
     $(ii)$ 
 ${\color{cyan}P}$ and $  { \color{violet} {\sf rt}({ Q})}$ are  commuting 
 and 
 $(iii)$
 ${\color{cyan}P} = { \color{violet} {\sf rt}({ Q})}$.
}
\label{3cases-river}
\end{figure}

 Subcase $(ii)$. 
 We now assume  that  $P$ and  $ {\sf rt}(Q)$ commute. Or equivalently,   that $P$ is to the left of $Q$ as in  the central diagram in \cref{3cases-river}. 
We have that 
\begin{align}\label{eqq4}
 {\mathbbm 1}_\la   \mathbb D ( + {\sf rt}(Q) ) \mathbb D ( - {\sf rt}(Q) )  \mathbb D(+P) 
    &={\mathbbm 1}_\la   \mathbb D(+P)  \mathbb D ( + {\sf rt}(Q) )    \mathbb D ( - {\sf rt}(Q) )
\end{align}
by two applications of the commuting relation \eqref{rel3}. 
Substituting  \eqref{eqq4} into  \eqref{latter term222} we obtain the following 
\begin{align*}
\mathbb L^\la_\la(-Q  ) \mathbb D^{\la}_{\la+P}&=  
-
  {\mathbbm 1}_\la   \mathbb D(+P) 
\mathbb L^{\la+P}_{\la+P} 
-
  (-1)^{b({\sf rt}(Q))}		
  {\mathbbm 1}_\la   \mathbb D(+P)  \mathbb D ( +{\sf rt}(Q) )    \mathbb D ( -{\sf rt}(Q) ) \\
&=  \mathbb D^{\la}_{\la+P}(-\mathbb L^{\la+P}_{\la+P} -  (-1)^{b({\sf rt}(Q))}\mathbb D^{\la}_{\la + P} \mathbb D_{\la+ {\sf rt}(Q)+P}^{\la + P} \mathbb D^{\la+ {\sf rt}(Q)+P}_{\la+P})\\
&=  \mathbb D^{\la}_{\la+P} \mathbb L^{\la+P}_{\la+P}(-\langle P\sqcup Q\rangle_{\la+P}).
\end{align*}
where the final equality follows as ${\sf rt}(Q)={\sf rt}(\langle P\sqcup Q\rangle_{\la+P})$. 

Subcase $(iii)$. 
We now assume that  $P={\sf rt}(Q)$, as in the rightmost diagram in  \cref{3cases-river}.
We have that 
\begin{align}\label{eqq5}
  {\mathbbm 1}_\la   \mathbb D ( + {\sf rt}(Q) ) \mathbb D ( - {\sf rt}(Q) )  \mathbb D(+P) 
    &
    =
 2(-1)^{b(P)+1}
  {\mathbbm 1}_\la \mathbb D( + P) \mathbb L_{\la+ P}^{\la+ P} 
\end{align}
by cubic relation \eqref{cubic}. 
Substituting   \eqref{eqq5} into  \eqref{latter term222} we obtain the following 
\begin{align*}
\mathbb L^\la_\la(-Q  ) \mathbb D^{\la}_{\la+P}&= 
 -{\mathbbm 1}_\la\mathbb D(+P)\mathbb L^{\la+P}_{\la+P}
 - 2 (-1)^{2b({\sf rt}(Q))+1}  {\mathbbm 1}_\la 
     \mathbb D( + P) \mathbb L_{\la+ P}^{\la+ P}
   \\ 
&= (-1+2) {\mathbbm 1}_\la \mathbb D (+P)\mathbb L^{\la+P}_{\la+P}\\
&=  \mathbb D^{\la}_{\la+P}\mathbb L^{\la+P}_{\la+P}(-\langle P\sqcup Q\rangle_{\la+P}) 
\end{align*}
where the first two   equalities follow  as $b({\sf rt}(Q))=b(P)$ and so the signs cancel;
 the third equality follows as $\langle P\sqcup Q\rangle_{\la+P}	\in {\rm DRem}_0(\la)$  is the rightmost such removable Dyck path.
\end{proof}

\begin{prop} \label{looploop}
 For $P, Q \in {\rm DRem}(\la)$ we have that 
 \begin{equation}
    \mathbb L^\la_\la(-P)
  \mathbb L^\la_\la(-Q)
  = \mathbb L^\la_\la(-Q) \mathbb L^\la_\la(-P) 
 \end{equation}
 and moreover, if $P \in {\rm DRem}_0(\la)$ then we have that 
$$(  \mathbb L^\la_\la(-P))^2=0.$$
 \end{prop}

\begin{proof}
As in previous proofs, we assume that  $m=n$ (the   $m\neq n$ case is similar, but easier). 
We first suppose that   ${\sf last}(Q)=m-1$. In which case, we have that 
$ \mathbb L^\la_\la  (-Q)=  \mathbb L^\la_\la  $ by \cref{L=lincombo}. 
For any $Q\neq P\in {\rm DRem}(\la)$ one can check that 
$$\mathbb L^\la_\la  (-P)\mathbb L^\la_\la =\mathbb L^\la_\la \mathbb L^\la_\la  (-P)$$
by expanding out $\mathbb L^\la_\la  (-P)$ as prescribed by \cref{L=lincombo} and then noticing that $\mathbb L^\la_\la$ commutes past every single   term  in the expansion, by  relation \eqref{loop-relation}.
On the other hand if $P=Q$ with ${\sf last}(Q)=m-1$, then $ \mathbb L^\la_\la(-Q)^2=0$  by relation \eqref{loop-relation}.  
Therefore we can assume for the remainder of the proof that ${\sf last}(P)\leq {\sf last}(Q)<m$. 

 \smallskip
\noindent     \textbf{Case 1.} For $P, Q \in {\rm DRem}_0(\la)$, we claim  that 
\begin{equation}\label{looploop0}
    \mathbb L^\la_\la(-P)
  \mathbb L^\la_\la(-Q)
  = (1-\delta_{PQ})\mathbb L^\la_\la(-Q) \mathbb L^\la_\la(-P).
\end{equation}
We first consider the case  that  $P=Q\in {\rm DRem}_0(\la) $.  
 We  have that 
\begin{align*}
     &\mathbb L^\la_\la(-P)^2		
     \\=& 
    (- \mathbb L^\la_\la  -(-1)^{b({\sf rt}(P))}	\mathbb D(+{\sf rt}(P))\mathbb D(-{\sf rt}(P)))^2\\
     =& (\mathbb L^\la_\la) ^2 + 2(-1)^{b({\sf rt}(P))} 
      \mathbb L^\la_\la
      \mathbb D(+{\sf rt}(P))\mathbb D(-{\sf rt}(P))+ \mathbb D(+{\sf rt}(P))\mathbb D(-{\sf rt}(P)) \mathbb D(+{\sf rt}(P))\mathbb D(-{\sf rt}(P))
      \\
     =&   2(-1)^{b({\sf rt}(P))} 
      \mathbb L^\la_\la
      \mathbb D(+{\sf rt}(P))\mathbb D(-{\sf rt}(P))+ \mathbb D(+{\sf rt}(P))\mathbb D(-{\sf rt}(P)) \mathbb D(+{\sf rt}(P))\mathbb D(-{\sf rt}(P)).
   \\
     =&   2(-1)^{b({\sf rt}(P))} 
      \mathbb L^\la_\la
      \mathbb D(+{\sf rt}(P))\mathbb D(-{\sf rt}(P)) 
  + 2(-1)^{b({\sf rt}(P))+1} 
      \mathbb D(+{\sf rt}(P))
            \mathbb L^{\la+{\sf rt}(P)}_{\la+{\sf rt}(P)} \mathbb D(-{\sf rt}(P))
      \\    =&   2(-1)^{b({\sf rt}(P))} 
      \mathbb L^\la_\la
      \mathbb D(+{\sf rt}(P))\mathbb D(-{\sf rt}(P)) 
  + 2(-1)^{b({\sf rt}(P))+1} 
            \mathbb L^{\la  }_{\la  }   \mathbb D(+{\sf rt}(P))\mathbb D(-{\sf rt}(P))
      \\
        =&0
\end{align*}
where the first equality follows from \cref{L=lincombo};
 the second is expanding out the brackets;
the third and fifth equalities follow from 
    \eqref{loop-relation}  and rearranging terms;
  the fourth follows from the cubic relation \eqref{cubic}.

 We can now assume that  $P\neq Q$ and so   ${\sf last}(P)\leq{\sf first}(Q)-2$.
 We can expand out the product $\mathbb L^\la_\la(-P)  \mathbb L^\la_\la(-Q)$ as follows 
\begin{align*} 
     (- \mathbb L^\la_\la  -(-1)^{b({\sf rt}(P))}	  \mathbb D^\la_{\la+{\sf rt}(P) }\mathbb D_\la^{\la+{\sf rt}(P) } 		)(- \mathbb L^\la_\la -(-1)^{b({\sf rt}(Q))}	  \mathbb D^\la_{\la+{\sf rt}(Q) }\mathbb D_\la^{\la+{\sf rt}(Q) } 		)
     \end{align*}
 and one can expand out $   \mathbb L^\la_\la(-Q)\mathbb L^\la_\la(-P)$ similarly; 
all of the products involving  $  \mathbb L^\la_\la$ can be easily rearranged using  relation \eqref{loop-relation}; thus,  in order to deduce \cref{looploop0} it will suffice  to verify the following claim:
%
%
$$(\mathbb D^\la_{\la+{\sf rt}(P) }\mathbb D_\la^{\la+{\sf rt}(P) } 	)(  \mathbb D^\la_{\la+{\sf rt}(Q) }\mathbb D_\la^{\la+{\sf rt}(Q) } 	)=
 (  \mathbb D^\la_{\la+{\sf rt}(Q) }\mathbb D_\la^{\la+{\sf rt}(Q) } 	)(\mathbb D^\la_{\la+{\sf rt}(P) }\mathbb D_\la^{\la+{\sf rt}(P) } 	). $$
We now set about proving this claim. 
Recall that  $P,Q \in {\rm DRem}_0(\la)$ and  $P$ is to the left of $Q$; therefore there
 exists a Dyck path $T \in {\rm DRem}_1(\la+{\sf rt}(P))$ 
   such that $${\sf split}_T({\sf rt}(P))=S \sqcup {\sf rt}(Q)$$
 for some Dyck path $S \in {\rm DRem}(\la+{\sf rt}(P)-T)$ which is 
  adjacent to $T$.   In which case $\langle T\cup S\rangle_{\la +{\sf rt}(P)}={\sf rt}(P)$.
  This is pictured in \cref{alksghsalgdfghjdfghjd}.
 We have that 
\begin{align*}
&
 \mathbb D(+{\sf rt}(P)  )\mathbb D(-{\sf rt}(P))
  \mathbb D(+{\sf rt}(Q) ) \mathbb D(-{\sf rt}(Q))  
 \\
 =&
(-1)^{  b({\sf rt}(P))+b(S)}
 \mathbb D(+{\sf rt}(P)  )\mathbb D(-T) 
  \mathbb D(-S ) \mathbb D(-{\sf rt}(Q))  
\\
 =&
(-1)^{  b({\sf rt}(P))+b(S)}
\Big ( \mathbb D(+{\sf rt}(P)  )\mathbb D(-T)\Big)
\Big(  \mathbb D(-S ) \mathbb D(-{\sf rt}(Q))  \Big)
\\ =&
(-1)^{  b({\sf rt}(P))+b(S)}
\Big(   \mathbb D(+{\sf rt}(Q) \mathbb D(+S ))  \Big)
\Big (  \mathbb D(+T)\mathbb D(-{\sf rt}(P)  )\Big)
\\
=& 
  \mathbb D(+{\sf rt}(Q) ) \mathbb D(-{\sf rt}(Q))  
  \mathbb D(+{\sf rt}(P)  )\mathbb D(-{\sf rt}(P))
 \end{align*}
 where the  first and final equations  follow  from
  applying relation \eqref{adjacent} to the   pair of adjacent Dyck paths 
 $S, T $; the second equation is merely re-bracketing; 
  the third equation follows by applying 
  relation  \eqref{rel4}  for 
  the non-commuting pair 
  $T\prec {\sf right}(P)$ 
for both the left and right  bracketed term.

\begin{figure}[ht!]

$$
 \begin{tikzpicture}[scale=0.35]
  \path(0,0)--++(135:2) coordinate (hhhh);
 \draw[very thick] (hhhh)--++(135:7)--++(45:7)--++(-45:7)--++(-135:7);
 \clip (hhhh)--++(135:7)--++(45:7)--++(-45:7)--++(-135:7);
\path(0,0) coordinate (origin2);
  \path(0,0)--++(135:2) coordinate (origin3);

      \foreach \i in {0,1,2,3,4,5,6,7,8}
{
\path (origin3)--++(45:0.5*\i) coordinate (c\i); 
\path (origin3)--++(135:0.5*\i)  coordinate (d\i); 
  }

   \foreach \i in {0,1,2,3,4,5,6,7,8}
{
\path (origin3)--++(45:1*\i) coordinate (c\i); 
\path (c\i)--++(-45:0.5) coordinate (c\i); 
\path (origin3)--++(135:1*\i)  coordinate (d\i); 
\path (d\i)--++(-135:0.5) coordinate (d\i); 
\draw[thick,densely dotted] (c\i)--++(135:10);
\draw[thick,densely dotted] (d\i)--++(45:10);
  }

\path(0,0)--++(135:2) coordinate (hhhh);

\fill[opacity=0.2](hhhh)
--++(135:6)  coordinate (JJ)
--++(45:1)
--++(-45:1)--++(45:2)
--++(-45:2)  coordinate (JJ1) --++(45:1)
--++(-45:1)--++(45:2)--++(-45:1)--++(45:1)--++(-45:1);

\fill[opacity=0.4,magenta](hhhh)
  (JJ)
--++(135:1)
 --++(45:2)
--++(-45:2)--++(-135:1)--++(135:1)
;



\fill[opacity=0.4,cyan](hhhh)
  (JJ1)
--++(135:1)
 --++(45:2)
--++(-45:2)--++(-135:1)--++(135:1)
;

\path(origin3)--++(45:-0.5)--++(135:7.5) coordinate (X)coordinate (start);

\path (start)--++(135:1)coordinate (start);
 
\path (start)--++(45:5)--++(-45:6) coordinate (X) ;

\path (start)--++(45:2)--++(-45:3) coordinate (X) ; 
 
\path(X)--++(45:1)--++(135:1) coordinate (X) ;
 
\path (X)--++(-45:1) coordinate (X) ;
 \fill[violet](X) circle (4pt);
\draw[ thick, violet](X)--++(45:1) coordinate (X) ;
\fill[violet](X) circle (4pt);
\draw[ thick, violet](X)--++(-45:1) coordinate (X) ;
\fill[violet](X) circle (4pt);
\path(X)--++(45:1) coordinate (X) ;
\fill[orange](X) circle (4pt);
\draw[ thick, orange](X)--++( 45:1) coordinate (X) ;
\fill[orange](X) circle (4pt);
\draw[ thick, orange](X)--++(-45:1) coordinate (X) ;
\fill[orange](X) circle (4pt);
 \path(X)--++(-45:1) coordinate (X) ;
\fill[darkgreen](X) circle (4pt);
\draw[ thick, darkgreen](X)--++( 45:1) coordinate (X) ;
\fill[darkgreen](X) circle (4pt);
   \draw[ thick, darkgreen](X)--++(- 45:1) coordinate (X) ;
\fill[darkgreen](X) circle (4pt);

\end{tikzpicture}$$
\caption{We picture ${\color{magenta}P}, {\color{cyan}Q} \in {\rm DRem}(\la)$ for $\la=(7^2,5,4^2,2,1)$. 
We have  also pictured ${\color{darkgreen}{\sf rt}(Q)}, 
{\color{violet}S} \in {\rm DAdd}_1(\la)$ and ${\color{orange}T}\in {\rm DRem}
(\la+{\sf rt}(P))$.  We note that ${\sf rt}(P)= 
{\color{darkgreen}{\sf rt}(Q)}\sqcup
{\color{violet}S}\sqcup {\color{orange}T}$.
}
\label{alksghsalgdfghjdfghjd}
\end{figure}
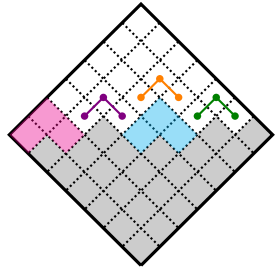

\smallskip
\noindent     \textbf{Case 2.} For $P, Q \in {\rm DRem}_{>0}(\la) $, we claim  that 
\begin{equation*}
   \mathbb L^\la_\la(-P)
  \mathbb L^\la_\la(-Q)
  :=
(   \mathbb D^\la_{\la-P}   \mathbb D_\la^{\la-P})
(   \mathbb D^\la_{\la-Q}   \mathbb D_\la^{\la-Q})
=
(   \mathbb D^\la_{\la-Q}   \mathbb D_\la^{\la-Q})(   \mathbb D^\la_{\la-P}   \mathbb D_\la^{\la-P})  
  =:
   \mathbb L^\la_\la(-Q) \mathbb L^\la_\la(-P). 
\end{equation*}
 If $P$ and $ Q$ commute, then the claim follows by applying the commuting relations \eqref{rel3} and if $P=Q$ then the statement is obvious. Then without loss of generality, we can assume that ${\sf ht}(P)={\sf ht}(Q)+1$ and $Q \setminus P =Q^1 \sqcup Q^2$ where $Q^1, Q^2 \in {\rm DRem}(\la-P)$. In which case we have that 
\begin{align*}
    \mathbb L^\la_\la(-P)\mathbb L^\la_\la(-Q)
     &
     =
 {\mathbbm 1}_\la \mathbb D (-P)
 \mathbb D (+P)
  \mathbb D (-Q)
 \mathbb D (+Q)
\\     &
     =
  {\mathbbm 1}_\la \mathbb D (-P)
 \mathbb D (-Q^1)
  \mathbb D (-Q^2)
 \mathbb D (+Q)
\\       &
       =
  {\mathbbm 1}_\la \Big(\mathbb D (-P)
 \mathbb D (-Q^1)\Big)\Big(
  \mathbb D (-Q^2)
 \mathbb D (+Q)\Big)
\\       &
       =
  {\mathbbm 1}_\la \Big(\mathbb D (-Q)
 \mathbb D (+Q^2)\Big)\Big(
  \mathbb D (+Q^1)
 \mathbb D (+P)\Big)
  \\       &
       =
  {\mathbbm 1}_\la  \mathbb D (-Q)
   \mathbb D (+Q ) 
  \mathbb D (-P)
 \mathbb D (+P) 
 \\
 &=    \mathbb L^\la_\la(-P)\mathbb L^\la_\la(-Q)
 \end{align*}
where the first and final equalities are by \cref{L=lincombo};
the second and penultimate inequalities follow  by applying relation \eqref{rel4}
to the non-commuting pair $P\prec Q$;
the third equality is merely suggestive rebracketing;
the fourth equality follows by applying relation \eqref{adjacent} for the adjacent Dyck paths 
$Q^1$ and $P$ to {\em both} bracketed pairs.

%
%
%
\smallskip
\noindent     \textbf{Case 3.} For $P \in {\rm DRem}(\la)_0$, $Q \in {\rm DRem}(\la)_{>0}$, we claim  that 
\begin{equation*}
   \mathbb L^\la_\la(-P)
  \mathbb L^\la_\la(-Q)
  :=
  \mathbb L^\la_\la(-P)( \mathbb D^\la_{\la-Q}   \mathbb D_\la^{\la-Q})
  =
  (   \mathbb D^\la_{\la-Q}   \mathbb D_\la^{\la-Q})  \mathbb  L^\la_\la(-P)
  =:
   \mathbb L^\la_\la(-Q) \mathbb L^\la_\la(-P). 
\end{equation*}If $P$ and $Q$ commute, then the claim  follows from  \cref{h0loop_sq,lemmar3}.
So we can assume that $P,Q$ are adjacent. In particular, we have that $P \setminus  Q=P^1 \sqcup P^2$ where $P^1, P^2 \in {\rm DRem}(\la-Q)$. Under this assumption we have that 
\begin{align*}
    \mathbb L^\la_\la(-P)\mathbb L^\la_\la(-Q) 
    &
    =
     (-1)^{b(Q)} 
     {\mathbbm 1}_\la
     \mathbb L (-P) 
     \mathbb D (-Q)     \mathbb D (+Q)\\
    &= (-1)^{b(Q)}
     \mathbb D (-Q)
    \mathbb L (-P^2)    \mathbb D (+Q)    
    \\
    &=(-1)^{b(Q)}
    \mathbb D (-Q) \mathbb D (+Q) 
    \mathbb L (-\langle P^2 \sqcup Q\rangle_\la)
     \\  
    &=\mathbb L^\la_\la(-Q)\mathbb L^\la_\la(-P).
\end{align*}
where the second equality follows from   \cref{h0loop_sq}
and the third equality follows from  \cref{lemmar3}.
\end{proof}

Notice that every product we have considered thus far has been a product of 
a pair of element of degree at most 2 (each element corresponding to adding/removing a Dyck path of degree 1, or a loop of degree 2). 
We are now ready to consider more complicated products --- in particular, products of arbitrarily high degree.  In particular we will consider loops of arbitrarily high degree as follows. 
 Given $\la\setminus \alpha= (\la \setminus \alpha)_{\leq0} = \sqcup_{1\leq i \leq k} P^i$ a Dyck pair, we define 
 \begin{equation}\label{def: Lalpha}
     \mathbb L^\la_\la(\alpha)= \prod  _{1\leq i \leq k} \mathbb L^\la_\la(-P^i)
 \end{equation}
 which we observe is independent of the ordering on the Dyck paths in the tiling, by \cref{looploop}.
 With this notation in place, we can now look at the effect of multiplying loop tilings together.

 \begin{prop} \label{looploop-big}
Let 
 $(\la\setminus \alpha)= (\la \setminus \alpha)_{\leq0} = \sqcup_{i\in I} P^i$, and $(\la\setminus \beta)= (\la \setminus \beta)_{\leq0} = \sqcup_{j\in J} Q^j $.
We have that 
 $$
 \mathbb L^\la_\la(\alpha)
  \mathbb L^\la_\la(\beta)
  =
  \begin{cases}
   \mathbb L^\la_\la(\alpha\cap \beta)		&\text{if } P^i\neq Q^j \text { for } i \in I,j\in J	\\
   0		&\text{otherwise.}
   \end{cases}
 $$
  \end{prop}

\begin{proof}
 We first consider the easy case in which $ (\la \setminus \alpha)\cap (\la \setminus \beta)=\emptyset$.    In this case   $\la \setminus (\alpha\cap\beta)$ is a Dyck pair with tiling
$$\la \setminus(\alpha\cap\beta)= \bigsqcup_{i\in I} P^i \sqcup \bigsqcup_{j\in J} Q^j = (\la \setminus\alpha) \sqcup (\la \setminus\beta).$$
In the notation of \eqref{def: Lalpha}
 we have that 
 $ \mathbb L^\la_\la(\alpha)  \mathbb L^\la_\la(\beta)=:\mathbb L^\la_\la(\alpha\cap \beta)\neq 0 $
(and we remind the reader that this product
  is entirely independent of any ordering on Dyck paths, by  \cref{looploop}).

Now suppose  that there exists $T=P^{i_0 } \in {\rm DRem}_h(\la)$ and $T=Q^{j_0} \in {\rm DRem}_h (\la)$ for $i_0\in I , j_0\in J$ and we can assume that $h$ is minimal with respect to this property. 
We claim that $h=0$.  
Suppose not, 
then our assumption that 
 $(\la\setminus \alpha)= (\la \setminus \alpha)_{\leq0} $ implies that $T=P^{i_0} \prec P^{i' } $ for some $i' \in I$ such that $P^{i' } \in {\rm DRem}_0(\la)$. Arguing similarly for $T=Q^{j_0}$ we deduce the claim. 
 Therefore we can rewrite the product  as follows
\begin{equation}\label{reordering}
  \mathbb L^\la_\la(\alpha)  \mathbb L^\la_\la(\beta)= \mathbb L^\la_\la(-T)  \mathbb L^\la_\la(-T) \prod _{i \in I \setminus \{i_0\}} \mathbb L^\la_\la(-P^i)\prod _{j \in J\setminus\{j_0\}} \mathbb L^\la_\la(-Q^j)  =0 
\end{equation}
  where the latter equality follows by  \eqref{looploop} and the fact that $T \in {\rm  DRem}_0(\la)$.
 \end{proof}

 \begin{prop}
 \label{loopremove-big1}
Let 
 $(\la\setminus \alpha)= (\la \setminus \alpha)_{\leq0} = \sqcup_{i=0}^{k} T^i$ such that $T^k \prec T^{k-1} \prec \dots \prec T^1\prec T^0$ with ${\sf ht}^{\la}_{\alpha}(T^i)=-i \in \ZZ_{\leq 0}$. 
 Then
 $$
\Big( \prod_{i=0}^k \mathbb L^\la_\la(-T^i) \Big) \mathbb D^\la_{\la-T^k}=0.
 $$
 \end{prop}

 \begin{proof} For ease of notation, we set $b_i=b(T^{i})$ for all $i\geq 0$.
We proceed by induction on $k\geq1$.
 We first consider the base case in which   $k=1$. Then we have $T^0 \setminus T^1=P\sqcup Q$ for some $P,Q \in {\rm DRem}_0(\la-T^1)$. 
 Since $T^1 \in {\rm DRem}_1(\la)$, we have that 
\begin{equation*}
    {\mathbbm 1}_{\la}\mathbb L(-T^0) \mathbb L(-T^1)  \mathbb D(-T^1) {\mathbbm 1}_{\la-T^1} = 
   (-1)^{b_{1}}
      {\mathbbm 1}_{\la}\mathbb L(-T^0) 
   \mathbb D(-T^1)  \mathbb D(+T^1)\mathbb D(-T^1) {\mathbbm 1}_{\la-T^1}
\end{equation*}
simply by \cref{L=lincombo}. 
We now apply  to the  
subproduct $ \mathbb D(+T^1)\mathbb D(-T^1) {\mathbbm 1}_{\la-T^1}$
 the self-dual relation \eqref{rel2} and hence we obtain 
\begin{equation*}
    (-1)^{b_{1}}  \mathbbm 1_{\la}
 \mathbb L(-T^0)   
      \mathbb D(-T^1)
      \bigg((-1)^{b(P)-1}  \mathbb L(-P) + (-1)^{b(Q)-1}  \mathbb L(-Q) \bigg)
\end{equation*}
and, applying Lemma \ref{h0loop_rect} and \ref{h0loop_sq}, we obtain  
\begin{equation*}
    (-1)^{b_{1}}  \mathbbm 1_{\la}
 \mathbb L(-T^0)   
      \bigg((-1)^{b(P)-1}  \mathbb L(-T^0)       + (-1)^{b(Q)-1}  \mathbb L(-T^0) \bigg)
       \mathbb D(-T^1)
\end{equation*}
and hence, applying \cref{looploop} to the subproduct 
$ \mathbb L(-T^0)    \mathbb L(-T^0)   $ we obtain that the overall product is zero, thus completing the base case of our induction.

We now turn to the inductive step. We suppose that 
\begin{equation}\label{inductive_step}
  \Big( \prod_{i=0}^{j-1} \mathbb L^\la_\la(-T^i) \Big)\Big(
  \mathbbm 1_\la \mathbb L (-T^j) \mathbb D(-T^j)\Big)=0  
\end{equation}
 for all $j<k$. We want to show that \eqref{inductive_step} holds for $j=k$ (using essentially the same argument as in the base case). Then we have $T^{k-1} \setminus T^{k}=P\sqcup Q$ for some $P,Q \in {\rm DRem}_{k-1}(\la-T^k)$. Within  \eqref{inductive_step} with $j=k$, we consider   the   subproduct  of the form
\begin{equation*} 
    {\mathbbm 1}_{\la} \mathbb L(-T^k)  \mathbb D(-T^k)  
    =
      (-1)^{b _k } \mathbb D(-T^k)  \mathbb D(+T^k)\mathbb D(-T^k)  
\end{equation*}
 simply by \cref{L=lincombo}. 
We now apply to the 
subproduct $ \mathbb D(+T^k)\mathbb D(-T^k) {\mathbbm 1}_{\la-T^k}$
 the self-dual relation \eqref{rel2}  and we hence obtain 
\begin{equation*}
-\mathbbm 1_{\la}\mathbb D(-T^k)\Big((-1)^{b(P)}\mathbb L(-P)+ (-1)^{b(Q)}\mathbb L(-Q) 
    + 2(-1)^{b_{k-2}}\mathbb L(-T^{k-2}) + \cdots + 2(-1)^{b_{0}}\mathbb L(-T^{0})\Big). 
\end{equation*}
Applying  Lemma \ref{h0loop_rect} and \ref{h0loop_sq}, we get
\begin{equation*}
-\mathbbm 1_{\la}\Big((-1)^{b(P)}\mathbb L(-T^{k-1})+ (-1)^{b(Q)}\mathbb L(-T^{k-1}) 
    + 2(-1)^{b_{k-2}}\mathbb L(-T^{k-2}) + \cdots + 2(-1)^{b_{0} }\mathbb L(-T^{0})\Big)\mathbb D(-T^k)
\end{equation*}
and  we observe that every summand is a scalar multiplied by some $\mathbbm 1_\la \mathbb L(-T^i) \mathbb D(-T^k)$ for $0\leq i<k$; therefore when we substitute these terms into 
the rightmost bracketed term of  \eqref{inductive_step} the resulting product is zero by the loop-nilpotency relation \eqref{loop-relation} and by \cref{looploop}. The result follows.
\end{proof}

 \begin{lem} \label{loopadd-big1}
  Let $T^k \prec T^{k-1} \prec \dots \prec T^1\prec T^0$ be elements of ${\rm DRem}(\la)$ with ${\sf ht}(T^i)=i \in \ZZ_{\geq 0}$ and set $\alpha=\la-T^0-T^1-\dots - T^k$. 
Let $S \in {\rm DAdd}(\la)$  be adjacent to $T^k$.  We have that 
 $$
\mathbbm1_\la \mathbb L(\alpha) 
\mathbb D(+S)=
\begin{cases}  
\mathbbm1_\la 
\mathbb D(+S) \mathbb L(\alpha+S+T^k)  
\mathbb L   (-\langle T^k\cup S\rangle _{\la+S })   
     &\text{ if $ \langle T^k\cup S\rangle _{\la+S  }$  		exists;		} 			\\
0. &\text{ otherwise}
\end{cases}
 $$
 \end{lem}
An example of the  Dyck tilings involved in \cref{loopadd-big1} is given in \cref{alksghsalgdfghjdfghjd2345}.

\begin{proof}
Recall that the loop generators commute and so we can rewrite the product as follows
\begin{align*}
\mathbbm1_\la \mathbb L(\alpha) 
\mathbb D(+S)
&=
\mathbbm1_\la \mathbb L(\alpha+T^k) 
\mathbb L(-T^k)
\mathbb D(+S)
\\&=
\begin{cases}
\mathbbm1_\la \mathbb L(\alpha+T^k) 
\mathbb D(+S) \mathbb L(-\langle T^k\cup S\rangle_{\la+S})   &\text{ if $ \langle T^k\cup S\rangle _{\la+S  }$  		exists;		} 			\\
0. &\text{ otherwise}
\end{cases}
\end{align*}
by \cref{uuuuuuuu2}. 
Now $S$ necessarily commutes with the $T^j$ for $0\leq j<k$ 
and so the result follows by \cref{uuuuuuuu1}.
\end{proof}

\begin{figure}[ht!]

$$
 \begin{tikzpicture}[scale=0.35]
  \path(0,0)--++(135:2) coordinate (hhhh);
 \draw[very thick] (hhhh)--++(135:6)--++(45:6)--++(-45:6)--++(-135:6);
 \clip (hhhh)--++(135:6)--++(45:6)--++(-45:6)--++(-135:6);
\path(0,0) coordinate (origin2);
  \path(0,0)--++(135:2) coordinate (origin3);

      \foreach \i in {0,1,2,3,4,5,6,7,8}
{
\path (origin3)--++(45:0.5*\i) coordinate (c\i); 
\path (origin3)--++(135:0.5*\i)  coordinate (d\i); 
  }

   \foreach \i in {0,1,2,3,4,5,6,7,8}
{
\path (origin3)--++(45:1*\i) coordinate (c\i); 
\path (c\i)--++(-45:0.5) coordinate (c\i); 
\path (origin3)--++(135:1*\i)  coordinate (d\i); 
\path (d\i)--++(-135:0.5) coordinate (d\i); 
\draw[thick,densely dotted] (c\i)--++(135:10);
\draw[thick,densely dotted] (d\i)--++(45:10);
  }

\path(0,0)--++(135:2) coordinate (hhhh);

\fill[opacity=0.2](hhhh)
--++(135:4)  coordinate (JJ)
--++(45:1)
--++(-45:1)--++(45:2)
--++(-45:2)  coordinate (JJ1) --++(45:1)
--++(-45:1);

\fill[opacity=0.2](hhhh)
  (JJ)
--++(135:1)coordinate (JJ)
 --++(45:2)
--++(-45:2)--++(-135:1)--++(135:1)
;

\fill[opacity=0.4,darkgreen](hhhh)
  (JJ)
--++(135:1) 
 --++(45:3)
--++(-45:2) 
 --++(45:1)
--++(-45:1)  --++(45:2)
--++(-45:3) 
 --++(-135:1)
--++(135:2) 
 --++(-135:2)
--++(135:1)  --++(-135:1)
--++(135:2) 
;



\fill[opacity=0.4,cyan](hhhh)
  (JJ1)
--++(135:1)
 --++(45:2)
--++(-45:2)--++(-135:1)--++(135:1)
;

\path(origin3)--++(45:-0.5)--++(135:7.5) coordinate (X)coordinate (start);

\path (start)--++(135:1)coordinate (start);
 
\path (start)--++(45:5)--++(-45:6) coordinate (X) ;

\path (start)--++(45:2)--++(-45:5) coordinate (X) ; 
 
\path(X)--++(45:2)--++(135:2) coordinate (X) ;
 
\path (X)--++(-45:1) coordinate (X) ;
 \fill[magenta](X) circle (4pt);
\draw[ thick, magenta](X)--++(45:1) coordinate (X) ;
\fill[magenta](X) circle (4pt);
\draw[ thick, magenta](X)--++(-45:1) coordinate (X) ;
\fill[magenta](X) circle (4pt);
%
%

\end{tikzpicture}
\qquad
\begin{tikzpicture}[scale=0.35]
  \path(0,0)--++(135:2) coordinate (hhhh);
 \draw[very thick] (hhhh)--++(135:6)--++(45:6)--++(-45:6)--++(-135:6);
 \clip (hhhh)--++(135:6)--++(45:6)--++(-45:6)--++(-135:6);
\path(0,0) coordinate (origin2);
  \path(0,0)--++(135:2) coordinate (origin3);

      \foreach \i in {0,1,2,3,4,5,6,7,8}
{
\path (origin3)--++(45:0.5*\i) coordinate (c\i); 
\path (origin3)--++(135:0.5*\i)  coordinate (d\i); 
  }

   \foreach \i in {0,1,2,3,4,5,6,7,8}
{
\path (origin3)--++(45:1*\i) coordinate (c\i); 
\path (c\i)--++(-45:0.5) coordinate (c\i); 
\path (origin3)--++(135:1*\i)  coordinate (d\i); 
\path (d\i)--++(-135:0.5) coordinate (d\i); 
\draw[thick,densely dotted] (c\i)--++(135:10);
\draw[thick,densely dotted] (d\i)--++(45:10);
  }

\path(0,0)--++(135:2) coordinate (hhhh);

\fill[opacity=0.2](hhhh)
--++(135:4)  coordinate (JJ)
--++(45:1)
--++(-45:1)--++(45:1)coordinate (JJX)--++(45:1)
--++(-45:2)  coordinate (JJ1) --++(45:1)
--++(-45:1);

\fill[opacity=0.4,violet](hhhh)
  (JJ)
--++(135:1)coordinate (JJ)
 --++(45:2)
--++(-45:2)--++(-135:1)--++(135:1)
;

\fill[opacity=0.4,darkgreen](hhhh)
  (JJ)
--++(135:1) 
 --++(45:3)
--++(-45:1)  
 --++(45:2)
--++(-45:2)  --++(45:1)
--++(-45:3) 
 --++(-135:1)
--++(135:2) 
 --++(-135:1)
--++(135:2)  --++(-135:2)
--++(135:1) 
;

\fill[opacity=0.4,violet](hhhh)
  (JJX)
--++(135:1)
 --++(45:2)
--++(-45:2)--++(-135:1)--++(135:1)
;



\fill[opacity=0.4,violet](hhhh)
  (JJ1)
--++(135:1)
 --++(45:2)
--++(-45:2)--++(-135:1)--++(135:1)
;

\path(origin3)--++(45:-0.5)--++(135:7.5) coordinate (X)coordinate (start);

\path (start)--++(135:1)coordinate (start);
 
\path (start)--++(45:5)--++(-45:6) coordinate (X) ;

\path (start)--++(45:2)--++(-45:5) coordinate (X) ; 
 
\path(X)--++(45:2)--++(135:2) coordinate (X) ;
%
%

\end{tikzpicture}$$
\caption{ On the left we depict a pair 
${\color{cyan}T^{1}}\prec \color{darkgreen}T^0$ in ${\rm DRem}(\la)$
and ${\color{magenta}S } \in {\rm DAdd}(\la)$
 for $\la=(6^3,4,3^3)$ as in \cref{loopadd-big1}. 
 We note that  ${\color{magenta}S }$ and ${\color{cyan}T^{1}}$ are adjacent.  
 On the right we depict the paths 
 $ \color{darkgreen}T^0$ and $ \color{violet}\langle T^1\cup S\rangle_{\la+S} $.
}
\label{alksghsalgdfghjdfghjd2345}
\end{figure}
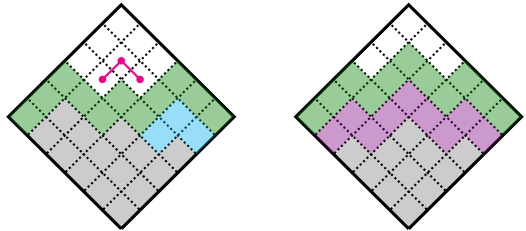

 \begin{prop}
 \label{chris-lem}
 Let $Q^1\prec Q^2 \prec Q^3$ be removable Dyck paths of $\la \in \mathscr{R}_{m,n}$.
  Suppose that $Q^3 \in {\rm DRem}_0(\la)$ and 
  that $Q^1$ does not commute with $Q^2$ and that $Q^2$ does not commute with $Q^3$.  
 We have that 
 $$
 \mathbb L^\la_\la(-Q^3)
  \mathbb D^\la_{\la-Q^2} 
  \mathbb L^{\la-Q^2}_{\la-Q^2}(-Q^1)
  =
 \mathbb L^\la_\la(-Q^3)
   \mathbb L^{\la}_{\la}(-Q^1)
  \mathbb D^\la_{\la-Q^2} .
 $$
 \end{prop}

 \begin{proof}
 We focus on the $m = n$ case (the $m \neq n$ case is similar, but easier). 
 By our assumptions  
  $Q^2 \in {\rm DRem}_{  1}(\la)$, $Q^1 \in {\rm DRem}_{  2}(\la)$
   and 
    $Q^1\in {\rm DRem}_{  0}(\la-Q^2)$ is 
    such that ${\sf last}(Q^2)<m-1$. 
     By \cref{L=lincombo} we have that
 \begin{align*}
 \mathbb L^{\la-Q^2}_{\la-Q^2}(-Q^1) 
 &= 
 - \mathbb L^{\la-Q^2}_{\la-Q^2} - (-1)^{b({\sf rt}(Q^1))} \mathbbm 1_{\la-Q^2}\mathbb D(+{\sf rt}(Q^1))\mathbb D(-{\sf rt}(Q^1))
 \\
    \mathbb L^{\la}_{\la}(-Q^1) &=  (-1)^{b(Q^1)} \mathbbm 1_{\la}\mathbb D(-Q^1)\mathbb D(+Q^1).
 \end{align*}
 Moreover, since $Q^1,Q^2$ do not commute and $Q^1 \prec Q^2$, we have that 
 $ Q^2 \setminus Q^1=P \sqcup T$ for some $P , T \in {\rm DRem}_1(\la-Q^1)$ and we can assume that $P$ is to the left of $T$.
 
 We first consider the case that $Q^3 \in {\rm DRem}_0(\la)$ is the rightmost removable Dyck path, in which case 
 $\mathbb L^{\la}_{\la}(-Q^3)=\mathbb L^{\la}_{\la}$. 
Now,  since  $Q^1,Q^2$ do not commute, our assumption on $Q^3$ implies  that $T = {\sf rt}(Q^1)$.
 We have that 
\begin{align}
  \mathbbm 1_\la
    \mathbb D(-Q^2)
    \mathbb L (-Q^1)&=
 {\mathbbm 1}_\la   
 \mathbb D(-Q^2)
  (- \mathbb L^{\la-Q^2}_{\la-Q^2}
   -
    (-1)^{b({\sf rt}(Q^1))} \mathbb D(+{\sf rt}(Q^1))\mathbb D(-{\sf rt}(Q^1)))  & \notag\\
&     = - \mathbb L^{\la}_{\la} \mathbb D(-Q^2)
      -(-1)^{b({\sf rt}(Q^1))}
{\mathbbm 1}_\la \mathbb D(-Q^2) \mathbb D ( + {\sf rt}(Q^1) ) \mathbb D ( - {\sf rt}(Q^1) )  & \notag\\
&
     = -{\mathbbm 1}_\la\mathbb L(-Q^3) \mathbb D(-Q^2)
        -(-1)^{b({\sf rt}(Q^1))}
{\mathbbm 1}_\la \mathbb D(-Q^2) \mathbb D ( + {\sf rt}(Q^1) ) \mathbb D ( - {\sf rt}(Q^1) )  
\notag
\\
  &  =  -{\mathbbm 1}_\la\mathbb L(-Q^3) \mathbb D(-Q^2)
   -(-1)^{b({\sf rt}(Q^1))+b(Q^2)-b(P)}
   {\mathbbm 1}_\la \mathbb D(-Q^1) \mathbb D ( -P) \mathbb D ( - {\sf rt}(Q^1) )  & \notag\\
     & 
   =   -{\mathbbm 1}_\la\mathbb L(-Q^3) \mathbb D(-Q^2)
   -(-1)^{b({\sf rt}(Q^1))+b(Q^2)-b(P)}
{\mathbbm 1}_\la \mathbb D(-Q^1) \mathbb D ( + Q^1 ) \mathbb D ( -Q^2 )  \notag
\end{align}
where the second equality follows from relation  \eqref{loop-relation};
 the third from \cref{L=lincombo}; the fourth equality follows by applying the relation \eqref{adjacent} to the adjacent pair $P,Q^1$;
 the fifth   equality follows by applying the non-commuting relation to the pair $Q^1 \prec Q^2$. 
We therefore have that
  \begin{align}
 &\mathbbm 1_\la 
\mathbb L(-Q^3)
  \mathbb D (-Q^2)
  \mathbb L (-Q^1)
 \notag \\
= &   -{\mathbbm 1}_\la\mathbb L(-Q^3)^2  \mathbb D(-Q^2)
    -(-1)^{b({\sf rt}(Q^1))+b(Q^2)-b(P)}
{\mathbbm 1}_\la \mathbb L(-Q^3)\mathbb D(-Q^1) \mathbb D ( + Q^1 ) \mathbb D ( -Q^2 ) & \notag \\
 = & 0  -(-1)^{b({\sf rt}(Q^1))+b(Q^2)-b(P)}
{\mathbbm 1}_\la \mathbb L(-Q^3) \mathbb D(-Q^1) \mathbb D ( + Q^1 ) \mathbb D ( -Q^2 ) & \notag \\
 =&  {\mathbbm 1}_\la \mathbb L(-Q^3) \mathbb L(-Q^1) \mathbb D ( -Q^2 )
 \notag
  \end{align}
where the 
first equality follows from substituting our expansion of  
$\mathbb D (-Q^2)
  \mathbb L (-Q^1)
$ from  above; 
the second   follows from  \cref{looploop};
 the third  by   \cref{L=lincombo} and the fact that 
$b(Q^2)= b(Q^1) + b(P) + b(T)=  b(Q^1)+b(P) + b({\sf rt}(Q^1))$.

It remains to consider the case in which ${\sf last}(Q^3)<m-1$. 
 By   \cref{L=lincombo} we can express these loops in terms of ${\sf rt}(Q^3)$ and the case already considered above, as follows
$$\mathbb L^{\la}_{\la}(-Q^3) = - \mathbb L^{\la}_{\la} - (-1)^{b({\sf rt}(Q^3))} \mathbbm 1_{\la}\mathbb D(+{\sf rt}(Q^3))\mathbb D(-{\sf rt}(Q^3)).$$
Our assumptions allow us to fix the following notation 
$$Q^3 \setminus Q^2 = P_1 \sqcup P_2, \quad Q^2 \setminus Q^1 = S_1 \sqcup S_2$$
where $P_1,P_2 \in {\rm DRem}(\la-Q^2)$ and $S_1,S_2 \in {\rm DRem}(\la-Q^1)$. 
We assume that $P_1$ and $S_1$ are to the left of $P_2$ and $S_2$ respectively. 
Moreover, notice that ${\sf rt}(Q^3)={\sf rt}(P_2).$  
We have that 
\begin{align*}
\mathbbm 1_\la 
\mathbb L(-Q^3)
\mathbb L(-Q^1)
\mathbb D(-Q^2)
&
=
(-1)^{b(Q^1)}
\mathbbm 1_\la 
\mathbb L(-Q^3)
\mathbb D(-Q^1)\mathbb D(+Q^1)
\mathbb D(-Q^2)
\\
&
=
(-1)^{b(Q^1)}
\mathbbm 1_\la 
\mathbb L(-Q^3)
\mathbb D(-Q^1)
\mathbb D(-S_1)
\mathbb D(-S_2)
\\
&
=
(-1)^{2b(Q^1)	+b(S_2)			}
\mathbbm 1_\la 
\mathbb L(-Q^3)
\mathbb D(-Q^2)
\mathbb D(+S_2)
\mathbb D(-S_2)
\\
&
=
\mathbbm 1_\la 
\mathbb L(-Q^3)
\mathbb D(-Q^2)
\big( 
\mathbb L(-Q^1)+
\mathbb L(-P_2)\big)
\end{align*}
where the first equality follows from \cref{L=lincombo};
the second from \eqref{rel4} for the non-commuting pair 
$Q^1\prec Q^2$; 
the third  from \eqref{adjacent} for the adjacent  pair 
$Q^1,S_1$ and the fact that 
$b(Q^1)+b(Q^2)-b(S_1)=2b(Q^1)+b(S_2)$; the fourth from the self-dual relation for 
$S_2$ and the fact that $S_2$ is adjacent to 
$P_2,Q^1 \in {\rm DRem}(\la-Q^2)$.  
Thus it suffices to prove 
$\mathbb L(-Q^3)
\mathbb D(-Q^2)
\mathbb L(-P_2)=0$ and the result will follow.  
We have that 
\begin{align*}
\mathbbm 1_\la 
\mathbb L(-Q^3)
\mathbb D(-Q^2)
\mathbb L(-P_2)
&
=
\mathbbm 1_\la 
\mathbb L(-Q^3)
\mathbb D(-Q^2)
 (
- \mathbb L^{\la-Q^2}
_{\la-Q^2}
-
(-1)^{b({\sf rt}(Q^3))}
\mathbb D(+{\sf rt}(Q^3))
\mathbb D(-{\sf rt}(Q^3))
 )
\end{align*}
by  \cref{L=lincombo} and the fact that ${\sf rt}(P_2)={\sf rt}(Q^3)$. Thus it will suffice to show that 
$$ (-1)^{b({\sf rt}(Q^3))}
 \mathbb L(-Q^3)
\mathbb D(-Q^2)
\mathbb D(+{\sf rt}(Q^3))
\mathbb D(-{\sf rt}(Q^3))
 )=
 - 
 \mathbbm 1_\la 
\mathbb L(-Q^3)
\mathbb L^{\la}
_{\la}
\mathbb D(-Q^2).
$$
We have that 
\begin{align*}
&	(-1)^{b({\sf rt}(Q^3))}
\mathbb L(-Q^3)
\mathbb D(-Q^2)
\mathbb D(+{\sf rt}(Q^3))
\mathbb D(-{\sf rt}(Q^3))
\\
 =&
\mathbb L(-Q^3)
 \mathbb D(+{\sf rt}(Q^3))
\mathbb D(-{\sf rt}(Q^3))
\mathbb D(-Q^2)
\\
=&
\mathbb L(-Q^3)
(  -
\mathbb   L^{\la}_{\la}	
-
	\mathbb  L(-Q^3)	)
  \mathbb D(-Q^2)
\\
=&
-\mathbb L(-Q^3)
\mathbb   L^{\la}_{\la} \mathbb D(-Q^2)	
-
	\mathbb  L(-Q^3)^2
  \mathbb D(-Q^2)
  \\
=&
-\mathbb L(-Q^3)
\mathbb   L^{\la}_{\la} \mathbb D(-Q^2)	
\end{align*}as required; 
where the first equality follows from the commutativity relation \eqref{rel3}; 
the second follows from the self-dual relation for $Q^3$; 
the third equality expands out the brackets;
the fourth equality follows from relation \eqref{loop-relation}.
The result follows. 
\end{proof}

 \begin{lem}
 \label{chris-lem2}
 Let $Q^1\prec Q^2 \prec Q^3$ be removable Dyck paths of $\la \in \mathscr{R}_{m,n}$.
  Suppose that $Q^3 \in {\rm DRem}_{>0}(\la)$ and 
  that $Q^1$ does not commute with $Q^2$ and that $Q^2$ does not commute with $Q^3$.  
  We set $Q^3\setminus Q^2= P^1 \sqcup P^2$.  
 We have that 
 $$
 \mathbb D(-Q^3) \mathbb D(+Q^3)
  \mathbb D(-Q^2)
  \mathbb D(-Q^1)
  =
    \mathbb D(-Q^1)
 \mathbb D(-Q^3) \mathbb D(+P^1)
  \mathbb D(+P^2).
 $$
 \end{lem}

 \begin{proof}
This follows simply using the ``old relations" and is left as an exercise for the reader. 
 \end{proof}


 \subsection{A spanning set of the symmetric Dyck path algebra}
Let $\mu\setminus \alpha$ be a Dyck tiling.
We recall our favourite path from $\alpha$   to $\mu$ is given as follows:
we  first regularise 
\begin{align*} 
 \alpha={\sf reg}_d(\alpha) \subseteq {\sf reg}_{d+1}(\alpha)
\dots  \subseteq {\sf reg}_{-1}(\alpha)
  \subseteq {\sf reg}_{0}(\alpha)=   {\sf reg} (\alpha)
\end{align*} for  $d(\alpha)=d$ and $P^k = ({\sf reg}_{k}(\alpha) \setminus {\sf reg}_{k-1}(\alpha))$ the maximal breadth addable Dyck path  of height $k$. 
After regularising, we   then split
\begin{align*} 
 {\sf reg}(\alpha)
 ={\sf split}_d(\mu\setminus\alpha) \supseteq {\sf split}_{1-d}(\mu\setminus \alpha)
\dots  \supseteq {\sf  split}_{-1}(\mu\setminus \alpha)
  \supseteq {\sf  split}_{0}(\mu\setminus \alpha)=   
 \alpha\sqcup  ( \mu\setminus \alpha)_{\leq 0}
\end{align*}
with 
$
({\sf split}_{k}(\mu\setminus\alpha) )\setminus ( {\sf split}_
{k+1}(\mu\setminus \alpha))
=
 R_1^k\sqcup  R_2^k \sqcup\dots \sqcup  R_K^k
$
a disjoint union of commuting Dyck paths such that 
\begin{align}\label{The-Rs}
P^k - R_1^k- R_2^k- \dots -R_K^k= (\mu\setminus\alpha)_k.
\end{align}
After splitting, we  then add  as follows
\begin{align*} 
 \alpha\sqcup( \mu\setminus \alpha)_{\leq 0}={\sf add}_0(\mu\setminus \alpha) \subseteq {\sf add}_{1}(\mu\setminus \alpha)
\dots  \subseteq {\sf add}_{a-1}(\mu\setminus \alpha)  
  \subseteq {\sf add}_{a}(\mu\setminus \alpha)=   \mu
\end{align*}
   where 
\begin{align}\label{The-As}{\sf add}_k(\mu\setminus \alpha)\setminus 
   {\sf add}_{k-1}(\mu\setminus \alpha)=
    A_1^k\sqcup  A_2^k \sqcup\dots \sqcup  A_K^k
\end{align} is a union of addable Dyck paths of height 
$1\leq k\leq a$.
    We now lift these paths to  elements of $\mathcal{A}_{m,n}$. 
For $d(\alpha)\leq k \leq 0$ we lift the $k$th step in \cref{Step1} to the element 
 $$
 \mathbb L ^k(\alpha)=
 \mathbb L^{{\sf reg}(\alpha)}_{{\sf reg}(\alpha)}( -P^k) $$
and for $d(\alpha)\leq k \leq 0$ we lift  the $k$th step in \cref{Step2} to the element 
\begin{align}\label{sdasdfasdfasdfdfdf2e23232322}
\mathbb R_{\mu\setminus\alpha}^k= 	\mathbb D^{{\sf  split}_{k-1}(\mu\setminus \alpha)}
_{{\sf  split}_{k }(\mu\setminus \alpha)}
=
	\mathbb D^{{\sf  split}_{k-1}(\mu\setminus \alpha)}
_{{\sf  split}_{k -1}(\mu\setminus \alpha)-R_1^k}
	\mathbb D^{{\sf  split}_{k-1}(\mu\setminus \alpha)-R_1^k}
_{{\sf  split}_{k -1}(\mu\setminus \alpha)-R_1^k-R_2^k}\dots 
\end{align}
and for $1\leq k\leq a $ we lift  the $k$th step in \cref{Step3} to the element 
\begin{align}\label{sdasdfasdfasdfdfdf2e2323232}
\mathbb A_{\mu\setminus\alpha}^k = 	\mathbb D^{{\sf  add}_{k-1}(\mu\setminus \alpha)}
_{{\sf  add}_{k }(\mu\setminus \alpha)}
=
	\mathbb D^{{\sf  add}_{k-1}(\mu\setminus \alpha)}
_{{\sf  add}_{k -1}(\mu\setminus \alpha)+A_1^k}
	\mathbb D^{{\sf  add}_{k-1}(\mu\setminus \alpha)+A_1^k}
_{{\sf  add}_{k -1}(\mu\setminus \alpha)+A_1^k+A_2^k}\dots 
\end{align}
We note that the ordering in \eqref{sdasdfasdfasdfdfdf2e23232322} and \eqref{sdasdfasdfasdfdfdf2e2323232} does not matter, as all these paths commute. 
Finally, we put all this together 
$$
\mathbb  L (\alpha) = \prod _{d(\alpha)\leq k \leq 0 }
 \mathbb  L ^k(\alpha)
 \qquad
\mathbb  R_{\mu\setminus \alpha} = \prod _{d(\alpha)\leq k \leq 0 }
\mathbb  R ^k_{\mu\setminus \alpha}
 \qquad
\mathbb  A _{\mu\setminus \alpha} = \prod _{  k \geq 1 }
\mathbb  A ^k_{\mu\setminus \alpha} $$
(and we extend this to duals in the usual fashion). 
We define the elements
$$
\mathbb  R^{< h} (\mu\setminus \alpha) 
=   \prod_{ d(\alpha) \leq k<h	}	\mathbb R ^k_{\mu\setminus \alpha}    
\qquad
\mathbb  R^{> h} (\mu\setminus \alpha) 
=   \prod_{ h < k\leq 0	}	\mathbb R ^k_{\mu\setminus \alpha}    
$$
and extend this to weak inequalities and similar subproducts of $\mathbb  A (\mu\setminus \alpha)$ in the obvious fashion.
With this notation in place, we are able to state the main result of this subsection.

 \begin{thm}\label{the spanning set}
 The symmetric  Dyck algebra $\mathcal A_{m,n}$ has spanning set 
 \begin{align}\label{bassis2}
 \{ \mathbb A_{\la\setminus \alpha}^\ast
 \mathbb R_{\la\setminus \alpha}^\ast
  \mathbb L (\alpha)
  \mathbb R_{\mu\setminus \alpha} 
\mathbb A_{\mu\setminus \alpha} 
 \mid \alpha \in \mathscr{P}_{m,n} , 
\la,\mu \in \mathscr{R}_{m,n} 
\text { and }   \mu\setminus \alpha ,  \la\setminus \alpha  \text{ are Dyck pairs} \}.
\end{align}

 \end{thm}

  \begin{proof} We will prove that the 2-sided ideal 
$\mathcal A_{m,n}^{\leq \alpha}=
\langle {\mathbbm 1}_\sigma , \mathbb L^{\sigma}_\sigma(\tau)	\mid	\tau\leq  \sigma \leq 	\alpha	\rangle$ has spanning set given by 
$$
{\rm Dyck}_{\leq \alpha}=\{
 \mathbb A_{\la\setminus \beta}^\ast
 \mathbb R_{\la\setminus \beta}^\ast
  \mathbb L (\beta)
  \mathbb R_{\mu\setminus \beta} 
\mathbb A_{\mu\setminus \beta} 
\mid 
\la,\mu \in \mathscr{R}_{m,n} 
\text { and }\mu\setminus \beta, \la\setminus \beta  \text{ are Dyck pairs}
\text { for }
\beta \leq \alpha 
 \}.
$$
We proceed by induction on the Bruhat ordering on $\alpha$, refining the induction by the degree 
of the element 
$$
 \mathbb A_{\la\setminus \beta}^\ast
 \mathbb R_{\la\setminus \beta}^\ast
  \mathbb L (\beta)
  \mathbb R_{\mu\setminus \beta} 
\mathbb A_{\mu\setminus \beta} 
$$
for all $\beta\leq \alpha$.
  Our aim is to show the following:
\begin{align*}
(
  \mathbb L (\alpha)
  \mathbb R_{\mu\setminus \alpha} 
\mathbb A_{\mu\setminus \alpha} )
  X
  =
  \begin{cases}  
  \mathbb L (\alpha)
  \mathbb R_{\mu\setminus \alpha} 
\mathbb A_{\mu\setminus \alpha} 
&\text {if $X=\mathbbm 1_\mu$}
\\
 \mathbb L (\alpha)
  \mathbb R_{\nu\setminus \alpha} 
\mathbb A_{\nu\setminus \alpha}
 &\text {if $X= \mathbb D(+Q)$ and $(\mu+Q) \setminus \alpha  $ a Dyck pair}
\\
  \pm \mathbb L (\alpha)
  \mathbb R_{\nu\setminus \alpha} 
\mathbb A_{\nu\setminus \alpha}
 &\text {if $X= \mathbb D(-Q)$ for $  Q\not \in \mu\setminus \alpha$ 
 and $(\mu-Q) \setminus \alpha $ a Dyck pair}
\\
 Y  \in \mathcal{A}_{m,n}^{< \alpha} &\text{otherwise }
   \end{cases}
   \end{align*}
for   any $\mu\setminus\alpha$ such that 
$
{\rm deg}(\mu\setminus\alpha)
=
2|d(\alpha)|+k $ and will refine our   induction by the degree $k \geq0$.  
In the $\alpha=\varnothing $ case,  
we first observe that ${\sf reg}(\varnothing)=(m^m)\in \mathcal{P}_{m,n}$. 
Now, we note that the unique non-zero element of \eqref{bassis2} is given by
$$\mathbb L (\varnothing)
 = \textstyle \prod _{P \in {\rm DRem}(m^m)		}\mathbb L^{(m^m)}_{(m^m)}   (-P)
 $$ and is of degree $2 m$ (note that this varies over the $m$ distinct removable Dyck paths, of breadth 
$b(P)=p \in \{1,3,\dots, 2m-1\}$). 
 Thus  it will suffice to show that 
 $$
\mathbb L (\varnothing) X =
 \begin{cases}
 \mathbb L (\varnothing)&\text{ for }	X=	\mathbbm 1_{(m^m)}\\
  0						&\text{ otherwise. }
  \end{cases}
  	$$
Let $Q \in {\rm DRem}(m^m)$. For  $X= \mathbb L (- Q)$  the  
desired equality follows from \cref{looploop-big}.
For  $X=\mathbb D(-Q)$    the desired equality follows from
\cref{loopremove-big1}. 
For  $X=\mathbb D (+ Q)$    the desired equality follows from
 \cref{loopadd-big1}.
 The  case that $X$ is an idempotent is trivial.

Thus we can assume the result holds for all $\alpha' <\alpha$ and proceed to the $\alpha$ case.
 We   consider each type of generator $ X\in K^m_n$ in turn: loops, adding generators, and removing generators (idempotent generators are trivial).

\bigskip
\noindent{\bf Loops. }
Consider the product $(
  \mathbb L (\alpha)
  \mathbb R_{\mu\setminus \alpha} 
\mathbb A_{\mu\setminus \alpha} )
  X$ with $X=\mathbb L^\mu_\mu(-Q)$ with
$Q \in {\rm DRem}(\mu)$. 
 There are four cases to consider: $(i)$
$Q =H \in \mu\setminus \alpha$ 
or $(ii)$  
$Q\not \in \mu\setminus \alpha$ and $Q$ commutes with   $ \mu\setminus \alpha$
  $(iii)$ 
$Q\not \in \mu\setminus \alpha$ but 
$Q\prec H \in \mu\setminus \alpha$ for some $H$ which does not commute with $Q$
{  $(iv)$ 
$Q\not \in \mu\setminus \alpha$ but 
$Q\succ H \in \mu\setminus \alpha$ for some $H$ with ${\sf ht}^{\mu}_{\alpha}(H)>0$ which does not commute with $Q$.}

Case $(ii)$ is not as simple as one might first think. 
 By assumption   $Q$  commutes with all of $\mu\setminus \alpha$.  
Thus   $Q$  commutes with   the   $A_i^k$ 
for $i\geq 1$ and $1\leq k \leq a$ 
 as in \eqref{The-As} for the Dyck pair $\mu\setminus \alpha$ 
 and we claim that 
  $\mathbb L^\mu_\mu({-Q})$ commutes with   $\mathbb A_{\mu\setminus\alpha}$.  
The claim follows   by applying \cref{h0loop_rect,h0loop_sq} in the case that $ {\sf ht}^\mu(Q)=0$ or 
 by using the factorisation $\mathbb L^\mu_\mu({-Q})= (-1)^{b(Q)}	\mathbb D^\la _{\la-Q}\mathbb D_\la ^{\la-Q}	$ and \eqref{rel3} in the  case that $ {\sf ht}^\mu(Q)>0$. 
 We thus obtain 
 \begin{align*}
  (\mathbb L (\alpha)
  \mathbb R_{\mu\setminus \alpha} 
\mathbb A_{\mu\setminus \alpha}  \mathbbm 1_\mu)
\mathbb  L^\mu_\mu(-Q)
=
  \mathbb L (\alpha)
  \mathbb R_{\mu\setminus \alpha} 
  \mathbb  L 
  (-Q) 
\mathbb A_{\mu\setminus \alpha}   \mathbbm 1_\mu
\end{align*}
If $Q$  commutes with  all the Dyck paths   $R_i^k$ for all $i\geq 1$ and $d(\alpha)\leq k\leq 0$ in \cref{The-Rs} then 
(again, applying the commutation relations of \cref{h0loop_rect,h0loop_sq} or \eqref{rel3} as above) we obtain 
\[
  \mathbb L (\alpha)
  \mathbb R_{\mu\setminus \alpha} 
  \mathbb  L 
  (-Q) 
  \mathbb A_{\mu\setminus \alpha}   \mathbbm 1_\mu
  = 
  \mathbb L (\alpha)
  \mathbb  L 
  (-Q) 
    \mathbb R_{\mu\setminus \alpha} 
    \mathbb A_{\mu\setminus \alpha}   \mathbbm 1_\mu
    \in \mathcal{A}_{m,n}^{\leq \alpha-Q}
\]
 by \cref{looploop-big}.  
Now assume that   $Q$ does not commute with  all of the Dyck paths $R_i^k$ in \cref{The-Rs} (but recall   our ongoing assumption that $Q$ commutes with all Dyck paths in $\mu \setminus\alpha$).  
In which case,    $Q$ commutes with a pair
 of Dyck paths 
  $H \in\mu\setminus \alpha $ (to the left of $Q$) and $H' \in \mu\setminus \alpha$ (to the right of $Q$)   
  such that ${\sf ht}^{\mu}_{\alpha}(H)={\sf ht}^{\mu}_{\alpha}(H')=h\in \ZZ_{\leq0}$ 
 and such that  
 \begin{align}\label{HandHprime}
 H \sqcup H' = {\rm split}_{R^h_i}(H'')
 \end{align}
 for some $i\geq 1$ and some  (necessarily unique) $R^h_i$ which does not commute with $Q$. We remark that in such a case  the Dyck path  $R^h_i$,
  in turn, does not commute with $H''$.  
An  example   of how this can happen is given in \cref{figloopcase2}. 

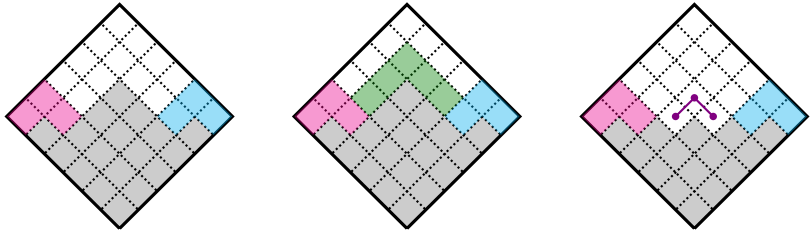
\begin{figure}[ht!]

$$
 \begin{tikzpicture}[scale=0.35]
  \path(0,0)--++(135:2) coordinate (hhhh);
 \draw[very thick] (hhhh)--++(135:6)--++(45:6)--++(-45:6)--++(-135:6);
 \clip(hhhh)--++(135:6)--++(45:6)--++(-45:6)--++(-135:6);
\path(0,0) coordinate (origin2);
  \path(0,0)--++(135:2) coordinate (origin3);

      \foreach \i in {0,1,2,3,4,5,6,7,8}
{
\path (origin3)--++(45:0.5*\i) coordinate (c\i); 
\path (origin3)--++(135:0.5*\i)  coordinate (d\i); 
  }

   \foreach \i in {0,1,2,3,4,5,6,7,8}
{
\path (origin3)--++(45:1*\i) coordinate (c\i); 
\path (c\i)--++(-45:0.5) coordinate (c\i); 
\path (origin3)--++(135:1*\i)  coordinate (d\i); 
\path (d\i)--++(-135:0.5) coordinate (d\i); 
\draw[thick,densely dotted] (c\i)--++(135:10);
\draw[thick,densely dotted] (d\i)--++(45:10);
  }

\path(0,0)--++(135:2) coordinate (hhhh);

\fill[opacity=0.2](hhhh)
--++(135:5)  coordinate (JJ)
--++(45:1)
--++(-45:1)--++(45:3)
--++(-45:3)  coordinate (JJ1) --++(45:1)
--++(-45:1)--++(45:2)--++(-45:1)--++(45:1)--++(-45:1);

\fill[opacity=0.4,magenta](hhhh)
  (JJ)
--++(135:1)
 --++(45:2)
--++(-45:2)--++(-135:1)--++(135:1)
;

\fill[opacity=0.4,cyan](hhhh)
  (JJ1)
--++(135:1)
 --++(45:2)
--++(-45:2)--++(-135:1)--++(135:1)
;

\path(origin3)--++(45:-0.5)--++(135:7.5) coordinate (X)coordinate (start);

%
%
%
%
%
%
%

\end{tikzpicture}\qquad
\begin{tikzpicture}[scale=0.35]
  \path(0,0)--++(135:2) coordinate (hhhh);
 \draw[very thick] (hhhh)--++(135:6)--++(45:6)--++(-45:6)--++(-135:6);
 \clip(hhhh)--++(135:6)--++(45:6)--++(-45:6)--++(-135:6);
\path(0,0) coordinate (origin2);
  \path(0,0)--++(135:2) coordinate (origin3);

      \foreach \i in {0,1,2,3,4,5,6,7,8}
{
\path (origin3)--++(45:0.5*\i) coordinate (c\i); 
\path (origin3)--++(135:0.5*\i)  coordinate (d\i); 
  }

   \foreach \i in {0,1,2,3,4,5,6,7,8}
{
\path (origin3)--++(45:1*\i) coordinate (c\i); 
\path (c\i)--++(-45:0.5) coordinate (c\i); 
\path (origin3)--++(135:1*\i)  coordinate (d\i); 
\path (d\i)--++(-135:0.5) coordinate (d\i); 
\draw[thick,densely dotted] (c\i)--++(135:10);
\draw[thick,densely dotted] (d\i)--++(45:10);
  }

\path(0,0)--++(135:2) coordinate (hhhh);

\fill[opacity=0.2](hhhh)
--++(135:5)  coordinate (JJ)
--++(45:1)
--++(-45:1)--++(45:1) coordinate (JJ2)--++(45:2)
--++(-45:3)  coordinate (JJ1) --++(45:1)
--++(-45:1)--++(45:2)--++(-45:1)--++(45:1)--++(-45:1);

\fill[opacity=0.4,magenta](hhhh)
  (JJ)
--++(135:1)
 --++(45:2)
--++(-45:2)--++(-135:1)--++(135:1)
;

\fill[opacity=0.4,cyan](hhhh)
  (JJ1)
--++(135:1)
 --++(45:2)
--++(-45:2)--++(-135:1)--++(135:1)
;

\fill[opacity=0.4,darkgreen](hhhh)
  (JJ2)
--++(135:1)
 --++(45:3)
--++(-45:3)--++(-135:1)--++(135:2)
;

\path(origin3)--++(45:-0.5)--++(135:7.5) coordinate (X)coordinate (start);

%
%
%
%
%
%
%

\end{tikzpicture}
\qquad
 \begin{tikzpicture}[scale=0.35]
  \path(0,0)--++(135:2) coordinate (hhhh);
 \draw[very thick] (hhhh)--++(135:6)--++(45:6)--++(-45:6)--++(-135:6);
 \clip(hhhh)--++(135:6)--++(45:6)--++(-45:6)--++(-135:6);
\path(0,0) coordinate (origin2);
  \path(0,0)--++(135:2) coordinate (origin3);

      \foreach \i in {0,1,2,3,4,5,6,7,8}
{
\path (origin3)--++(45:0.5*\i) coordinate (c\i); 
\path (origin3)--++(135:0.5*\i)  coordinate (d\i); 
  }

   \foreach \i in {0,1,2,3,4,5,6,7,8}
{
\path (origin3)--++(45:1*\i) coordinate (c\i); 
\path (c\i)--++(-45:0.5) coordinate (c\i); 
\path (origin3)--++(135:1*\i)  coordinate (d\i); 
\path (d\i)--++(-135:0.5) coordinate (d\i); 
\draw[thick,densely dotted] (c\i)--++(135:10);
\draw[thick,densely dotted] (d\i)--++(45:10);
  }

\path(0,0)--++(135:2) coordinate (hhhh);

\fill[opacity=0.2](hhhh)
--++(135:5)  coordinate (JJ)
--++(45:1)
--++(-45:1)
--++(45:1)
--++(-45:1)  --++(45:1)
--++(-45:1)  --++(45:1)
--++(-45:1)  
coordinate (JJ1) --++(45:1)
--++(-45:1)--++(45:2)--++(-45:1)--++(45:1)--++(-45:1);

\fill[opacity=0.4,magenta](hhhh)
  (JJ)
--++(135:1)
 --++(45:2)
--++(-45:2)--++(-135:1)--++(135:1)
;

\fill[opacity=0.4,cyan](hhhh)
  (JJ1)
--++(135:1)
 --++(45:2)
--++(-45:2)--++(-135:1)--++(135:1)
;

\path(origin3)--++(45:-0.5)--++(135:5.5) coordinate (X)coordinate (start);

\path (start)--++(135:1)coordinate (start);
 
\path (start)--++(45:5)--++(-45:6) coordinate (X) ;

\path (start)--++(45:2)--++(-45:3) coordinate (X) ; 
 
\path(X)--++(45:1)--++(135:1) coordinate (X) ;
 
\path (X)--++(-45:1) coordinate (X) ;
 \fill[violet](X) circle (4pt);
\draw[ thick, violet](X)--++(45:1) coordinate (X) ;
\fill[violet](X) circle (4pt);
\draw[ thick, violet](X)--++(-45:1) coordinate (X) ;
\fill[violet](X) circle (4pt);

\end{tikzpicture} 
$$
  \caption{ An example of case $(ii)$  for loop generators. 	On the left we depict $\mu\setminus \alpha= {\color{magenta}H}\sqcup {\color{cyan}H'}$. 
In the middle we depict ${\sf reg}(\alpha)\setminus \alpha$ and we highlight  the Dyck path 
$   \color{darkgreen}R^h_1 $. On the right we depict $\nu \setminus \alpha$ where $\nu=\mu- \color{violet}Q$.  Notice that $\color{violet}Q$ commutes with the tiling $\mu\setminus \alpha= {\color{magenta}H}\sqcup {\color{cyan}H'}$ but does not commute with $   \color{darkgreen}R^h_1 $.
  	}
    \label{figloopcase2}

\end{figure}

For Dyck paths as in \cref{HandHprime} 
  we have that 
 \begin{align*}
 \textstyle   \mathbb L (\alpha)
  \mathbb R_{\mu\setminus \alpha}
  \mathbb  L (-Q) 
    \mathbb A_{\mu\setminus \alpha}   \mathbbm 1_\mu
   = & \textstyle  
 ( \prod_{
k\neq h 
} \mathbb L (-P^k) )
 	\mathbb R ^{<h}_{\mu\setminus \alpha}   
\mathbb L (-P^h)
\mathbb  R ^h_{\mu\setminus \alpha}
  \mathbb  L (-Q) 
   	\mathbb R ^{>h}_{\mu\setminus \alpha}   
 \mathbb A_{\mu\setminus \alpha}   \mathbbm 1_\mu
  \\
  = & \textstyle  
 ( \prod_{
 k\neq h 
} \mathbb L (-P^k) )
 	\mathbb R ^{<h}_{\mu\setminus \alpha}   
\mathbb L (-P^h)  \mathbb  L (-Q) 
\mathbb  R ^h_{\mu\setminus \alpha}
  	\mathbb R ^{>h}_{\mu\setminus \alpha}    
 \mathbb A_{\mu\setminus \alpha}   \mathbbm 1_\mu
 \\
 = & \textstyle   
 ( \prod_{
 k\neq h 
} \mathbb L (-P^k) )\mathbb L (-P^h)  \mathbb  L (-Q) 
  	\mathbb R ^{<h}_{\mu\setminus \alpha}   
\mathbb  R ^h_{\mu\setminus \alpha}
  	\mathbb R ^{>h}_{\mu\setminus \alpha}    
 \mathbb A_{\mu\setminus \alpha}   \mathbbm 1_\mu
 \\
 =& \textstyle   \mathbb L (\alpha)  \mathbb  L (-Q) 
  \mathbb R_{\mu\setminus \alpha}
 \mathbb A_{\mu\setminus \alpha}   \mathbbm 1_\mu
\end{align*}
 where   
 first  and third equalities follow  from the commuting relations or  \cref{h0loop_rect,h0loop_sq};   the second equalities  from applying 
\cref{chris-lem} for $h=0$  (or   \cref{chris-lem2} if $h>0$)  to the pair $R^h_i$ and $Q$
and  the commutation relations \eqref{rel3} to the pairs  $R^h_j$   and  $Q$ for $j\neq i$ ;
the fourth equality follow from the definitions of the products. 
We have that $\mathbb L (\alpha)  \mathbb  L (-Q)\in \mathcal{A}_{m,n}^{\leq \alpha-Q} $ 
by \cref{looploop-big}.

Case $(iii)$. We now suppose that 
$Q\not \in \mu\setminus \alpha$, but  $Q\prec H  \in \mu\setminus \alpha $ for some $H$ which does not commute with $Q$.  
By assumption ${\sf ht}^{\mu }(Q)>0$ and so $\mathbb L^\mu_\mu({-Q})= (-1)^{b(Q)}	\mathbb D^\mu _{\mu-Q}\mathbb D_\mu ^{\mu-Q}	$. 
First, we suppose that    ${\sf ht}^{\mu}_{\alpha}(H)\leq 0$. 
In which case $Q$ commutes with every Dyck path $A^j_k$, $R^i_l$ in \eqref{The-Rs} and  \eqref{The-As} (see \cref{figloopcase3} for an example). 
By the commutation relations \eqref{rel3}, we have that 
 \begin{align*}
( \mathbb L (\alpha)
  \mathbb R_{\mu\setminus \alpha} 
\mathbb A_{\mu\setminus \alpha} \mathbbm 1_\mu)
\mathbb  L^\mu_\mu(-Q)
=
  \mathbb L (\alpha)				\mathbb  L(-Q)
  \mathbb R_{\mu\setminus \alpha} 
\mathbb A_{\mu\setminus \alpha}  \mathbbm 1_\mu \in \mathcal{A}_{m,n}^{\leq \alpha-Q} 
\end{align*}
by  \cref{looploop-big}.

\begin{figure}[ht!]

$$
 \begin{tikzpicture}[scale=0.35]
  \path(0,0)--++(135:2) coordinate (hhhh);
 \draw[very thick] (hhhh)--++(135:10)--++(45:10)--++(-45:10)--++(-135:10);
 \clip (hhhh)--++(135:10)--++(45:10)--++(-45:10)--++(-135:10);
\path(0,0) coordinate (origin2);
  \path(0,0)--++(135:2) coordinate (origin3);

      \foreach \i in {0,1,2,3,4,5,6,7,8,9,10}
{
\path (origin3)--++(45:0.5*\i) coordinate (c\i); 
\path (origin3)--++(135:0.5*\i)  coordinate (d\i); 
  }

   \foreach \i in {0,1,2,3,4,5,6,7,8,9,10}
{
\path (origin3)--++(45:1*\i) coordinate (c\i); 
\path (c\i)--++(-45:0.5) coordinate (c\i); 
\path (origin3)--++(135:1*\i)  coordinate (d\i); 
\path (d\i)--++(-135:0.5) coordinate (d\i); 
\draw[thick,densely dotted] (c\i)--++(135:12);
\draw[thick,densely dotted] (d\i)--++(45:12);
  }

\path(0,0)--++(135:2) coordinate (hhhh);

\fill[opacity=0.2](hhhh)
--++(135:8)  coordinate (JJ)
--++(45:2)
--++(-45:2)	coordinate (JJ7)		--++(45:2)
--++(-45:2)   --++(45:2)
--++(-45:2)coordinate (JJ1)  --++(45:2)--++(-45:1)coordinate (JJ23)--++(45:1)--++(-45:1);

\fill[opacity=0.4,orange](hhhh)
  (JJ23)
--++(135:1) 
 --++(45:2)
--++(-45:2)--++(-135:1)--++(135:1)
;

\fill[opacity=0.4,darkgreen](hhhh)
  (JJ)
--++(135:1)coordinate (JJ)
 --++(45:1)
--++(-45:1)--++(-135:1)--++(135:1)
;

\fill[opacity=0.4,cyan](hhhh)
  (JJ7)
--++(135:1)
 --++(45:3)
--++(-45:3)--++(-135:1)--++(135:2)
;

\fill[opacity=0.4,orange](hhhh)
  (JJ)
--++(135:1)
 --++(45:2)
--++(-45:1)
 --++(45:1)
 --++(-45:1)
 --++(45:3)
--++(-45:3)
 --++(45:1)
  --++(-45:1) --++(45:1)
--++(-45:2)
--++(-135:1)--++(135:1)--++(-135:1)--++(135:1)
--++(-135:1)--++(135:3)
--++(-135:3)--++(135:1)--++(-135:1)--++(135:1)
;


%

\fill[opacity=0.4,darkgreen](hhhh)
  (JJ1)
--++(135:1)
 --++(45:1)
--++(-45:1)--++(-135:1)--++(135:1)
;

\path(origin3)--++(45:-0.5)--++(135:7.5) coordinate (X)coordinate (start);

\path (start)--++(135:1)coordinate (start);
 
\path (start)--++(45:6)--++(-45:5) coordinate (X) ;

\path (start)--++(45:2)--++(-45:3) coordinate (X) ; 
 
\path(X)--++(45:3)--++(135:3) coordinate (X) ;
 
\path (X)--++(-45:1) coordinate (X) ;
 \fill[violet](X) circle (4pt);
\draw[ thick, violet](X)--++(45:1) coordinate (X) ;
\fill[violet](X) circle (4pt);
\draw[ thick, violet](X)--++(-45:1) coordinate (X) ;
\fill[violet](X) circle (4pt);

\end{tikzpicture}
\qquad
 \begin{tikzpicture}[scale=0.35]
  \path(0,0)--++(135:2) coordinate (hhhh);
 \draw[very thick] (hhhh)--++(135:10)--++(45:10)--++(-45:10)--++(-135:10);
 \clip (hhhh)--++(135:10)--++(45:10)--++(-45:10)--++(-135:10);
\path(0,0) coordinate (origin2);
  \path(0,0)--++(135:2) coordinate (origin3);

      \foreach \i in {0,1,2,3,4,5,6,7,8,9,10}
{
\path (origin3)--++(45:0.5*\i) coordinate (c\i); 
\path (origin3)--++(135:0.5*\i)  coordinate (d\i); 
  }

   \foreach \i in {0,1,2,3,4,5,6,7,8,9,10}
{
\path (origin3)--++(45:1*\i) coordinate (c\i); 
\path (c\i)--++(-45:0.5) coordinate (c\i); 
\path (origin3)--++(135:1*\i)  coordinate (d\i); 
\path (d\i)--++(-135:0.5) coordinate (d\i); 
\draw[thick,densely dotted] (c\i)--++(135:12);
\draw[thick,densely dotted] (d\i)--++(45:12);
  }

\path(0,0)--++(135:2) coordinate (hhhh);

\fill[opacity=0.2](hhhh)
--++(135:8)  coordinate (JJ)
--++(45:2)
--++(-45:2)	coordinate (JJ7)		--++(45:2)
--++(-45:2)   --++(45:2)
--++(-45:2)coordinate (JJ1)  --++(45:2)--++(-45:1)coordinate (JJ23)--++(45:1)--++(-45:1);

\fill[opacity=0.2](hhhh)
  (JJ23)
--++(135:1) 
 --++(45:2)
--++(-45:2)--++(-135:1)--++(135:1)
;

\fill[opacity=0.2](hhhh)
  (JJ)
--++(135:1)coordinate (JJ)
 --++(45:1)
--++(-45:1)--++(-135:1)--++(135:1)
;

\fill[opacity=0.2](hhhh)
  (JJ7)
--++(135:1)
 --++(45:3)
--++(-45:3)--++(-135:1)--++(135:2)
;

\fill[opacity=0.2](hhhh)
  (JJ)
--++(135:1)
 --++(45:2)
--++(-45:1) coordinate (XY)
 --++(45:1)
 --++(-45:1)
 --++(45:3)
--++(-45:3)
 --++(45:1)
  --++(-45:1) --++(45:1)
--++(-45:2)
--++(-135:1)--++(135:1)--++(-135:1)--++(135:1)
--++(-135:1)--++(135:3)
--++(-135:3)--++(135:1)--++(-135:1)--++(135:1)
;


%

\fill[opacity=0.4,magenta] 
  (XY)
--++(135:1)
 --++(45:2)
--++(-45:2)--++(-135:1)--++(135:1)
;

\path(XY)--++(45:4)--++(-45:4) coordinate (XY);

\fill[opacity=0.4,cyan] 
  (XY)
--++(135:1)
 --++(45:2)
--++(-45:2)--++(-135:1)--++(135:1)
;

\path(XY)--++(45:2)--++(-45:3) coordinate (XY);

\fill[opacity=0.4,orange] 
  (XY)
--++(135:1)
 --++(45:1)
--++(-45:1)--++(-135:1)--++(135:1)
;

\fill[opacity=0.2] 
  (JJ1)
--++(135:1)
 --++(45:1)
--++(-45:1)--++(-135:1)--++(135:1)
;

\path(origin3)--++(45:-0.5)--++(135:7.5) coordinate (X)coordinate (start);

\path (start)--++(135:1)coordinate (start);
 
\path (start)--++(45:6)--++(-45:5) coordinate (X) ;

\path (start)--++(45:2)--++(-45:3) coordinate (X) ; 
 
\path(X)--++(45:3)--++(135:3) coordinate (X) ;
 
\path (X)--++(-45:1) coordinate (X) ;
 \fill[violet](X) circle (4pt);
\draw[ thick, violet](X)--++(45:1) coordinate (X) ;
\fill[violet](X) circle (4pt);
\draw[ thick, violet](X)--++(-45:1) coordinate (X) ;
\fill[violet](X) circle (4pt);

\end{tikzpicture}
$$
  \caption{ An example of case $(iii)$  for loop generators. On the left we depict the Dyck tiling of 
 $\mu\setminus\alpha$ and the Dyck path $\color{violet}X$ which commutes with all of 
  $\mu\setminus\alpha$ 
   except for ${\color{violet}X}\prec \color{cyan}H$ which is of height $-1$.
     On the right we depict the Dyck tiling of  $ {\sf reg}(\alpha)\setminus \mu $ (consisting of the Dyck paths 
      $\color{magenta}R^{-1}_1 $, $\color{cyan}R^{-1}_2  $, and $\color{orange}R^{0}_1 $)  
   and the Dyck path  $\color{violet}X$ which commutes with all of 
$ {\sf reg}(\alpha)\setminus \mu $. 		}
    \label{figloopcase3}
\end{figure}
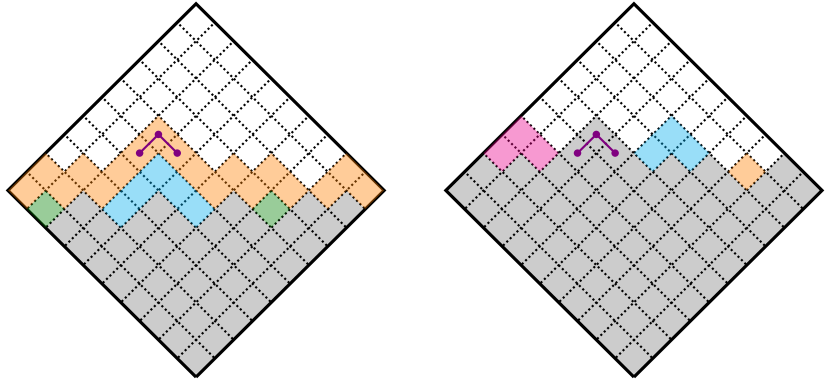

 We now consider the case that ${\sf ht}^{\mu}_{\alpha}(H)=h>0$ and $Q$ does not commute with some $Q\prec H=A^h_1$ in  \eqref{The-As}. We note that  $H\setminus Q=H^1\sqcup H^2$.
  We have that 
  \begin{align*}
 \textstyle  ( \mathbb L (\alpha)
  \mathbb R_{\mu\setminus \alpha} 
\mathbb A_{\mu\setminus \alpha} \mathbbm 1_\mu )
\mathbb  L^\mu_\mu(-Q)
=& \textstyle 
\mathbb L (\alpha)
  \mathbb R_{\mu\setminus \alpha} 
   	\mathbb A ^{<h}_{\mu\setminus \alpha}    
 \mathbb D(+H)
\mathbb  L (-Q)
  ( \prod_{ j\neq 1	}	\mathbb D(+ A ^h_ j)  )  
   	\mathbb A ^{>h}_{\mu\setminus \alpha}   
\mathbbm 1_\mu 
\\
=& \textstyle 
\mathbb L (\alpha)
  \mathbb R_{\mu\setminus \alpha} 
   	\mathbb A ^{<h}_{\mu\setminus \alpha}    
  \mathbb D(+H^1) 
   \mathbb D(+H^2)
 \mathbb D(+Q)
  ( \prod_{ j\neq 1	}	\mathbb D(+ A ^h_ j)  )  
   	\mathbb A ^{>h}_{\mu\setminus \alpha}   
\mathbbm 1_\mu 
\\
=& \textstyle 
\mathbb L (\alpha)
  \mathbb R_{\mu\setminus \alpha} 
   	\mathbb A ^{<h}_{\mu\setminus \alpha}    
\mathbb D(+H^1) 
   \mathbb D(-H^1)
 \mathbb D(+H)
  ( \prod_{ j\neq 1	}	\mathbb D(+ A ^h_ j)  )  
   	\mathbb A ^{>h}_{\mu\setminus \alpha}   
\mathbbm 1_\mu 
\\
=& \textstyle 
\mathbb L (\alpha)
  \mathbb R_{\mu\setminus \alpha} 
   	\mathbb A ^{<h}_{\mu\setminus \alpha}    
\mathbb D(+H^1) 
   \mathbb D(-H^1)
   	\mathbb A ^{\geq h}_{\mu\setminus \alpha}    
\mathbbm 1_\mu 
\\
 =&    \sum_{P \in {\rm DRem}(\nu) }c_P
\textstyle\mathbb L (\alpha)
  \mathbb R_{\mu\setminus \alpha} 
   	\mathbb A ^{<h}_{\mu\setminus \alpha}    
\mathbb L(-P) 
   	\mathbb A ^{\geq h}_{\mu\setminus \alpha}    
\mathbbm 1_\mu 
\end{align*}
for $\nu=\alpha \cup (\mu\setminus\alpha)_{<h}$ and $c_P \in \Bbbk $ some coefficients which can be calculated explicitly using the self-dual relation \eqref{rel2};
 the second and third  equalities follow 
 by the non-commuting relation \eqref{rel4} and adjacency relations  \eqref{adjacent}. 
Now, we observe that 
$$\textstyle\mathbb L (\alpha)
  \mathbb R_{\mu\setminus \alpha} 
 	\mathbb A ^{<h}_{\mu\setminus \alpha} 
\mathbb L(-P) 
=
\mathbb L (\alpha)
  \mathbb R_{\nu\setminus \alpha} 
 \mathbb A  _{\nu\setminus \alpha} 
\mathbb    L  (-P) 
\in 
\mathcal{A}_{m,n} ^{< \alpha } $$
by our inductive assumption on the degree; 
therefore 
   $ ( \mathbb L (\alpha)
  \mathbb R_{\mu\setminus \alpha} 
\mathbb A_{\mu\setminus \alpha} \mathbbm 1_\mu )
\mathbb  L^\mu_\mu(-Q)\in \mathcal{A}_{m,n}^{<\alpha}$.

Case $(i)$.  We now suppose that $Q =H \in \mu\setminus \alpha$. 
 We first  assume that   ${\sf ht}^{\mu}_{\alpha}(H)=h\leq 0$.
There are two distinct subcases to consider. First suppose that 
$H$ is the unique element of $(\mu\setminus\alpha)_h$.
In this case, $Q$ commutes with all the Dyck paths $A^j_k$, $R^i_l$ in \eqref{The-Rs} and  \eqref{The-As} and so  we have that 
\[
( \mathbb L (\alpha)
  \mathbb R_{\mu\setminus \alpha} 
\mathbb A_{\mu\setminus \alpha}  \mathbbm 1_\mu)
\mathbb  L (-Q)
=
 \mathbb L (\alpha)\mathbb  L (-Q)
  \mathbb R_{\mu\setminus \alpha} 
\mathbb A_{\mu\setminus \alpha}  \mathbbm 1_\mu \in \mathcal{A}_{m,n}^{\leq \alpha-Q}
\]
  by the commutation relations \eqref{rel3} and  \cref{looploop-big}. 
We must now consider the case that 
 $ H  $ is not the unique Dyck path in $(\mu\setminus\alpha)_h$ for $h\leq 0$. 
Our assumptions    imply that   $Q$ is adjacent to $H':=R^h_1$ and  possibly $H'':=R^h_2$ (if the latter exists)  removed in the $h$ step  of \eqref{The-Rs}; moreover,  
 $Q$ commutes with   every  other Dyck path $A^j_k$, $R^i_l$ in \eqref{The-Rs} and  \eqref{The-As} not equal to $H'$ or $H''$.
 Our assumptions further imply that 
  ${\sf ht}^\mu(Q)>0$ and so $ \mathbb L(-Q)= 
(-1)^{b(Q)}	\mathbbm1_\mu \mathbb D(-Q)\mathbb D(+Q)$ by 
\cref{L=lincombo}.

\begin{figure}[ht!]

  $$
 \begin{tikzpicture}[scale=0.35]
  \path(0,0)--++(135:2) coordinate (hhhh);
 \draw[very thick] (hhhh)--++(135:10)--++(45:10)--++(-45:10)--++(-135:10);
 \clip (hhhh)--++(135:10)--++(45:10)--++(-45:10)--++(-135:10);
\path(0,0) coordinate (origin2);
  \path(0,0)--++(135:2) coordinate (origin3);

      \foreach \i in {0,1,2,3,4,5,6,7,8,9,10}
{
\path (origin3)--++(45:0.5*\i) coordinate (c\i); 
\path (origin3)--++(135:0.5*\i)  coordinate (d\i); 
  }

   \foreach \i in {0,1,2,3,4,5,6,7,8,9,10}
{
\path (origin3)--++(45:1*\i) coordinate (c\i); 
\path (c\i)--++(-45:0.5) coordinate (c\i); 
\path (origin3)--++(135:1*\i)  coordinate (d\i); 
\path (d\i)--++(-135:0.5) coordinate (d\i); 
\draw[thick,densely dotted] (c\i)--++(135:12);
\draw[thick,densely dotted] (d\i)--++(45:12);
  }

\path(0,0)--++(135:2) coordinate (hhhh);

\fill[opacity=0.2](hhhh)
--++(135:8)  coordinate (JJ)
--++(45:2)
--++(-45:2)	coordinate (JJ7)		--++(45:2)
--++(-45:2)   --++(45:2)
--++(-45:2)coordinate (JJ1)  --++(45:2)--++(-45:1)coordinate (JJ23)--++(45:1)--++(-45:1);

\fill[opacity=0.4,violet](hhhh)
  (JJ)
--++(135:1) 
 --++(45:3)
--++(-45:2)
 --++(45:2)
--++(-45:2)
 --++(45:2)
--++(-45:3)
 --++(-135:1)--++(135:2)
  --++(-135:2)--++(135:2)
   --++(-135:2)--++(135:2)
;

\path(JJ)--++(135:1) coordinate(JJ);

\fill[opacity=0.4,orange](hhhh)
  (JJ)
--++(135:1) 
 --++(45:4)
--++(-45:2)
 --++(45:2)
--++(-45:2)
 --++(45:2)
--++(-45:3)
--++(45:1)
--++(-45:1)
--++(45:1)
--++(-45:2)
 --++(-135:1)--++(135:1)
  --++(-135:1)--++(135:1)
 --++(-135:1)--++(135:3)
  --++(-135:2)--++(135:2)
   --++(-135:2)--++(135:2)
;

\end{tikzpicture}\qquad
 \begin{tikzpicture}[scale=0.35]
  \path(0,0)--++(135:2) coordinate (hhhh);
 \draw[very thick] (hhhh)--++(135:10)--++(45:10)--++(-45:10)--++(-135:10);
 \clip (hhhh)--++(135:10)--++(45:10)--++(-45:10)--++(-135:10);
\path(0,0) coordinate (origin2);
  \path(0,0)--++(135:2) coordinate (origin3);

      \foreach \i in {0,1,2,3,4,5,6,7,8,9,10}
{
\path (origin3)--++(45:0.5*\i) coordinate (c\i); 
\path (origin3)--++(135:0.5*\i)  coordinate (d\i); 
  }

   \foreach \i in {0,1,2,3,4,5,6,7,8,9,10}
{
\path (origin3)--++(45:1*\i) coordinate (c\i); 
\path (c\i)--++(-45:0.5) coordinate (c\i); 
\path (origin3)--++(135:1*\i)  coordinate (d\i); 
\path (d\i)--++(-135:0.5) coordinate (d\i); 
\draw[thick,densely dotted] (c\i)--++(135:12);
\draw[thick,densely dotted] (d\i)--++(45:12);
  }

\path(0,0)--++(135:2) coordinate (hhhh);

\fill[opacity=0.2](hhhh)
--++(135:8)  coordinate (JJ)
--++(45:2)
--++(-45:2)	coordinate (JJ7)		--++(45:2)
--++(-45:2)   --++(45:2)
--++(-45:2)coordinate (JJ1)  --++(45:2)--++(-45:1)coordinate (JJ23)--++(45:1)--++(-45:1);

\fill[opacity=0.4,orange](hhhh)
  (JJ23)
--++(135:1) 
 --++(45:2)
--++(-45:2)--++(-135:1)--++(135:1)
;

\fill[opacity=0.4,violet](hhhh)
  (JJ)
--++(135:1)coordinate (JJ)
 --++(45:1)
--++(-45:1)--++(-135:1)--++(135:1)
;

\fill[opacity=0.4,violet](hhhh)
  (JJ7)
--++(135:1)
 --++(45:3)
--++(-45:3)--++(-135:1)--++(135:2)
;

\fill[opacity=0.4,orange](hhhh)
  (JJ)
--++(135:1)
 --++(45:2)
--++(-45:1)
 --++(45:1)
 --++(-45:1)
 --++(45:3)
--++(-45:3)
 --++(45:1)
  --++(-45:1) --++(45:1)
--++(-45:2)
--++(-135:1)--++(135:1)--++(-135:1)--++(135:1)
--++(-135:1)--++(135:3)
--++(-135:3)--++(135:1)--++(-135:1)--++(135:1)
;


%

\fill[opacity=0.4,violet](hhhh)
  (JJ1)
--++(135:1)
 --++(45:1)
--++(-45:1)--++(-135:1)--++(135:1)
;

\path(origin3)--++(45:-0.5)--++(135:7.5) coordinate (X)coordinate (start);

\path (start)--++(135:1)coordinate (start);
 
\path (start)--++(45:6)--++(-45:5) coordinate (X) ;

\path (start)--++(45:2)--++(-45:3) coordinate (X) ; 
 
\path(X)--++(45:2)--++(135:3) coordinate (X) ;
 
\path (X)--++(-45:1) coordinate (X) ;
 \fill[cyan](X) circle (4pt);
\draw[ thick, cyan](X)--++(45:1) coordinate (X) ;
\fill[cyan](X) circle (4pt);
\draw[ thick, cyan](X)--++(45:1) coordinate (X) ;
\fill[cyan](X) circle (4pt);
\draw[ thick, cyan](X)--++(-45:1) coordinate (X) ;
\fill[cyan](X) circle (4pt);\draw[ thick, cyan](X)--++(-45:1) coordinate (X) ;
\fill[cyan](X) circle (4pt);

\end{tikzpicture}
\qquad
 \begin{tikzpicture}[scale=0.35]
  \path(0,0)--++(135:2) coordinate (hhhh);
 \draw[very thick] (hhhh)--++(135:10)--++(45:10)--++(-45:10)--++(-135:10);
 \clip (hhhh)--++(135:10)--++(45:10)--++(-45:10)--++(-135:10);
\path(0,0) coordinate (origin2);
  \path(0,0)--++(135:2) coordinate (origin3);

      \foreach \i in {0,1,2,3,4,5,6,7,8,9,10}
{
\path (origin3)--++(45:0.5*\i) coordinate (c\i); 
\path (origin3)--++(135:0.5*\i)  coordinate (d\i); 
  }

   \foreach \i in {0,1,2,3,4,5,6,7,8,9,10}
{
\path (origin3)--++(45:1*\i) coordinate (c\i); 
\path (c\i)--++(-45:0.5) coordinate (c\i); 
\path (origin3)--++(135:1*\i)  coordinate (d\i); 
\path (d\i)--++(-135:0.5) coordinate (d\i); 
\draw[thick,densely dotted] (c\i)--++(135:12);
\draw[thick,densely dotted] (d\i)--++(45:12);
  }

\path(0,0)--++(135:2) coordinate (hhhh);

\fill[opacity=0.2](hhhh)
--++(135:8)  coordinate (JJ)
--++(45:2)
--++(-45:2)	coordinate (JJ7)		--++(45:2)
--++(-45:2)   --++(45:2)
--++(-45:2)coordinate (JJ1)  --++(45:2)--++(-45:1)coordinate (JJ23)--++(45:1)--++(-45:1);

\fill[opacity=0.2](hhhh)
  (JJ23)
--++(135:1) 
 --++(45:2)
--++(-45:2)--++(-135:1)--++(135:1)
;

\fill[opacity=0.2](hhhh)
  (JJ)
--++(135:1)coordinate (JJ)
 --++(45:1)
--++(-45:1)--++(-135:1)--++(135:1)
;

\fill[opacity=0.2](hhhh)
  (JJ7)
--++(135:1)
 --++(45:3)
--++(-45:3)--++(-135:1)--++(135:2)
;

\fill[opacity=0.2](hhhh)
  (JJ)
--++(135:1)
 --++(45:2)
--++(-45:1) coordinate (XY)
 --++(45:1)
 --++(-45:1)
 --++(45:3)
--++(-45:3)
 --++(45:1)
  --++(-45:1) --++(45:1)
--++(-45:2)
--++(-135:1)--++(135:1)--++(-135:1)--++(135:1)
--++(-135:1)--++(135:3)
--++(-135:3)--++(135:1)--++(-135:1)--++(135:1)
;


%

\fill[opacity=0.4,magenta] 
  (XY)
--++(135:1)
 --++(45:2)
--++(-45:2)--++(-135:1)--++(135:1)
;

\path(XY)--++(45:4)--++(-45:4) coordinate (XY);

\fill[opacity=0.4,magenta] 
  (XY)
--++(135:1)
 --++(45:2)
--++(-45:2)--++(-135:1)--++(135:1)
;

\path(XY)--++(45:2)--++(-45:3) coordinate (XY);

\fill[opacity=0.4,orange] 
  (XY)
--++(135:1)
 --++(45:1)
--++(-45:1)--++(-135:1)--++(135:1)
;

\fill[opacity=0.2] 
  (JJ1)
--++(135:1)
 --++(45:1)
--++(-45:1)--++(-135:1)--++(135:1)
;

\path(origin3)--++(45:-0.5)--++(135:7.5) coordinate (X)coordinate (start);

\path (start)--++(135:1)coordinate (start);
 
\path (start)--++(45:6)--++(-45:5) coordinate (X) ; 
  
\path (start)--++(45:2)--++(-45:3) coordinate (X) ; 
 
\path(X)--++(45:2)--++(135:3) coordinate (X) ;
 
\path (X)--++(-45:1) coordinate (X) ;
 \fill[cyan](X) circle (4pt);
\draw[ thick, cyan](X)--++(45:1) coordinate (X) ;
\fill[cyan](X) circle (4pt);
\draw[ thick, cyan](X)--++(45:1) coordinate (X) ;
\fill[cyan](X) circle (4pt);
\draw[ thick, cyan](X)--++(-45:1) coordinate (X) ;
\fill[cyan](X) circle (4pt);\draw[ thick, cyan](X)--++(-45:1) coordinate (X) ;
\fill[cyan](X) circle (4pt);
\end{tikzpicture}
$$
    \caption{ An example of case $(i)$  for loop generators. 	
We depict the Dyck tilings of 
${\sf reg}(\alpha)\setminus \alpha$,  
 $\mu\setminus\alpha$,
   and  $ {\sf reg}(\alpha)\setminus \mu $  respectively. 
In the two rightmost pictures, we also depict the  Dyck path 
$\color{cyan}X$ which is equal to a
 ${ \color{violet}H} \in \mu\setminus\alpha$ which is of height $-1$.
	}
    \label{figloopcase1}
\end{figure}
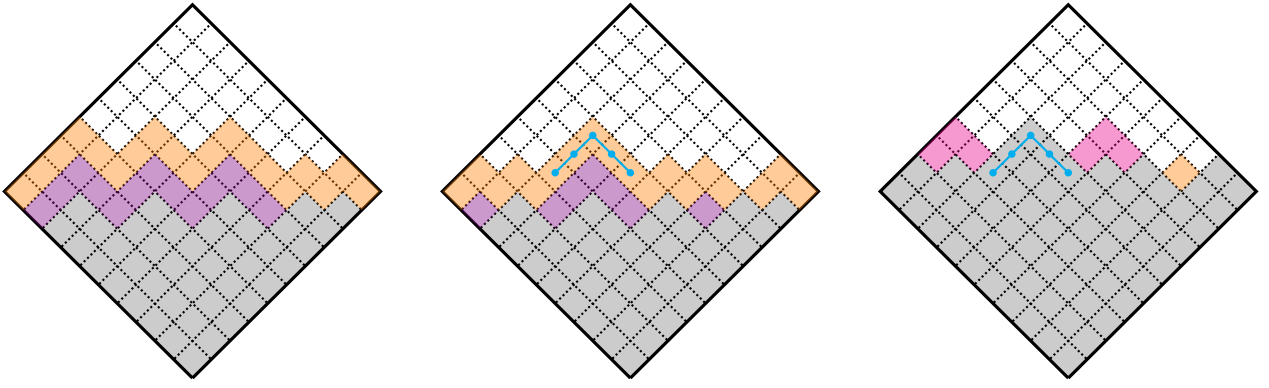

We first consider  the case that $H'=R^h_1$ is the unique  Dyck path in \eqref{The-Rs} which is adjacent  to $Q$. 
We have that 
 \begin{align*}
\textstyle  ( \mathbb L (\alpha)
  \mathbb R_{\mu\setminus \alpha} 
\mathbb A_{\mu\setminus \alpha} \mathbbm 1_\mu )
\mathbb  L^\mu_\mu(-Q)
=& \textstyle 
\mathbb L (\alpha)
   	\mathbb R ^{< h}_{\mu\setminus \alpha}   
 \mathbb D(-H')  
\mathbb  L (-Q)
  ( \prod_{ j\neq 1 
 	}	\mathbb D(-R ^h_ j)  )  
   	\mathbb R ^{> h}_{\mu\setminus \alpha}    
  \mathbb A_{\mu\setminus \alpha} 
\mathbbm 1_\mu 
 \\
=&\pm \textstyle 
\mathbb L (\alpha)
   	\mathbb R ^{< h}_{\mu\setminus \alpha}   
 \mathbb  L 
  (- \langle H' \cup Q\rangle_{\nu }  ) \mathbb D(-H') 
  ( \prod_{ j\neq 1 
 	}	\mathbb D(-R ^h_ j)  )  
   	\mathbb R ^{> h}_{\mu\setminus \alpha}    
  \mathbb A_{\mu\setminus \alpha} 
\mathbbm 1_\mu 
 \end{align*}
for $\nu=\alpha \cup (\mu\setminus\alpha)_{<h}$;   
the first equality follows from the commutativity relations \eqref{rel3}; 
 the second    equality follows  
from the adjacent relation   \eqref{adjacent}.
Therefore 
   $  ( \mathbb L (\alpha)
  \mathbb R_{\mu\setminus \alpha} 
\mathbb A_{\mu\setminus \alpha} \mathbbm 1_\mu )
\mathbb  L^\mu_\mu(-Q)\in \mathcal{A}_{m,n}^{<\alpha}$ by induction on degree.

We now consider  the case that $H'=R^h_1$ and $H''=R^h_2$ are the two  Dyck paths in \eqref{The-Rs} which are adjacent  to $Q$.  
 We have that 
 \begin{align*}
& 
\textstyle  ( \mathbb L (\alpha)
  \mathbb R_{\mu\setminus \alpha} 
\mathbb A_{\mu\setminus \alpha} \mathbbm 1_\mu )
\mathbb  L (-Q)
 \\
=& \textstyle 
\mathbb L (\alpha)
   	\mathbb R ^{< h}_{\mu\setminus \alpha}   
   \mathbb D(-H'') \mathbb D(-H') 
\mathbb  L^\mu_\mu(-Q)
  ( \prod_{ j\neq 1  ,2
 	}	\mathbb D(-R ^h_ j)  )  
   	\mathbb R ^{> h}_{\mu\setminus \alpha}    
  \mathbb A_{\mu\setminus \alpha} 
\mathbbm 1_\mu 
 \\
=& \textstyle 
\pm\mathbb L (\alpha)
   	\mathbb R ^{< h}_{\mu\setminus \alpha}   
{\mathbbm 1}_\nu 
\mathbb D(-H'') 
 \mathbb  L 
  (- \langle H' \cup Q\rangle_{\nu }  ) \mathbb D(-H') 
  ( \prod_{ j\neq 1  ,2
 	}	\mathbb D(-R ^h_ j)  )  
   	\mathbb R ^{> h}_{\mu\setminus \alpha}    
  \mathbb A_{\mu\setminus \alpha} 
\mathbbm 1_\mu 
 \\
=& \textstyle 
\pm\mathbb L (\alpha)
   	\mathbb R ^{< h}_{\mu\setminus \alpha}   
{\mathbbm 1}_\nu 
\mathbb  L 
  (- \langle H' \cup H'' \cup Q\rangle_{\nu }  ) 
 \mathbb D(-H'')  \mathbb D(-H') 
  ( \prod_{ j\neq 1  ,2
 	}	\mathbb D(-R ^h_ j)  )  
	\mathbb R ^{> h}_{\mu\setminus \alpha}    
	  \mathbb A_{\mu\setminus \alpha} 
\mathbbm 1_\mu 
 \end{align*}
for $\nu=\alpha \cup (\mu\setminus\alpha)_{<h}$;   
the first equality follows from the commutativity relations \eqref{rel3}; 
 the second  and third   equalities follow 
from the adjacent relation   \eqref{adjacent}.
Therefore 
   $  ( \mathbb L (\alpha)
  \mathbb R_{\mu\setminus \alpha} 
\mathbb A_{\mu\setminus \alpha} \mathbbm 1_\mu )
\mathbb  L^\mu_\mu(-Q)\in \mathcal{A}_{m,n}^{<\alpha}$ by induction on degree.

Now  assume
      that ${\sf ht}^{\mu}_{\alpha}(H)>0$. 
      In this case $Q$ is equal to 
precisely one       Dyck path, $Q=H=A^h_1$ say, in \eqref{The-As} and $Q$ 
commutes with all other paths in 
$A^j_k$, $R^i_l$ in \eqref{The-As} and \eqref{The-Rs} not equal to $H$.
We have that 
\begin{align*}( \mathbb L (\alpha)
  \mathbb R_{\mu\setminus \alpha} 
\mathbb A_{\mu\setminus \alpha}  \mathbbm 1_\mu)
\mathbb  L (-Q)
& \textstyle  =
  \mathbb L (\alpha)
  \mathbb R_{\mu\setminus \alpha} 
  	\mathbb A ^{< h}_{\mu\setminus \alpha}   
 \mathbb D (+Q)  
 \mathbb L (-Q)  
 (\prod_{j\neq 1} \mathbb  D(+A ^h_j))
   	\mathbb A ^{> h}_{\mu\setminus \alpha}     \mathbbm 1_\mu
\\
& \textstyle  =
\mathbb L (\alpha)
  \mathbb R_{\mu\setminus \alpha} 
  	\mathbb A ^{< h}_{\mu\setminus \alpha}   
    \mathbb D (+Q)      \mathbb D (-Q)  
    \mathbb D (+Q)  
 (\prod_{j\neq 1} \mathbb  D(+A ^h_j))
   	\mathbb A ^{> h}_{\mu\setminus \alpha}     \mathbbm 1_\mu
 \\
& \textstyle  =
 \sum_{
\begin{subarray}c
 P  \in {\rm DRem}(\nu)
\end{subarray}
} 
c_P 
\mathbb L (\alpha)
  \mathbb R_{\mu\setminus \alpha} 
  	\mathbb A ^{< h}_{\mu\setminus \alpha}   
\mathbb    L  (-P) 
    	\mathbb A ^{\geq h}_{\mu\setminus \alpha}    \mathbbm 1_\mu
\end{align*}
for $\nu=\alpha \cup (\mu\setminus\alpha)_{<h}$ and $c_P \in \Bbbk $ some coefficients which can be calculated explicitly using the self-dual relation \eqref{rel2}; 
   the second equality follows  
by \cref{L=lincombo}.
 Therefore 
 $( \mathbb L (\alpha)
  \mathbb R_{\mu\setminus \alpha} 
\mathbb A_{\mu\setminus \alpha}  \mathbbm 1_\mu)
\mathbb  L (-Q)\in 
\mathcal{A}_{m,n} ^{< \alpha } 
 $ by induction on degree. 

 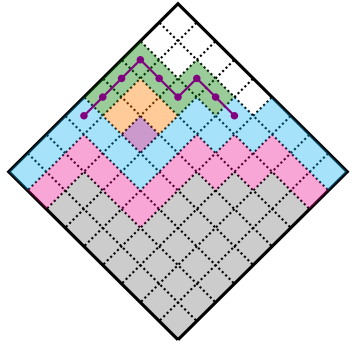
\begin{figure}[ht!]
 
 $$ \begin{tikzpicture}[scale=0.35]
  \path(0,0)--++(135:2) coordinate (hhhh);
 \draw[very thick] (hhhh)--++(135:9)--++(45:9)--++(-45:9)--++(-135:9);
 \clip (hhhh)--++(135:9)--++(45:9)--++(-45:9)--++(-135:9);
\path(0,0) coordinate (origin2);
  \path(0,0)--++(135:2) coordinate (origin3);

      \foreach \i in {0,1,2,3,4,5,6,7,8,9,10}
{
\path (origin3)--++(45:0.5*\i) coordinate (c\i); 
\path (origin3)--++(135:0.5*\i)  coordinate (d\i); 
  }

   \foreach \i in {0,1,2,3,4,5,6,7,8,9,10}
{
\path (origin3)--++(45:1*\i) coordinate (c\i); 
\path (c\i)--++(-45:0.5) coordinate (c\i); 
\path (origin3)--++(135:1*\i)  coordinate (d\i); 
\path (d\i)--++(-135:0.5) coordinate (d\i); 
\draw[thick,densely dotted] (c\i)--++(135:12);
\draw[thick,densely dotted] (d\i)--++(45:12);
  }

\path(0,0)--++(135:2) coordinate (hhhh);

\fill[opacity=0.2](hhhh)
--++(135:7)  coordinate (JJ)
--++(45:2)
--++(-45:3)	coordinate (JJ7)		--++(45:3)
--++(-45:1)   --++(45:1)
--++(-45:1)coordinate (JJ1)  --++(45:1)--++(-45:2);

\path   (JJ7)--++(45:2)--++(135:2) coordinate (JJ7);

%
%

\fill[opacity=0.4,violet](hhhh)
  (JJ7)
--++(135:1)coordinate   (JJ7)
 --++(45:1)
--++(-45:1); 
;

\fill[opacity=0.4,orange](hhhh)
  (JJ7)
--++(135:1) coordinate   (JJ7)
 --++(45:2)
--++(-45:2)--++(-135:1)--++(135:1)
;

\fill[opacity=0.4,darkgreen](hhhh)
  (JJ7)
--++(135:1) coordinate   (JJ7)
 --++(45:3)
--++(-45:2)--++(45:1)--++(-45:2)
--++(-135:1)--++(135:1)--++(-135:1)--++(135:2)
;

\fill[opacity=0.35,magenta](hhhh)
  (JJ)
--++(135:1)
 --++(45:1) coordinate (JJ)--++(45:2)
--++(-45:3)
 --++(45:3)
 --++(-45:1) --++(45:1) --++(-45:2)
 coordinate (X) 
    --++(-135:1)--++(135:1)--++(-135:1)--++(135:1)
 --++(-135:3)--++(135:3)
;

\fill[opacity=0.35,magenta](hhhh)
  (X)--++(135:1)--++(45:1)--++(-45:3)--++(-135:1);

\fill[opacity=0.35,cyan](hhhh)
  (JJ)
  --++(-135:1)
--++(135:1)--++(45:1)
 --++(45:1) coordinate (JJ)--++(45:2)
--++(-45:3)
 --++(45:3)
 --++(-45:1) --++(45:1) --++(-45:2)
 coordinate (X) 
    --++(-135:1)--++(135:1)--++(-135:1)--++(135:1)
 --++(-135:3)--++(135:3)
 ;

\fill[opacity=0.35,cyan](hhhh)
  (X)--++(135:1)--++(45:1)--++(-45:4)--++(-135:1);

%
%
%
%
%
%
%

%
%
%
%
%
%
%
%
%
%

\path(origin3)--++(45:-0.5)--++(135:7.5) coordinate (X)coordinate (start);

\path (start)--++(135:2)coordinate (start);
 
\path (start)--++(45:6)--++(-45:5) coordinate (X) ;

\path (start)--++(45:2)--++(-45:3) coordinate (X) ; 
 
\path(X)--++(45:2)--++(135:3) coordinate (X) ;
 
\path (X)--++(-45:1) coordinate (X) ;
 \fill[violet](X) circle (4pt);
\draw[ thick, violet](X)--++(45:1) coordinate (X) ;
\fill[violet](X) circle (4pt);
\draw[ thick, violet](X)--++(45:1) coordinate (X) ;
\fill[violet](X) circle (4pt);
\draw[ thick, violet](X)--++(45:1) coordinate (X) ;
\fill[violet](X) circle (4pt);
\draw[ thick, violet](X)--++(-45:1) coordinate (X) ;
\fill[violet](X) circle (4pt);\draw[ thick, violet](X)--++(-45:1) coordinate (X) ;
\fill[violet](X) circle (4pt);
\draw[ thick, violet](X)--++(45:1) coordinate (X) ;
\fill[violet](X) circle (4pt);
\draw[ thick, violet](X)--++(-45:1) coordinate (X) ;
\fill[violet](X) circle (4pt);\draw[ thick, violet](X)--++(-45:1) coordinate (X) ;
\fill[violet](X) circle (4pt);

\end{tikzpicture}$$
\caption{Case $(iv)$ of the loop generator. Notice that $\color{violet}Q$ commutes with   the big pink and cyan Dyck paths (as we are not in case $(iii)$ by assumption).	The Dyck paths 
${\color{violet}Q}\setminus {\color{darkgreen}H} = H^1 \sqcup H^2$ commute with the rest of the Dyck tiling 
(again as we are not in case $(iii)$).		}
\label{Case 4}
\end{figure}

Finally, suppose we are in case $(iv)$ and we assume that 
$Q$ is not as in case $(ii)$  or  $(iii)$, an example is depicted in \cref{Case 4}. 
 Since 
${\sf ht}^\mu(Q)\in \ZZ_{\geq 0}$, this implies that 
${\sf ht}^{\mu}_{\alpha}(H)=h\in \ZZ_{>0}$. 
Our assumptions imply that 
 $Q\setminus H=H^1 \sqcup H^2$.
  and  that 
$H^1, H^2$ commute with $\nu\setminus\alpha$ for $\nu = \mu - H$.
We have that 
\begin{align*}
 \mathbb A_{\mu\setminus \alpha} 
 \mathbbm 1_\mu 
 \mathbb  L(-Q)
= 
 \mathbb A_{\nu\setminus \alpha} 
 \mathbb  D(+H) 
 \mathbb  L (-Q) \mathbbm 1_\mu 
= 
\mathbb A_{ \nu \setminus \alpha} 
 \mathbb  L (-H^1) 
 \mathbb  D(+H) 
  \mathbbm 1_\mu  
\end{align*}
by either \cref{lemmar3} if ${\sf ht}^\mu(Q)=0$
 or application of the adjacent relation \eqref{adjacent} if ${\sf ht}^\mu(Q)>0$ (using the factorisation of \cref{L=lincombo}). In either case, the result again follows by induction on the degree.

\color{black}
\bigskip
\noindent{\bf Adding a Dyck Path. }
We now consider product where $X=\mathbb D(+Q)$ corresponding to adding a Dyck path $Q\in {\rm DAdd}(\mu)$ (to obtain $\nu=\mu\cup Q$).  There are three cases to consider:
$Q$ is adjacent to $i=0, 1,$ or 2 Dyck paths in $\mu\setminus \alpha$.

{\em The $i=0$ case. }
Assume that  $Q\in {\rm DAdd}(\mu)$  is adjacent to zero Dyck paths  in the Dyck tiling of  $\mu\setminus \alpha$.
In which case,    $(\mu\setminus \alpha) \cup Q$ is itself a Dyck tiling and 
\[
( \mathbb L (\alpha)
  \mathbb R_{\mu\setminus \alpha} 
\mathbb A_{\mu\setminus \alpha} \mathbbm  1_\mu)  \mathbb  D(+Q)
=
  \mathbb L (\alpha)
  \mathbb R_{\mu\setminus \alpha} 
(\mathbb A_{\mu\setminus \alpha}    \mathbb  D(+Q))\mathbbm  1_\nu
=   
\mathbb L (\alpha)
  \mathbb R_{\mu\setminus \alpha} 
\mathbb A_{\nu\setminus \alpha}  \mathbbm  1_\nu
=   
\mathbb L (\alpha)
  \mathbb R_{\nu\setminus \alpha} 
\mathbb A_{\nu\setminus \alpha}  \mathbbm  1_\nu
\]
is an element of the claimed spanning set (the   third equality follows immediately from  the definition of 
the elements, the second equality might also involve some application of the commutativity relations \eqref{rel3}). 

{\em The $i=1$ case. }
We now assume that  $Q$ is adjacent to one Dyck path  $H\in\mu\setminus \alpha$ 
 with ${\sf ht}^{\mu}_{\alpha}(H)=h\in \ZZ$. 
We first suppose that   $\langle H\cup Q  	\rangle_{\mu+Q}$ does not exist, in which case
  \begin{align*}
( \mathbb L (\alpha) 
  \mathbb R_{\mu\setminus \alpha} 
\mathbb A_{\mu\setminus \alpha}\mathbbm  1_\mu )
\mathbb  D(+Q)
=0
\end{align*}
using many applications of the commuting relation \eqref{rel3} and one 
application of either the adjacency relation \ref{adjacent} or \cref{loopadd-big1} 		for $h>0$ or $h\leq 0$ respectively.  We can now assume that    $\langle H\cup Q  	\rangle_{\mu+Q}$ does exist.
Our assumption that $H$ is the {\em unique} Dyck path in $\mu\setminus\alpha$ that is adjacent to  $Q$  
implies that 
  $\langle H\cup Q  	\rangle_{\mu+Q}   - Q  - H =H'  \in {\rm DRem}(\alpha)$. 
We   consider the case that $h\leq 0$ (the case that $h>0$ is almost identical) 
where  we have that   
 \begin{align*}
( \mathbb L (\alpha)
  \mathbb R_{\mu\setminus \alpha} 
\mathbb A_{\mu\setminus \alpha} {\mathbbm1}_\mu )
\mathbb  D(+Q)
 =& \textstyle   \mathbb L (\alpha)
   	\mathbb R ^{< h}_{\mu\setminus \alpha}  
 \mathbb D (-H) \mathbb D (+Q)
 (\prod _{j\neq 1} \mathbb D(-R^j_h))   	\mathbb R ^{> h}_{\mu\setminus \alpha}   
 \mathbb A_{\mu\setminus \alpha} \mathbbm 1_{\mu+Q}
\\
 =& \textstyle   \mathbb L (\alpha)
   	\mathbb R ^{< h}_{\mu\setminus \alpha}  
 \mathbb D (-H') \mathbb D (+(H\cup H' \cup Q))
 (\prod _{j\neq 1} \mathbb D(-R^j_h))   	\mathbb R ^{> h}_{\mu\setminus \alpha}   
 \mathbb A_{\mu\setminus \alpha} \mathbbm 1_{\mu+Q}
\\
=& \textstyle   \mathbb L (\alpha)
\mathbb D (-H') 
   	\mathbb R ^{< h}_{\mu\setminus \alpha}  
\mathbb D (+(H\cup H' \cup Q))
 (\prod _{j\neq 1} \mathbb D(-R^j_h))   	\mathbb R ^{> h}_{\mu\setminus \alpha}   
 \mathbb A_{\mu\setminus \alpha} \mathbbm 1_{\mu+Q}
\\
=& \textstyle \mathbb D (-H')   \mathbb L (\alpha-H')
   	\mathbb R ^{< h}_{\mu\setminus \alpha}  
\mathbb D (+(H\cup H' \cup Q))
 (\prod _{j\neq 1} \mathbb D(-R^j_h))   	\mathbb R ^{> h}_{\mu\setminus \alpha}   
 \mathbb A_{\mu\setminus \alpha} \mathbbm 1_{\mu+Q}
 \end{align*}
where the second equality follows from the
 adjacency relation \eqref{adjacent}  for $h\neq0$
 and   \cref{loopadd-big1} for $h= 0$; 
the first and third equalities follow
  by the commuting relations \eqref{rel3}; 
  the final equality follows by  \cref{h0loop_rect,h0loop_sq}.  
 Thus we conclude that 
 $
 ( \mathbb L (\alpha)
  \mathbb R_{\mu\setminus \alpha} 
\mathbb A_{\mu\setminus \alpha} {\mathbbm1}_\mu )
\mathbb  D(+Q) \in 
   \mathcal A_{m,n}^{\leq \alpha- H'}
$.

{\em The $i=2$ case. }
 It remains to consider the case that there are precisely
  two   Dyck paths adjacent to $Q$; these 
  must be of the same height $h\in \ZZ$ and we label them by $H'=A_1^h$ and $H''=A_2^h$.
If $h\in \ZZ_{>0}$, we  apply  the commutativity relation \eqref{rel3} and the adjacent relation \eqref{adjacent} and hence obtain the following  
\begin{align*} 
 \mathbb A_{\mu\setminus \alpha}  \mathbbm 1_\mu 
\mathbb  D (+Q)
=&\textstyle 
   	\mathbb A ^{< h}_{\mu\setminus \alpha}   
\mathbb  A ^h_{\mu\setminus \alpha}
  \mathbb D (+Q) 
   	\mathbb A ^{> h}_{\mu\setminus \alpha}     \mathbbm 1_{\mu+Q}
\\
=&\textstyle 
    	\mathbb A ^{< h}_{\mu\setminus \alpha}   
  \mathbb D (+H')  \mathbb D (+H'')    \mathbb D (+Q) 
  (\prod _{j\neq 1,2}\mathbb D(+A^h_j))
   	\mathbb A ^{> h}_{\mu\setminus \alpha}     \mathbbm 1_{\mu+Q}
\\
=&\textstyle 
     	\mathbb A ^{< h}_{\mu\setminus \alpha}   
  \mathbb D (+H')  \mathbb D (-H')    \mathbb D (+Q\cup H \cup H') 
  (\prod _{j\neq 1,2}\mathbb D(+A^h_j))
   	\mathbb A ^{> h}_{\mu\setminus \alpha}     \mathbbm 1_{\mu+Q}
\\
 =&\textstyle 
 \sum_{
\begin{subarray}c
 P  \in {\rm DRem}(\nu)
\end{subarray}
} 
c_P 
   	\mathbb A ^{< h}_{\mu\setminus \alpha}   
\mathbb    L  (-P) 
    \mathbb D (+Q\cup H \cup H') 
  (\prod _{j\neq 1,2}\mathbb D(+A^h_j))
   	\mathbb A ^{> h}_{\mu\setminus \alpha}     \mathbbm 1_{\mu+Q}
 \end{align*}
 for $\nu = \alpha \cup (\mu\setminus\alpha)_{<h}$  and $c_P \in \Bbbk $ some coefficients which can be calculated explicitly using the self-dual relation \eqref{rel2}.
Substituting  the above into the overall product, we deduce that 
$\textstyle ( \mathbb L (\alpha) 
  \mathbb R_{\mu\setminus \alpha} 
\mathbb A_{\mu\setminus \alpha}\mathbbm  1_\mu )
\mathbb  D(+Q)
 \in  \mathcal{A}_{m,n}^{<\alpha} $ by  induction on degree.

We now assume that $h \in \ZZ_{\leq 0}$ (and continue with out assumption that $Q$ is adjacent to two Dyck paths in $\mu\setminus\alpha$).
We note that $Q$ commutes past all Dyck paths 
 $A^j_k$ in   \eqref{The-As} 
 but is actually {\em equal to} some path $Q=R^h_1$   in \eqref{The-Rs}. 
This is because $Q$ is adjacent to two Dyck paths $H'$ and $H''$ of height $h\in \ZZ_{\leq 0}$
 and these paths were formed by first doing a loop and then splitting with the path $ R^h_1$.  
We hence   obtain the following
\begin{align*}
\textstyle ( \mathbb L (\alpha)
  \mathbb R_{\mu\setminus \alpha} 
\mathbb A_{\mu\setminus \alpha} \mathbbm 1_\mu) 
\mathbb  D(+Q)
=& \textstyle   \mathbb L (\alpha)
   	\mathbb R ^{< h}_{\mu\setminus \alpha}  
 \mathbb R ^h_{\mu\setminus \alpha} \mathbb D (+Q)
   	\mathbb R ^{> h}_{\mu\setminus \alpha}   
 \mathbb A_{\mu\setminus \alpha} \mathbbm 1_{\mu+Q}
\\
 =& \textstyle   \mathbb L (\alpha)
   	\mathbb R ^{< h}_{\mu\setminus \alpha}  
    \mathbb D(-Q) \mathbb D(+Q) \prod _{j\neq 1} \mathbb D(-R^j_h)
   	\mathbb R ^{> h}_{\mu\setminus \alpha}   
 \mathbb A_{\mu\setminus \alpha} \mathbbm 1_{\mu+Q}
\\
=& \textstyle   \mathbb L (\alpha)
   	\mathbb R ^{< h}_{\mu\setminus \alpha}  
    \mathbb L(-Q)  \prod _{j\neq 1} \mathbb D(-R^j_h)
   	\mathbb R ^{> h}_{\mu\setminus \alpha}   
 \mathbb A_{\mu\setminus \alpha} \mathbbm 1_{\mu+Q}
\end{align*}
where the first and second  equalities follow  from the commutation relations \eqref{rel3};
and the third equality  follows from \cref{L=lincombo}. 
Again, we deduce that $( \mathbb L (\alpha) 
  \mathbb R_{\mu\setminus \alpha} 
\mathbb A_{\mu\setminus \alpha}\mathbbm  1_\mu )
\mathbb  D(+Q)
 \in  \mathcal{A}_{m,n}^{<\alpha} $ by induction on degree.

\bigskip
\noindent{\bf Removing a Dyck Path. }
We now consider product where $X=\mathbb D(-Q)$ corresponding to removing a Dyck path $Q$. 
 There are three cases to consider: $(i)$
$Q =H \in \mu\setminus \alpha$ 
or $(ii)$  
$Q\not \in \mu\setminus \alpha$ and $Q$ commutes with   $ \mu\setminus \alpha$
 or $(iii)$ 
$Q\not \in \mu\setminus \alpha$ and 
$Q\prec H \in \mu\setminus \alpha$ for some $H$ which does not commute with $Q$ 
    $(iv)$ 
$Q\not \in \mu\setminus \alpha$ but 
$Q\succ H \in \mu\setminus \alpha$ 
 which does not commute with $Q$.

Case $(ii)$ is (again, as in the loop case)  not as simple as one might first think. 
 By assumption   $Q$  commutes with all of $\mu\setminus \alpha$.  
Thus   $Q$  commutes with   the
  $A^j_l$ in  \cref{The-As}  
and (as in the loop case) 
  $\mathbb D (-Q)$ commutes with   $\mathbb A_{\mu\setminus\alpha}$.  
Now, either $\mathbb D (-Q)$ commutes with all the   of the Dyck paths $R^i_k$  in \cref{The-Rs}
 or there exists a pair $H$ and $H'$  as in \cref{HandHprime}.  
 We have that 
 $$ 
 ( \mathbb L (\alpha)
  \mathbb R_{\mu\setminus \alpha} 
\mathbb A_{\mu\setminus \alpha} \mathbbm1_{\mu})
\mathbb  D (-Q) 
=
 \mathbb L (\alpha)  \mathbb  D (-Q) 
  \mathbb R_{\mu\setminus \alpha}
  \mathbb A_{\mu\setminus \alpha}   \mathbbm 1_\mu
=
 \mathbb  D (-Q)  \mathbb L (\alpha-Q)  
  \mathbb R_{\mu\setminus \alpha}
  \mathbb A_{\mu\setminus \alpha}   \mathbbm 1_\mu \in \mathcal{A}_{m,n}^{<\alpha}
 $$
 where the first equality follows from an identical argument to that used 
in case $(i)$ of the loop case (with  ${\sf ht}^{\mu}_{\alpha}(H)=h>0$); the second equality follows from  
  \cref{h0loop_rect,h0loop_sq} as $Q$ commutes with ${\sf reg}(\alpha)\setminus \alpha$.

Case $(iii)$. Suppose that 
$Q \not \in \mu\setminus\alpha$ but 
$Q \prec H \in \mu\setminus \alpha$ of height $h \in \ZZ$. 
In this case, we have that  $H\setminus Q= Q^1 \sqcup Q^2$ for some $Q^1$ and $Q^2$ and moreover that 
$ \nu \setminus \alpha$ for $\nu=\mu-Q$ is a Dyck tiling.  If ${\sf ht}^{\mu}_{\alpha}(H) > 0$ then we have that 
 \[
(   \mathbb L  (\alpha)
  \mathbb R_{\mu\setminus \alpha}  
    \mathbb A_{\mu\setminus \alpha}  \mathbbm 1_\mu)
\mathbb  D (-Q)
=
   \mathbb L  (\alpha)
  \mathbb R_{\mu\setminus \alpha}  
   \mathbb A_{(\mu-Q)\setminus \alpha}   \mathbbm 1_{\mu-Q}
=
   \mathbb L  (\alpha)
  \mathbb R_{(\mu-Q)\setminus \alpha}  
   \mathbb A_{(\mu-Q)\setminus \alpha}   
 \]
 where the first equality follows by the non-commuting relation \eqref{rel4} for the Dyck paths $H$ and $Q$  (together with the usual 
 commutation relations  \eqref{rel3}); the second equality follows as $  \mathbb R_{\mu\setminus \alpha}  =  \mathbb R_{(\mu-Q)\setminus \alpha}  $. 
If 
   ${\sf ht}^{\mu}_{\alpha}(H)\leq 0$ we have that 
$$
(   \mathbb L  (\alpha)
  \mathbb R_{\mu\setminus \alpha}  
    \mathbb A_{\mu\setminus \alpha}   \mathbbm 1_\mu)
\mathbb  D (-Q)
 =
    \mathbb L  (\alpha)
(  \mathbb R_{\mu\setminus \alpha}  \mathbb  D (-Q))
    \mathbb A_{\mu\setminus \alpha}   \mathbbm 1_{\mu-Q}
 =
   \mathbb L  (\alpha)
    \mathbb R_{(\mu-Q)\setminus \alpha}  
        \mathbb A_{\mu\setminus \alpha}    \mathbbm 1_{\mu-Q}
$$
 where the first equality follows by the  
 commutation relations  \eqref{rel3}; the second equality follows as 
 $ \mathbb R_{(\mu-Q)\setminus \alpha}  =  \mathbb R_{\mu\setminus \alpha}  \mathbb  D (-Q)$ (and  the  
 commutation relations  \eqref{rel3}) and 
 $  \mathbb A_{\mu\setminus \alpha}  =  \mathbb A_{(\mu-Q)\setminus \alpha}  $.

%
%

Case $(i)$. 
We suppose that $Q =H \in \mu\setminus \alpha$ for some ${\sf ht}^{\mu}_{\alpha}(H)=h\in\ZZ$.  
    We first  assume that   ${\sf ht}^{\mu}_{\alpha}(H)\leq 0$.
 First suppose that 
$H$ is the unique element of $(\mu\setminus\alpha)_h$.
In this case, $Q$ commutes with all the Dyck paths $A^j_k$, $R^i_l$ in \eqref{The-Rs} and  \eqref{The-As} and so  we have that 
 $$( \mathbb L (\alpha)
  \mathbb R_{\mu\setminus \alpha} 
\mathbb A_{\mu\setminus \alpha}  \mathbbm 1_\mu)
\mathbb  D (-Q)
=
 \mathbb L (\alpha)\mathbb  D (-Q)
  \mathbb R_{\mu\setminus \alpha} 
\mathbb A_{\mu\setminus \alpha}  \mathbbm 1_{\mu-Q} =0
 $$
 where the final equality follows by \cref{loopremove-big1} and 
 the fact that  $Q \in \mu \setminus\alpha$.
 We must now consider the case that 
 $ H  $ is not the unique Dyck path in $(\mu\setminus\alpha)_h$. 
 In which case, there exist Dyck path(s) 
$H':=R^h_1$ and possibly $H'':=R^h_2$ removed in the $h$ step (an example is depicted in \cref{figloopcase1}). 
Our assumptions  further  imply that   
 $Q$ is adjacent to $H'$ and  $H''$ (if the latter exists) and that $Q$ commutes with   every  other Dyck path $A^j_k$, $R^i_l$ in \eqref{The-Rs} and  \eqref{The-As}. 
 We set $P^h$ to be the unique Dyck path in ${\sf reg}(\alpha)\setminus \alpha$ of height $h$.
 We note that $P^h$ commutes with all Dyck paths 
 $R^i_k$ for $i<h$.  
 
 We first consider  the case that $H'  =R^h_1$ is the unique  Dyck path in \eqref{The-Rs} which is adjacent  to $Q$.  
We note that $H'\prec P^h$ is a non-commuting pair with 
 $P^h\setminus H' = Q \sqcup Q'$ for some $Q'$ of height $h$. 
 We have that 
\begin{align*} 
( \mathbb L (\alpha)
  \mathbb R_{\mu\setminus \alpha} 
\mathbb A_{\mu\setminus \alpha} \mathbbm 1_\mu)
\mathbb  D (-Q)
  =&\textstyle
  \mathbb L (\alpha)   	\mathbb R ^{< h}_{\mu\setminus \alpha}   
   \mathbb D (+H')
   \mathbb D (-Q)
   ( \prod_{j\neq 1}     \mathbb D(+    R ^h_j)  )
   	\mathbb R ^{> h}_{\mu\setminus \alpha}   
 \mathbb A_{\mu\setminus \alpha} \mathbbm 1_{\mu-Q}
\\
  =&\textstyle
    \pm  \mathbb L (\alpha)
     	\mathbb R ^{<h}_{\mu\setminus \alpha}   
  \mathbb D (-P^h ) 
       \mathbb D (+Q')
   ( \prod_{j\neq 1}     \mathbb D(+    R ^h_j)  )
   	\mathbb R ^{> h}_{\mu\setminus \alpha}   
 \mathbb A_{\mu\setminus \alpha} \mathbbm 1_{\mu-Q}
\\
    =&\textstyle
  \pm  \mathbb L (\alpha)  \mathbb D (-P^h ) 
     	\mathbb R ^{<h}_{\mu\setminus \alpha}   
       \mathbb D (+Q')
   ( \prod_{j\neq 1}     \mathbb D(+    R ^h_j)  )
   	\mathbb R ^{> h}_{\mu\setminus \alpha}   
 \mathbb A_{\mu\setminus \alpha} \mathbbm 1_{\mu-Q}
 \\
 =&0 
  \end{align*}
where the second equality follows from  applying relation \eqref{adjacent} to the pair $H'$, $Q$ of adjacent Dyck paths;
 the final equality follows from \cref{loopremove-big1} (as $P^h \in {\sf reg}(\alpha)\setminus \alpha$); all other equalities  follow from the commuting relations \eqref{rel3}.

  Continuing with the case  that $h\in \ZZ_{\leq 0}$, 
we now   suppose that  $H'=R^h_1$ and $H''=R^h_2$ are the two  Dyck paths in \eqref{The-Rs} which are adjacent  to $Q$. 
We note that $H', H''\prec P^h$ are non-commuting pairs with 
 $P^h\setminus H' = Q \sqcup Q'$ and 
 $P^h\setminus H'' = Q' \sqcup Q''$
  for some $Q', Q''$ of height $h$. 
 We have that 
\begin{align*} 
 ( \mathbb L (\alpha)
  \mathbb R_{\mu\setminus \alpha} 
\mathbb A_{\mu\setminus \alpha} \mathbbm 1_\mu)
\mathbb  D (-Q)
  =& \textstyle
  \mathbb L (\alpha)   	\mathbb R ^{< h}_{\mu\setminus \alpha}   
   \mathbb D(+H')   \mathbb D(+H'')
   \mathbb D (-Q)
    ( \prod_{j\neq 1,2}     \mathbb D(+    R ^h_j)  )
   	\mathbb R ^{> h}_{\mu\setminus \alpha}   
 \mathbb A_{\mu\setminus \alpha} \mathbbm 1_{\mu-Q}
\\
  =&\textstyle
  \pm 
   \mathbb L (\alpha)
     	\mathbb R ^{<h}_{\mu\setminus \alpha}   
  \mathbb D (-P^h ) 
        \mathbb D (+Q')       \mathbb D (+Q'')
   ( \prod_{j\neq 1,2}     \mathbb D(+    R ^h_j)  )
   	\mathbb R ^{> h}_{\mu\setminus \alpha}   
 \mathbb A_{\mu\setminus \alpha} \mathbbm 1_{\mu-Q}
\\
    =&\textstyle
  \pm    \mathbb L (\alpha)  \mathbb D (-P^h ) 
     	\mathbb R ^{<h}_{\mu\setminus \alpha}   
       \mathbb D (+Q')       \mathbb D (+Q'')
   ( \prod_{j\neq 1,2}     \mathbb D(+    R ^h_j)  )
   	\mathbb R ^{> h}_{\mu\setminus \alpha}   
 \mathbb A_{\mu\setminus \alpha} \mathbbm 1_{\mu-Q}
 \\
  =&0 
  \end{align*}
where the  second equality follows from two applications of the  adjacent relation \eqref{adjacent};
 the final equation follows from \cref{loopremove-big1}; all others follow from the commuting relations \eqref{rel3}.

If   $h>0$ then $H=A^1_h$  for some $A^1_h$ as in  \cref{The-As}.  
Therefore we have that 
\begin{align*} 
\phantom{  = }
(\mathbb L (\alpha)
   \mathbb R_{\mu\setminus \alpha} 
 \mathbb A _{\mu\setminus \alpha}
 \mathbbm 1_\mu)
  \mathbb D(-Q)
   &\textstyle
  =
  \mathbb L (\alpha)
   \mathbb R_{\mu\setminus \alpha} 
     	\mathbb A ^{< h}_{\mu\setminus \alpha}   
 \mathbb D (+Q)   \mathbb D (-Q) 
 (  \prod _{j\neq 1} \mathbb D(+A^j_h))
   	\mathbb A ^{> h}_{\mu\setminus \alpha}     \mathbbm 1_{\mu-Q}
\\   &\textstyle
  =
  \sum_{P \in {\rm DRem}(\nu)} c_P
  \mathbb L (\alpha)
   \mathbb R_{\mu\setminus \alpha} 
     	\mathbb A ^{< h}_{\mu\setminus \alpha}   
 \mathbb L(-P)
 (  \prod _{j\neq 1} \mathbb D(+A^j_h))
   	\mathbb A ^{> h}_{\mu\setminus \alpha}     \mathbbm 1_{\mu-Q}\end{align*}
for 
$\nu=\alpha\cup (\mu\setminus\alpha)_{<h}$ and  $c_P \in 
\Bbbk$  coefficients which can be calculated explicitly using the self-dual relation \eqref{rel2}.   Thus $(\mathbb L (\alpha)
   \mathbb R_{\mu\setminus \alpha} 
 \mathbb A _{\mu\setminus \alpha}
 \mathbbm 1_\mu)  \mathbb D(-Q)\in \mathcal{A}_{m,n}^{<\alpha}
$, by induction on degree.

Finally, suppose we are in case $(iv)$ and we assume that 
$Q$ is not as in case $(ii)$  or  $(iii)$, an example is depicted in \cref{Case 4}. 
Our assumptions imply that 
 $Q\setminus H=H^1 \sqcup H^2$.
Recall our assumption that $H \prec Q \not \preceq H' \in \nu \setminus \alpha$ unless 
  $H'$ commutes with $Q$;
  this assumption  further implies that 
 $H^1$ and $H^2$ commute with all 
  of the Dyck paths $R^i_k$, $A^j_l$ in \eqref{The-Rs} and  \eqref{The-As} for the pair $\nu \setminus \alpha$. 
  Therefore   
$$
\mathbb L (\alpha)
\mathbb R_{ \mu \setminus \alpha} \mathbb A_{ \nu \setminus \alpha} 
 \mathbb D(-H^1)
 \mathbb D(-H^2)\mathbbm 1_\nu 
   =
  \mathbb D(-H^1)
  \mathbb D(-H^2)
\mathbb L (\alpha-H^1-H^2)
\mathbb R_{ \nu \setminus \alpha} \mathbb A_{ \nu \setminus \alpha} \in \mathcal{A}_{m,n}^{<\alpha}
$$
by applying the  non-commuting relation \eqref{rel4} to $\mathbb D(+H) \mathbb D(-Q)$ in
 the relevant place and then the  commuting relations \eqref{rel3} and \cref{h0loop_rect,h0loop_sq}  for all other products.
   \end{proof}

%
%
%
%
%
%

\subsection{The isomorphism theorem}

Having established the spanning set of the algebra $\mathcal{A}_{m,n} $ in the previous section, 
we are now ready to prove the main result of the paper.
 It will be convenient  to set 
 $$
   D (-Q) := 
 \sum_{
 \begin{subarray}c
  {\la\in  \mathscr{R}_{m,n}}
  \\
 Q \in {\rm DRem}(\la)
 \end{subarray}
 }\!\!\!
 { \sf  1}_\la   D^\la_{\la-Q}
 \qquad
   D (+Q) := 
\sum_{
 \begin{subarray}c
  {\la\in  \mathscr{R}_{m,n}}
  \\
Q \in {\rm DAdd}(\la)
 \end{subarray}
 }\!\!\!
 { \sf  1}_\la   D^\la_{\la+Q} 
 $$

\begin{thm}
For $m\leq n$ the map 
 $
\varphi : \mathcal{A}_{m,n} \to H^m_n
 $
 given by
 $$
  \varphi(\mathbbm 1_\mu) = {\sf 1}_\mu 
  \qquad
 \varphi(\mathbb D^\la_\mu)=D^\la_\mu
 \qquad
  \varphi(\mathbb L^\la_\la(-P))=
  (-1)^{b(P)} 
  {\sf 1}_\la D(-P)D(+P)
 $$
 for $\la,\mu \in \mathcal{R}_{m,n}$ 
 is a $\ZZ$-graded $\Bbbk$-algebra isomorphism. 
\end{thm}

\begin{proof} 
We first verify that the map is a $\Bbbk$-algebra homomorphism. 
Clearly  \eqref{rel1}, \eqref{rel2}, \eqref{rel3}, \eqref{rel4}, \eqref{adjacent} hold in $ K^m_n$ as they are verbatim the same (just change the font!). 
We must check that the cubic relation \eqref{cubic} and loop relations \eqref{loop-relation} hold. 

We first check the cubic relation. 
Let $P \in {\rm DAdd}_1(\mu)$ be such ${\sf last}(P)$ is maximal with respect to this property. 
Since $P \in {\rm DAdd}_1(\mu)$, we have that 
   $$\{Q \mid Q \in {\rm DRem}(\mu) \text { and } P\prec Q \}=\emptyset.$$If ${\sf last}(P)\leq m-2$, then we let $Q\in {\rm DRem}_0(\mu+P)$ be the unique element such that $P\prec Q$
 and we set $Q\setminus P=Q^1\sqcup Q^2$. 
Otherwise, 
${\sf last}(P)= m$ and there does not exist any $P\prec Q$ and we set $Q^1\in {\rm DRem}_0(\mu)$ to be the unique element that is adjacent to $P$. 
We have that 
\begin{align}\notag
&\phantom{=}\varphi(
\mathbb   D_{\mu}^{\mu+P}
\mathbb  D^{\mu}_{\mu+P}
\mathbb   D_{\mu}^{\mu+P}
)
\\
\notag
&=
 D(-P)
  D(+P)
  D(-P)
{\sf  1}_\mu 
\\
\label{label1}
&=\begin{cases} 
(-1)^{b(P)+1}
 D(-P)
( D(-Q^1) D(+Q^1)
+D(-Q^2) D(+Q^2)){\sf  1}_\mu  &\text{if ${\sf last}(P) \leq m-2$}\\
(-1)^{b(P)+1}
 D(-P)
  D(-Q^1) D(+Q^1){\sf  1}_\mu 
&\text{if ${\sf last}(P)=m $}
\end{cases}
 \end{align}
 where the second equality follows from the self-dual relation \eqref{selfdualrel}.  
Note that $\mu \in \mathscr{R}_{m,n}$,
   $P \in {\rm DAdd}_1(\mu)$, and ${\sf last}(P)$ being 
    maximal with respect to this property 
 together imply that 
$(i)$
 ${\sf last}(P)\leq m-2$ if and only if $m=n$
$(ii)$  
 ${\sf last}(P)= m $ if and only if $m<n$ as stated in the cubic relation.

We consider the first case of \eqref{label1} where (by the above) $m=n$. 
 For $m=n$ we observe that $ \langle P \cup Q^1\rangle_{\mu+P} =\langle P \cup Q^2\rangle_{\mu+P}$ 
 and we denote this Dyck path by $Q$. 
We have that 
 $$
 D(-P)
( D(-Q^1) D(+Q^1)
+D(-Q^2) D(+Q^2)){\sf  1}_\mu =
\pm 2 
 D(-Q)D(+Q)
 D(-P)
{\sf  1}_\mu 
 $$by applying the adjacent relation \eqref{adjacent-OG} first to the pairs $(P,Q^1)$ and $(P,Q^2)$ followed by   the commuting relation to $Q^1, Q^2$. We note that ${\sf last}(Q)=m-1$ and so  
relation  \eqref{adjacent} is preserved by $\varphi$ for $m=n$.
 
 We consider the second case of \eqref{label1} where (by the above) $m<n$. 
  For $m<n$ we have that $\langle P \cup Q^1\rangle_{\mu+P}$ does not exist and so 
  $  D(-P) D(-Q^1) D(+Q^1){\sf  1}_\mu =0 $ and so  
relation  \eqref{adjacent} is preserved by $\varphi$ for $m<n$.

It remains to verify that the loop relations \eqref{loop-relation} are preserved by $\varphi$ for $m=n$.  
Let $P \in {\rm DRem}_0(\mu)$ with ${\sf last}(P)=m-1$. 
We first check the loop nilpotency relation.
We have that
 \begin{align*}
\varphi(\mathbb   L_{\mu}^{\mu}(-P)
 \mathbb   L_{\mu}^{\mu}(-P))
&=
{\sf 1}_\mu
D(-P)
D(+P)
D(-P)
D(+P)
=0
\end{align*}by applying the self-dual relation \eqref{selfdualrel} to the innermost pair and observing that
since  $P\in {\rm Add}_0(\mu-P)$ and $\mu \in \mathcal{R}_{m,m}$ this implies that there are no 
$Q\in  {\rm DRem}_{<0}(\mu-P)$ in order to provide non-zero terms on the righthand-side of the
 self-dual relation. Thus the loop nilpotency relation holds. 

Finally, we must verify the loop-commutation relation.  We continue to let $P \in {\rm DRem}_0(\mu)$ with ${\sf last}(P)=m-1$. We note that if $\la=\mu\pm Q$ where $Q$ commutes with $P$ then 
$\varphi$ preserves the loop relation trivially. 
 Thus it remains to consider the cases $(i)$ $\la=\mu-Q$ with $Q \prec P$ a non-commuting pair and $(ii)$ $\la= \mu +Q$ with $P,Q$ an adjacent pair. 
In case $(i)$ we note that 
$P\setminus Q=Q^1\sqcup Q^2$ where ${\sf last}(Q^2)=m-1$. Therefore we have that 
\begin{align*}
\varphi(\mathbb D^\la_\mu \mathbb L_\mu^\mu)&=
\varphi(\mathbb D^\la_\mu \mathbb L_\mu^\mu(-P))
\\
&=
  (-1)^{b(P)} 
{\sf 1}_\la  D (+Q) D (-P)  D (+P)
\\
&=
  (-1)^{b(P)} 
  {\sf 1}_\la  D (-Q^2) D (-Q^1)  D (+P)
\\
&=
(-1)^{b(Q^2)}
{\sf 1}_\la  D (-Q^2) D (+Q^2)  D (+Q)
 \\
&=
\varphi(\mathbb   L_\la^\la(-Q^2) \mathbb D^\la_\mu )
\\
&=\varphi(\mathbb   L_\la^\la \mathbb D^\la_\mu )
\end{align*}
where the third  equality follows from the non-commuting
 relation \eqref{noncommuting} applied to $Q,P$, 
the fourth equality follows from the adjacent relation \eqref{adjacent-OG} applied to $Q^2, Q$. 
Case $(ii)$ follows from a similar argument.

We must now check that the homomorphism is surjective.  
This will immediately imply that the spanning set of \cref{the spanning set}
is actually a basis (as the spanning set of $\mathcal{A}_{m,n}$ has the same size
 as the basis of $K^m_n$)
 and we  hence will deduce that  $\varphi$ is an isomorphism.  
We need only show that the algebra $K^m_n$ is generated by the 
elements $  {\sf 1}_\mu $, 
$D^\la_\mu$ and 
$ D^\la_{\la-P}D_\la^{\la-P}$ for $\la,\mu \in \mathcal{R}_{m,n}$ and $P \in {\rm DRem}(\la)$. 
Thus it suffices to write the cellular basis in terms of these claimed generators which we have already done in \cref{basis22}. The result follows.
\end{proof}

  \section{The $\rm Ext$-quiver and relations of $e K^m_ne$ \\ and the Kleshchev--Martin conjecture.}

We have shown that the algebra $H^m_n=e K^m_ne$ is a basic algebra
 generated in degrees 0, 1, and 2.  Thus for $\la \subseteq \mu$,  we have that 
\begin{align}
\label{radical}
\dim_\Bbbk({\rm Ext}^1_{H^m_n}(D(\la), D(\mu))) = 
\dim_\Bbbk({\rm Hom}_{H^m_n}
(\rad_1(P(\la)e), D(\mu)\langle k \rangle )) =
 0
 \intertext{unless $k=1,2$. 
 By the cellular self-duality, we have that }
  \dim_\Bbbk({\rm Ext}^1_{H^m_n}(D(\la), D(\mu)\langle k \rangle )) = 
\dim_\Bbbk({\rm Ext}^1_{H^m_n}(D(\mu), D(\la)\langle k \rangle )) 
\end{align}
and so we will be able to focus solely on the first and second grading 
layers of the projective  $  {H^m_n}$-module $P(\la)e$ for each  $\la  \in  \mathscr{R}_{m,n}$.

\begin{lem}\label{asbefore}
For $\la\neq \mu \in  \mathscr{R}_{m,n}$, we have that 
$$\dim_\Bbbk({\rm Ext}^1_{H^m_n}(D(\la), D(\mu))) = \begin{cases}
1 &\text{if $\la = \mu \pm P$ for $P$ a Dyck path};
\\
0	&\text{otherwise.}
\end{cases} $$
\end{lem}
\begin{proof}
If $\la =\mu \pm P$, then 
$\dim_\Bbbk({\rm Hom}_{H^m_n}
(\rad_1(P(\la)e), D(\mu)\langle 1 \rangle )) =1$
 and by the parity on the grading, we have that 
$\dim_\Bbbk({\rm Hom}_{H^m_n}
(\rad_1(P(\la)e), D(\mu)\langle 2 \rangle )) =0$. 
Now suppose that $\mu\neq \la\pm P$. 
By \cref{radical}, we can assume that $\la \subseteq \mu$ is such that 
$(i)$ $\mu=\nu+ P $ for $\nu=\la-Q$ and $\nu \in  \mathscr{P}_{m,n}$
or $(ii)$ $\mu=\nu+ P $ for $\nu=\la+Q$ and $\nu \in  \mathscr{R}_{m,n}$.
Case $(ii)$ is trivial. 
  Case $(i)$ is non-trivial only when $\nu \not\in  \mathscr{R}_{m,n}$ 
  (equivalently $Q \in {\rm DRem}_0(\la)$)
   in which case our assumption that $\mu \in  \mathscr{R}_{m,n}$ implies that 
   $P=Q$ and so $\la=\mu$, as required.
\end{proof}

 This is already enough to deduce  the ${\rm Ext}$-quiver for $m\neq n$.

\begin{thm}
Let $m \neq n$.  
The $\rm Ext$-quiver  of $H^m_n$ has       vertex set 
$\{\mathbbm 1_\la \mid \la \in   \mathscr{R}_{m,n} \}$  and 
  arrows   
    $\mathbb D^\la_\mu: \la   \to  \mu$ 
and 
   $\mathbb D^\mu_\la:\mu \to \la$ for every $\la=\mu - P$ with $P\in {\rm DRem}_{>0}(\mu)$.  
   The symmetric algebra  $H^m_n$ is  the path algebra of this quiver modulo relations \eqref{rel1}, \eqref{rel2}, \eqref{rel3}, \eqref{rel4}, \eqref{adjacent} and \eqref{cubic}.
\end{thm}

%
%
%

    For the proofs of the last two theorems, we will first need to enumerate the removable Dyck paths of height zero from left to right as follows
\[
P^1, P^2, \dots ,P^r \in {\rm DRem}_0(\la).
\]
For each pair $1\leq j<k\leq r$ we let 
$Q_{j,k}\in {\rm DAdd}_1(\la)$ denote the addable Dyck path which is adjacent to 
$P^j$ and $P^k$.  Examples are depicted in \cref{dyckeg!}.

\begin{figure}[ht!]
$$  \begin{tikzpicture}[scale=0.5]
  \path(0,0)--++(135:2) coordinate (hhhh);
 \draw[very thick] (hhhh)--++(135:5)--++(45:5)--++(-45:5)--++(-135:5);
 \clip (hhhh)--++(135:5)--++(45:5)--++(-45:5)--++(-135:5);
\path(0,0) coordinate (origin2);
  \path(0,0)--++(135:2) coordinate (origin3);

      \foreach \i in {0,1,2,3,4,5,6,7,8,9,10}
{
\path (origin3)--++(45:0.5*\i) coordinate (c\i); 
\path (origin3)--++(135:0.5*\i)  coordinate (d\i); 
  }

\path(0,0)--++(135:2) coordinate (hhhh);

\fill[opacity=0.2](hhhh)
--++(135:4)   coordinate (JJ1)
--++(45:1)
--++(-45:1)	 coordinate (JJ)	--++(45:1)
--++(-45:1)    --++(45:1)
--++(-45:1)coordinate (JJ2) --++(45:1)
--++(-45:1);

\fill[opacity=0.4,magenta] 
  (JJ1)
--++(135:1)
 --++(45:1)
--++(-45:1) 
;

\fill[opacity=0.7,yellow!80!orange] 
  (JJ)
--++(135:1)
 --++(45:2)
--++(-45:2)--++(-135:1)--++(135:1)
;

\fill[opacity=0.4,cyan] 
  (JJ2)
--++(135:1)
 --++(45:2)
--++(-45:2)--++(-135:1)--++(135:1)
;

 \path(JJ1)--++(135:0.5)--++(45:1.5) coordinate (orange);

 \fill[orange] (orange) circle (7pt);

   \foreach \i in {0,1,2,3,4,5,6,7,8,9,10}
{
\path (origin3)--++(45:1*\i) coordinate (c\i); 
\path (c\i)--++(-45:0.5) coordinate (c\i); 
\path (origin3)--++(135:1*\i)  coordinate (d\i); 
\path (d\i)--++(-135:0.5) coordinate (d\i); 
\draw[thick,densely dotted] (c\i)--++(135:12);
\draw[thick,densely dotted] (d\i)--++(45:12);
  }

\end{tikzpicture}
\qquad
  \begin{tikzpicture}[scale=0.5]
  \path(0,0)--++(135:2) coordinate (hhhh);
 \draw[very thick] (hhhh)--++(135:5)--++(45:5)--++(-45:5)--++(-135:5);
 \clip (hhhh)--++(135:5)--++(45:5)--++(-45:5)--++(-135:5);
\path(0,0) coordinate (origin2);
  \path(0,0)--++(135:2) coordinate (origin3);

      \foreach \i in {0,1,2,3,4,5,6,7,8,9,10}
{
\path (origin3)--++(45:0.5*\i) coordinate (c\i); 
\path (origin3)--++(135:0.5*\i)  coordinate (d\i); 
  }

\path(0,0)--++(135:2) coordinate (hhhh);

\fill[opacity=0.2](hhhh)
--++(135:4)   coordinate (JJ1)
--++(45:1)
--++(-45:1)	 coordinate (JJ)	--++(45:1)
--++(-45:1)    --++(45:1)
--++(-45:1)coordinate (JJ2) --++(45:1)
--++(-45:1);

\fill[opacity=0.4,magenta] 
  (JJ1)
--++(135:1)
 --++(45:1)
--++(-45:1) 
;

\fill[opacity=0.7,yellow!80!orange] 
  (JJ)
--++(135:1)
 --++(45:2)
--++(-45:2)--++(-135:1)--++(135:1)
;

\fill[opacity=0.4,cyan] 
  (JJ2)
--++(135:1)
 --++(45:2)
--++(-45:2)--++(-135:1)--++(135:1)
;

 \path(JJ1)--++(135:0.5)--++(45:1.5) coordinate (orange);

 \fill[violet] (orange) circle (7pt);
\draw[very thick , violet] (orange)--++(45:1) coordinate (orange) ;
  \fill[violet] (orange) circle (7pt);
  \draw[very thick , violet] (orange)--++(45:1) coordinate (orange) ;
  \fill[violet] (orange) circle (7pt);
  \draw[very thick , violet] (orange)--++(-45:1) coordinate (orange) ;
  \fill[violet] (orange) circle (7pt);
  \draw[very thick , violet] (orange)--++(-45:1) coordinate (orange) ;
  \fill[violet] (orange) circle (7pt);

   \foreach \i in {0,1,2,3,4,5,6,7,8,9,10}
{
\path (origin3)--++(45:1*\i) coordinate (c\i); 
\path (c\i)--++(-45:0.5) coordinate (c\i); 
\path (origin3)--++(135:1*\i)  coordinate (d\i); 
\path (d\i)--++(-135:0.5) coordinate (d\i); 
\draw[thick,densely dotted] (c\i)--++(135:12);
\draw[thick,densely dotted] (d\i)--++(45:12);
  }

\end{tikzpicture}\qquad
  \begin{tikzpicture}[scale=0.5]
  \path(0,0)--++(135:2) coordinate (hhhh);
 \draw[very thick] (hhhh)--++(135:5)--++(45:5)--++(-45:5)--++(-135:5);
 \clip (hhhh)--++(135:5)--++(45:5)--++(-45:5)--++(-135:5);
\path(0,0) coordinate (origin2);
  \path(0,0)--++(135:2) coordinate (origin3);

      \foreach \i in {0,1,2,3,4,5,6,7,8,9,10}
{
\path (origin3)--++(45:0.5*\i) coordinate (c\i); 
\path (origin3)--++(135:0.5*\i)  coordinate (d\i); 
  }

\path(0,0)--++(135:2) coordinate (hhhh);

\fill[opacity=0.2](hhhh)
--++(135:4)   coordinate (JJ1)
--++(45:1)
--++(-45:1)	 coordinate (JJ)	--++(45:1)
--++(-45:1)    --++(45:1)
--++(-45:1)coordinate (JJ2) --++(45:1)
--++(-45:1);

\fill[opacity=0.4,magenta] 
  (JJ1)
--++(135:1)
 --++(45:1)
--++(-45:1) 
;

\fill[opacity=0.7,yellow!80!orange] 
  (JJ)
--++(135:1)
 --++(45:2)
--++(-45:2)--++(-135:1)--++(135:1)
;

\fill[opacity=0.4,cyan] 
  (JJ2)
--++(135:1)
 --++(45:2)
--++(-45:2)--++(-135:1)--++(135:1)
;

 \path(JJ2)--++(45:0.5)--++(135:1.5) coordinate (orange);

 \fill[darkgreen] (orange) circle (7pt);
        
   \foreach \i in {0,1,2,3,4,5,6,7,8,9,10}
{
\path (origin3)--++(45:1*\i) coordinate (c\i); 
\path (c\i)--++(-45:0.5) coordinate (c\i); 
\path (origin3)--++(135:1*\i)  coordinate (d\i); 
\path (d\i)--++(-135:0.5) coordinate (d\i); 
\draw[thick,densely dotted] (c\i)--++(135:12);
\draw[thick,densely dotted] (d\i)--++(45:12);
  }

\end{tikzpicture}
$$
\caption{We picture the three removable Dyck paths ${\color{magenta}P^1}, 
{\color{yellow!60!orange}P^2}, {\color{cyan}P^3} \in {\rm DRem}_0(\la).$ 
The corresponding $3 \choose 2$ addable Dyck paths  
$\color{orange!70!brown}Q_{1,2}$, 
$\color{violet}Q_{1,3}$, 
and 
$\color{darkgreen}Q_{1,2}$
in ${\rm DAdd}_1(\la)$ are also pictured. 
}
\label{dyckeg!}
\end{figure}
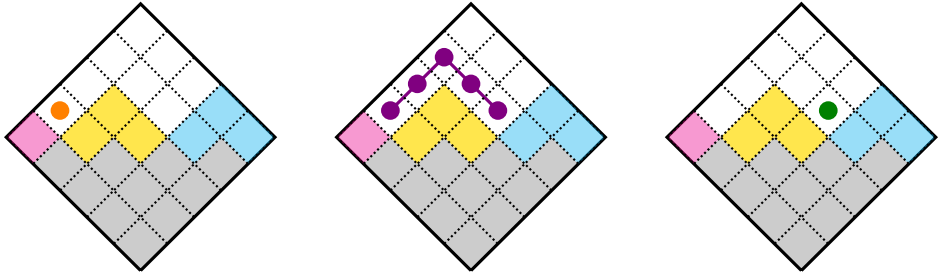

With this notation in place, we can rewrite the self-dual relation for these Dyck paths as follows:
\begin{equation}\label{oohcolours}
\mathbb  D_{{\la+\color{violet}Q_{j,k}}}^{\la} \mathbb  D_{\la}^{{\la+\color{violet}Q_{j,k}}}
=
\mathbb  L^\la_\la (- {\color{magenta}P^j}) 
+
\mathbb  L^\la_\la (- {\color{cyan}P^k}) 
\end{equation}
We notice that the righthand-side consists of degree 2 terms which factor through an idempotent 
 labelled by a non-regular partition.  
 If we can rewrite the  terms on the righthand-side using solely products  of the form on  the lefthand-side, then  these loops will not appear in our ${\rm Ext}$-quiver.  This is best illustrated via  examples.

\begin{eg}
Let $p>2$ and $\la= (5,4^2,2^3)$ as in \cref{dyckeg!}.
We have that 
\begin{align*} 
\mathbb  D_{{\la+\color{orange!70!brown}Q_{1,2}}}^{\la} \mathbb  D_{\la}^{{\la+\color{orange!70!brown}Q_{1,2}}}
=
\mathbb  L^\la_\la (- {\color{magenta}P^1}) 
+
\mathbb  L^\la_\la (- {\color{yellow!60!orange}P^2}) 
\\
\mathbb  D_{{\la+\color{violet}Q_{1,3}}}^{\la} \mathbb  D_{\la}^{{\la+\color{violet}Q_{1,3}}}
=
\mathbb  L^\la_\la (- {\color{magenta}P^1}) 
+
\mathbb  L^\la_\la (- {\color{cyan}P^3}) 
\\
\mathbb  D_{{\la+\color{darkgreen}Q_{2,3}}}^{\la} \mathbb  D_{\la}^{{\la+\color{darkgreen}Q_{2,3}}}
=
\mathbb  L^\la_\la (- {\color{yellow!60!orange}P^2}) 
+
\mathbb  L^\la_\la (- {\color{cyan}P^3}) 
\end{align*}
and inverting this we obtain 
\begin{align*} 
\mathbb  L^\la_\la (- {\color{magenta}P^1})  
&=
\tfrac{1}{2}\left(\mathbb  D_{{\la+\color{orange!70!brown}Q_{1,2}}}^{\la} \mathbb  D_{\la}^{{\la+\color{orange!70!brown}Q_{1,2}}}+
\mathbb  D_{{\la+\color{violet}Q_{1,3}}}^{\la} \mathbb  D_{\la}^{{\la+\color{violet}Q_{1,3}}}-
\mathbb  D_{{\la+\color{darkgreen}Q_{2,3}}}^{\la} \mathbb  D_{\la}^{{\la+\color{darkgreen}Q_{2,3}}}\right)
\\ 
\mathbb  L^\la_\la (- {\color{yellow!60!orange}P^2})  
&=
\tfrac{1}{2}\left(\mathbb  D_{{\la+\color{orange!70!brown}Q_{1,2}}}^{\la} \mathbb  D_{\la}^{{\la+\color{orange!70!brown}Q_{1,2}}}-
\mathbb  D_{{\la+\color{violet}Q_{1,3}}}^{\la} \mathbb  D_{\la}^{{\la+\color{violet}Q_{1,3}}}+
\mathbb  D_{{\la+\color{darkgreen}Q_{2,3}}}^{\la} \mathbb  D_{\la}^{{\la+\color{darkgreen}Q_{2,3}}}\right)
\\ 
\mathbb  L^\la_\la (- {\color{cyan}P^3})  
&=
\tfrac{1}{2}\left(-\mathbb  D_{{\la+\color{orange!70!brown}Q_{1,2}}}^{\la} \mathbb  D_{\la}^{{\la+\color{orange!70!brown}Q_{1,2}}}+
\mathbb  D_{{\la+\color{violet}Q_{1,3}}}^{\la} \mathbb  D_{\la}^{{\la+\color{violet}Q_{1,3}}}+
\mathbb  D_{{\la+\color{darkgreen}Q_{2,3}}}^{\la} \mathbb  D_{\la}^{{\la+\color{darkgreen}Q_{2,3}}}\right).
  \end{align*}
Therefore we deduce that none of the loops labelled by 
 ${\color{magenta}P^1}, 
{\color{yellow!60!orange}P^2}, {\color{cyan}P^3}$ are required in the ${\rm Ext}$-quiver 
of $e K^m_ne$, as they can be written as  linear combinations of other paths.

\end{eg}

 In fact, this is the most complicated thing we have to deal with in constructing these ${\rm Ext}$-quivers.  Therefore the following lemma will be used repeatedly during the upcoming proofs.
 
\begin{lem}\label{kjsahgljkdsfhgjlkdfshgjfdshgljkdfshg}
Let $\Bbbk$ be a field of characteristic $p>2$.  Then 
$$
\left(\begin{array}{ccc}1 & 1 & 0 \\1 & 0 & 1 \\0 & 1 & 1\end{array}\right)^{-1}
=\frac{1}{2}
 \left(\begin{array}{ccc} 1 & 1 & -1 \\1 & -1 & 1 \\-1 & 1 & 1\end{array}\right)
$$and otherwise the matrix on the left is not invertible.
\end{lem}

%
%
%
%

\begin{thm}\label{pnot2}
Let   $p\neq 2$.  
The $\rm Ext$-quiver  of $H^m_m$ has      vertex set 
$\{\mathbbm 1_\la \mid \la \in   \mathscr{R}_{m,m} \}$  and 
  arrows   
    $\mathbb D^\la_\mu: \la   \to  \mu$  
and 
   $\mathbb D^\mu_\la:\mu \to \la$ for every $\la=\mu - P$ with $P\in {\rm DRem}_{>0}(\mu)$ 
   together with the loops 
      $\mathbb L^\la_\la:\la \to \la$ for any $\la=(m^a, (m-a)^{m-a}) $ for $1\leq a \leq m$. 
   The symmetric algebra  $H^m_m$ is given by the path algebra of this quiver modulo relations \eqref{rel1}, \eqref{rel2}, \eqref{rel3}, \eqref{rel4}, \eqref{adjacent}, \eqref{cubic}, and \eqref{loop-relation}.
\end{thm}

\begin{proof}
We first suppose that  $\la \in \mathscr{R}_{m,m}$ is such that 
$\la\neq (m^a, (m-a)^{m-a}) $ for some $1\leq a \leq m$.  
There are two subcases to consider, either $|{\rm DAdd}_{>1}(\la)|>0$ 
or $|{\rm DAdd}_{1}(\la)|>1$.
We first suppose that there exists $Q\in {\rm DAdd}_{>1}(\la)$, in which case there exists a unique 
$Q\prec P \in {\rm DRem}_0 (\la)$ and 
\[
 \mathbb  L^\la_\la (-  P) 
  =
\tfrac{1}{2}\mathbb  D^{\la}_{\la+Q} \mathbb D_{\la}^{\la+Q} 
+
\sum_{Q\prec R \prec P} \alpha_R D^{\la}_{\la-R} \mathbb D_{\la}^{\la-R} 
\]where the $ \alpha_R \in \Bbbk$ can be determined explicitly with the self-dual relation.
If $P$   is the unique element of $ {\rm DRem}_0(\la)$ then we are done.
Otherwise, let 
$P^1,P^2,\dots, P^r$ denote all removable Dyck paths of height zero with $P^j=P$ for some $1\leq j \leq r$. 
For $i< j <k$, we have that 
$$
 \mathbb  L^\la_\la (- {\color{magenta} P^i}) =
\mathbb  D^{\la}_{\la+{\color{orange}Q_{i,j}}}  \mathbb D_{\la}^{\la+{\color{orange}Q_{i,j}}} 
-\mathbb  L^\la_\la (-  {\color{yellow!80!orange}P^j}) 
\qquad 
 \mathbb  L^\la_\la (-  {\color{cyan}P^k}) =
\mathbb  D^{\la}_{\la+{\color{darkgreen}Q_{j,k}}}  \mathbb D_{\la}^{\la+{\color{darkgreen}Q_{j,k}}} 
-\mathbb  L^\la_\la (-   {\color{yellow!80!orange}P^j}) 
$$and so we can rewrite every loop as a linear combination of other paths in the quiver (hence the loop at $\la$ can be deleted from the quiver).  

We now suppose that $|{\rm DAdd}_{1}(\la)|>1$. Therefore 
  $|{\rm DRem}_0(\la)|>2$ and we let 
$P^1,P^2,\dots, P^r$ denote these removable Dyck paths of height zero. 
For each $1\leq i\leq   j \leq l\leq r$ we have that 
have that 
\begin{align*} 
\mathbb  D_{{\la+\color{orange!70!brown}Q_{i,j}}}^{\la} \mathbb  D_{\la}^{{\la+\color{orange!70!brown}Q_{i,j}}}
&=
\mathbb  L^\la_\la (- {\color{magenta}P^i}) 
+
\mathbb  L^\la_\la (- {\color{yellow!60!orange}P^j}) 
\\
\mathbb  D_{{\la+\color{violet}Q_{i,k}}}^{\la} \mathbb  D_{\la}^{{\la+\color{violet}Q_{i,k}}}
&=
\mathbb  L^\la_\la (- {\color{magenta}P^i}) 
+
\mathbb  L^\la_\la (- {\color{cyan}P^k}) 
\\
\mathbb  D_{{\la+\color{darkgreen}Q_{j,k}}}^{\la} \mathbb  D_{\la}^{{\la+\color{darkgreen}Q_{j,k}}}
&=
\mathbb  L^\la_\la (- {\color{yellow!60!orange}P^j}) 
+
\mathbb  L^\la_\la (- {\color{cyan}P^k}) .
\end{align*}
 We can invert this this system of equations using \cref{kjsahgljkdsfhgjlkdfshgjfdshgljkdfshg} and hence rewrite this loop as a linear combination of other paths in the quiver (hence this loop can be deleted from the quiver).  

For  the remainder of the proof, we let 
$  Q_{m,m}$ denote the  quiver with    vertex set 
$\{\mathbbm 1_\la \mid \la \in   \mathscr{P}_{m,m} \}$  and 
  arrows   
  \begin{itemize}
\item  $\mathbb D^\la_\mu: \la   \to  \mu$ 
and 
   $\mathbb D^\mu_\la:\mu \to \la$ for every $\la=\mu - P$ with $P\in {\rm DRem}(\mu)$; 
\item for $m=n$ we have additional    ``loops" of degree 2, 
$\mathbb L^\la_\la :\la\to \la$ for every $\la\in \mathscr{R}_{m,m} $.
\end{itemize}
The algebra $K^m_m$ is the quotient of the path algebra $\Bbbk Q_{m,m}$ by the  relations in \cref{L=lincombo,generatortheorem2}.
We will  detail bases of   $\Bbbk Q_{m,m}$-modules and determine the linear dependencies
arising from the relations in \cref{L=lincombo,generatortheorem2} and hence determine the 
generators needed for the ${\rm Ext}$-quiver of $H^m_n$.

We first suppose that  $\la=(m^m)$.  The degree 2 subspace  of 
$\mathbbm 1_{(m^m)}	\Bbbk Q_{m,m}	\mathbbm 1_{(m^m)}	$ 
is $m$-dimensional with basis 
\[ \{\mathbb D^{(m^m)}_{(m^m)-P^j}\mathbb D_{(m^m)}^{(m^m)-P^j}
\mid \text{$P^j \in {\rm DRem}_j(m^m)$ for $0\leq j \leq m$}\} \cup \{ \mathbb L^\la_\la \}.\]   
     When we project onto the quotient 
     $\mathbbm 1_{(m^m)}	 H^m_n	\mathbbm 1_{(m^m)}	$ 
     the unique relation we apply is 
     $$\mathbb D^{(m^m)}_{(m^m)-P^j}\mathbb D_{(m^m)}^{(m^m)-P^j}=
     \mathbb L^{(m^m)}_{(m^m)}$$where the lefthand-side cannot be factorised as a product of elements in $ H^m_n$ and so we cannot delete the loop from the regular quiver.

We now suppose that  $\la=(m^a, (m-a)^{m-a}) $ for $1\leq a \leq m$, 
then we set ${\color{magenta}P^1}, {\color{yellow!60!orange}P^2}\in {\rm DRem}_0(\la)$.
We suppose that  the sets 
$\{ Q \mid Q \prec {\color{magenta}P^1}\}$ and 
$\{ Q \mid Q \prec {\color{yellow!60!orange}P^2}\}$ have size $p_1,p_2 \in \ZZ_{>0}$ respectively. 
In which case  the degree 2 subspace  of 
$\mathbbm 1_{(m^m)} \Bbbk Q_{m,m}\mathbbm 1_{(m^m)}	$ 
is $(p_1 + p_2+2)$-dimensional with  basis 
\[ \{\mathbb D^{\la }_{\la -Q}\mathbb D_{\la }^{\la -Q}
\mid \text{$Q \preceq P^1$ or $Q\preceq P^2$}\}
\cup \{\mathbb  D_{{\la+\color{orange!70!brown}Q_{1,2}}}^{\la} \mathbb  D_{\la}^{{\la+\color{orange!70!brown}Q_{1,2}}}
\} \cup \{ \mathbb L^\la_\la \}.\]   
     When we project onto the quotient 
     $\mathbbm 1_{(m^m)}	 H^m_n	\mathbbm 1_{(m^m)}	$ 
     the two relations we apply are 
$$
\mathbb  D_{{\la+\color{orange!70!brown}Q_{1,2}}}^{\la} \mathbb  D_{\la}^{{\la+\color{orange!70!brown}Q_{1,2}}}
=\mathbb D^{\la }_{\la -\color{magenta}P^1}\mathbb D_{\la }^{{\la -\color{magenta}P^1}}+
 \mathbb D^{\la }_{\la -\color{yellow!60!orange}P^2}\mathbb D_{\la }
^{{\la -\color{yellow!60!orange}P^2}}, \quad 
\qquad  \mathbb L^\la_\la=
 \mathbb D^{\la }_{\la -\color{yellow!60!orange}P^2}\mathbb D_{\la }
^{{\la -\color{yellow!60!orange}P^2}}, \quad 
$$thus projecting onto a $2$-dimensional space. 
We have that  $\la-{ \color{magenta}P^1}, {\la -\color{yellow!60!orange}P^2}\not \in \mathcal{R}_{m,m}$ and so  
 we cannot delete   the loop generator in this case, as above.\end{proof}

\begin{thm}
Let   $p = 2$.  
The $\rm Ext$-quiver  of $H^m_m $ has      vertex set 
$\{\mathbbm 1_\la \mid \la \in   \mathscr{R}_{m,m} \}$  and 
  arrows   
    $\mathbb D^\la_\mu: \la   \to  \mu$  
and 
   $\mathbb D^\mu_\la:\mu \to \la$ for every $\la=\mu - P$ with $P\in {\rm DRem}_{>0}(\mu)$ 
   together with all possible  loops 
      $\mathbb L^\la_\la:\la \to \la$ for every $\la \in   \mathscr{R}_{m,m}$. 
   The symmetric algebra  $H^m_m $ is given by the path algebra of this quiver modulo relations \eqref{rel1}, \eqref{rel2}, \eqref{rel3}, \eqref{rel4}, \eqref{adjacent}, \eqref{cubic}, and \eqref{loop-relation}.
\end{thm}

\begin{proof}
The case  $|{\rm DRem}_0(\la)|=1,2$ are identical to the proof for the $p\neq 2$ case.  
For $|{\rm DRem}_0(\la)|>2$ the argument for $p\neq 2$ cannot be passed through because the matrix of \cref{kjsahgljkdsfhgjlkdfshgjfdshgljkdfshg} is no longer invertible.  In fact, the linear dependencies listed between paths in  the proof of \cref{pnot2} are exhuastive (by inspection of the relations in the presentation) and  it is impossible to rewrite the loops as linear combinations of paths, and so we cannot delete any of these loops from the quiver. 
\end{proof}

We now wish to discuss the existence of self-extensions of simple modules (in other words, the existence of loops in the ${\rm Ext}$-quiver)
in the context of faithfulness of quasi-hereditary covers. 
This necessitates us recalling the main result of \cite{BDDHMS}.
\begin{thm}\label{citeme}
The extended arc algebras   $K^m_n $ are $(|m-n|-1)$-faithful covers of the 
Khovanov arc algebras $H^m_n $ for $m,n \in \NN$.  In other words,
$$
{\rm Ext}^i_{K^m_n}	(M,N) \cong
{\rm Ext}^i_{H^m_n}	(eM ,eN)  
$$for $M,N$ a pair of standard-filtered modules and $0\leq i < |m-n|$.  

\end{thm}

 \begin{cor}
The extended arc algebras   $K^m_n $ are $0$-faithful covers of the 
Khovanov arc algebras $H^m_n $ if and only if $m\neq n$.  
\end{cor}
 
 \begin{proof}One direction is immediate from \cref{citeme}.  To see that the $m=n$ case cannot have a 0-faithful cover, we observe that 
 $S_{m,m}(m^m) \cong D_{m,m}(m^m)\cong S_{m,m}(\varnothing)$. 
 \end{proof}

Putting together the results of this section and those from our previous work, we obtain the following.

\begin{thm}
The ${\rm Ext}$-quiver of the symmetric algebra $ H^m_n $ is loop-free if and only if 
$   K^m_n$ is an $i$-faithful quasi-hereditary cover for some $i\geq 0$. 
\end{thm}

 Finally, we recall that this behaviour was predicted in another context  and a slightly different language.  Firstly, the conjecture is somewhat folkloreish and so the best citation we have for its formulation is the very recent work of Geranios--Morotti--Kleshchev.  
In their  statement  the  condition for the existence of  loops is that $p=2$, 
 but this can reformulated in terms of faithfulness of quasi-hereditary covers using 
  \cite[Corollary 3.9.1]{MR2050037}
or  \cite{donkinhandbook}  
 to obtain the following.

\begin{conj}[The  Kleshchev--Martin conjecture {\cite[Introduction]{MR4391729}}]
The ${\rm Ext}$-quiver of the group algebra of the symmetric group $\Bbbk S_r$
 is loop-free if and only if 
the classical Schur algebra $S_\Bbbk (r,r)$  is an $i$-faithful quasi-hereditary cover for some $i\geq 0$. 
\end{conj}

\begin{rmk}
For some historical context, we remark that the first results on this conjecture were implicit 
in \cite{MR1308984} and later made explicit and extended in \cite{MR1728406}.  
While the conjecture is well-known and often referenced in the literature, 
progress in this direction  has been incredibly limited.
Arguably the first major step towards resolving this  conjecture was recently taken in \cite{MR4391729},
 where the authors  ``generically" verify this conjecture by proving that it holds for all 
RoCK blocks (which constitute ``most blocks"). 
\end{rmk}

The original Kleshchev--Martin context concerns truncation from a quasi-hereditary algebra
to a symmetric algebra 
using a highest weight idempotent ${\sf 1}_\omega$ for $\omega=(1^r,0,0,\dots)$.  
Our truncation is from a quasi-hereditary algebra to a symmetric algebra 
by a highest weight idempotent $\sum_{\rho\subseteq \la} e_\la$ where   
$\rho=(m,m-1,\dots,2,1)\in \mathscr{P}_{m,n}$.  
Such (``co-saturated") truncations appear in many different contexts in Lie theory, for example they were the subject of   conjectures of Khovanov  \cite{MR2078414} proven in \cite{MR2369489}.
All of the quasi-hereditary algebras discussed above are Morita equivalent to 
(singular) anti-spherical Hecke categories. 
We propose the following  vast generalisation of the above self-extension conjecture.

Let $(W,P)$ denote a parabolic Coxeter system with generators $s_i\in S_W$. 
Let $\mathcal{H}_{(W,P)}$ denote the category algebra of the Elias--Libedinsky--Williamson
 anti-spherical Hecke category (see \cite{MR4510171} for the definition of this category algebra and \cite{MR4437613,MR3555156} for the original definition of the anti-spherical  Hecke category).  Given a  reduced word $\w=s_{i_1}s_{i_2}\dots s_{i_k}$ for some  $w \in {^PW}$,  we have an idempotent 
${\sf 1}_\w={\sf 1}_{s_{i_1}}\otimes {\sf 1}_{s_{i_2}} \otimes \dots\otimes  {\sf 1}_{s_{i_k}}\in \mathcal{H}_{(W,P)}$.

\begin{conj}[Generalised Kleshchev--Martin conjecture]
Let   $x,w\in {^PW}$ and consider the subalgebra  
$$\textstyle
(\sum_{x<y\leq w}{\sf 1}_{\y})\mathcal{H}_{(W,P)}(\sum_{x<y\leq w}{\sf 1}_{\y})
\subseteq  (\sum_{ y\leq w}{\sf 1}_{\y})\mathcal{H}_{(W,P)}(\sum_{ y\leq w}{\sf 1}_{\y})
 $$where the latter is a quasi-hereditary cover of the former,  by construction. 
 If the subalgebra is a symmetric algebra, then 
  its ${\rm Ext}$-quiver is loop-free if and only if the  quasi-hereditary cover 
  is   $i$-faithful, for some $i\geq 0$.
\end{conj}

In order to truly generalise the classical Kleshchev--Martin conjecture, we should state the above conjecture for {\em singular} Hecke categories (and indeed we do believe this holds).  However these objects  are in relative infancy and even defining reduced  words in this context is a difficult problem and the subject of very recent work of   \cite{geordiethesis,MR4563864}.

        \bibliographystyle{amsalpha}   
\bibliography{master}

 \end{document}